\documentclass[a4paper,11pt,twoside]{article}
\usepackage[left=2cm,right=2cm,top=2cm,bottom=3cm]{geometry}

\usepackage[T1]{fontenc}
\usepackage[utf8]{inputenc} 
\usepackage[english]{babel} 

\usepackage[super]{nth}  

\usepackage{comment}
\usepackage{xcolor}

\usepackage[usestackEOL]{stackengine}
\usepackage{amsmath} 
\usepackage{amssymb} 
\usepackage{amsthm}
\usepackage{amsfonts}
\usepackage{mathtools}
\usepackage{mathrsfs}
\usepackage{braket} 
\usepackage{graphicx}
\usepackage{graphics}
\usepackage{stmaryrd}
\usepackage{textcomp}
\usepackage[export]{adjustbox}
\usepackage{cancel}
\usepackage{mathabx,epsfig}
\usepackage{cases}
\usepackage[overload]{empheq}
\usepackage{esvect}
\usepackage{nicefrac}
\usepackage{mathtools}
\mathtoolsset{showonlyrefs=true,showmanualtags=true}

\usepackage[all]{xy}

\newcommand{\dsp}{\hspace{0.2mm}}
\newcommand{\msp}{\hspace{-0.3mm}}

\newcommand{\numberset}{\mathbb}
\newcommand{\N}{\numberset{N}}

\newcommand{\R}{\numberset{R}}
\newcommand{\T}{\numberset{T}}

\newcommand{\F}{\mathcal{F}}  
\newcommand{\dist}{\text{dist}} 
\newcommand{\id}{\text{I}}

\newcommand{\tr}{\text{tr}} 
\newcommand{\mix}{\text{mix}}  
\newcommand{\canc}{\text{canc}}  

\DeclareMathOperator{\p}{\numberset{P}} 
\DeclareMathOperator{\meas}{meas}
\DeclareMathOperator{\sgn}{\text{sgn}} 
\DeclareMathOperator{\ueta}{u^{\eta,\epsilon}} 
\DeclareMathOperator{\chieta}{\chi^{\eta,\epsilon}} 
\DeclareMathOperator{\peta}{p^{\eta,\epsilon}} 
\DeclareMathOperator{\qeta}{q^{\eta,\epsilon}} 

\theoremstyle{definition}
\newtheorem{definition}{Definition}[section]
\newtheorem{theorem}[definition]{Theorem}
\newtheorem{lemma}[definition]{Lemma}
\newtheorem{corollary}[definition]{Corollary}
\newtheorem{proposition}[definition]{Proposition}

\newtheorem{remark}[definition]{Remark}

\newtheorem*{notation*}{Notation}
\newtheorem*{definition*}{Definition}
\newtheorem*{theorem*}{Theorem}
\newtheorem*{lemma*}{Lemma}
\newtheorem*{corollary*}{Corollary}
\newtheorem*{proposition*}{Proposition}
\newtheorem*{fact*}{Fact}
\newtheorem*{example*}{Example}
\newtheorem*{claim*}{Claim}
\newtheorem*{remark*}{Remark}
\newtheorem*{conjecture*}{Conjecture}

\numberwithin{equation}{section}

\allowdisplaybreaks[4]

\title{Porous media equations with nonlinear gradient noise and Dirichlet boundary conditions}
\author{Andrea Clini \thanks{Mathematical Institute, University of Oxford, Oxford OX2 6GG, UK (andrea.clini@maths.ox.ac.uk)}}
\date{\today} 

\begin{document}

\maketitle

\begin{abstract}
    We establish pathwise existence of solutions for stochastic porous media and fast diffusion equations of type \eqref{formula/stochastic porous media equation}, in the full regime $m\in(0,\infty)$ and for any initial data $u_0\in L^2(Q)$.
    Moreover, if the initial data is positive, solutions are pathwise unique. 
    In turn, the solution map to \eqref{formula/stochastic porous media equation} is a continuous function of the driving noise and it generates an associated random dynamical system.
    Finally, in the regime $m\in\{1\}\cup(2,\infty)$, all the aforementioned results also hold for signed initial data.
\end{abstract} 




\section{Introduction}\label{introduction}

In this paper we consider stochastic porous media and fast diffusion equations with nonlinear, conservative noise of the form
\begin{align}[left ={\empheqlbrace}]\label{formula/stochastic porous media equation}
    \begin{aligned}
        &\partial_tu =\Delta (|u|^{m-1}u)+\nabla\cdot(A(x,u)\circ dz_t) &\text{in } Q\times(0,\infty),
        \\
        &u=0 &\text{on } \partial Q\times (0,\infty),
        \\
        &u=u_0 &\text{on }  Q\times \{0\},
    \end{aligned}
\end{align}
for a diffusion exponent $m\in(0,\infty)$, initial data $u_0\in L^2(Q)$, and an $n$-dimensional, $\alpha$-H\"older continuous geometric rough path $z$, which in particular covers the case where $z$ is an $n$-dimensional Brownian motion.
The domain $Q$ is a smooth bounded domain in $\R^d$.
The matrix valued nonlinearity  
\begin{equation}
    A(x,\xi):Q\times\R\to\R^{d\times n} \nonumber
\end{equation}
is assumed to be regular, with regularity dictated by the regularity of the rough path $z$.

We establish path-by-path existence and uniqueness of \eqref{formula/stochastic porous media equation} for positive initial data $u_0\in L^2_+(Q)$ in the full regime $m\in (0,\infty)$ (cf. Theorems \ref{theorem/uniqueness of pathwise kinetic solutions for nonnegative initial data} and \ref{theorem/existence of pathwise kinetic solutions for signed data}), in terms of the central notion of \emph{pathwise kinetic solution} (cf. Definition \ref{definition/pathwise kinetic solution}).
Moreover, we show the solution map to \eqref{formula/stochastic porous media equation} generates a random dynamical system associated to the equation (cf. Theorem \ref{theorem/random dynamical system for nonnegative initial data}) and is a continuous function of the noise (cf. Theorem \ref{theorem/continuous dependence on the noise for nonnegative initial data}).
This has consequences on the support properties of the solution to \eqref{formula/stochastic porous media equation} around the solution of its deterministic version \eqref{formula/deterministic porous media equation 0} (cf. Remark \ref{remark/support properties of the solution}).
Finally, the existence of solutions to \eqref{formula/stochastic porous media equation} is in fact proved for any signed data $u_0\in L^2(Q)$ and in the full regime $m\in(0,\infty)$ (cf. Theorem \ref{theorem/existence of pathwise kinetic solutions for signed data}). In particular, when $m=1$ or $m\in(2,\infty)$, all the aforementioned results extend to general signed data (cf. Theorem \ref{theorem/extension of the results to signed initial data}).

Stochastic porous media equations of this type arise, for example, as a continuum limit of mean-field stochastic differential equations with common noise \cite{Coghi-gess,kurtz-xiong-particle-representation}, with notable relation to the theory of mean field games \cite{lasry-lions-mean-field-games,lasry-lions-jeux-a-champ-moyenne-la-cas-stationnaire}; or in the graph formulation of the stochastic mean curvature and curve shortening flow \cite{sarhir-renesse-ergodicity-of-stochastic-curve-shortening,dirr-novaga-luckaus-a-stochastic-selection,souganidis-yip-uniqueness-of-motion-by-mean-curvature}.
Moreover, equation \eqref{formula/stochastic porous media equation} can be regarded as an approximate model for the fluctuating hydrodynamics of the zero-range particle process around its hydrodynamic limit \cite{dirr-stamakis-zimmer,fehrman-gess-large-deviations-for-conservative,ferrari-presutti-vares--nonequilibrium-fluctuations-for-a-zero}, as an approximation to the Dean-Kawasaki equation arising in fluid dynamics \cite{marconi-tarazona-dynamical-density-functional,dean-langevin-equation,kawasaki-microscopic-analyses,Cornalba-A-Regularized-Dean-Kawasaki-Model:-Derivation-and-Analysis,fehrman-gess-Well-posedness-of-the-Dean-Kawasaki-and-the-nonlinear-Dawson-Watanabe-equation-with-correlated-noise}, and as a model for thin films of Newtonian fluids with negligible surface tension \cite{grun-thin-film-flow}.
We refer to \cite[Section 1.2]{fehrman-gess-Well-posedness-of-the-Dean-Kawasaki-and-the-nonlinear-Dawson-Watanabe-equation-with-correlated-noise} and \cite[Section 1.1]{fehrman-gess-Wellposedness-of-Nonlinear-Diffusion-Equations-with-Nonlinear-Conservative-Noise} for more details on these applications and the references therein.

Our approach to \eqref{formula/stochastic porous media equation} is crucially based on the work \cite{fehrman-gess-Wellposedness-of-Nonlinear-Diffusion-Equations-with-Nonlinear-Conservative-Noise}, where Fehrman and Gess proved analogous results for this equation posed on the $d$-dimensional torus with periodic boundary conditions.
In turn, this was motivated by the works of Lions, Perthame and Souganidis \cite{lions-perthame-souganidis-Scalar-conservation-laws-with-rough-(stochastic)-fluxes,lions-perthame-souganidis-Scalar-conservation-laws-with-rough-(stochastic)-fluxes-the-spatially-dependent-case}, and Gess and Souganidis \cite{gess-souganidis-Scalar-conservation-laws-with-multiple-rough-fluxes,gess-souganidis-Long-Time-Behavior-Invariant-Measures-and-Regularizing-Effects-for-Stochastic-Scalar-Conservation-Laws,gess-souganidis-stochastic-nonisotropic} on stochastic conservation laws and simpler versions of \eqref{formula/stochastic porous media equation}.
The method is essentially based on passing to the equation's kinetic formulation, introduced by Chen and Perthame \cite{chen-perthame-Well-posedness-for-non-isotropic-degenerate-parabolic-hyperbolic-equations,perthame-kinetic-formulation-of-conservation-laws}, for which the noise enters as a linear transport, and then on analytic techniques and rough path analysis.

The main difficulty with respect to the periodic case \cite{fehrman-gess-Wellposedness-of-Nonlinear-Diffusion-Equations-with-Nonlinear-Conservative-Noise} is the handling of the Dirichlet boundary conditions.
The imposition of boundary conditions is a challenging problem.
In the stochastic case, even in the simple linear one-dimensional case $\partial u=\partial_{xx}u+(\partial_xu)\circ \dot{z}$, it is nontrivial \cite{krylov_brownian_trajectory} and in fact not always possible to enforce them in a classical sense.
The conditions can be recast in a weak sense, but even in the deterministic case, this is notoriously difficult for nonlinear equations.

In our case, depending on the diffusion exponent $m\in(0,\infty)$, we obtain $H^1$-regularity for the power $|u|^{m-1}u$ and impose the zero boundary conditions by requiring the trace of $|u|^{m-1}u$ to vanish.
The behaviour of the noise at the boundary is controlled in terms of the Dirichlet conditions imposed on the solution.
This requires a sharp ad-hoc treatment depending on the particular exponent $m\in(0,\infty)$ and the observation of crucial cancellations among the boundary error terms.

The well-posedness of \eqref{formula/stochastic porous media equation} has been an open question for long time, even in the \emph{probabilistic} (i.e. non patwhise) setting and even in the case $z_t$ is given by a Brownian motion $B_t$.
Generalized stochastic porous media equations of the form 
\begin{equation}
    d\,u=\Delta\phi(u) \,dt + \sigma(x,u)\,dB_t \nonumber
\end{equation}
have attracted considerable interest and their well-posedness has been obtained for several classes of nonlinearities $\phi$, noise coefficients $\sigma(x,u)$ and boundary conditions.
We refer to the monographs \cite{barbu-daprato-rockner-stochastic-porous-media,liu-rockner-Stochastic-Partial-Differential-Equations:-An-Introduction}, and to \cite{fehrman-gess-Path-by-path-well-posedness-of-nonlinear-diffusion-equations-with-multiplicative-noise,gess-hofmanova-weel-posedness-and-regularity-for-quasilinear,Debussche-Hofmanova-Vovelle-degenerate-parabolic-stochastic,barbu-rockner-nonlinear-fokker-planck-equations-driven,barbu-rockner-an-operatorial-approach,bauzet-vallet-wittbold-a-degenerate-parabolic-hyperbolic} for recent contributions.
To some extent, the case of a \emph{linear} gradient noise, that is $A(x,u)=h(x)u$ in \eqref{formula/stochastic porous media equation}, could be treated with similar methods (cf. \cite{dareiotis-gess-supremum-estimates,munteanu-rockner-the-total-variation-flow,Tolle-estimates-for-nonlinear}).
However, the nonlinear structure of the gradient noise in \eqref{formula/stochastic porous media equation} requires entirely different techniques.

The aforementioned works \cite{lions-perthame-souganidis-Scalar-conservation-laws-with-rough-(stochastic)-fluxes,lions-perthame-souganidis-Scalar-conservation-laws-with-rough-(stochastic)-fluxes-the-spatially-dependent-case,gess-souganidis-Scalar-conservation-laws-with-multiple-rough-fluxes,gess-souganidis-Long-Time-Behavior-Invariant-Measures-and-Regularizing-Effects-for-Stochastic-Scalar-Conservation-Laws,gess-souganidis-stochastic-nonisotropic} started developing a kinetic approach to tackle scalar conservation laws and simplified versions of \eqref{formula/stochastic porous media equation}.
The work of Fehrman and Gess \cite{fehrman-gess-Wellposedness-of-Nonlinear-Diffusion-Equations-with-Nonlinear-Conservative-Noise} was the first to prove similar results for \eqref{formula/stochastic porous media equation} in the pathwise context, at the cost of high regularity assumptions on the noise coefficient needed to overcome the roughness of the signal (cf. assumption \eqref{formula/assumption on A smooth} below).
Dareiotis and Gess \cite{Dareioris-gess-Nonlinear-diffusion-equations-with-nonlinear-gradient-noise} then managed to lower the regularity assumptions up to $A\in C^3_b(\T^d\times\R)$ and proved well-posedness of \eqref{formula/stochastic porous media equation} with periodic boundary conditions, in the probabilistic sense.
Furthermore, Fehrman and Gess \cite{fehrman-gess-Well-posedness-of-the-Dean-Kawasaki-and-the-nonlinear-Dawson-Watanabe-equation-with-correlated-noise} proved probabilistic well-posedness of \eqref{formula/stochastic porous media equation} with periodic boundary conditions, when $z_t$ is a Brownian motion and, for example, when $A(x,u)=f(x)\sqrt{u}$ with $f\in C_b^2(\T^d)$.

Despite requiring higher regularity, the interest in pathwise results is nonetheless retained and twofold.
First, it is well-known that solutions to stochastic differential equations do not depend continuously on the driving noise (see for instance \cite{lyons-on-the-nonexistence-of-path-integrals}), and even more so if the noise coefficient is a nonlinear function of the solution itself.
However, the continuity of the solution can be recovered by means of a finer rough path topology \cite{Lyons-differential-equations-driven-by-rough-signals}.
In the same fashion, Theorem \ref{theorem/continuous dependence on the noise for nonnegative initial data} below establishes the continuous dependence in the noise of the solution map to \eqref{formula/stochastic porous media equation}.

Secondly, the pathwise nature of the existence and uniqueness Theorems \ref{theorem/uniqueness of pathwise kinetic solutions for nonnegative initial data} and \ref{theorem/existence of pathwise kinetic solutions for signed data} immediately implies the existence of a random dynamical system associated to \eqref{formula/stochastic porous media equation}, stated in Theorem \ref{theorem/random dynamical system for nonnegative initial data}.
This is a notoriously difficult problem (see e.g. \cite{Flandoli-regularity-theory-and-stochastic-flows-for-parabolic-spdes,mohammed-the-stable-manifold-theorem,gess-Random-Attractors-for-Degenerate-Stochastic-Partial-Differential-Equations}).
Even in the linear case $m=1$, the existence of a random dynamical system for a nonlinear SPDE with nonlinear $x$-dependent noise could not be proved before \cite{fehrman-gess-Wellposedness-of-Nonlinear-Diffusion-Equations-with-Nonlinear-Conservative-Noise}.
Indeed, all the aforementioned works on \eqref{formula/stochastic porous media equation} in the probabilistic setting could not obtain these two consequences.

\subsection{Structure of the paper}
The material is organized as follows.
We end Section \ref{introduction} with an overview of the methods and the arguments employed in this work.
In Section \ref{section/preliminaries and main results} we introduce our hypotheses and notations, then we present the main results of the paper.
In Section \ref{section/definition and motivation of pathwise kinetic solutions} we bring forward the kinetic formulation of the equation; after analyzing the associated system of characteristics, we motivate and present the definition of pathwise kinetic solution.
In Section \ref{section/uniqueness of pathwise kinetic solutions} we prove the uniqueness of solutions to \eqref{formula/stochastic porous media equation}.
Section \ref{section/existence of pathwise kinetic solutions} is devoted to the proof of existence of solutions to \eqref{formula/stochastic porous media equation} and of their continuous dependence on the driving noise.
In Section \ref{section/rough path estimates} we present some stability results from the theory of rough paths and we gather some estimates needed throughout the paper.

\subsection{Overview of the methods}\label{section/overview of the methods}
The methods of this paper build upon the kinetic approach put forward in \cite{lions-perthame-souganidis-Scalar-conservation-laws-with-rough-(stochastic)-fluxes,lions-perthame-souganidis-Scalar-conservation-laws-with-rough-(stochastic)-fluxes-the-spatially-dependent-case,gess-souganidis-Scalar-conservation-laws-with-multiple-rough-fluxes,gess-souganidis-Long-Time-Behavior-Invariant-Measures-and-Regularizing-Effects-for-Stochastic-Scalar-Conservation-Laws,gess-souganidis-stochastic-nonisotropic,fehrman-gess-Wellposedness-of-Nonlinear-Diffusion-Equations-with-Nonlinear-Conservative-Noise,fehrman-gess-Path-by-path-well-posedness-of-nonlinear-diffusion-equations-with-multiplicative-noise} for stochastic conservation laws and variants of the stochastic porous media equation.
The aim of Section \ref{section/definition and motivation of pathwise kinetic solutions} is to explain the pathwise and kinetic approach to the \eqref{formula/stochastic porous media equation}, and to derive and motivate the central notion of \emph{pathwise kinetic solution}.
First, we pass to the kinetic formulation of the PDE \cite{chen-perthame-Well-posedness-for-non-isotropic-degenerate-parabolic-hyperbolic-equations}.
This is an equation in $d+2$ variables: the time and space variables $t$ and $x$, and an additional \emph{velocity} variable $\xi$, which corresponds to the magnitude of the solution.
The interest in such formulation of \eqref{formula/stochastic porous media equation} is that now the noise enters the equation as a linear transport.
The transport is well-defined for rough driving signals, when the underlying system is interpreted as a rough differential equation.

Following the strategy used in \cite{lions-perthame-souganidis-Scalar-conservation-laws-with-rough-(stochastic)-fluxes,lions-perthame-souganidis-Scalar-conservation-laws-with-rough-(stochastic)-fluxes-the-spatially-dependent-case,gess-souganidis-Scalar-conservation-laws-with-multiple-rough-fluxes,gess-souganidis-Long-Time-Behavior-Invariant-Measures-and-Regularizing-Effects-for-Stochastic-Scalar-Conservation-Laws,gess-souganidis-stochastic-nonisotropic,fehrman-gess-Wellposedness-of-Nonlinear-Diffusion-Equations-with-Nonlinear-Conservative-Noise,fehrman-gess-Path-by-path-well-posedness-of-nonlinear-diffusion-equations-with-multiplicative-noise}, and previously put forward in the theory of stochastic viscosity solutions \cite{Lions-souganidis-Fully-nonlinear-stochastic-partial-differential-equations:-non-smooth-equations-and-applications,Lions-souganidis-Fully-nonlinear-stochastic-partial-differential-equations,lions-souganidis-Fully-nonlinear-stochastic-pde-with-semilinear-stochastic-dependence,Lions-souganidis-uniqueness-of-solutions-of-Fully-nonlinear-stochastic-partial-differential-equations,Lions-souganidis-viscosity-solutions-of-of-Fully-nonlinear-stochastic-partial-differential-equations}, we test the kinetic equation against a restricted class of test functions only, precisely test functions transported by an underlying conservative system of stochastic characteristics.
When testing against these functions, the terms involving the noise automatically cancel out.
Such test functions are simply obtained with the method of characteristics: namely by flowing an arbitrary initial data back in time along the characteristic curves of the rough differential equations prescribing the transport.
Informally, this is of course the same as flowing the kinetic solution forward in time along the characteristics.
This restricted class of test functions in nonetheless large enough to provide a comprehensive characterization of the solutions to \eqref{formula/stochastic porous media equation}, or better to its kinetic version, which is sufficient to prove its well-posedness.
The precise notion of pathwise kinetic solution is given in Definition \ref{definition/pathwise kinetic solution}.

In comparison to \cite{lions-perthame-souganidis-Scalar-conservation-laws-with-rough-(stochastic)-fluxes} and \cite{fehrman-gess-Path-by-path-well-posedness-of-nonlinear-diffusion-equations-with-multiplicative-noise}, owing to the $x$-dependent gradient structure of the noise, the characteristics' equations cannot be solved explicitly.
Therefore, the solutions need to be controlled with the rough path estimates from Section \ref{section/rough path estimates}.
Moreover, as time passes, the characteristics also move in space (cf. \cite{fehrman-gess-Wellposedness-of-Nonlinear-Diffusion-Equations-with-Nonlinear-Conservative-Noise,lions-perthame-souganidis-Scalar-conservation-laws-with-rough-(stochastic)-fluxes-the-spatially-dependent-case}).
With respect to \cite{fehrman-gess-Wellposedness-of-Nonlinear-Diffusion-Equations-with-Nonlinear-Conservative-Noise,lions-perthame-souganidis-Scalar-conservation-laws-with-rough-(stochastic)-fluxes-the-spatially-dependent-case}, the introduction of Dirichlet boundary conditions further complicates the picture: as time evolves the space characteristics might indeed escape the domain $Q$, and the resulting test functions flowed along these characteristics would not be compactly supported within $Q$ if too much time has passed.
To overcome this difficulty, we make further assumptions on the noise coefficient $A(x,u)$ defining the system of rough characteristics \eqref{formula/forward stochastic characteristics} so as to directly prevent the characteristics from escaping the domain $Q$ (cf. \eqref{formula/assumption on A xi derivative zero on the boundary} and \eqref{formula/smooth characteristics never leave the domain}).
These assumptions simplify the handling of the boundary conditions and are fundamental in our proofs. 
They are justified since we have in mind equation \eqref{formula/stochastic porous media equation} as an approximate model for a noise term $\nabla\cdot(\sigma(x,u)\eta)$, with $\eta$ a space-time white noise.
Indeed, it is always possible to find a cylindrical expansion of $\eta$, which we then truncate at some order, such that the aforementioned conditions are satisfied.
See Remark \ref{remark/assumptions compatible with white noise approximation} for details.

Section \ref{section/uniqueness of pathwise kinetic solutions} is devoted to the proof of uniqueness.
The formal proof of uniqueness follows the same outline presented in \cite{chen-perthame-Well-posedness-for-non-isotropic-degenerate-parabolic-hyperbolic-equations} in the deterministic case.
However, to justify the formal computation, care must be taken to avoid the product of $\delta$-distributions.
This is achieved with a regularization in the space and velocity variables. 
Moreover, to cancel the noise from the equation, we are allowed to use transported test functions only.
Additional error terms arise due to the transport of test functions along characteristics, which are handled using a time-splitting argument that relies crucially on the conservative structure of the equation.
In this setting, the interaction between the $x$-dependent characteristics and the nonlinear diffusion term further complicates the arguments: the error terms generated by the regularization procedure need to be controlled with sharp estimates (cf. Proposition \ref{proposition/singular moments for defect measures d in (0,1)} and \ref{proposition/singular moments for defect measures d=1}) for singular moments of the parabolic defect measure (cf. Definition \ref{definition/pathwise kinetic solution}), especially in the case of small diffusion exponents $m\in(0,1)\cup(1,2]$.

The imposition of Dirichlet boundary conditions makes the analysis even more difficult: to keep everything compactly supported within the domain $Q$ it is necessary to introduce a cutoff function.
This, combined with the displacement of the space characteristics, generates new boundary layer error terms which are controlled with the zero-trace conditions imposed on the solution $u$.
Concretely, the decay at the boundary of $u$ is quantified by means of the $H^1$-regularity of $|u|^{m-1}u$ prescribed by the nonlinear diffusion.
In turn, this forces us to choose the cutoff function $\phi=\phi(m)$ as a function of the particular exponent $m\in(0,\infty)$.
Finally, besides the aforesaid estimates, the handling of both the interior and boundary error terms relies on some fundamental cancellations among them.
Indeed, the gradient structure of the noise implies that the characteristics preserve the underlying Lebesgue measure, and this is crucial both for the above mentioned time-splitting argument and to observe these cancellations.

Finally, in Section \ref{section/existence of pathwise kinetic solutions} we prove the existence of pathwise kinetic solutions and their continuous dependence on the noise in the rough path topology.
This is obtained by proving stable estimates for the solutions of smoothed PDEs approximating \eqref{formula/stochastic porous media equation}, and then using weak convergence and compactness arguments to prove these solutions converge to a solution of \eqref{formula/stochastic porous media equation}.
The aforementioned estimates hold for the limiting solution as well.
Furthermore, they hold uniformly for solutions of \eqref{formula/stochastic porous media equation} corresponding to rough signals $\{z_t^n\}_{n\in\N}$ close to each other in rough path topology.
Then, repeating the same weak convergence and compactness arguments used to show existence, and combining them with the uniqueness of solutions, we prove that the solution map to \eqref{formula/stochastic porous media equation} depends continuously on the noise.

Concretely, we first prove stable estimates for singular moments of the parabolic defect measure of the approximating PDEs \eqref{formula/regularized porous media equation 1} totally akin to those used in the uniqueness proof (cf. Proposition \ref{proposition/stable estimate for L^2 norm and defect measures of smoothed solutions} and Proposition \ref{proposition/singular moments for defect measures d in (0,1)}).
Then we exploit these estimates to show that the corresponding kinetic solutions are uniformly bounded in suitable fractional Sobolev spaces $W_{x,\xi}^{\ell,1}$ for the space and velocity variables (cf. Proposition \ref{proposition/stable estimate for transported kinetic functions}).
Moreover, they have enough regularity in time (cf. Proposition \ref{proposition/stable estimate for time derivative of transported kinetic functions}) to invoke the Aubin-Lions-Simons Lemma \cite{simon-Compact-sets-in-the-spaceLp} and obtain strong convergence.

In this regard, it is worth mentioning that the weak convergence arguments developed in \cite{gess-souganidis-Scalar-conservation-laws-with-multiple-rough-fluxes} for scalar conservation laws with $x$-dependent noise do not apply in the parabolic case \eqref{formula/stochastic porous media equation}.
Indeed, the corresponding class of pathwise entropy solutions is not closed under weak convergence, since the second order structure of the equation cannot ensure the weak limiting object actually is a solution.
Therefore, in \cite{gess-souganidis-stochastic-nonisotropic} the same authors introduced a strong convergence method based on uniform $BV$-estimates to tackle the parabolic case.
However, these arguments are probably restricted to the $x$-independent case as a uniform $BV$-estimate for solutions to \eqref{formula/stochastic porous media equation} does not seem available.



\section{Preliminaries and main results}\label{section/preliminaries and main results}

\subsection{Hypotheses and notations}\label{section/hypotheses and notations}

The domain $Q$ is a smooth bounded domain in $\R^d$ with $d\geq1$.
The diffusion exponent is $m\in(0,\infty)$, and for the signed power we shall use the shorthand 
\begin{equation}
    u^{[m]}:=|u|^{m-1}u. \nonumber
\end{equation}
We shall denote by $L^2_+(Q)$ the closed subspace of $L^2$-integrable functions which are a.e. nonnegative.
The noise is a geometric rough path: for $n\geq1$ and a H\"older exponent $\alpha\in(0,1)$, for each $T>0$
\begin{equation} \label{formula/z is a geometric rough path}
    z_t=(z_t^1,\dots,z_t^n)\in C^{0,\alpha}\left([0,T];G^{\left\lfloor\frac{1}{\alpha}\right\rfloor}(\R^n)\right),
\end{equation}
where $C^{0,\alpha}\big([0,T];G^{\left\lfloor\nicefrac{1}{\alpha}\right\rfloor}(\R^n)\big)$ is the space of $n$-dimensional $\alpha$-H\"older geometric rough paths on $[0,T]$.
We denote by $d_\alpha$ the $\alpha$-H\"older metric defined on this space. See Section \ref{section/rough path estimates} for references on rough path theory.

As regards the noise coefficient $A(x,\xi)$, we assume it is smooth with bounded derivatives.
Precisely, for some $\gamma>\frac{1}{\alpha}$ we assume
\begin{equation}\label{formula/assumption on A smooth}
   D_xA(x,\xi)\in C_b^{\gamma+2}(\R^d\times\R;\R^{d\times d\times n})\quad\text{and}\quad \partial_{\xi}A(x,\xi)\in C_b^{\gamma+2}(\R^d\times\R;\R^{d\times n}).
\end{equation}
This regularity is necessary in order to obtain the rough path estimates of Proposition \ref{proposition/stability results for RDEs}.
In particular, as the regularity of the noise decreases, more regularity is required for the coefficients.

We also need to impose conditions on $A$ so as to control the direction of the characteristics.
To start with, we assume the nonlinearity $A(x,\xi)$ satisfies 
\begin{equation}\label{formula/assumption on A divergence 0 in 0}
    \nabla_x\cdot A(x,0)=\sum_{i=1}^d\partial_{x_i}A_{i,\cdot}(x,0)=0\quad\text{ for each } x\in \R^d.
\end{equation}
This assumption guarantees that the underlying stochastic characteristics preserve the sign of the velocity variable.
Even in the case of smooth driving signal, this condition is necessary to ensure that the evolution of \eqref{formula/stochastic porous media equation} does not increase the mass of the initial condition.

Next, we impose two conditions to govern the behaviour of the space characteristics near the boundary.
These assumptions are \emph{crucial} in our arguments to impose Dirichlet boundary conditions and are justified by Remark \ref{remark/assumptions compatible with white noise approximation} below.
Namely, we impose that $A(x,\xi)$ satisfies
\begin{equation}\label{formula/assumption on A xi derivative zero on the boundary}
    \partial_{\xi} A(x,\xi)=D_x\partial_{\xi} A(x,\xi)=0 \quad\text{on } \partial Q\times\R.
\end{equation}
As already mentioned in Section \ref{section/overview of the methods}, the assumption $\partial_{\xi} A_{|\partial Q\times\R}=0$ ensures that, as time evolves, the space characteristics never leave the domain $Q$.
This condition is necessary to guarantee that the transport of a test function along the characteristics stays compactly supported in $Q$, and thus it is an admissible test function in the kinetic formulation of \eqref{formula/stochastic porous media equation}.
The further hypothesis $D_x\partial_{\xi} A_{|\partial Q\times\R}=0$ is slightly more technical and is needed to effectively exploit the condition $\partial_{\xi} A_{|\partial Q\times\R}=0$ and further quantify the informal fact that the space characteristics move slower as they start closer to the boundary.
Intuitively, it means that the strength of the transport term $\partial_{\xi}A(x,u) \nabla u\circ dz_t$ featuring in \eqref{formula/stochastic porous media equation} decreases more than linearly in terms of the distance $\dist(x,\partial Q)$.

\begin{remark}\label{remark/assumptions compatible with white noise approximation}
We stress that our main interest in \eqref{formula/stochastic porous media equation} is to consider it as a space correlated approximation to the, possibly ill-posed, stochastic porous media equation 
\begin{equation}
    \partial_tu=\Delta u^{[m]}+\nabla\cdot(\sigma(x,u)\circ \eta), \nonumber
\end{equation}
where $\eta$ is a $d$-dimensional space-time white noise and $\sigma(x,\xi)$ is a possibly nonsmooth nonlinearity with $\sigma(x,0)=0$ for every $x\in Q$, for example $\sigma(x,\xi)=\sqrt{\xi}$.
We observe that the assumptions \eqref{formula/assumption on A divergence 0 in 0}-\eqref{formula/assumption on A xi derivative zero on the boundary} are perfectly compatible with this strategy.
Indeed, a standard construction of the space-time white noise on $Q\times [0,\infty)$ is $\eta=\sum_{i=1}^\infty\rho_i(x)dB_t^i$, where $\{B_t^i\}_{i\in\N}$ are independent $d$-dimensional Brownian motions and $\{\rho_i(x)\}_{i\in\N}$ is an orthonormal basis of $L^2(Q)$.
Now consider a basis such that $\rho_i\in C^\infty(\Bar{Q})$ and such that $\rho_i=D_x\rho_i=0$ on $\partial Q$ for each $i\in\N$ (for example, consider a spectral basis of $L^2(Q)$, contained in $H^2_0(Q)$, for the symmetric compact operator $\Delta^{-2}$).
For each $m\in\N$, we can take $n=md$ and set
\begin{equation} \nonumber
    A(x,\xi)=\sigma_m(x,\xi)\Big[\rho_1(x)\id_d\mid\rho_2(x)\id_d\mid\cdots\mid\rho_m(x)\id_d\Big],\quad
    z_t=(B_t^1,\dots,B_t^m),
\end{equation}
where $\id_d$ denote the identity matrix of dimension $d$, and $\sigma_m$ is a smooth approximation to $\sigma$ with $\sigma_m(x,0)=0$ for every $x\in Q$.
Then $A(x,u)\circ dz_t=\sigma_m(x,u)\sum_{i=1}^m\rho_i(x)\circ dB^i_t$ is indeed converging to $\sigma(x,u)\circ\eta$ as $m\to\infty$, and elementary computations show that $A(x,\xi)$ satisfies the assumptions \eqref{formula/assumption on A divergence 0 in 0}-\eqref{formula/assumption on A xi derivative zero on the boundary}.
\end{remark}

\subsection{Main results}\label{section/main results}

We now present the main results of the paper.
The precise notion of \emph{pathwise kinetic solution} is given in Definition \ref{definition/pathwise kinetic solution} below.
Our first result, proved in Section \ref{section/uniqueness of pathwise kinetic solutions}, is a contraction principle for pathwise kinetic solutions with nonnegative initial data, which in particular implies their uniqueness.

\begin{theorem}\label{theorem/uniqueness of pathwise kinetic solutions for nonnegative initial data}
Let $m\in(0,\infty)$ and let $u_0^1,u_0^2\in L_+^2(Q)$.
Under the assumptions \eqref{formula/assumption on A smooth}-\eqref{formula/assumption on A xi derivative zero on the boundary}, pathwise kinetic solutions $u^1$ and $u^2$ of \eqref{formula/stochastic porous media equation} with initial data $u_0^1$ and $u_0^2$ satisfy
\begin{equation}\label{theorem/uniqueness of pathwise kinetic solutions for nonnegative initial data/formula 1}
    \|u^1-u^2\|_{L^{\infty}([0,\infty);L^1(Q))}\leq\|u_0^1-u_0^2\|_{L^1(Q)}. 
\end{equation}
In particular, pathwise kinetic solutions with nonnegative initial data are unique.
\end{theorem} 

As already mentioned in Section \ref{introduction}, our uniqueness result heavily relies on sharp a priori estimates needed to tackle the nonlinear diffusion and take care of the zero boundary conditions.
In turn, these are coupled with other stable estimates for smoothed equations approximating \eqref{formula/stochastic porous media equation} both in space and time.
Then, using weak convergence and compactness arguments, in Section \ref{section/existence of pathwise kinetic solutions} we prove the existence of pathwise kinetic solutions with a limit procedure.

\begin{theorem}\label{theorem/existence of pathwise kinetic solutions for signed data}
Let $m\in(0,\infty)$ and let $u_0\in L^2(Q)$.
Under the assumptions \eqref{formula/assumption on A smooth}-\eqref{formula/assumption on A xi derivative zero on the boundary}, there exists a pathwise kinetic solution of \eqref{formula/stochastic porous media equation} with initial data $u_0$.
Furthermore, if $u_0\in L^2_+(Q)$, the corresponding solution stays nonnegative, that is
\begin{equation}
    u(x,t)\geq0\quad \text{almost everywhere in }Q\times[0,\infty). \nonumber
\end{equation}
\end{theorem}

The pathwise nature of Theorem \ref{theorem/uniqueness of pathwise kinetic solutions for nonnegative initial data} and \ref{theorem/existence of pathwise kinetic solutions for signed data} immediately implies the existence of a \emph{random dynamical systems} associated to \eqref{formula/stochastic porous media equation}.
Precisely, suppose that our driving noise $[0,\infty)\ni t\mapsto z_t=z_t(\omega)$ arises from the sample paths of a stochastic process defined on a probability space $(\Omega, \F, \p)$, in such a way that $z_t(\omega)$ is indeed an $\alpha$-H\"older geometric rough path for almost every $\omega\in\Omega$.
Then we have the following result.

\begin{theorem}\label{theorem/random dynamical system for nonnegative initial data}
Assume the hypotheses \eqref{formula/assumption on A smooth}-\eqref{formula/assumption on A xi derivative zero on the boundary} and let $m\in(0,\infty)$.
When interpreted in the sense of pathwise kinetic solutions, equation \eqref{formula/stochastic porous media equation} defines a random dynamical system on $L^2_+(Q)$.
Let $u(u_0,s,t,z_{\cdot}(\omega))$ denote the solution at time $t$ to \eqref{formula/stochastic porous media equation} started at time $s$ with initial data $u_0\in L^2_+(Q)$ and driving signal $z_{\cdot}(\omega)$.
Then, for almost every $\omega\in (\Omega,\F,\p)$, we have
\begin{equation}
    u(u_0,s,t,z_{\cdot}(\omega))=u(u_0,0,t-s,z_{\cdot+s}(\omega))\quad\forall\, 0\leq s\leq t\,\,\forall\, u_0\in L^2_+(Q). \nonumber
\end{equation}
Moreover, the contraction principle \eqref{theorem/uniqueness of pathwise kinetic solutions for nonnegative initial data/formula 1} implies this dynamical system is continuous when considered with values in $L^1(Q)$.
\end{theorem}

Next we present a result stating that the solution map is continuous with respect to the driving noise.
Indeed, all the aforementioned estimates needed for the existence theorem are pathwise estimates and depend on the driving noise $z_t$ cosindered as a geometric rough path; in particular, they are uniform for geometric rough paths close to each other in the $\alpha$-H\"older metric $d_\alpha$.
In Section \ref{section/existence of pathwise kinetic solutions}, using the same compactness arguments as for the existence proof and exploiting the uniqueness of solutions to \eqref{formula/stochastic porous media equation}, we prove the following continuity result.
Unfortunately, this method does not yield an explicit estimate quantifying the convergence of solutions in terms of the convergence of the driving signals.

\begin{theorem}\label{theorem/continuous dependence on the noise for nonnegative initial data}
Assume the hypotheses \eqref{formula/assumption on A smooth}-\eqref{formula/assumption on A xi derivative zero on the boundary}.
Let $m\in(0,\infty)$ and $u_0\in L^2_+(Q)$.
For any $T>0$, let $\{z^k\}_{k\in\N}$ and $z$ be a sequence of $n$-dimensional $\alpha$-H\"older geometric rough paths on $[0,T]$ such that
\begin{equation}
    \lim_{k\to\infty}d_\alpha(z^k,z)=0. \nonumber
\end{equation}
Let $\{u^k\}_{k\in\N}$ and $u$ denote the pathwise kinetic solutions to \eqref{formula/stochastic porous media equation} on $[0,T]$ with initial data $u_0$ and driving signals $\{z^k\}_{k\in\N}$ and $z$ respectively.
Then we have
\begin{equation}
    \lim_{k\to\infty}\|u^k-u\|_{L^1([0,T];L^1(Q))}=0. \nonumber
\end{equation}
\end{theorem}

\begin{remark}\label{remark/support properties of the solution}
Theorem \ref{theorem/continuous dependence on the noise for nonnegative initial data} has immediate consequences on the support properties of the solutions to \eqref{formula/stochastic porous media equation}.
For simplicity, let us suppose that our driving noise $t\mapsto z_t$ arises from the sample paths of an $n$-dimensional Brownian motion $B_t(\omega)$ defined on a probability space $(\Omega, \F, \p)$, in the sense that the mapping $t\mapsto \mathbb{B}^{\text{Strat}}_t(\omega)$ defines $\p$ almost surely an $\alpha$-H\"older geometric rough path $z_t(\omega)$, where $\mathbb{B}^{\text{Strat}}_t$ denotes the Brownian motion enhanced with its iterated Stratonovich integrals (see e.g.  \cite{fritz-hairer-a-course-on-rough-apths}).
For $\p$ a.e. $\omega\in\Omega$, let $u(\omega)$ denote the pathwise kinetic solution of \eqref{formula/stochastic porous media equation} with initial data $u_0\in L^2_+(Q)$ and driving signal $t\mapsto \mathbb{B}^{\text{Strat}}_t(\omega)$.
Given any smooth path $g:[0,\infty)\to\R^n$, let $\Bar{u}_g$ denote the pathwise kinetic solution, with the same initial data $u_0$, of the deterministic porous media equation with convective term
\begin{equation}\label{formula/deterministic porous media equation 0}
    \partial_t \Bar{u}_g=\Delta \Bar{u}_g^{[m]}+\nabla\cdot\big(A(x,\bar{u}_g)\,\dot{g}\big),
\end{equation}
that is of equation \eqref{formula/stochastic porous media equation} driven by the smooth path $z_t:= g(t)$ (cf. \ref{section/rough path estimates}).
In this setting, Theorem \ref{theorem/continuous dependence on the noise for nonnegative initial data} implies that the probability of $u$ being arbitrarily close to the deterministic solution $\Bar{u}_g$ in $L^1$-norm is always nonzero.
Indeed, properties of the Stratonovich enhanced Brownian motion ensure that $\p\left(d_{\alpha}\left(\mathbb{B}^{\text{Strat}}_{\cdot},g\right)\!\leq\epsilon\,\right)>0$ for any $T\geq 0$ and any $\epsilon>0$ (cf. \cite[Chapter 13]{friz-victoir-multidimensiona-stochastic-processes-as-rough-paths}).
Theorem \ref{theorem/continuous dependence on the noise for nonnegative initial data} immediately implies that, for any $T\geq 0$,
\begin{equation}
    \p\left(\|u-\Bar{u}_g\|_{L^1([0,T];L^1(Q))}\!\leq\epsilon\,\right)>0 \quad \forall\epsilon>0.
\end{equation}
That is to say, the support of the law of $u$ in $L^1([0,T];L^1(Q))$ contains the solution of \eqref{formula/deterministic porous media equation 0} for every smooth path $g$.
In fact, since almost every sample path $\mathbb{B}_{\cdot}(\omega)$ arises as $d_{\alpha}$-limit of smooth paths, we actually have that the support of the law of $u$ is the closure of $\{\bar{u}_g\mid\text{$g$ smooth path}\}$ in $L^1([0,T];L^1(Q))$.
\end{remark}

\medskip
Finally, we notice that the methods of this paper apply to general initial data in $L^2(Q)$ provided the diffusion exponent satisfies $m=1$ or $m>2$.
Indeed, the nonnegativity of the solution, and thus of the initial data, is only required in the a-priori estimate presented in Proposition \ref{proposition/singular moments for defect measures d=1}.
In turn, this estimate is only needed to tackle the case of small diffusion exponents $m\in(0,1)\cup(1,2]$ (cf. Remark \ref{remark/extension to signed data}).
As a consequence we get the following result.

\begin{theorem}\label{theorem/extension of the results to signed initial data}
Let $m=1$ or $m>2$.
Under the assumptions \eqref{formula/assumption on A smooth}-\eqref{formula/assumption on A xi derivative zero on the boundary}, for every $u_0\in L^2(Q)$ there exists a unique pathwise kinetic solution of \eqref{formula/stochastic porous media equation} and the analogous results of Theorem \ref{theorem/uniqueness of pathwise kinetic solutions for nonnegative initial data} and Theorems \ref{theorem/continuous dependence on the noise for nonnegative initial data} and \ref{theorem/random dynamical system for nonnegative initial data} hold.
\end{theorem}



\section{Definition and motivation of pathwise kinetic solutions}\label{section/definition and motivation of pathwise kinetic solutions}

The aim of this section is to understand equation \eqref{formula/stochastic porous media equation}, and motivate and present the notion of \emph{pathwise kinetic solution} given in Definition \ref{definition/pathwise kinetic solution}.
For this purpose, we shall first consider a uniformly elliptic regularization of \eqref{formula/stochastic porous media equation} driven by smooth noise.
The assumption \eqref{formula/z is a geometric rough path} ensures that there exists a sequence of smooth paths $\{z^\epsilon:[0,\infty)\to\R^n\}_{\epsilon\in(0,1)}$ such that, for each $T>0$,
\begin{equation}\label{formula/smooth paths converging to rough path}
    \lim_{\epsilon\to0}d_{\alpha}\left(z,z^{\epsilon}\right)=0,
\end{equation}
where $d_\alpha$ denotes the $\alpha$-H\"older metric on the space of geometric rough paths $C^{0,\alpha}\left([0,T];G^{\left\lfloor\nicefrac{1}{\alpha}\right\rfloor}(\R^n)\right)$.
In what follows, for $\epsilon\in(0,1)$, we will denote by $\Dot{z}^{\epsilon}$ the derivative of the smooth path.

Furthermore, it is necessary to introduce an $\eta$-perturbation by the Laplacian, for $\eta\in(0,1)$, in order to remove the degeneracy of the porous media diffusion.
Therefore, for each $\eta\in(0,1)$ and $\epsilon\in(0,1)$, we consider the equation
\begin{align}[left ={\empheqlbrace}]\label{formula/regularized porous media equation 1}
    \begin{aligned}
        &\partial_t\ueta=\Delta (\ueta)^{[m]}+\eta\Delta\ueta+\nabla\cdot(A(x,\ueta)\Dot{z}_t^\epsilon) &\text{in } Q\times(0,\infty),
        \\
        &\ueta=0 &\text{on } \partial Q\times (0,\infty),
        \\
        &\ueta=u_0 &\text{on }  Q\times \{0\}.
    \end{aligned}
\end{align}
We shall derive a formulation of the equation that is well-defined for singular driving signals.
Namely, we shall pass to the kinetic form of \eqref{formula/regularized porous media equation 1}, where the noise enters as a linear transport, and then derive a formulation that is well-defined even after passing to the limit with respect to the regularization.

The following proposition establishes the well-posedness of \eqref{formula/regularized porous media equation 1} in the classical sense.
The proof is a small modification of \cite[Proposition A.1]{fehrman-gess-Wellposedness-of-Nonlinear-Diffusion-Equations-with-Nonlinear-Conservative-Noise}, and thus is omitted.

\begin{proposition}\label{proposition/existence of classical solutions for smoothed equation}
For each $\eta\in(0,1)$, each $\epsilon\in(0,1)$, and each $u_0\in L^2(Q)$, there exists a classical solution $\ueta$ of equation \eqref{formula/regularized porous media equation 1} such that
\begin{align}
    \begin{aligned}
        u\in H^1\left([0,T];H^1_0(Q),H^{-1}(Q)\right),\quad\text{and}\quad u^{[m]},u^{\left[\nicefrac{m+1}{2}\right]}\in L^2\left([0,T];H^1_0(Q)\right).
    \end{aligned}
\end{align}
\end{proposition}

We now pass to the kinetic form of \eqref{formula/regularized porous media equation 1},
complete details of the following derivation are given in \cite[Appendix A]{fehrman-gess-Wellposedness-of-Nonlinear-Diffusion-Equations-with-Nonlinear-Conservative-Noise}.
This formulation is obtained by introducing the kinetic function $\Bar{\chi}:\R^2\to\{-1,0,1\}$ defined by
\begin{equation}
    \Bar{\chi}(v,\xi):= \begin{cases}
                           1\qquad\text{if $0<\xi<v$},
                           \\
                           -1\quad\text{ if $v<\xi<0$},
                           \\
                           0\qquad\text{else.}
                        \end{cases}
    \nonumber
\end{equation}
We then define, for each $\eta\in(0,1)$ and $\epsilon\in(0,1)$, for $\ueta$ the solution of \eqref{formula/regularized porous media equation 1}, the composition
\begin{equation}\label{formula/kinetic function of smooth solutions}
    \chieta(x,\xi,t):=\bar{\chi}(\ueta(x,t),\xi). 
\end{equation}
Proposition \ref{proposition/weak kinetic formulation of the smoothed equation} below states that the kinetic function $\chieta$ is a distributional solution of the equation 
\begin{align}\label{formula/strong kinetic formulation of the smoothed equation}
\begin{aligned}
    \partial_t\chi^{\eta,\epsilon}\!\!=\,m|\xi|^{m-1}\!\Delta_x\chieta\!+\eta\Delta_x\chieta\!
    +(\partial_{\xi}A(x,\xi)\dot{z}_t^\epsilon)\!\cdot\!\nabla_x\chieta\!-\!\left(\nabla_x\!\cdot\! A(x,\xi)\dot{z}^\epsilon_t\right)\partial_{\xi}\!\chieta\!+\partial_{\xi}(\peta+\qeta)  
\end{aligned}
\end{align}
in $Q\times\R\times(0,\infty)$, with Dirichlet boundary conditions and initial data $\bar{\chi}(u_0(x),\xi)$.
Here, the measure $\peta$ is the \emph{entropy defect measure}
\begin{equation}\label{formula/entropy defect measure smooth}
    \peta(x,\xi,t):=\delta_0(\xi-\ueta(x,t))\,\eta\,\left|\nabla\ueta(x,t)\right|^2,
\end{equation}
and the measure $\qeta$ is the \emph{parabolic defect measure}
\begin{equation}\label{formula/parabolic defect measure smooth}
    \qeta(x,\xi,t):=\delta_0(\xi-\ueta(x,t))\frac{4m}{(m+1)^2}\left|\nabla(\ueta)^{\left[\frac{m+1}{2}\right]}(x,t)\right|^2, 
\end{equation}
where $\delta_0$ denotes the one-dimensional Dirac mass centered at the origin.
The sense in which the kinetic function satisfies \eqref{formula/strong kinetic formulation of the smoothed equation} is made precise by the following proposition.
The proof is again a small modification of \cite[Proposition A.2]{fehrman-gess-Wellposedness-of-Nonlinear-Diffusion-Equations-with-Nonlinear-Conservative-Noise} and is omitted.

\begin{proposition}\label{proposition/weak kinetic formulation of the smoothed equation}
For each $\eta\in(0,1)$, $\epsilon\in(0,1)$, and $u_0\in L^2(Q)$, let $\ueta$ be the solution of \eqref{formula/regularized porous media equation 1} from Proposition \ref{proposition/existence of classical solutions for smoothed equation}.
The associated kinetic function $\chieta$ defined in \eqref{formula/kinetic function of smooth solutions} is a distributional solution of equation \eqref{formula/strong kinetic formulation of the smoothed equation} in the sense that, for every $t_1\leq t_2\in[0,\infty)$ and for every $\psi\in C_c^{\infty}(Q\times\R\times[t_1,t_2])$, we have
\begin{align}
        \hspace{-2mm}\int_{Q\times\R}\!\!\!\!\!\!\!\!\chieta(x,\xi,r)\psi(x,\xi,r)\,dx\,d\xi\Big|_{r=t_1}^{r=t_2}
        \!\!\!=&\int_{t_1}^{t_2}\!\!\!\int_{Q\times\R}\!\!\!\chieta\partial_t\psi\,dx\,d\xi\,dr \nonumber
        \\
        &+\!\!
        \int_{t_1}^{t_2}\!\!\!\int_{Q\times\R}\!\!\!\left(m|\xi|^{m-1}+\eta\right)\chieta\Delta_x\psi\,dx\,d\xi\,dr \label{formula/weak kinetic formulation of the smoothed equation}
        \\
        &-\!\!\int_{t_1}^{t_2}\!\!\!\int_{Q\times\R}\!\!\!\!\!\!\!\!\chieta\!\left(\nabla_{\!\!x}\!\!\cdot\!((\partial_\xi A(x,\xi)\dot{z}_t^\epsilon)\psi)\!-\!\partial_\xi((\nabla_{\!\!x}\!\!\cdot\! A(x,\xi)\dot{z}_t^\epsilon)\psi)\right)\,dx\,d\xi\,dr \nonumber
        \\
        &-
        \int_{t_1}^{t_2}\!\!\!\int_{Q\times\R}\!\!\!\!\left(\peta(x,\xi,r)+\qeta(x,\xi,r)\right)\partial_{\xi}\psi\,dx\,d\xi\,dr. \nonumber
\end{align}
\end{proposition}

The purpose of this section is to remove the dependency of equation \eqref{formula/weak kinetic formulation of the smoothed equation} on the derivative of the noise, so as to get a formulation of \eqref{formula/stochastic porous media equation} which is well-defined even for rough driving signals.
To achieve this, rather than testing equation \eqref{formula/weak kinetic formulation of the smoothed equation} against arbitrary test functions $\psi$, we shall limit ourselves to only use the solutions of the following transport equations, for any $t_0\in[0,t_1]$ and any initial data $\rho_0\in C_c^{\infty}(Q\times\R)$:
\begin{align}[left ={\empheqlbrace}]\label{formula/smooth underlying transport equation}
    \begin{aligned}
        &\partial_t\rho=\partial_\xi A(x,\xi)\dot{z}_t^\epsilon\cdot\nabla_{\!\!x}\rho -(\nabla_{\!\!x}\!\cdot\! A(x,\xi))\dot{z}_t^\epsilon\,\partial_{\xi}\rho &\text{in } \R^d\times\R\times(t_0,\infty),
        \\
        &\rho=\rho_0 &\text{on }  \R^d\times\R\times \{t_0\}.
    \end{aligned}
\end{align}
Indeed, using such test functions $\rho(x,\xi,t)$ in \eqref{formula/weak kinetic formulation of the smoothed equation}, the first and third term on the right-hand side automatically cancel out.

Owing to assumption \eqref{formula/assumption on A smooth} on $A$ and to the smoothness of the paths $z^{\epsilon}$, equation \eqref{formula/smooth underlying transport equation} is a standard first-order PDE with smooth coefficients, and its solution is represented using the associated characteristics, which we now analyze.
The forward characteristic $(X_{t_0,t}^{x_0,\xi_0,\epsilon},\Xi_{t_0,t}^{x_0,\xi_0,\epsilon})$ associated to \eqref{formula/smooth underlying transport equation} beginning at time $t_0\geq0$ from $(x_0,\xi_0)\in\R^d\times\R$ is defined as the solution of the system
\begin{align}[left ={\empheqlbrace}]\label{formula/forward smooth characteristics}
    \begin{aligned}
        &\dot{X}_{t_0,t}^{x_0,\xi_0,\epsilon}=-\partial_{\xi}A(X_{t_0,t}^{x_0,\xi_0,\epsilon},\Xi_{t_0,t}^{x_0,\xi_0,\epsilon})\dot{z}_t^{\epsilon}&\text{in } (t_0,\infty),
        \\
        &\dot{\Xi}_{t_0,t}^{x_0,\xi_0,\epsilon}=(\nabla_{\!\!x}\!\cdot\! A(X_{t_0,t}^{x_0,\xi_0,\epsilon},\Xi_{t_0,t}^{x_0,\xi_0,\epsilon}))\dot{z}_t^\epsilon&\text{in } (t_0,\infty),
        \\
        &(X_{t_0,t_0}^{x_0,\xi_0,\epsilon},\Xi_{t_0,t_0}^{x_0,\xi_0,\epsilon})=(x_0,\xi_0).
    \end{aligned}
\end{align}
The associated backward characteristic $(Y_{t_0,s}^{x_0,\xi_0,\epsilon},\Pi_{t_0,s}^{x_0,\xi_0,\epsilon})$ beginning at time $0$ from $(x_0,\xi_0)\in\R^d\times\R$ and terminating at $t_0$, is simply defined as the forward characteristic $(X_{t_0,t}^{x_0,\xi_0,\epsilon},\Xi_{t_0,t}^{x_0,\xi_0,\epsilon})$ run backward in time.
Namely, we define
\begin{equation}\label{formula/backward smooth characteristics are forward smooth characteristics run backward in time}
    Y_{t_0,s}^{x_0,\xi_0,\epsilon}:=X_{t_0,t_0-s}^{x_0,\xi_0,\epsilon}\quad\text{and}\quad \Pi_{t_0,s}^{x_0,\xi_0,\epsilon}:=\Xi_{t_0,t_0-s}^{x_0,\xi_0,\epsilon}\qquad \text{for } s\in[0,t_0].
\end{equation}
The forward and backward characteristics are mutually inverse, in the sense that, for each $(x,\xi)\in\R^d\times\R$, for each $t_0\geq0$ and $t\geq t_0$, and for each $s_0\geq 0$ and $s\in[0,s_0]$, we have the relation
\begin{equation}\label{formula/inverse relation for smooth characteristics}
    \left(X_{t_0,t}^{Y_{t,t-t_0}^{x,\xi,\epsilon},\Pi_{t,t-t_0}^{x,\xi,\epsilon},\epsilon},\Xi_{t_0,t}^{Y_{t,t-t_0}^{x,\xi,\epsilon},\Pi_{t,t-t_0}^{x,\xi,\epsilon},\epsilon}\right)=\left(Y_{s_0,s}^{X_{s_0-s,s}^{x,\xi,\epsilon},\Xi_{s_0-s,s}^{x,\xi,\epsilon},\epsilon},\Pi_{s_0,s}^{X_{s_0-s,s}^{x,\xi,\epsilon},\Xi_{s_0-s,s}^{x,\xi,\epsilon},\epsilon}\right)=(x,\xi).
\end{equation}
Furthermore, for each $t_0\geq0$, we define the reversed path
\begin{equation}
    z^{\epsilon}_{t_0,t}:=z^{\epsilon}_{t-t_0}\quad\text{ for $t\in[0,t_0]$}. \nonumber
\end{equation}
It is then easy to check that the backward characteristic $(Y_{t_0,s}^{x_0,\xi_0,\epsilon},\Pi_{t_0,s}^{x_0,\xi_0,\epsilon})$ coincides with the solution of the system
\begin{align}[left ={\empheqlbrace}]\label{formula/backward smooth characteristics}
    \begin{aligned}
        &\dot{Y}_{t_0,t}^{x_0,\xi_0,\epsilon}=-\partial_{\xi}A(Y_{t_0,t}^{x_0,\xi_0,\epsilon},\Pi_{t_0,t}^{x_0,\xi_0,\epsilon})\dot{z}_{t_0,t}^{\epsilon}&\text{in } (0,t_0),
        \\
        &\dot{\Pi}_{t_0,t}^{x_0,\xi_0,\epsilon}=(\nabla_{\!\!x}\!\cdot\! A(Y_{t_0,t}^{x_0,\xi_0,\epsilon},\Pi_{t_0,t}^{x_0,\xi_0,\epsilon}))\dot{z}_{t_0,t}^\epsilon&\text{in } (0,t_0),
        \\
        &(Y_{t_0,t_0}^{x_0,\xi_0,\epsilon},\Pi_{t_0,t_0}^{x_0,\xi_0,\epsilon})=(x_0,\xi_0).
    \end{aligned}
\end{align}

The solution of \eqref{formula/smooth underlying transport equation} is the transport of the initial data along the backward characteristics \eqref{formula/backward smooth characteristics are forward smooth characteristics run backward in time}.
Precisely, for each $\rho_0\in C_c^{\infty}(Q\times\R)$, a direct computation proves that $\rho_0(Y_{t,t-t_0}^{x,\xi,\epsilon},\Pi_{t,t-t_0}^{x,\xi,\epsilon})$ is indeed the solution of \eqref{formula/smooth underlying transport equation}.
For each $t_0\geq 0$ and $\rho_0\in C_c^{\infty}(Q\times\R)$, we shall use the notation
\begin{equation}\label{formula/transported smooth test function}
    \rho_{t_0,t}^{\epsilon}(x,\xi):=\rho_0(Y_{t,t-t_0}^{x,\xi,\epsilon},\Pi_{t,t-t_0}^{x,\xi,\epsilon})
\end{equation}
for the solution of the transport equation \eqref{formula/smooth underlying transport equation}.

Furthermore, as a consequence of \eqref{formula/smooth underlying transport equation}, the characteristics preserve the Lebesgue measure on $\R^d\times\R$.
That is, for every $0\leq t_0<t_1$ and $0\leq s_1<s_0$, for every $\psi\in L^1(\R^d\times\R)$, 
\begin{equation}\label{formula/lebesgue measure is preserved by smooth caharacteristics}
    \int_{R^d\times\R}\psi(x,\xi)\,dx\,d\xi=\int_{R^d\times\R}\psi(X_{t_0,t_1}^{x,\xi,\epsilon},\Xi_{t_0,t_1}^{x,\xi,\epsilon})\,dx\,d\xi=\int_{R^d\times\R}\psi(Y_{s_0,s_1}^{x,\xi,\epsilon},\Pi_{s_0,s_1}^{x,\xi,\epsilon})\,dx\,d\xi.
\end{equation}
This property in turn comes from the conservative structure of the noise term in \eqref{formula/regularized porous media equation 1}, and it is essential to the proof of uniqueness in Section \ref{section/uniqueness of pathwise kinetic solutions}.

Finally, we analyze the consequences for the characteristics of the assumptions \eqref{formula/assumption on A divergence 0 in 0} and \eqref{formula/assumption on A xi derivative zero on the boundary} on the noise coefficient $A(x,\xi)$.
It is immediate from the second line of \eqref{formula/forward smooth characteristics} and \eqref{formula/backward smooth characteristics} that the hypothesis $\nabla_{\!\!x}\cdot A(x,0)\equiv0$ implies the characteristics preserve the sign of the velocity variable.
That is, for each $(x,\xi)\in\R^d\times\R$, for each $t_0\geq0$ and $t\geq t_0$, and for each $s_0\geq 0$ and $s\in[0,s_0]$, we have
\begin{equation}\label{formula/sign of velocity is preserved by smooth characteristics}
    \Xi_{t_0,t}^{x,\xi,\epsilon}=\Pi_{s_0,s}^{x,\xi,\epsilon}=0 \,\,\text{ if and only if \,\,$\xi=0$,\,\, and\,\, }\sgn(\xi)=\sgn(\Xi_{t_0,t}^{x,\xi,\epsilon})=\sgn(\Pi_{s_0,s}^{x,\xi,\epsilon}) \text{ \,\,if\,\, $\xi\neq0$}.
\end{equation}

As regards the assumption $\partial_{\xi}A(x,\xi)_{|\partial Q\times\R}\equiv 0$, the first line of \eqref{formula/forward smooth characteristics} and \eqref{formula/backward smooth characteristics} ensure that space characteristics starting from the boundary do not move. That is, for any $t_0\geq0$,
\begin{equation}\label{formula/smooth characteristics starting from the boundary do not move}
    \text{if \, $(x,\xi)\in\partial Q\times\R$,\, then}\,\,X^{x,\xi,\epsilon}_{t_0,t}=Y^{x,\xi,\epsilon}_{t_0,s}=x\,\,\, \text{for all $t\geq 0$ and all $s\in[0,t_0]$}.
\end{equation}
The uniqueness of solutions then implies that the space characteristics cannot cross the boundary $\partial Q$ and thus, when starting within $Q$, they never leave the domain.
That is, for any $t_0\geq0$,
\begin{equation} \label{formula/smooth characteristics never leave the domain}
    \text{if \,\,$(x,\xi)\in Q\times\R$, \, then }\,X^{x,\xi,\epsilon}_{t_0,t}\in Q\,\,\,\text{for all $t\geq 0$, \, and }\,\, Y^{x,\xi,\epsilon}_{t_0,s}\in Q\,\,\, \text{for all $s\in[0,t_0]$}. 
\end{equation}
In fact, owing to the smoothness hypothesis \eqref{formula/assumption on A smooth}, more is true: the closer to $\partial Q$ is the space initial data $x\in \R^d$, the slower the associated space characteristic moves.
Rigorously speaking, for any $T>0$ we have
\begin{equation}\label{formula/space smooth characteristics move slower the closer they get to the boundary}
    \left|X_{t_0,t}^{x,\xi,\epsilon}-x\right|+\left|Y_{t_0,s}^{x,\xi,\epsilon}-x\right|\leq C\, \dist(x,\partial Q) \quad\forall(x,\xi)\in\R^d\times\R\quad\forall t_0,t\in[0,T]\quad\forall s\in[0,t_0],
\end{equation}
for a constant $C=C(T,A,z)$ depending on the time $T$, the noise coefficient $A(x,\xi)$ defining the systems \eqref{formula/forward smooth characteristics} and \eqref{formula/backward smooth characteristics}, and the rough signal $z$, but uniform for all the paths $\{z^{\epsilon}\}_{\epsilon\in(0,1)}$, which are close to the rough path $z$ in the metric $d_{\alpha}$.
This and other estimates are proved in Appendix \ref{section/rough path estimates}.

Finally, the assumption $\partial_{\xi}A(x,\xi)_{|\partial Q\times\R}\equiv 0$ has fundamental consequences on the transported functions \eqref{formula/transported smooth test function}.
Indeed estimate \eqref{formula/space smooth characteristics move slower the closer they get to the boundary} ensures that
\begin{equation}\label{formula/transported compactly supported functions stays compactly supported - smooth case}
    \text{if $\rho_0\!\in\! C_c(Q\times\R)$, then $\rho^{\epsilon}_{t_0,t}(x,\xi)\!=\!\rho_0(Y_{t,t-t_0}^{x,\xi,\epsilon},\Pi_{t,t-t_0}^{x,\xi,\epsilon})\!\in\! C_c(Q\!\times\!\R\!\times\![t_0,T])$ for any $T\geq t_0\geq0$.} 
\end{equation}
That is, the transport along characteristics of a functions which is compactly supported within $Q\times\R$ stays compactly supported, and thus, in particular, is an admissible test function in equation \eqref{formula/weak kinetic formulation of the smoothed equation}.

Now we go back to equation \eqref{formula/weak kinetic formulation of the smoothed equation}.
The following corollary makes precise the idea of testing the equation against functions transported along the inverse characteristics, so as to get rid of the noise in the equation.
The proof is an immediate consequence of the discussion above, in particular equation \eqref{formula/smooth underlying transport equation}, the representation \eqref{formula/transported smooth test function} and property \eqref{formula/transported compactly supported functions stays compactly supported - smooth case}, and Proposition \ref{proposition/weak kinetic formulation of the smoothed equation}.

\begin{corollary}\label{corollary/transported weak kinetic formulation of the smoothed equation}
Let $\eta,\epsilon\in(0,1)$ and $u_0\in L^2(Q)$.
The kinetic function $\chieta$ from Proposition \ref{proposition/weak kinetic formulation of the smoothed equation} satisfies, for each $t_0\leq t_1\in[0,\infty)$ and $\rho_0\in C_c^{\infty}(Q\times\R)$, for the solution $\rho_{t_0,r}^\epsilon(x,\xi):=\rho(Y^{x,\xi,\epsilon}_{r,r-t_0},\Pi^{x,\xi,\epsilon}_{r,r-t_0})$ of \eqref{formula/smooth underlying transport equation},
\begin{align}\label{formula/transported weak kinetic formulation of the smoothed equation 1}
    \vspace{-2mm}
    \begin{aligned}
        \int_{Q\times\R}\!\!\!\!\!\!\chieta(x,\xi,r)\rho_{t_0,r}^{\epsilon}(x,\xi)\,dx\,d\xi\Big|_{r=t_0}^{r=t_1}
        =&
        \int_{t_0}^{t_1}\!\!\!\int_{Q\times\R}\!\!\!\!\left(m|\xi|^{m-1}+\eta\right)\chieta(x,\xi,r)\Delta_x\rho_{t_0,r}^{\epsilon}(x,\xi)\,dx\,d\xi\,dr
        \\
        &-
        \int_{t_0}^{t_1}\!\!\!\int_{Q\times\R}\!\!\!\!\!\!\left(\peta(x,\xi,r)+\qeta(x,\xi,r)\right)\partial_{\xi}\rho_{t_0,r}^{\epsilon}(x,\xi)\,dx\,d\xi\,dr.
    \end{aligned}
\end{align}
\end{corollary}

The essential observation is that, in the passage to the singular limit $\epsilon\to 0$, the system of characteristics \eqref{formula/forward smooth characteristics} is well-posed for rough noise when interpreted as a rough differential equation.
In view of the representation \eqref{formula/transported smooth test function}, this implies the well-posedness of equation \eqref{formula/smooth underlying transport equation} for rough signals as well.
Precisely, for each $(x_0,\xi_0)\in\R^d\times\R$ and $t_0\geq 0$, we define the rough forward characteristic $(X_{t_0,t}^{x_0,\xi_0},\Xi_{t_0,t}^{x_0,\xi_0})$ as the solution of the rough differential equation
\begin{align}[left ={\empheqlbrace}]\label{formula/forward stochastic characteristics}
    \begin{aligned}
        &dX_{t_0,t}^{x_0,\xi_0}=-\partial_{\xi}A(X_{t_0,t}^{x_0,\xi_0},\Xi_{t_0,t}^{x_0,\xi_0})\circ dz_t&\text{in } (t_0,\infty),
        \\
        &d\hspace{0.5mm}\Xi_{t_0,t}^{x_0,\xi_0}=(\nabla_{\!\!x}\!\cdot\! A(X_{t_0,t}^{x_0,\xi_0},\Xi_{t_0,t}^{x_0,\xi_0}))\circ dz_t&\text{in } (t_0,\infty),
        \\
        &(X_{t_0,t_0}^{x_0,\xi_0},\Xi_{t_0,t_0}^{x_0,\xi_0})=(x_0,\xi_0).
    \end{aligned}
\end{align}
The continuity properties of the It\^o--Lyons map, summarized in Proposition \ref{proposition/stability results for RDEs}, and the smoothness assumptions \eqref{formula/assumption on A smooth} on $A(x,\xi)$ imply that, for each $T>0$,
\begin{equation}
    \lim_{\epsilon\to0}\left(X_{t_0,t}^{x,\xi,\epsilon},\Xi_{t_0,t}^{x,\xi,\epsilon}\right)=\left(X_{t_0,t}^{x,\xi},\Xi_{t_0,t}^{x,\xi}\right) \,\,\text{uniformly for $t_0\!\leq\!t\!\in\![0,T]$ and  $(x,\xi)\in\R^d\times\R$.} 
\end{equation}

In analogy to the smooth case, we now list the properties of the rough characteristics.
As before, the rough inverse characteristic $(Y_{t_0,s}^{x_0,\xi_0},\Pi_{t_0,s}^{x_0,\xi_0})$ is defined as the forward characteristic run backward in time. 
Namely, we define
\begin{equation}\label{formula/backward stochastic characteristics are forward smooth characteristics run backward in time}
    Y_{t_0,s}^{x_0,\xi_0}:=X_{t_0,t_0-s}^{x_0,\xi_0}\quad\text{and}\quad \Pi_{t_0,s}^{x_0,\xi_0}:=\Xi_{t_0,t_0-s}^{x_0,\xi_0}\qquad \text{for } s\in[0,t_0]. \nonumber
\end{equation}
As before, for each $t_0\geq0$, we define the reversed path $z_{t_0,t}:=z_{t-t_0}$ for $t\in[0,t_0]$, and the backward rough characteristic $(Y_{t_0,s}^{x_0,\xi_0,\epsilon},\Pi_{t_0,s}^{x_0,\xi_0,\epsilon})$ coincides with the solution of the rough differential equation
\begin{align}[left ={\empheqlbrace}]\label{formula/backward stochastic characteristics}
    \begin{aligned}
        &dY_{t_0,t}^{x_0,\xi_0}=-\partial_{\xi}A(Y_{t_0,t}^{x_0,\xi_0},\Pi_{t_0,t}^{x_0,\xi_0})\circ dz_{t_0,t}&\text{in } (0,t_0),
        \\
        &d\hspace{0.5mm}\Pi_{t_0,t}^{x_0,\xi_0}=(\nabla_{\!\!x}\!\cdot\! A(Y_{t_0,t}^{x_0,\xi_0},\Pi_{t_0,t}^{x_0,\xi_0}))\circ dz_{t_0,t}&\text{in } (0,t_0),
        \\
        &(Y_{t_0,t_0}^{x_0,\xi_0},\Pi_{t_0,t_0}^{x_0,\xi_0})=(x_0,\xi_0).
    \end{aligned} 
\end{align}

All the properties \eqref{formula/inverse relation for smooth characteristics}-\eqref{formula/transported compactly supported functions stays compactly supported - smooth case} of the smooth forward and backward characteristics continue to hold for the limiting rough characteristics, just by deleting the symbol $\epsilon$ form the formulas.
The rough forward and backward characteristics are inverse of each other in the sense of formula \eqref{formula/inverse relation for smooth characteristics}, and they preserve the Lebesgue measure as prescribed by \eqref{formula/lebesgue measure is preserved by smooth caharacteristics}.
The assumptions $\nabla_{\!\!x}\cdot A(x,0)\equiv0$ still implies that the rough characteristics preserve the sign of the velocity variable as in \eqref{formula/sign of velocity is preserved by smooth characteristics}. 
Similarly, the assumption $\partial_{\xi}A(x,\xi)_{|\partial Q\times\R}\equiv 0$ guarantees that the rough space characteristics do not move when starting from the boundary, and that, when starting within $Q$, cannot leave the domain and move slower the closer they get to the boundary, as specified by \eqref{formula/smooth characteristics starting from the boundary do not move}-\eqref{formula/space smooth characteristics move slower the closer they get to the boundary}.
These concepts are discussed in Section \ref{section/rough path estimates}.

Finally, it follows from the well-posedness of the characteristic systems \eqref{formula/forward stochastic characteristics} and \eqref{formula/backward stochastic characteristics} that the rough transport equation
\begin{align}[left ={\empheqlbrace}]\label{formula/stochastic underlying transport equation}
    \begin{aligned}
        &\partial_t\rho=\left(\partial_\xi A(x,\xi)\circ dz_t\right)\cdot\nabla_{\!\!x}\rho -\left((\nabla_{\!\!x}\!\cdot\! A(x,\xi))\circ dz_t\right)\,\partial_{\xi}\rho &\text{in } \R^d\times\R\times(t_0,\infty),
        \\
        &\rho=\rho_0 &\text{on }  \R^d\times\R\times \{t_0\},
    \end{aligned}
\end{align}
is well-posed for any $t_0\geq 0$ and any initial data $\rho_0\in C^{\infty}(\R^d\times\R)$.
Indeed, in analogy with \eqref{formula/transported smooth test function}, the solution is represented by the transport of the initial data along the inverse rough characteristics \eqref{formula/backward stochastic characteristics}.
That is, for each $t_0\geq0$ and $\rho_0\in C^{\infty}(\R^d\times\R)$, the solution of \eqref{formula/stochastic underlying transport equation} can be written as
\begin{equation}\label{formula/transported stochastic test function}
    \rho_{t_0,t}(x,\xi):=\rho_0(Y_{t,t-t_0}^{x,\xi},\Pi_{t,t-t_0}^{x,\xi}).
\end{equation}
Then, as in the smooth case, the assumption $\partial_{\xi}A(x,\xi)_{|\partial Q\times\R}\equiv 0$ is fundamental to ensure that the transport along rough characteristics of a function that is compactly supported within $Q\times\R$ stays compactly supported.
Namely, estimate \eqref{formula/space smooth characteristics move slower the closer they get to the boundary} implies that
\begin{equation}\label{formula/transported compactly supported functions stays compactly supported - rough case}
    \text{if $\rho_0\!\in\! C_c(Q\times\R)$, then $\rho_{t_0,t}(x,\xi)\!=\!\rho_0(Y_{t,t-t_0}^{x,\xi},\Pi_{t,t-t_0}^{x,\xi})\!\in\! C_c(Q\!\times\!\R\!\times\![t_0,T])$ for any $T\geq t_0\geq0$.} 
\end{equation}

We are now prepared to present the definition of pathwise kinetic solution.
Proposition \ref{proposition/stable estimate for L^2 norm and defect measures of smoothed solutions} below proves that the solutions $\ueta$ of \eqref{formula/regularized porous media equation 1} satisfy, for each $T>0$, for $C=C(m,Q,T,A,z)$,
\begin{align}\label{formula/stable estimate for L^2 norm and defect measures of smoothed solutions}
    \begin{aligned}
        \hspace{-4mm}\|\ueta\|_{L^{\infty}([0,T];L^2(Q))}^2\!\!+\left\|\nabla(\ueta)^{[\nicefrac{m+1}{2}]}\right\|_{L^2([0,T];L^2(Q))}^2\!\!\!+\left\|\eta^{\frac{1}{2}}\nabla(\ueta)\right\|_{L^2([0,T];L^2(Q))}^2
        \!\!\leq
        C\left(1+\|u_0\|_{L^2(Q)}^2\right), 
    \end{aligned}\nonumber
\end{align}
uniformly for $\eta\in(0,1)$ and $\epsilon\in(0,1)$.
Using similar estimates, and weak convergence and compactness arguments, we would expect the smooth solutions $\ueta$ to converge to some limiting function $u$ as we let $\eta,\epsilon\to 0$.
In turn, we would expect $\chieta=\bar{\chi}(\ueta,\xi)$ to converge to $\bar{\chi}(u,\xi)$.
Moreover, the definitions \eqref{formula/entropy defect measure smooth}-\eqref{formula/parabolic defect measure smooth} of the entropy and parabolic defect measures and the above estimate on their total mass immediately imply that $\peta$ and $\qeta$ converge weakly in the $\eta,\epsilon\to 0$ limit.
Finally, for each $T>0$, Proposition \ref{proposition/stability results for RDEs} ensures that, as $\epsilon\to0$, 
\begin{equation}
    Y_{t,s}^{x,\xi,\epsilon}\to Y_{t,s}^{x,\xi}\quad\text{and}\quad \Pi_{t,s}^{x,\xi,\epsilon}\to \Pi_{t,s}^{x,\xi}\quad\text{uniformly for $s\leq t\in[0,T]$ and $(x,\xi)\in\R^d\times\R$}.\nonumber
\end{equation}
As a consequence, we would expect equation \eqref{formula/transported weak kinetic formulation of the smoothed equation 1} to pass to the $\eta,\epsilon\to 0$ limit as well.
This informal argument motivates the following definition.

\begin{definition}\label{definition/pathwise kinetic solution}
Let $u_0\in L^2(Q)$. 
A \emph{pathwise kinetic solution} with initial data $u_0$ is a function $u\in L^{\infty}_{\text{loc}}([0,\infty);L^2(Q))$ that satisfies the following properties.
\begin{itemize}
    \item [(i)] For each $T>0$,
        \begin{equation}
            u^{\left[\frac{m+1}{2}\right]}\in L^2([0,T];H^1_0(Q)). \nonumber
        \end{equation}
        In particular, for each $T>0$, the parabolic defect measure 
        \begin{equation}\label{formula/parabolic defect measure}
            q(x,\xi,t):=\frac{4m}{(m+1)^2}\delta_0(\xi-u(x,t))\left|\nabla u^{\left[\frac{m+1}{2}\right]}\right|^2 
        \end{equation}
        is finite on $Q\times\R\times[0,T]$.
    \item [(ii)] For the kinetic function $\chi(x,\xi,t)=\bar{\chi}(u(x,t),\xi)$, there exists a nonnegative entropy defect measure $p$ on $Q\times\R\times[0,\infty)$, which is finite on $Q\times\R\times[0,T]$ for each $T>0$, and a subset $\mathcal{N}\subset (0,\infty)$ of Lebesgue measure zero such that, for every $t_0<t_1\in [0,\infty)\setminus \mathcal{N}$ and every $\rho_0\in C_c^{\infty}(Q\times\R)$, for $\rho_{t_0,t}(x,\xi)$ defined by \eqref{formula/transported stochastic test function},
        \begin{align}\label{definition/pathwise kinetic solution/formula 1}
        \vspace{-2mm}
            \begin{aligned}
            \int_{Q\times\R}\!\!\!\!\!\!\chi(x,\xi,r)\rho_{t_0,r}(x,\xi)\,dx\,d\xi\bigg|_{r=t_0}^{r=t_1}
            =&
            \int_{t_0}^{t_1}\!\!\!\int_{Q\times\R}\!\!\!\!\!m|\xi|^{m-1}\chi(x,\xi,r)\,\Delta_x\rho_{t_0,r}(x,\xi)\,dx\,d\xi\,dr
            \\
            &-
            \int_{t_0}^{t_1}\!\!\!\int_{Q\times\R}\!\!\!\!\!\!\left(p(x,\xi,r)+q(x,\xi,r)\right)\,\partial_{\xi}\rho_{t_0,r}(x,\xi)\,dx\,d\xi\,dr.
            \end{aligned}
        \end{align}
    Moreover, the initial condition is enforced in the sense that, when $t_0=0$,
        \begin{equation}
            \int_{Q\times\R}\!\!\!\!\!\!\chi(x,\xi,0)\,\rho_{0,0}(x,\xi)\,dx\,d\xi=\int_{Q\times\R}\!\!\!\!\!\!\bar{\chi}(u_0(x),\xi)\,\rho_0(x,\xi)\,dx\,d\xi. \nonumber
        \end{equation}
\end{itemize}
\end{definition}

\begin{remark}
We observe that in fact, as $\eta\to0$, the entropy defect measures $\peta$ converge weakly to $0$, owing to the regularity implied by the parabolic defect measures $\qeta$.
However, due to the weak lower semicontinuity of the $L^2$-norm, along a subsequence, the weak limit of the parabolic defect measures $\qeta$ may lose mass in the limit, since the gradients $\nabla(\ueta)^{[\nicefrac{m+1}{2}]}$ will, in general, converge only weakly. 
The entropy defect measure appearing in Definition \ref{definition/pathwise kinetic solution} is therefore necessary to account for this potential loss of mass.
The complete details of this phenomenon are given in the proof of Theorem \ref{theorem/existence of pathwise kinetic solutions for signed data} from Section \ref{section/existence of pathwise kinetic solutions}.
\end{remark}

\begin{remark}
We remark that \eqref{definition/pathwise kinetic solution/formula 1} is equivalent to requiring that the kinetic function $\chi$ satisfies, for each $\varphi\in C_c^{\infty}([0,\infty))$, each $t_0\leq t_1\in[0,\infty)\setminus\mathcal{N}$ and each $\rho_0\in C_c^{\infty}(Q\times\R)$, for the solution $\rho_{t_0,t}(x,\xi)$ of \eqref{formula/stochastic underlying transport equation},
\begin{align}\label{formula/pathwise kientic solution equation enhanced version}
    \vspace{-2mm}
    \begin{aligned}
        \int_{t_0}^{t_1}\int_{Q\times\R}\!\!\!\!\!\!\chi(x,\xi,r)\,\rho_{t_0,r}(x,\xi)\,\dot{\varphi}(r)\,dx\,d\xi\,dr
        =&
        \int_{Q\times\R}\!\!\!\!\!\!\chi(x,\xi,r)\rho_{t_0,r}(x,\xi)\varphi(r)\,dx\,d\xi\bigg|_{r=t_0}^{r=t_1}
        \\
        &-\!
        \int_{t_0}^{t_1}\!\!\int_{Q\times\R}\!\!\!\!\!\!\!\!m|\xi|^{m-1}\chi(x,\xi,r)\,\Delta_x\rho_{t_0,r}(x,\xi)\,\varphi(r)\,dx\,d\xi\,dr
        \\
        &+\!
        \int_{t_0}^{t_1}\!\!\int_{Q\times\R}\!\!\!\!\!\!\!\!\!\!\left(p(x,\xi,r)+q(x,\xi,r)\right)\partial_{\xi}\rho_{t_0,r}(x,\xi)\,\varphi(r)\,dx\,d\xi\,dr.
    \end{aligned}
\end{align}
\end{remark}

\begin{remark}
An easy approximation argument shows that in fact we can test equation \eqref{definition/pathwise kinetic solution/formula 1} against the transport along characteristics of any $\rho_0\in C^2(Q\times\R)$ as long as it has bounded first and second derivatives and it satisfies $\rho_0\in C_c(Q\times[-M,M])$ for each $M>0$, i.e. it is compactly supported in $Q$ locally in $\xi$.
\end{remark}

Finally, we observe that the regularity and zero-trace condition imposed by Definition \ref{definition/pathwise kinetic solution}.(i) imply that every pathwise kinetic solution satisfies the following integration by parts formula.
The proof is a small modification of \cite[Lemma 3.6]{fehrman-gess-Wellposedness-of-Nonlinear-Diffusion-Equations-with-Nonlinear-Conservative-Noise} and is omitted.
\begin{lemma}\label{lemma/integration by parts formula}
Let $u$ be a pathwise kinetic solution of \eqref{formula/stochastic porous media equation} in the sense of Defintion \ref{definition/pathwise kinetic solution} and let $\chi(x,\xi,t)$ be the associated kinetic function.
For each $\psi\in C_c^{\infty}(Q\times\R\times[0,\infty))$ and each $t_0\geq0$, we have 
\begin{align}\label{formula/integration by parts formula}
    \int_0^{t_0}\int_{Q\times\R}\frac{m+1}{2}|\xi|^{\frac{m-1}{2}}\chi(x,\xi,r)\nabla\psi(x,\xi,r)\,dx \,d\xi\,dr
    =
    -\int_0^{t_0}\int_Q\nabla u^{\left[\frac{m+1}{2}\right]}\,\psi(x,u(x,r),r)\,dx\,dr.
\end{align}
\end{lemma}



\section{Uniqueness of pathwise kinetic solutions}\label{section/uniqueness of pathwise kinetic solutions}

In this section we prove that pathwise kinetic solutions are unique.
In order to motivate and give an overview of the actual proof, we sketch the uniqueness argument in the case of the deterministic porous media equation
\begin{align}[left ={\empheqlbrace}]\label{formula/deterministic porous media equation}
    \begin{aligned}
        &\partial_tu=\Delta u^{[m]} &\text{in } Q\times(0,\infty),
        \\
        &u=0 &\text{on } \partial Q\times (0,\infty),
        \\
        &u=u_0 &\text{on }  Q\times {0}. 
    \end{aligned}
\end{align}
The corresponding kinetic formulation is 
\begin{align}[left ={\empheqlbrace}]\label{formula/deterministic porous media equation kinetic formulation}
    \begin{aligned}
        &\partial_t\chi=m|\xi|^{m-1}\Delta_x\chi+\partial_{\xi}(p+q) &\text{in } Q\times\R\times(0,\infty),
        \\
        &\chi=0 &\text{on } \partial Q\times\R\times (0,\infty),
        \\
        &\chi=\bar{\chi}(u_0;\xi) &\text{on }  Q\times\R\times \{0\},
    \end{aligned}
\end{align}
where $p\geq0$ is the nonnegative entropy defect measure and $q$ is the parabolic defect measure defined by \eqref{formula/parabolic defect measure}.

In this setting, the following proof of uniqueness is due to Chen and Perthame \cite{chen-perthame-Well-posedness-for-non-isotropic-degenerate-parabolic-hyperbolic-equations}.
Suppose that $u^1$ and $u^2$ are two kinetic solutions of \eqref{formula/deterministic porous media equation} in the sense that the associated kinetic functions $\chi^1$ and $\chi^2$ solve \eqref{formula/deterministic porous media equation kinetic formulation}.
Properties of the kinetic function yield the identity
\begin{align}\label{proof/formal proof of uniqueness/formula 1}
    \begin{aligned}
        \int_Q\left|u^1-u^2\right|\,dx
        &=\int_\R\int_Q\left|\chi^1-\chi^2\right|^2\,dx\,d\xi
        =\int_\R\int_Q\left|\chi^1\right|+\left|\chi^2\right|-2\chi^1\chi^2\,dx\,d\xi
        \\
        &=\int_\R\int_Q\chi^1\sgn(\xi)+\chi^2\sgn(\xi)-2\chi^1\chi^2\,dx\,d\xi.
    \end{aligned}
\end{align}
Now we formally take the derivative in time of this equation and apply equation \eqref{formula/deterministic porous media equation kinetic formulation}.
Then we integrate by parts in space. Using the distributional equalities
\begin{equation}\label{proof/formal proof of uniqueness/formula 2}
    \partial_\xi\chi^i(x,\xi,t)=\delta_0(\xi)-\delta_0(\xi-u^i(x,t))\,\text{ and }\nabla_x\chi^i=\delta_0(\xi-u^i(x,t))\nabla u^i(x,t),
\end{equation}
for $i=1,2$, and exploiting the zero boundary conditions we eventually get
\begin{align}\label{proof/formal proof of uniqueness/formula 3}
    \begin{aligned}
        \partial_t\int_Q|u^1-u^2|\,dx
        =
        &\frac{16m}{(m+1)^2}\int_\R\int_Q\delta_0(\xi-u^1(x,t))\delta_0(\xi-u^2(x,t))\nabla(u^1)^{\left[\frac{m+1}{2}\right]}\nabla(u^2)^{\left[\frac{m+1}{2}\right]}
        \\ 
        &-2\int_\R\int_Q\delta_0(\xi-u^1(x,t))\left(p^2(x,\xi,t)+q^2(x,\xi,t)\right)
        \\
        &-2\int_\R\int_Q\delta_0(\xi-u^2(x,t))\left(p^1(x,\xi,t)+q^1(x,\xi,t)\right).
    \end{aligned}
\end{align}
Applications of H\"older's inequality and Young's inequality, together with the definition of the parabolic defect measure and the nonnegativity of the entropy defect measure, prove the right-hand side of \eqref{proof/formal proof of uniqueness/formula 3} is nonpositive.
Integrating in time then completes the proof of uniqueness.

The formal argument leading to \eqref{proof/formal proof of uniqueness/formula 3} provides the outline for the proof of Theorem \ref{theorem/uniqueness of pathwise kinetic solutions for nonnegative initial data}.
However, even to justify the formal computation, care must be taken to avoid the product of $\delta$-distributions and exploit the boundary conditions.
This is achieved by regularizing the sgn and kinetic functions in the space and velocity variables.
Additional error terms arise due to the transport of test functions by the inverse characteristics, which are handled using a time-splitting argument that relies crucially on the conservative structure of the equation.

In this setting, the main difficulty with respect to the case with periodic boundary conditions comes from the class of admissible test functions for which equation \eqref{definition/pathwise kinetic solution/formula 1} actually holds.
Indeed, equation \eqref{definition/pathwise kinetic solution/formula 1} shall play the same role as equation \eqref{formula/deterministic porous media equation kinetic formulation} in the outline of the proof sketched above.
Imposing Dirichlet boundary conditions implies that in \eqref{definition/pathwise kinetic solution/formula 1} the admissible test functions need to be compactly supported within the domain $Q$.
In turn, when writing down the rigorous version of the argument above, we shall need to introduce a cutoff function so as to effectively exploit \eqref{definition/pathwise kinetic solution/formula 1}.
To counter the effect of the transport along characteristics imposed by \eqref{definition/pathwise kinetic solution/formula 1}, we consider a cutoff function that has already been transported along the backward characteristics. 
Having introduced a cutoff, new boundary error terms arise to avoid the product of $\delta$-distributions.
We shall have to choose the speed the cutoff function is approaching the boundary with according to the diffusion regime $m\in(0,\infty)$.
This is needed to exploit the zero boundary conditions and the Sobolev regularity of $u^{[m]}$ (cf. Lemma \ref{lemma/sobolev regularity of u^m} below) in order to control the derivatives of the cutoff.

We can now prove the uniqueness of pathwise kinetic solutions.
Besides the aforementioned regularization, time-splitting and cutoff procedures, the rigorous justification of the argument above relies on sharp estimates for singular moments of the parabolic and entropy defect measures.
These estimates are postponed to Proposition \ref{proposition/singular moments for defect measures d in (0,1)} and \ref{proposition/singular moments for defect measures d=1} below.

\begin{remark}
For the remainder of the paper, after applying the integration by parts formula \eqref{formula/integration by parts formula}, we will frequently encounter derivatives of functions $f(x,\xi,r):Q\times\R\times[0,\infty)\to\R$ evaluated at $\xi=u(x,r)$.
In order to simplify the notation, we make the convention that
\begin{equation}\label{convention on derivatives}
    \nabla_xf(x,u(x,r),r)=\nabla_xf(x,\xi,r)_{|\xi=u(x,r)}, 
\end{equation}
and analogous conventions for all possible derivatives.
That is, in every case, the notation indicates the derivative of $f$ evaluated at $(x,u(x,r),r)$ as opposed to the derivative of the full composition.
\end{remark}

\begin{proof}[\textbf{Proof of Theorem \ref{theorem/uniqueness of pathwise kinetic solutions for nonnegative initial data}}]
\hfill

\noindent
The proof will proceed in 10 Steps.
In Step 0 we introduce the notation and the machinery needed in the proof.
Step 1 presents the scheme for the rigorous justification of the formal proof \eqref{proof/formal proof of uniqueness/formula 1}-\eqref{proof/formal proof of uniqueness/formula 3} outlined above, which consists of a time-splitting argument, a cutoff procedure in space, and a regularization in the space and velocity variables.
Step 2 analyzes the terms of \eqref{proof/formal proof of uniqueness/formula 1} involving the $\sgn$ function, and Step 3 considers the mixed term.
In Step 4 we rigorously observe the cancellation coming from the parabolic defect measure we formally noticed in \eqref{proof/formal proof of uniqueness/formula 3}.
In Step 5 we justify the cancellation coming from the integration by parts and the application of the zero boundary conditions formally performed in \eqref{proof/formal proof of uniqueness/formula 2}-\eqref{proof/formal proof of uniqueness/formula 3}.
In Step 6 we make the point and gather the error terms produced by the rigorous computations performed up to this point.
In Step 7 we analyze the internal error terms coming from the transport along characteristics of the $\sgn$ and kinetic functions.
Step 8 tackles the boundary error terms produced by the rigorous application of the integration by parts and the zero boundary conditions.
Finally, in Step 9, we pass to the limit with respect to the space and velocity regularization first, the space cutoff second, and the time splitting third, and we conclude the proof.

\medskip
\noindent
\textbf{Step 0: Set up.}
Let $u^1$ and $u^2$ be two pathwise kinetic solutions with initial data $u_0^1,u_0^2\in L_+^2(Q)$ respectively.
For each $i=1,2$, we shall write $\chi^i$ for the corresponding kinetic function, and $p^i$ and $q^i$ for the corresponding entropy and parabolic defect measures.
In order to simplify the notation, for each $(x,\xi,r)\in Q\times\R\times [0,\infty)$ we shall use 
\begin{equation} \label{0.1}
    \chi_r^i(x,\xi):=\chi(x,\xi,r),\quad p_r^i(x,\xi):=p^i(x,\xi,r),\quad q_r^i(x,\xi):=q^i(x,\xi,r).
\end{equation}
Moreover, we shall consider the transported kinetic functions
\begin{equation}\label{0.2}
    \Tilde{\chi}_{t,s}^i(x,\xi):=\chi_s^i(X_{t,s}^{x,\xi},\Xi_{t,s}^{x,\xi})\quad\text{for each $(x,\xi,t,s)\in Q\times\R\times[0,\infty)^2$,}
\end{equation}
for the rough characteristics \eqref{formula/forward stochastic characteristics}.

The proof will start with a time-splitting argument made possible by the conservative structure of the equation.
For $i=1,2$, let $\mathcal{N}^i\subseteq(0,\infty)$ be the null set corresponding to $u^i$ from Definition \ref{definition/pathwise kinetic solution}, and define $\mathcal{N}=\mathcal{N}^1\cup\mathcal{N}^2$.
Let $T\in[0,\infty)\setminus\mathcal{N}$ be fixed, but arbitrary.
We shall consider arbitrary partitions $\mathcal{P}=\{0=t_0<t_1<\dots<t_N=T\}$ of $[0,T]$, for any $N\in\N$, with the constraint that $\mathcal{P}\subseteq[0,T]\setminus\mathcal{N}$.
We shall denote by $|\mathcal{P}|=\max_{i}|t_{i+1}-t_i|$ the diameter of the partition.

The imposition of the zero boundary conditions requires us to test the kinetic equation \eqref{definition/pathwise kinetic solution/formula 1} against compactly supported test functions only, and we thus need to introduce a cutoff function.
First, define the signed distance function $d_{\partial Q}:\R^d\to\R$ by
\begin{align}\label{0.3}
    d_{\partial Q}(x)=
    \empheqlbrace
    \begin{aligned}
      &\dist(x,\partial Q)    \quad\quad\,\,\,\text{if } x\in \bar{Q},\\
      &-\dist(x,\partial Q)   \quad\text{ if } x\in Q^c.
    \end{aligned}
\end{align}
Since the domain $Q$ is smooth, this defines a smooth function.
Moreover, for $x\in\R^d$ close enough to $\partial Q$, we have the identity
\begin{equation}\label{0.4}
    \nabla d_{\partial Q}(x)=-\Tilde{n}(x),
\end{equation}
where $\Tilde{n}(x)$ denotes the extended unit outward normal to $\partial Q$.
Namely, for any $x_0\in \partial Q$, let $n(x_0)$ denote the actual unit outward normal to $\partial Q$.
For any $x\in\R^d$ close enough to $\partial Q$ there exists a unique pair $(x^*,t)\in \partial Q\times\R$ such that $x=x^*-tn(x^*)$, precisely given by $t=d_{\partial Q}(x)$ and $x^*$ the (unique) closest point on $\partial Q$, and we set $\Tilde{n}(x):=n(x^*)$.
For each $\ell\in(0,1)$, we shall denote
\begin{equation}\label{0.5}
    Q^{\ell}:=\{x\in Q\mid d_{\partial Q}(x)\leq c_1\ell\} \quad \text{and} \quad Q_{\ell}:=\{x\in Q\mid d_{\partial Q}(x)> c_1\ell\},
\end{equation}
for a constant $c_1=c_1(T,A,z,Q,d,m)>0$, possibly changing throughout the proof, but only depending on the parameters specified.
It is then immediate from the definition that
\begin{equation}\label{0.5bis}
    \meas\big(Q^{\ell}\big)\leq C\ell,
\end{equation}
for a constant $C=C(T,A,z,Q,m)$.
Moreover, we point out the following property, which is a simple consequence of estimate \eqref{formula/characteristics move slower the closer they get to the boundary}: for any $(x,\xi)\in Q\times\R$ and any $s\leq t_0\in[0,T]$, possibly after enlargening the constant $c_1$ in \eqref{0.5},
\begin{equation}
    \label{0.5tris}
    \text{if}\,\,\,Y_{t_0,s}^{x,\xi}\in Q^{\ell},\,\,\text{then}\,\,\, x\in Q^{\ell},\,\,\,\text{and viceversa}.
\end{equation}

Now, for each $\beta\in(0,1)$, we define a cutoff function $\phi_{\beta}=\phi_{\beta,m}$ within $Q$, whose shape depends on the particular diffusion exponent $m\in(0,\infty)$.
Precisely, set $\gamma_m=(m+2)\wedge 3$ and consider the piecewise linear function $\psi_{\beta}^{\circ}:\R\to[0,1]$ given by
\begin{align}\label{0.6}
    \psi_{\beta}^{\circ}(s)=\empheqlbrace
    \begin{aligned}
    &0\qquad\qquad\quad\qquad\,\text{if } s\leq\beta^{\gamma_m},
    \\
    &\beta^{-1}\left(s-\beta^{\gamma_m}\right)\,\,\,\quad\text{if }s\in\left[\beta^{\gamma_m},\beta+\beta^{\gamma_m}\right],
    \\
    &1\qquad\quad\quad\qquad\text{if } s\geq\beta+\beta^{\gamma_m}.
\end{aligned}
\end{align}
Let $\rho_1:\R\to[0,1]$ be a standard $1$-dimensional positive radial mollifier supported in the unitary ball, and set
\begin{equation}
    \label{0.7}
    \psi_{\beta}(s):=\rho_1^{\frac{1}{2}\beta^{\gamma_m}}\!\!\!\!\!*\psi_{\beta}^{\circ}(s) \quad\text{ for $s\in\R$,}
\end{equation}
where $\rho_1^{\frac{1}{2}\beta^{\gamma_m}}$ denotes the mollifier rescaled at order $\frac{1}{2}\beta^{\gamma_m}$.
Finally, we define a cutoff function $\phi_{\beta}:\R^d\to[0,1]$ by setting
\begin{equation}
    \label{0.8}
    \phi_{\beta}(x):=\psi_\beta(d_{\partial Q}(x)).
\end{equation}
This definition guarantees that $\phi_{\beta}\in C_c^{\infty}(Q)$ and it satisfies
\begin{align}\label{0.9}
    \phi_{\beta}(y)=\empheqlbrace
    \begin{aligned}
    &1\quad\text{if } d_{\partial Q}(y)>\beta+\frac{3}{2}\beta^{\gamma_m},
    \\
    &0\quad\text{if } d_{\partial Q}(y)<\frac{1}{2}\beta^{\gamma_m}.
\end{aligned}
\end{align}
Since $\beta+\frac{3}{2}\beta^{\gamma_m}\simeq \beta$, recalling the notation \eqref{0.5}, we shall write that $\phi_{\beta}\equiv1$ in $Q_{\beta}$ and $\phi_{\beta}\equiv0$ in $Q^{\beta^{\gamma_m}}$.
We also point out that, for $k\in\N$, for a constant $C=C(Q)$,
\begin{equation}
    \label{0.10}
    D^k\phi_\beta\equiv0\quad\text{in }Q_\beta\quad\text{ and }\quad |D^k\phi_\beta|\leq C\beta^{-k}\quad\text{in }Q^\beta.
\end{equation}
Moreover, we observe the following expression for the laplacian $\Delta\phi_{\beta}$.
The distributional derivatives of $\psi_{\beta}^{\circ}$ and the mollification imply that
\begin{equation}\label{0.12}
    \ddot{\psi}_{\beta}(s)=\frac{1}{\beta}\rho_1^{\frac{1}{2}\beta^{\gamma_m}}\!\!\big(s-\beta^{\gamma_m}\big)-\frac{1}{\beta}\rho_1^{\frac{1}{2}\beta^{\gamma_m}}\!\!\big(s-(\beta+\beta^{\gamma_m})\big).
\end{equation}
Then, the definition of $\phi_\beta$ and formula \eqref{0.4} yield
\begin{equation}\label{0.13}
    \Delta\phi_{\beta}=\ddot{\psi}_{\beta}(d_{\partial Q}(x))-\dot{\psi}_\beta(d_{\partial Q}(x))\nabla\cdot\Tilde{n}(x).
\end{equation}

Finally, we shall need a regularization procedure.
Let $\rho_1$ and $\rho_d$ be standard $1$-dimensional and $d$-dimensional positive radial mollifiers supported in the unitary ball.
For each $\epsilon\in(0,1)$, denote by $\rho_1^{\epsilon}$ and $\rho_d^{\epsilon}$ their rescaling of order $\epsilon$.
For $i=1,2$, with the notation \eqref{0.5}, for each $(y,\eta)\in Q_{\epsilon}\times\R$ and each $t\leq r\in[0,\infty)$, we define the smoothed transported kinetic functions
\begin{align}\label{0.14}
    \begin{aligned}
        \Tilde{\chi}_{t,r}^{i,\epsilon}(y,\eta):=\left(\Tilde{\chi}_{t_i,r}^i*\rho_d^{\epsilon}\rho_1^{\epsilon}\right)(y,\eta)
        &=\int_{Q\times\R}\chi_r^i(X_{t,r}^{x,\xi},\Xi_{t,r}^{x,\xi})\rho_d^{\epsilon}(x-y)\rho_1^{\epsilon}(\xi-\eta)\,dx\,d\xi
        \\
        &=\int_{Q\times\R}\chi_r^i(x,\xi)\rho_d^{\epsilon}(Y_{r,r-t}^{x,\xi}-y)\rho_1^{\epsilon}(\Pi_{r,r-t}^{x,\xi}-\eta)\,dx\,d\xi,
    \end{aligned}
\end{align}
where, in the last inequality, we used the inverse \\eqref{formula/inverse relation for smooth characteristics} and conservative \eqref{formula/lebesgue measure is preserved by smooth caharacteristics} properties of characteristics.
In particular, for each $(y,\eta)\in Q\times\R$ and $t\leq r\in[0,\infty)$,
\begin{equation}
    \label{0.15}
    \lim_{\epsilon\to0}\Tilde{\chi}_{t,r}^{i,\epsilon}(y,\eta)=\Tilde{\chi}_{t,r}^i(y,\eta).
\end{equation}
To simplify the notation, for $(x,y,\xi,\eta)\in Q^2\times\R^2$ and $t\leq r\in[0,\infty)$, we shall write
\begin{equation}
    \label{0.16}
    \rho_{t,r}^{\epsilon}(x,y,\xi,\eta):=\rho_d^{\epsilon}(Y_{r,r-t}^{x,\xi}-y)\rho_1^{\epsilon}(\Pi_{r,r-t}^{x,\xi}-\eta).
\end{equation}
According to \eqref{formula/transported stochastic test function}, this represents the solution to the rough differential equation \eqref{formula/stochastic underlying transport equation} beginning at time $t$ with initial data $\rho_d^{\epsilon}(\cdot-y)\rho_1^{\epsilon}(\cdot-\eta)$.
We observe that, for each fixed $(y,\eta)\in Q_{\epsilon}\times\R$ and $t\leq t'\in[0,\infty)\setminus\mathcal{N}$, the kinetic equation \eqref{definition/pathwise kinetic solution/formula 1} can be applied to \eqref{0.14}.

In conclusion, we shall also consider the regularized transported $\sgn$ function, for $(y,\eta)\in Q_{\epsilon}\times\R$ and $t\leq r\in[0,\infty)$,
\begin{align}\label{0.17}
    \begin{aligned}
        \Tilde{\sgn}_{t,r}^{\epsilon}(y,\eta):&=\int_{Q\times\R}\sgn(\Xi_{t,r}^{x,\xi})\rho_d^{\epsilon}(x-y)\rho_1^{\epsilon}(\xi-\eta)\,dx\,d\xi
        =\int_{Q\times\R}\sgn(\xi)\rho_{t,r}^{\epsilon}(x,y,\xi,\eta)\,dx\,d\xi.
    \end{aligned}
\end{align}
In fact, formula \eqref{formula/sign of velocity is preserved by smooth characteristics} ensures that $\tilde{\sgn}_{t,r}(x,\xi):=\sgn(\Xi_{t,r}^{x,\xi})=\sgn(\xi)$, so that we find
\begin{align}\label{0.18}
    \begin{aligned}
        \Tilde{\sgn}_{t,r}^{\epsilon}(y,\eta)&=\int_{Q\times\R}\sgn(\xi)\rho_d^{\epsilon}(x-y)\rho_1^{\epsilon}(\xi-\eta)\,dx\,d\xi
        =\int_{\R}\sgn(\xi)\rho_1^{\epsilon}(\xi-\eta)\,d\xi.
    \end{aligned}
\end{align}
Thus $\Tilde{\sgn}_{t,r}^{\epsilon}(y,\eta)=\sgn^{\epsilon}(\eta)$ is just the standard $\epsilon$-mollification of the $\sgn$ function and is independent of $t\leq r\in[0,\infty)$ and $y\in Q$.
It will nevertheless be useful to consider this regularized transported expression, since it will clarify important cancellations in the arguments to follow.

\medskip
\noindent
\textbf{Step 1: Time splitting, cutoff and regularization.}
Let $\mathcal{P}=\{0=t_0<t_1<\dots<t_N=T\}$ be a partition of $[0,T]$ with $\mathcal{P}\subseteq[0,T]\setminus\mathcal{N}$.
Properties of the kinetic function, the definition of the cutoff $\phi_\beta$ and property \eqref{formula/smooth characteristics never leave the domain}, and the conservative \eqref{formula/lebesgue measure is preserved by smooth caharacteristics} and inverse \eqref{formula/inverse relation for smooth characteristics} property of the characteristics imply that
\begin{align}\label{1.1}
    \begin{aligned}
        \int_Q\left|u^1(y,r)-u^2(y,r)\right|\,dy\Big|_{r=0}^{r=T}
        &=
        \sum_{i=0}^{N-1}\int_Q\left|u^1(y,r)-u^2(y,r)\right|\,dy\Big|_{r=t_i}^{r=t_{i+1}}
        \\
        &=
        \sum_{i=0}^{N-1}\int_{Q\times\R}\left|\chi^1_r(y,\eta)-\chi^2_r(y,\eta)\right|^2dy\,d\eta\Big|_{r=t_i}^{r=t_{i+1}}
        \\
        &=
        \lim_{\beta\to0}\sum_{i=0}^{N-1}\int_{Q\times\R}\left|\chi^1_{t_i,r}(y,\eta)-\chi^2_{t_i,r}(y,\eta)\right|^2\phi_{\beta}(Y_{r,r-t_i}^{y,\eta})\,dy\,d\eta\Big|_{r=t_i}^{r=t_{i+1}}
        \\
         &=
        \lim_{\beta\to0}\sum_{i=0}^{N-1}
        \!\int_{Q\times\R}\!\!\!\!\left|\tilde{\chi}^1_{t_i,r}(y,\eta)-\tilde{\chi}^2_{t_i,r}(y,\eta)\right|^2\phi_{\beta}(y)\,dy\,d\eta\Big|_{r=t_i}^{r=t_{i+1}}.
    \end{aligned}
\end{align}
Moreover, unfolding the square and introducing the $\epsilon$-regularization, we get
\begin{align}\label{1.3}
    \begin{aligned}
        \int_Q\big|u^1(y,r)&-u^2(y,r)\big|\,dy\Big|_{r=0}^{r=T}
        \\
        &=
        \lim_{\beta\to0}\sum_{i=0}^{N-1}
        \int_{Q\times\R}\!\!\!\!\big(\tilde{\chi}^1_{t_i,r}\tilde{\sgn}_{t_i,r}+\tilde{\chi}^2_{t_i,r}\tilde{\sgn}_{t,_i,r}-2\tilde{\chi}^1_{t_i,r}\tilde{\chi}^2_{t_i,r}\big)\phi_{\beta}(y)\,dy\,d\eta\Big|_{r=t_i}^{r=t_{i+1}}
        \\
        &= 
        \lim_{\beta\to0}\lim_{\epsilon\to0}\sum_{i=0}^{N-1}
        \!\int_{Q\times\R}\!\!\!\!\big(\tilde{\chi}^{1,\epsilon}_{t_i,r}\tilde{\sgn}^{\epsilon}_{t_i,r}+\tilde{\chi}^{2,\epsilon}_{t_i,r}\tilde{\sgn}^{\epsilon}_{t_i,r}-2\tilde{\chi}^{1,\epsilon}_{t_i,r}\tilde{\chi}^{2,\epsilon}_{t_i,r}\big)\phi_{\beta}(y)\,dy\,d\eta\Big|_{r=t_i}^{r=t_{i+1}}.
    \end{aligned}
\end{align}
The next four steps are devoted to the analysis of the difference
\begin{equation}\label{1.4}
\int_{Q\times\R}\!\!\!\!\big(\tilde{\chi}^{1,\epsilon}_{t_i,r}\tilde{\sgn}^{\epsilon}_{t_i,r}+\tilde{\chi}^{2,\epsilon}_{t_i,r}\tilde{\sgn}^{\epsilon}_{t_i,r}-2\tilde{\chi}^{1,\epsilon}_{t_i,r}\tilde{\chi}^{2,\epsilon}_{t_i,r}\big)\phi_{\beta}(y)\,dy\,d\eta\Big|_{r=t_i}^{r=t_{i+1}},
\end{equation}
for any $i=0,\dots,N-1$ and any $\epsilon,\beta\in(0,1)$, with $\epsilon\ll\beta$ and $\beta\ll|\mathcal{P}|$.
In Step 3 we consider the $\sgn$ term $\tilde{\chi}^{j,\epsilon}_{t_i,r}\tilde{\sgn}^{\epsilon}_{t_i,r}$ for $j=1,2$, and in Step 4 we consider the mixed term $\tilde{\chi}^{1,\epsilon}_{t_i,r}\tilde{\chi}^{2,\epsilon}_{t_i,r}$.
Finally, in Step 5 and 6 we shall observe crucial cancellations among these terms.

\medskip
\noindent
\textbf{Step 2: The $\sgn$ terms.}
We will first analyze the term involving the $\sgn$ function in \eqref{1.4}.
For the convolution kernel \eqref{0.16}, we shall write $(x,\xi)\in Q\times \R$ for the integration variables defining $\tilde{\chi}_{t_i,r}^{1,\epsilon}$ and we shall write $\rho_{t_i,r}^{1,\epsilon}$ for the corresponding convolution kernel.
The kinetic equation \eqref{definition/pathwise kinetic solution/formula 1} and \eqref{0.18} imply that, with the notation from \eqref{0.14} and \eqref{0.16},
\begin{align}\label{2.1}
\begin{aligned}
\int_{Q\times\R}\!\!\!\!\!\!\!\tilde{\chi}^{1,\epsilon}_{t_i,r}(y,\eta)\,\tilde{\sgn}^{\epsilon}_{t_i,r}
&
\phi_{\beta}(y)\,dy\,d\eta\Big|_{r=t_i}^{r=t_{i+1}}
\!\!\!
\\
=&
\int_{t_i}^{t_{i+1}}\!\!\!\int_{Q\times\R}\!\!\left(\int_{Q\times\R}m|\xi|^{m-1}\chi_r^1\,\Delta_x\rho_{t_i,r}^{1,\epsilon}\,dx\,d\xi\right)\tilde{\sgn}_{t_i,r}^{\epsilon}\phi_\beta(y)\, dy \, d\eta\, dr
\\
&-
\!
\int_{t_i}^{t_{i+1}}\!\!\!\int_{Q\times\R}\!\!\left(\int_{Q\times\R}(p_r^1+q_r^1)\,\partial_{\xi}\rho_{t_i,r}^{1,\epsilon}\,dx\,d\xi\right)\tilde{\sgn}_{t_i,r}^{\epsilon}\phi_\beta(y)\,dy\,d\eta\,dr.
\end{aligned}
\end{align}
The first and second term in \eqref{2.1} will be handled separately.
Observe that, from \eqref{0.16}, for each $(x,y,\xi,\eta,r)\in\R^{2d}\times\R^2\times[t_i,\infty)$,
\begin{equation}
    \label{2.2}
    \nabla_{x}\rho_{t_i,r}^{1,\epsilon}=-\nabla_{y}\rho_{t_i,r}^{1,\epsilon}(x,y,\xi,\eta)\,D_xY_{r,r-t_i}^{x,\xi}-\partial_{\eta}\rho_{t_i,r}^{1,\epsilon}(x,y,\xi,\eta)\,\nabla_{x}\Pi_{r,r-t_i}^{x,\xi},
\end{equation}
and 
\begin{equation}
    \label{2.3}
    \partial_{\xi}\rho_{t_i,r}^{1,\epsilon}=-\nabla_{y}\rho_{t_i,r}^{1,\epsilon}(x,y,\xi,\eta)\,\partial_\xi Y_{r,r-t_i}^{x,\xi}-\partial_{\eta}\rho_{t_i,r}^{1,\epsilon}(x,y,\xi,\eta)\,\partial_\xi \Pi_{r,r-t_i}^{x,\xi}.
\end{equation}
For the first term on the right-hand side of \eqref{2.1}, we use formula \eqref{2.2}  and then integrate by parts in the $(y,\eta)$ variables to get
\begin{align}\label{2.4}
\begin{aligned}
&
\hspace{-6mm}
\int_{t_i}^{t_{i+1}}
\!\!\!\int_{Q\times\R}\!\!\left(\int_{Q\times\R}m|\xi|^{m-1}\chi_r^1\,\Delta_x\rho_{t_i,r}^{1,\epsilon}\,dx\,d\xi\right)\tilde{\sgn}_{t_i,r}^{\epsilon}\phi_\beta(y)\, dy \, d\eta\, dr 
\\
=&
\int_{t_i}^{t_{i+1}}\!\!\!\int_{Q\times\R}\!\!\!m|\xi|^{m-1}\chi_r^1\nabla_x\cdot\left(\int_{Q\times\R}\nabla_x\rho_{t_i,r}^{1,\epsilon}\tilde{\sgn}_{t_i,r}^{\epsilon}\phi_\beta(y)\, dy \, d\eta\right)dx\,d\xi\, dr  
\\
=
&
\int_{t_i}^{t_{i+1}}\!\!\!\int_{Q\times\R}\!\!\!\!\!\!\!\!\!\!m|\xi|^{m-1}\chi_r^1\nabla_x\!\cdot\!\left(\int_{Q\times\R}\!\!\!\!\!\!\!\!\rho_{t_i,r}^{1,\epsilon}\big(\nabla_{y}\tilde{\sgn}_{t_i,r}^{\epsilon}D_xY_{r,r-t_i}^{x,\xi}\!\!\!+\!\partial_{\eta}\tilde{\sgn}_{t_i,r}^{\epsilon}\nabla_{x}\Pi_{r,r-t_i}^{x,\xi}\big)\phi_\beta(y)\dsp dy \dsp d\eta\!\right)dx\dsp d\xi\dsp dr  
\\
&+
\int_{t_i}^{t_{i+1}}\!\!\!\int_{Q\times\R}\!\!\!m|\xi|^{m-1}\chi_r^1\nabla_x\cdot\left(\int_{Q\times\R}\rho_{t_i,r}^{1,\epsilon}\tilde{\sgn}_{t_i,r}^{\epsilon}\nabla_y\phi_\beta(y)D_xY_{r,r-t_i}^{x,\xi}\, dy \, d\eta\right)dx\,d\xi\, dr.  
\end{aligned}
\end{align}
For the first term on the right-hand side of \eqref{2.4}, it follows from definition \eqref{0.17} and the computation \eqref{2.2} that, after adding and subtracting the terms $D_{x'}Y^{x'\!,\xi'}_{r,r-t_i}$ and $\nabla_{x'}\Pi^{x'\!,\xi'}_{r,r-t_i}$,
\begin{align}\label{2.5}
\begin{aligned}
\int_{t_i}^{t_{i+1}}
&
\!\!\!\int_{Q\times\R}\!\!\!\!\!\!\!\!m|\xi|^{m-1}\chi_r^1\nabla_x\!\cdot\!\left(\int_{Q\times\R}\!\!\!\!\!\!\rho_{t_i,r}^{1,\epsilon}\big(\nabla_{y}\tilde{\sgn}_{t_i,r}^{\epsilon}D_xY_{r,r-t_i}^{x,\xi}\!\!\!+\!\partial_{\eta}\tilde{\sgn}_{t_i,r}^{\epsilon}\nabla_{x}\Pi_{r,r-t_i}^{x,\xi}\big)\phi_\beta(y)\dsp dy \dsp d\eta\!\right)dx\dsp d\xi\dsp dr  
\\
&
=
IE_i^{\sgn1,1}
-
\int_{t_i}^{t_{i+1}}\!\!\!\int_{Q^3\times\R^3}\!\!\!\!\!m|\xi|^{m-1}\chi_r^1\nabla_x\rho_{t_i,r}^{1,\epsilon}\sgn(\xi')\nabla_{x'}\rho_{t_i,r}^{2,\epsilon}\,\phi_\beta(y)\, dx \dsp dx' dy\dsp d\xi\dsp d\xi' d\eta\dsp dr,  
\end{aligned}
\end{align}
for the internal error term relative to the $\sgn$ term, omitting the integration variables,
\begin{align}\label{2.6}
\begin{aligned}
IE_i^{\sgn1,1}
\!
:=
&
\!\!
\int_{t_i}^{t_{i+1}}\!\!\!\int_{Q^3\times\R^3}\!\!\!\!\!\!\!\!\!\!m|\xi|^{m-1}\chi_r^1\nabla_x\!\cdot\!\left(\rho_{t_i,r}^{1,\epsilon}\sgn(\xi')\phi_\beta(y)
\,\nabla_{y}\rho_{t_i,r}^{2,\epsilon}\big(D_xY_{r,r-t_i}^{x,\xi}\!\!\!-\!D_{x'}Y_{r,r-t_i}^{x'\!\!,\xi'}\big)\right)
\\
&
+
\int_{t_i}^{t_{i+1}}\!\!\!\int_{Q^3\times\R^3}\!\!\!\!\!\!\!\!\!\!m|\xi|^{m-1}\chi_r^1\nabla_x\!\cdot\!\left(\rho_{t_i,r}^{1,\epsilon}\sgn(\xi')\phi_\beta(y)
\,\partial_{\eta}\rho_{t_i,r}^{2,\epsilon}\big(\nabla_x\Pi_{r,r-t_i}^{x,\xi}\!\!\!-\!\nabla_{x'}\Pi_{r,r-t_i}^{x'\!\!,\xi'}\big)\right),
\end{aligned}
\end{align}
and where the last term of \eqref{2.5} vanishes after integrating by parts in the $x'$-variable.
That is,
\begin{align}\label{2.7}
    \begin{aligned}
        \int_{t_i}^{t_{i+1}}\!\!\!\int_{Q^3\times\R^3}\!\!\!\!\!m|\xi|^{m-1}\chi_r^1\nabla_x\rho_{t_i,r}^{1,\epsilon}\sgn(\xi')\nabla_{x'}\rho_{t_i,r}^{2,\epsilon}\,\phi_\beta(y)\dsp dx \dsp dx' dy\dsp d\xi\dsp d\xi' d\eta\dsp dr=0.
    \end{aligned}
\end{align}
For the second term of \eqref{2.4}, we first rewrite
\begin{align}\label{2.8}
    \begin{aligned}
        \int_{t_i}^{t_{i+1}}\!\!\!\int_{Q\times\R}
        &
        \!\!\!m|\xi|^{m-1}\chi_r^1\nabla_x\cdot\left(\int_{Q\times\R}\rho_{t_i,r}^{1,\epsilon}\,\tilde{\sgn}_{t_i,r}^{\epsilon}\nabla_y\phi_\beta(y)D_xY_{r,r-t_i}^{x,\xi}\, dy \, d\eta\right)dx\,d\xi\, dr  
        \\
        &
        =
        BE^{\sgn1,1}_i
        +
        \int_{t_i}^{t_{i+1}}\!\!\!\int_{Q^2\times\R^2}\!\!\!m|\xi|^{m-1}\chi_r^1\nabla_x\rho_{t_i,r}^{1,\epsilon}\,\tilde{\sgn}_{t_i,r}^{\epsilon}\nabla_y\phi_\beta(y)D_xY_{r,r-t_i}^{x,\xi}dx\,dy\,d\xi\,d\eta\,dr,
    \end{aligned}
\end{align}
for the boundary error term relative to the $\sgn$ term
\begin{align}\label{2.9}
    \begin{aligned}
        BE^{\sgn1,1}_i
        =
        \int_{t_i}^{t_{i+1}}\!\!\!\int_{Q^2\times\R^2}\!\!\!m|\xi|^{m-1}\chi_r^1\,\rho_{t_i,r}^{1,\epsilon}\,\tilde{\sgn}_{t_i,r}^{\epsilon}\nabla_y\phi_\beta(y)\Delta_xY_{r,r-t_i}^{x,\xi}dx\,dy\,d\xi\,d\eta\,dr.
    \end{aligned}
\end{align}
For the second term on the right-hand side of \eqref{2.8}, arguing as in \eqref{2.4}-\eqref{2.6}, we use formula \eqref{2.2}, add and subtract the terms $D_{x'}Y^{x'\!,\xi'}_{r,r-t_i}$ and $\nabla_{x'}\Pi^{x'\!,\xi'}_{r,r-t_i}$, and then integrate by parts in the $(y,\eta)$ variables to get
\begin{align}\label{2.10}
    \begin{aligned}
        &
        \hspace{-7mm}
        \int_{t_i}^{t_{i+1}}
        \!\!\!\int_{Q^2\times\R^2}\!\!\!m|\xi|^{m-1}\chi_r^1\nabla_x\rho_{t_i,r}^{1,\epsilon}\tilde{\sgn}_{t_i,r}^{\epsilon}\nabla_y\phi_\beta(y)D_xY_{r,r-t_i}^{x,\xi}\, dy \, d\eta\, dx\,d\xi\, dr
        \\
        =
        &
        \int_{t_i}^{t_{i+1}}\!\!\!\int_{Q^3\times\R^3}\!\!\!\!\!\!\!\!\!\!m|\xi|^{m-1}\chi_r^1\,\rho_{t_i,r}^{1,\epsilon}\rho_{t_i,r}^{2,\epsilon}\sgn(\xi')\,\tr\!\left(\!\big(D_xY_{r,r-t_i}^{x,\xi}\big)^{\!T} D_y^2\phi_\beta(y)D_xY_{r,r-t_i}^{x,\xi}\!\right)\dsp dx \dsp dx' dy\dsp d\xi\dsp d\xi' d\eta\dsp dr
        \\
        &-
        \int_{t_i}^{t_{i+1}}\!\!\!\int_{Q^3\times\R^3}\!\!\!\!\!\!m|\xi|^{m-1}\chi_r^1\,\rho_{t_i,r}^{1,\epsilon}\nabla_{x'}\rho_{t_i,r}^{2,\epsilon}\sgn(\xi')\nabla_y\phi_\beta(y)D_xY_{r,r-t_i}^{x,\xi}\dsp dx \dsp dx' dy\dsp d\xi\dsp d\xi' d\eta\dsp dr
        +
        BE_i^{\sgn1,2},
    \end{aligned}
\end{align}
for the boundary error term, omitting the integration variables,
\begin{align}\label{2.11}
\begin{aligned}
BE_i^{\sgn1,2}
\!
:=
&
\!\!
\int_{t_i}^{t_{i+1}}\!\!\!\int_{Q^3\times\R^3}\!\!\!\!\!\!\!\!\!\!m|\xi|^{m-1}\chi_r^1\,\rho_{t_i,r}^{1,\epsilon}\sgn(\xi')\nabla_y\phi_\beta(y)D_xY_{r,r-t_i}^{x,\xi}
\nabla_{y}\rho_{t_i,r}^{2,\epsilon}\big(D_xY_{r,r-t_i}^{x,\xi}\!\!\!-\!D_{x'}Y_{r,r-t_i}^{x'\!\!,\xi'}\big)
\\
&
+
\int_{t_i}^{t_{i+1}}\!\!\!\int_{Q^3\times\R^3}\!\!\!\!\!\!\!\!\!\!m|\xi|^{m-1}\chi_r^1\,\rho_{t_i,r}^{1,\epsilon}\sgn(\xi')\nabla_y\phi_\beta(y)D_xY_{r,r-t_i}^{x,\xi}
\partial_{\eta}\rho_{t_i,r}^{2,\epsilon}\big(\nabla_x\Pi_{r,r-t_i}^{x,\xi}\!\!\!-\!\nabla_{x'}\Pi_{r,r-t_i}^{x'\!\!,\xi'}\big),
\end{aligned}
\end{align}
and where the second term in \eqref{2.10} vanishes after integrating by parts in the $x'$-variable,
\begin{equation}
    \label{2.12}
    \int_{t_i}^{t_{i+1}}\!\!\!\int_{Q^3\times\R^3}\!\!\!m|\xi|^{m-1}\chi_r^1\,\rho_{t_i,r}^{1,\epsilon}\nabla_{x'}\rho_{t_i,r}^{2,\epsilon}\sgn(\xi')\,\nabla_y\phi_\beta(y)D_xY_{r,r-t_i}^{x,\xi}\dsp dx \dsp dx' dy\dsp d\xi\dsp d\xi' d\eta\dsp dr=0.
\end{equation}

We now consider the second term in \eqref{2.1}.
Using formula \eqref{2.3} and integrating by parts in the $(y,\eta)$ variables, we obtain
\begin{align}\label{2.13}
\begin{aligned}
&\int_{t_i}^{t_{i+1}}\!\!\!\int_{Q\times\R}\!\!\left(\int_{Q\times\R}(p_r^1+q_r^1)\,\partial_{\xi}\rho_{t_i,r}^{1,\epsilon}\,dx\,d\xi\right)\tilde{\sgn}_{t_i,r}^{\epsilon}\phi_\beta(y)\,dy\,d\eta\,dr
\\
&\,\,\,=
BE_i^{\sgn1,3}\!\!\!
+\!\!
\int_{t_i}^{t_{i+1}}\!\!\!\int_{Q^2\times\R^2}\!\!\!\!\!\!\!\!\!\!\!\!(p_r^1+q_r^1)\rho_{t_i,r}^{1,\epsilon}\,\phi_\beta(y)\big(\nabla_y\tilde{\sgn}_{t_i,r}^{\epsilon}\partial_{\xi}Y_{r,r-t_i}^{x,\xi}\!\!+\partial_{\eta}\tilde{\sgn}_{t_i,r}^{\epsilon}\partial_{\xi}\Pi_{r,r-t_i}^{x,\xi}\big)\,dx\,dy\,d\xi\,d\eta\,dr,
\end{aligned}
\end{align}
for the boundary error term
\begin{align}\label{2.14}
\begin{aligned}
BE_i^{\sgn1,3}
:=
\int_{t_i}^{t_{i+1}}\!\!\!\int_{Q^2\times\R^2}(p_r^1+q_r^1)\,\rho_{t_i,r}^{1,\epsilon}\,\tilde{\sgn}_{t_i,r}^{\epsilon}\nabla_y\phi_\beta(y)\partial_{\xi}Y_{r,r-t_i}^{x,\xi}dx\,dy\,d\xi\,d\eta\,dr.
\end{aligned}
\end{align}
For the second term on the right-hand side of \eqref{2.13}, we use definition \eqref{0.17}, add and subtract the derivatives $\partial_{\xi'}Y_{r,r-t_i}^{x'\!,\xi'}$ and $\partial_{\xi'}\Pi_{r,r-t_i}^{x'\!,\xi'}$, and recall formula \eqref{2.3} to write
\begin{align}\label{2.15}
\begin{aligned}
&\int_{t_i}^{t_{i+1}}\!\!\!\int_{Q^2\times\R^2}\!\!\!\!\!\!\!\!(p_r^1+q_r^1)\rho_{t_i,r}^{1,\epsilon}\,\phi_\beta(y)\big(\nabla_y\tilde{\sgn}_{t_i,r}^{\epsilon}\partial_{\xi}Y_{r,r-t_i}^{x,\xi}\!\!+\partial_{\eta}\tilde{\sgn}_{t_i,r}^{\epsilon}\partial_{\xi}\Pi_{r,r-t_i}^{x,\xi}\big)\,dx\,dy\,d\xi\,d\eta\,dr
\\
&\,\,\,=
IE_i^{\sgn1,2}\!\!\!
-
\int_{t_i}^{t_{i+1}}\!\!\!\int_{Q^3\times\R^3}\!\!\!\!\!\!\!\!(p_r^1+q_r^1)\,\rho_{t_i,r}^{1,\epsilon}\,\partial_{\xi'}\rho_{t_i,r}^{2,\epsilon}\,\sgn(\xi')\,\phi_\beta(y)\,dxdx'\,dy\,d\xi\,d\xi'\,d\eta\,dr,
\end{aligned}
\end{align}
for the internal error term, omitting the integration variables,
\begin{align}\label{2.16}
\begin{aligned}
IE_i^{\sgn1,2}
:=&
\int_{t_i}^{t_{i+1}}\!\!\!\int_{Q^3\times\R^3}\!\!\!\!\!\!(p_r^1+q_r^1)\,\rho_{t_i,r}^{1,\epsilon}\,\sgn(\xi')\,\phi_\beta(y)\,\nabla_y\rho_{t_i,r}^{2,\epsilon}\big(\partial_{\xi}Y_{r,r-t_i}^{x,\xi}\!\!-\partial_{\xi'}Y_{r,r-t_i}^{x'\!,\xi'}\big)
\\
&
+
\int_{t_i}^{t_{i+1}}\!\!\!\int_{Q^3\times\R^3}\!\!\!\!\!\!(p_r^1+q_r^1)\,\rho_{t_i,r}^{1,\epsilon}\,\sgn(\xi')\,\phi_\beta(y)\,\partial_{\eta}\rho_{t_i,r}^{2,\epsilon}\big(\partial_{\xi}\Pi_{r,r-t_i}^{x,\xi}\!\!-\partial_{\xi'}\Pi_{r,r-t_i}^{x'\!,\xi'}\big).
\end{aligned}
\end{align}
Moreover, after integrating by parts in the $\xi'$-variable and using the distributional equality $\partial_{\xi'}\sgn(\xi')=2\delta_0(\xi')$, the second term in \eqref{2.15} becomes
\begin{align}\label{2.17}
\begin{aligned}
-
\int_{t_i}^{t_{i+1}}\!\!\!\int_{Q^3\times\R^3}\!\!\!\!\!\!\!\!(p_r^1+q_r^1)
& 
\,\rho_{t_i,r}^{1,\epsilon}\,\partial_{\xi'}\rho_{t_i,r}^{2,\epsilon}\,\sgn(\xi')\,\phi_\beta(y)\,dx\,dx'\,dy\,d\xi\,d\xi'\,d\eta\,dr
\\
&=
2\int_{t_i}^{t_{i+1}}\!\!\!\int_{Q^3\times\R^2}\!\!\!\!\!\!\!\!(p_r^1+q_r^1)\,\rho_{t_i,r}^{1,\epsilon}\,\rho_{t_i,r}^{2,\epsilon}(x',y,0,\eta)\,\phi_\beta(y)\,dx\,dx'\,dy\,d\xi\,d\eta\,dr.
\end{aligned}
\end{align}

Returning to \eqref{2.1}, it follows from \eqref{2.4}, \eqref{2.5}, \eqref{2.7}, \eqref{2.8}, \eqref{2.10}, \eqref{2.12}, \eqref{2.13}, \eqref{2.15} and \eqref{2.17} that
\begin{align}\label{2.18}
\begin{aligned}
\int_{Q\times\R}&\!\!\!\!\!\!\!\tilde{\chi}^{1,\epsilon}_{t_i,r}(y,\eta)\,\tilde{\sgn}^{\epsilon}_{t_i,r}\,\phi_{\beta}(y)\,dy\,d\eta\Big|_{r=t_i}^{r=t_{i+1}}
\\
=&
\,\,IE_i^{\sgn1,1}-IE_i^{\sgn1,2}
+
BE_i^{\sgn1,1}+BE_i^{\sgn1,2}-BE_i^{\sgn1,3}
\\
&
+\!\!\int_{t_i}^{t_{i+1}}\!\!\!\int_{Q^3\times\R^3}\!\!\!\!\!\!\!\!\!\!\!\!\!m|\xi|^{m-1}\chi_r^1\,\rho_{t_i,r}^{1,\epsilon}\rho_{t_i,r}^{2,\epsilon}\sgn(\xi')\,\tr\!\left(\!\big(D_xY_{r,r-t_i}^{x,\xi}\big)^{\!T}\! D_y^2\phi_\beta(y)D_xY_{r,r-t_i}^{x,\xi}\!\right)\dsp dx \dsp dx' dy\dsp d\xi\dsp d\xi' d\eta\dsp dr
\\
&
-2\int_{t_i}^{t_{i+1}}\!\!\!\int_{Q^3\times\R^2}\!\!\!\!\!\!\!\!(p_r^1+q_r^1)\,\rho_{t_i,r}^{1,\epsilon}\,\rho_{t_i,r}^{2,\epsilon}(x',y,0,\eta)\,\phi_\beta(y)\,dx\,dx'\,dy\,d\xi\,d\eta\,dr.
\end{aligned}
\end{align}
Furthermore, identical considerations with $\chi^1$ replaced by $\chi^2$ prove that, after swapping the roles of $(x,\xi)$ and $(x',\xi')$,
\begin{align}\label{2.19}
\begin{aligned}
\int_{Q\times\R}&\!\!\!\!\!\!\!\tilde{\chi}^{2,\epsilon}_{t_i,r}(y,\eta)\,\tilde{\sgn}^{\epsilon}_{t_i,r}\,\phi_{\beta}(y)\,dy\,d\eta\Big|_{r=t_i}^{r=t_{i+1}}
\\
=&
\,\,IE_i^{\sgn2,1}-IE_i^{\sgn2,2}
+
BE_i^{\sgn2,1}+BE_i^{\sgn2,2}-BE_i^{\sgn2,3}
\\
&
+\!\!\int_{t_i}^{t_{i+1}}\!\!\!\int_{Q^3\times\R^3}\!\!\!\!\!\!\!\!\!\!\!\!\!\!m|\xi'|^{m-1}\chi_r^2\,\rho_{t_i,r}^{1,\epsilon}\rho_{t_i,r}^{2,\epsilon}\sgn(\xi)\,\tr\!\left(\!\big(D_{x'}Y_{r,r-t_i}^{x'\!\!,\xi'}\big)^{\!T}\!\! D_y^2\phi_\beta(y)D_{x'}Y_{r,r-t_i}^{x'\!\!,\xi'}\!\right) \!dx \dsp dx' dy\dsp d\xi\dsp d\xi' d\eta\dsp dr
\\
&
-2\int_{t_i}^{t_{i+1}}\!\!\!\int_{Q^3\times\R^2}\!\!\!\!\!\!\!\!(p_r^2+q_r^2)\,\rho_{t_i,r}^{2,\epsilon}\,\rho_{t_i,r}^{1,\epsilon}(x,y,0,\eta)\,\phi_\beta(y)\,dx\,dx'\,dy\,d\xi'\,d\eta\,dr,
\end{aligned}
\end{align}
for internal error terms $IE_i^{\sgn2,1}$, $IE_i^{\sgn2,2}$, and boundary error terms $BE_i^{\sgn2,1}$, $BE_i^{\sgn2,2}$ and $BE_i^{\sgn2,3}$, defined in exact analogy with \eqref{2.6}, \eqref{2.16}, and \eqref{2.9}, \eqref{2.11} and \eqref{2.14} respectively.

\medskip
\noindent
\textbf{Step 3: The mixed term.}
We shall now analyze the mixed term in \eqref{1.4}.
For the convolution kernel \eqref{0.16}, we shall write $(x,\xi)\in Q\times\R$ for the integration variables defining $\tilde{\chi}^{1,\epsilon}_{t_i,r}$, and we shall write $\rho_{t_i,r}^{1,\epsilon}$ for the corresponding convolution kernel.
We shall write $(x',\xi')\in Q\times\R$ for the integration variables defining $\tilde{\chi}^{2,\epsilon}_{t_i,r}$, and $\rho_{t_i,r}^{2,\epsilon}$ for the corresponding convolution kernel.
The kinetic equation \eqref{definition/pathwise kinetic solution/formula 1} implies that
\begin{align}\label{3.1}
\begin{aligned}
\int_{Q\times\R}\!\!\!\!\!\!\!\tilde{\chi}^{1,\epsilon}_{t_i,r}(y,\eta)\,\tilde{\chi}^{2,\epsilon}_{t_i,r}(y,\eta)\,
&
\phi_{\beta}(y)\,dy\,d\eta\Big|_{r=t_i}^{r=t_{i+1}}
\!\!\!
\\
=&
\int_{t_i}^{t_{i+1}}\!\!\!\int_{Q\times\R}\!\!\left(\int_{Q\times\R}m|\xi|^{m-1}\chi_r^1\,\Delta_x\rho_{t_i,r}^{1,\epsilon}\,dx\,d\xi\right)\tilde{\chi}^{2,\epsilon}_{t_i,r}\,\phi_\beta(y)\, dy \, d\eta\, dr
\\
&-
\!
\int_{t_i}^{t_{i+1}}\!\!\!\int_{Q\times\R}\!\!\left(\int_{Q\times\R}(p_r^1+q_r^1)\,\partial_{\xi}\rho_{t_i,r}^{1,\epsilon}\,dx\,d\xi\right)\tilde{\chi}^{2,\epsilon}_{t_i,r}\,\phi_\beta(y)\,dy\,d\eta\,dr
\\
&
+\int_{t_i}^{t_{i+1}}\!\!\!\int_{Q\times\R}\!\!\left(\int_{Q\times\R}m|\xi'|^{m-1}\chi_r^2\,\Delta_{x'}\rho_{t_i,r}^{2,\epsilon}\,dx'd\xi'\right)\tilde{\chi}^{1,\epsilon}_{t_i,r}\,\phi_\beta(y)\, dy \, d\eta\, dr
\\
&-
\!
\int_{t_i}^{t_{i+1}}\!\!\!\int_{Q\times\R}\!\!\left(\int_{Q\times\R}(p_r^2+q_r^2)\,\partial_{\xi'}\rho_{t_i,r}^{2,\epsilon}\,dx'd\xi'\right)\tilde{\chi}^{1,\epsilon}_{t_i,r}\,\phi_\beta(y)\,dy\,d\eta\,dr.
\end{aligned}
\end{align}
The first term in \eqref{3.1} is studied in analogy to \eqref{2.4}-\eqref{2.12}.
Precisely, we use formula \eqref{2.2} and then integrate by parts in the $(y,\eta)$ variables to write
\begin{align}\label{3.2}
\begin{aligned}
&
\hspace{-6mm}
\int_{t_i}^{t_{i+1}}
\!\!\!\int_{Q\times\R}\!\!\left(\int_{Q\times\R}m|\xi|^{m-1}\chi_r^1\,\Delta_x\rho_{t_i,r}^{1,\epsilon}\,dx\,d\xi\right)\tilde{\chi}^{2,\epsilon}_{t_i,r}\,\phi_\beta(y)\, dy \, d\eta\, dr 
\\
=
&
\int_{t_i}^{t_{i+1}}\!\!\!\int_{Q\times\R}\!\!\!\!\!\!\!\!\!\!m|\xi|^{m-1}\chi_r^1\nabla_x\!\cdot\!\left(\int_{Q\times\R}\!\!\!\!\!\!\!\!\rho_{t_i,r}^{1,\epsilon}\big(\nabla_{y}\tilde{\chi}^{2,\epsilon}_{t_i,r}D_xY_{r,r-t_i}^{x,\xi}\!\!\!+\!\partial_{\eta}\tilde{\chi}^{2,\epsilon}_{t_i,r}\nabla_{x}\Pi_{r,r-t_i}^{x,\xi}\big)\phi_\beta(y)\dsp dy \dsp d\eta\!\right)dx\dsp d\xi\dsp dr  
\\
&+
\int_{t_i}^{t_{i+1}}\!\!\!\int_{Q\times\R}\!\!\!m|\xi|^{m-1}\chi_r^1\nabla_x\cdot\left(\int_{Q\times\R}\rho_{t_i,r}^{1,\epsilon}\,\tilde{\chi}^{2,\epsilon}_{t_i,r}\nabla_y\phi_\beta(y)D_xY_{r,r-t_i}^{x,\xi}\, dy \, d\eta\right)dx\,d\xi\, dr.  
\end{aligned}
\end{align}
For the first term on the right-hand side of \eqref{3.2}, it follows from definition \eqref{0.14} and formula \eqref{2.2} that, after adding and subtracting the terms $D_{x'}Y^{x'\!,\xi'}_{r,r-t_i}$ and $\nabla_{x'}\Pi^{x'\!,\xi'}_{r,r-t_i}$,
\begin{align}\label{3.3}
\begin{aligned}
\int_{t_i}^{t_{i+1}}
&
\!\!\!\int_{Q\times\R}\!\!\!\!\!\!\!\!m|\xi|^{m-1}\chi_r^1\nabla_x\!\cdot\!\left(\int_{Q\times\R}\!\!\!\!\!\!\rho_{t_i,r}^{1,\epsilon}\big(\nabla_{y}\tilde{\chi}^{2,\epsilon}_{t_i,r}D_xY_{r,r-t_i}^{x,\xi}\!\!\!+\!\partial_{\eta}\tilde{\chi}^{2,\epsilon}_{t_i,r}\nabla_{x}\Pi_{r,r-t_i}^{x,\xi}\big)\phi_\beta(y)\dsp dy \dsp d\eta\!\right)dx\dsp d\xi\dsp dr  
\\
&
=
IE_i^{\mix1,1}
-
\int_{t_i}^{t_{i+1}}\!\!\!\int_{Q^3\times\R^3}\!\!\!\!\!m|\xi|^{m-1}\chi_r^1\chi_r^2\nabla_x\rho_{t_i,r}^{1,\epsilon}\nabla_{x'}\rho_{t_i,r}^{2,\epsilon}\,\phi_\beta(y)\, dx \dsp dx' dy\dsp d\xi\dsp d\xi' d\eta\dsp dr,  
\end{aligned}
\end{align}
for the internal error term relative to the mixed term, omitting the integration variables,
\begin{align}\label{3.4}
\begin{aligned}
IE_i^{\mix1,1}
\!
:=
&
\!\!
\int_{t_i}^{t_{i+1}}\!\!\!\int_{Q^3\times\R^3}\!\!\!\!\!\!\!\!\!\!m|\xi|^{m-1}\chi_r^1\nabla_x\!\cdot\!\left(\rho_{t_i,r}^{1,\epsilon}\chi_r^2\phi_\beta(y)
\,\nabla_{y}\rho_{t_i,r}^{2,\epsilon}\big(D_xY_{r,r-t_i}^{x,\xi}\!\!\!-\!D_{x'}Y_{r,r-t_i}^{x'\!\!,\xi'}\big)\right)
\\
&
+
\int_{t_i}^{t_{i+1}}\!\!\!\int_{Q^3\times\R^3}\!\!\!\!\!\!\!\!\!\!m|\xi|^{m-1}\chi_r^1\nabla_x\!\cdot\!\left(\rho_{t_i,r}^{1,\epsilon}\chi_r^2\phi_\beta(y)
\,\partial_{\eta}\rho_{t_i,r}^{2,\epsilon}\big(\nabla_x\Pi_{r,r-t_i}^{x,\xi}\!\!\!-\!\nabla_{x'}\Pi_{r,r-t_i}^{x'\!\!,\xi'}\big)\right).
\end{aligned}
\end{align}
For the second term of \eqref{3.2}, we first rewrite
\begin{align}\label{3.5}
    \begin{aligned}
        \int_{t_i}^{t_{i+1}}\!\!\!\int_{Q\times\R}
        &
        \!\!\!m|\xi|^{m-1}\chi_r^1\nabla_x\cdot\left(\int_{Q\times\R}\rho_{t_i,r}^{1,\epsilon}\,\tilde{\chi}_{t_i,r}^{2,\epsilon}\,\nabla_y\phi_\beta(y)D_xY_{r,r-t_i}^{x,\xi}\, dy \, d\eta\right)dx\,d\xi\, dr  
        \\
        &
        =
        BE^{\mix1,1}_i
        +
        \int_{t_i}^{t_{i+1}}\!\!\!\int_{Q^2\times\R^2}\!\!\!m|\xi|^{m-1}\chi_r^1\nabla_x\rho_{t_i,r}^{1,\epsilon}\,\tilde{\chi}_{t_i,r}^{2,\epsilon}\,\nabla_y\phi_\beta(y)D_xY_{r,r-t_i}^{x,\xi}dx\,dy\,d\xi\,d\eta\,dr,
    \end{aligned}
\end{align}
for the boundary error term relative to the mixed term
\begin{align}\label{3.6}
    \begin{aligned}
        BE^{\mix1,1}_i
        =
        \int_{t_i}^{t_{i+1}}\!\!\!\int_{Q^2\times\R^2}\!\!\!m|\xi|^{m-1}\chi_r^1\,\rho_{t_i,r}^{1,\epsilon}\,\tilde{\chi}_{t_i,r}^{2,\epsilon}\,\nabla_y\phi_\beta(y)\,\Delta_xY_{r,r-t_i}^{x,\xi}dx\,dy\,d\xi\,d\eta\,dr.
    \end{aligned}
\end{align}
For the second term on the right-hand side of \eqref{3.5}, arguing as in \eqref{3.2}-\eqref{3.4}, we use formula \eqref{2.2}, add and subtract the terms $D_{x'}Y^{x'\!,\xi'}_{r,r-t_i}$ and $\nabla_{x'}\Pi^{x'\!,\xi'}_{r,r-t_i}$, and then integrate by parts in the $(y,\eta)$ variables to get
\begin{align}\label{3.7}
    \begin{aligned}
        &
        \hspace{-7mm}
        \int_{t_i}^{t_{i+1}}
        \!\!\!\int_{Q^2\times\R^2}\!\!\!m|\xi|^{m-1}\chi_r^1\nabla_x\rho_{t_i,r}^{1,\epsilon}\tilde{\chi}_{t_i,r}^{2,\epsilon}\nabla_y\phi_\beta(y)D_xY_{r,r-t_i}^{x,\xi}\, dy \, d\eta\, dx\,d\xi\, dr
        \\
        =
        &
        \int_{t_i}^{t_{i+1}}\!\!\!\int_{Q^3\times\R^3}\!\!\!\!\!\!\!\!\!\!m|\xi|^{m-1}\chi_r^1\,\chi_r^2\,\rho_{t_i,r}^{1,\epsilon}\rho_{t_i,r}^{2,\epsilon}\tr\!\left(\!\big(D_xY_{r,r-t_i}^{x,\xi}\big)^{\!T} D_y^2\phi_\beta(y)D_xY_{r,r-t_i}^{x,\xi}\!\right)\dsp dx \dsp dx' dy\dsp d\xi\dsp d\xi' d\eta\dsp dr
        \\
        &-
        \int_{t_i}^{t_{i+1}}\!\!\!\int_{Q^3\times\R^3}\!\!\!\!\!\!m|\xi|^{m-1}\chi_r^1\,\chi_r^2\,\rho_{t_i,r}^{1,\epsilon}\nabla_{x'}\rho_{t_i,r}^{2,\epsilon}\nabla_y\phi_\beta(y)D_xY_{r,r-t_i}^{x,\xi}\dsp dx \dsp dx' dy\dsp d\xi\dsp d\xi' d\eta\dsp dr
        +
        BE_i^{\mix1,2},
    \end{aligned}
\end{align}
for the boundary error term, omitting the integration variables,
\begin{align}\label{3.8}
\begin{aligned}
BE_i^{\mix1,2}
\!
:=
&
\!\!
\int_{t_i}^{t_{i+1}}\!\!\!\int_{Q^3\times\R^3}\!\!\!\!\!\!\!\!\!\!m|\xi|^{m-1}\chi_r^1\,\chi_r^2\rho_{t_i,r}^{1,\epsilon}\nabla_y\phi_\beta(y)D_xY_{r,r-t_i}^{x,\xi}
\nabla_{y}\rho_{t_i,r}^{2,\epsilon}\big(D_xY_{r,r-t_i}^{x,\xi}\!\!\!-\!D_{x'}Y_{r,r-t_i}^{x'\!\!,\xi'}\big)
\\
&
+
\int_{t_i}^{t_{i+1}}\!\!\!\int_{Q^3\times\R^3}\!\!\!\!\!\!\!\!\!\!m|\xi|^{m-1}\chi_r^1\,\chi_r^2\,\rho_{t_i,r}^{1,\epsilon}\nabla_y\phi_\beta(y)D_xY_{r,r-t_i}^{x,\xi}
\partial_{\eta}\rho_{t_i,r}^{2,\epsilon}\big(\nabla_x\Pi_{r,r-t_i}^{x,\xi}\!\!\!-\!\nabla_{x'}\Pi_{r,r-t_i}^{x'\!\!,\xi'}\big).
\end{aligned}
\end{align}
In conclusion, using \eqref{3.2}, \eqref{3.3}, \eqref{3.5} and \eqref{3.7}, we rewrite the first term on the right-hand side of \eqref{3.1} as
\begin{align}\label{3.9}
\begin{aligned}
\int_{t_i}^{t_{i+1}}\!\!\!&\int_{Q\times\R}\!\!\left(\int_{Q\times\R}m|\xi|^{m-1}\chi_r^1\,\Delta_x\rho_{t_i,r}^{1,\epsilon}\,dx\,d\xi\right)\tilde{\chi}^{2,\epsilon}_{t_i,r}\,\phi_\beta(y)\, dy \, d\eta\, dr
\\
=&
\,\,IE_i^{\mix1,1}
+
BE_i^{\mix1,1}+BE_i^{\mix1,2}
\\
&
\!\!+\!\!\int_{t_i}^{t_{i+1}}\!\!\!\int_{Q^3\times\R^3}\!\!\!\!\!\!\!\!\!\!\!\!\!m|\xi|^{m-1}\chi_r^1\,\chi_r^2\,\rho_{t_i,r}^{1,\epsilon}\,\rho_{t_i,r}^{2,\epsilon}\,\tr\!\left(\!\big(D_xY_{r,r-t_i}^{x,\xi}\big)^{\!T}\! D_y^2\phi_\beta(y)D_xY_{r,r-t_i}^{x,\xi}\!\right)\dsp dx \dsp dx' dy\dsp d\xi\dsp d\xi' d\eta\dsp dr
\\
&  
\!\!-\!\!\int_{t_i}^{t_{i+1}}\!\!\!\int_{Q^3\times\R^3}\!\!\!\!\!\!\!\!\!\!\!\!\!m|\xi|^{m-1}\chi_r^1\,\chi_r^2\,\rho_{t_i,r}^{1,\epsilon}\,\nabla_{x'}\rho_{t_i,r}^{2,\epsilon}\,\nabla_y\phi_\beta(y)D_xY_{r,r-t_i}^{x,\xi} dx \dsp dx' dy\dsp d\xi\dsp d\xi' d\eta\dsp dr
\\
&  
\!\!-\!\!\int_{t_i}^{t_{i+1}}\!\!\!\int_{Q^3\times\R^3}\!\!\!\!\!\!\!\!\!\!\!\!\!m|\xi|^{m-1}\chi_r^1\,\chi_r^2\,\nabla_x\rho_{t_i,r}^{1,\epsilon}\,\nabla_{x'}\rho_{t_i,r}^{2,\epsilon}\,\phi_\beta(y)\dsp dx \dsp dx' dy\dsp d\xi\dsp d\xi' d\eta\dsp dr.
\end{aligned}
\end{align}
Furthermore, after swapping the roles of $\chi^1$ and $\chi^2$, identical considerations allow us to rewrite the third term on the right-hand side of \eqref{3.1} as
\begin{align}\label{3.10}
\begin{aligned}
\int_{t_i}^{t_{i+1}}\!\!\!&\int_{Q\times\R}\!\!\left(\int_{Q\times\R}m|\xi'|^{m-1}\chi_r^2\,\Delta_{x'}\rho_{t_i,r}^{2,\epsilon}\,dx'd\xi'\right)\tilde{\chi}^{1,\epsilon}_{t_i,r}\,\phi_\beta(y)\, dy \, d\eta\, dr
\\
=&
\,\,IE_i^{\mix2,1}
+
BE_i^{\mix2,1}+BE_i^{\mix2,2}
\\
&
\!\!+\!\!\int_{t_i}^{t_{i+1}}\!\!\!\int_{Q^3\times\R^3}\!\!\!\!\!\!\!\!\!\!\!\!\!m|\xi'|^{m-1}\chi_r^1\,\chi_r^2\,\rho_{t_i,r}^{1,\epsilon}\,\rho_{t_i,r}^{2,\epsilon}\,\tr\!\left(\!\big(D_{x'}Y_{r,r-t_i}^{x'\!\!,\xi'}\big)^{\!T}\! D_y^2\phi_\beta(y)D_{x'}Y_{r,r-t_i}^{x'\!\!,\xi'}\!\right)\dsp dx \dsp dx' dy\dsp d\xi\dsp d\xi' d\eta\dsp dr
\\
&  
\!\!-\!\!\int_{t_i}^{t_{i+1}}\!\!\!\int_{Q^3\times\R^3}\!\!\!\!\!\!\!\!\!\!\!\!\!m|\xi'|^{m-1}\chi_r^1\,\chi_r^2\,\nabla_x\rho_{t_i,r}^{1,\epsilon}\,\rho_{t_i,r}^{2,\epsilon}\,\nabla_y\phi_\beta(y)D_{x'}Y_{r,r-t_i}^{x'\!\!,\xi'} dx \dsp dx' dy\dsp d\xi\dsp d\xi' d\eta\dsp dr
\\
&  
\!\!-\!\!\int_{t_i}^{t_{i+1}}\!\!\!\int_{Q^3\times\R^3}\!\!\!\!\!\!\!\!\!\!\!\!\!m|\xi'|^{m-1}\chi_r^1\,\chi_r^2\,\nabla_x\rho_{t_i,r}^{1,\epsilon}\,\nabla_{x'}\rho_{t_i,r}^{2,\epsilon}\,\phi_\beta(y)\dsp dx \dsp dx' dy\dsp d\xi\dsp d\xi' d\eta\dsp dr,
\end{aligned}
\end{align}
for internal and boundary error terms $IE_i^{\sgn2,1}$, $BE_i^{\mix2,1}$ and $BE_i^{\mix2,2}$ defined in exact analogy to \eqref{3.4}, \eqref{3.6} and \eqref{3.8} respectively.

We now treat the second and fourth term in \eqref{3.1} in analogy to \eqref{2.13}-\eqref{2.16}.
Using formula \eqref{2.3} and integrating by parts in the $(y,\eta)$ variables, we obtain
\begin{align}\label{3.11}
\begin{aligned}
&\int_{t_i}^{t_{i+1}}\!\!\!\int_{Q\times\R}\!\!\left(\int_{Q\times\R}(p_r^1+q_r^1)\,\partial_{\xi}\rho_{t_i,r}^{1,\epsilon}\,dx\,d\xi\right)\tilde{\chi}_{t_i,r}^{2,\epsilon}\phi_\beta(y)\,dy\,d\eta\,dr
\\
&\,\,\,=
BE_i^{\mix1,3}\!\!\!
+\!\!
\int_{t_i}^{t_{i+1}}\!\!\!\int_{Q^2\times\R^2}\!\!\!\!\!\!\!\!\!\!\!\!(p_r^1+q_r^1)\rho_{t_i,r}^{1,\epsilon}\,\phi_\beta(y)\big(\nabla_y\tilde{\chi}_{t_i,r}^{2,\epsilon}\partial_{\xi}Y_{r,r-t_i}^{x,\xi}\!\!+\partial_{\eta}\tilde{\chi}_{t_i,r}^{2,\epsilon}\partial_{\xi}\Pi_{r,r-t_i}^{x,\xi}\big)\,dx\,dy\,d\xi\,d\eta\,dr,
\end{aligned}
\end{align}
for the boundary error term
\begin{align}\label{3.12}
\begin{aligned}
BE_i^{\mix1,3}
:=
\int_{t_i}^{t_{i+1}}\!\!\!\int_{Q^2\times\R^2}(p_r^1+q_r^1)\,\rho_{t_i,r}^{1,\epsilon}\,\tilde{\chi}_{t_i,r}^{2,\epsilon}\nabla_y\phi_\beta(y)\partial_{\xi}Y_{r,r-t_i}^{x,\xi}dx\,dy\,d\xi\,d\eta\,dr.
\end{aligned}
\end{align}
For the second term on the right-hand side of \eqref{3.11}, we use definition \eqref{0.14}, add and subtract the derivatives $\partial_{\xi'}Y_{r,r-t_i}^{x'\!,\xi'}$ and $\partial_{\xi'}\Pi_{r,r-t_i}^{x'\!,\xi'}$, and recall formula \eqref{2.3} to write
\begin{align}\label{3.13}
\begin{aligned}
&\int_{t_i}^{t_{i+1}}\!\!\!\int_{Q^2\times\R^2}\!\!\!\!\!\!\!\!(p_r^1+q_r^1)\rho_{t_i,r}^{1,\epsilon}\,\phi_\beta(y)\big(\nabla_y\tilde{\chi}_{t_i,r}^{2,\epsilon}\partial_{\xi}Y_{r,r-t_i}^{x,\xi}\!\!+\partial_{\eta}\tilde{\chi}_{t_i,r}^{2,\epsilon}\partial_{\xi}\Pi_{r,r-t_i}^{x,\xi}\big)\,dx\,dy\,d\xi\,d\eta\,dr
\\
&\,\,\,=
IE_i^{\mix1,2}\!\!\!
-
\int_{t_i}^{t_{i+1}}\!\!\!\int_{Q^3\times\R^3}\!\!\!\!\!\!\!\!(p_r^1+q_r^1)\,\rho_{t_i,r}^{1,\epsilon}\,\partial_{\xi'}\rho_{t_i,r}^{2,\epsilon}\,\chi_r^2\,\phi_\beta(y)\,dxdx'\,dy\,d\xi\,d\xi'\,d\eta\,dr,
\end{aligned}
\end{align}
for the internal error term, omitting the integration variables,
\begin{align}\label{3.14}
\begin{aligned}
IE_i^{\mix1,2}
:=&
\int_{t_i}^{t_{i+1}}\!\!\!\int_{Q^3\times\R^3}\!\!\!\!\!\!(p_r^1+q_r^1)\,\rho_{t_i,r}^{1,\epsilon}\,\chi_r^2\,\phi_\beta(y)\,\nabla_y\rho_{t_i,r}^{2,\epsilon}\big(\partial_{\xi}Y_{r,r-t_i}^{x,\xi}\!\!-\partial_{\xi'}Y_{r,r-t_i}^{x'\!,\xi'}\big)
\\
&
+
\int_{t_i}^{t_{i+1}}\!\!\!\int_{Q^3\times\R^3}\!\!\!\!\!\!(p_r^1+q_r^1)\,\rho_{t_i,r}^{1,\epsilon}\,\chi_r^2\,\phi_\beta(y)\,\partial_{\eta}\rho_{t_i,r}^{2,\epsilon}\big(\partial_{\xi}\Pi_{r,r-t_i}^{x,\xi}\!\!-\partial_{\xi'}\Pi_{r,r-t_i}^{x'\!,\xi'}\big).
\end{aligned}
\end{align}
For the second term on the right-hand side of \eqref{3.13}, the distributional equality
\begin{equation}
    \label{3.15}
    \partial_{\xi'}\chi^2(x',\xi',r)=\delta_0(\xi')-\delta_0(u^2(x',r)-\xi')\quad\text{ for } (x',\xi',r)\in Q\times\R\times[0,\infty),
\end{equation}
implies that
\begin{align}\label{3.16}
\begin{aligned}
-
\int_{t_i}^{t_{i+1}}\!\!\!&\int_{Q^3\times\R^3}\!\!\!\!\!\!\!\!(p_r^1+q_r^1)
\,\rho_{t_i,r}^{1,\epsilon}\,\partial_{\xi'}\rho_{t_i,r}^{2,\epsilon}\,\chi_r^2\,\phi_\beta(y)\,dx\,dx'\,dy\,d\xi\,d\xi'\,d\eta\,dr
\\
=&
\int_{t_i}^{t_{i+1}}\!\!\!\int_{Q^3\times\R^2}\!\!\!\!\!\!\!\!(p_r^1+q_r^1)\,\rho_{t_i,r}^{1,\epsilon}\,\rho_{t_i,r}^{2,\epsilon}(x',y,0,\eta)\,\phi_\beta(y)\,dx\,dx'\,dy\,d\xi\,d\eta\,dr
\\
&-
\int_{t_i}^{t_{i+1}}\!\!\!\int_{Q^3\times\R^2}\!\!\!\!\!\!\!\!(p_r^1+q_r^1)\,\rho_{t_i,r}^{1,\epsilon}\,\rho_{t_i,r}^{2,\epsilon}(x',y,u^2(x',r),\eta)\,\phi_\beta(y)\,dx\,dx'\,dy\,d\xi\,d\eta\,dr.
\end{aligned}
\end{align}
Hence, returning to \eqref{3.11}, it follows from \eqref{3.13} and \eqref{3.16} that
\begin{align}\label{3.17}
\begin{aligned}
\int_{t_i}^{t_{i+1}}\!\!\!\int_{Q\times\R}
&
\!\!\left(\int_{Q\times\R}(p_r^1+q_r^1)\,\partial_{\xi}\rho_{t_i,r}^{1,\epsilon}\,dx\,d\xi\right)\tilde{\chi}^{2,\epsilon}_{t_i,r}\,\phi_\beta(y)\,dy\,d\eta\,dr
\\
=& IE_i^{\mix1,2}+BE_i^{\mix1,3}
\\
&+
\int_{t_i}^{t_{i+1}}\!\!\!\int_{Q^3\times\R^2}\!\!\!\!\!\!\!\!(p_r^1+q_r^1)\,\rho_{t_i,r}^{1,\epsilon}\,\rho_{t_i,r}^{2,\epsilon}(x',y,0,\eta)\,\phi_\beta(y)\,dx\,dx'\,dy\,d\xi\,d\eta\,dr
\\
&-
\int_{t_i}^{t_{i+1}}\!\!\!\int_{Q^3\times\R^2}\!\!\!\!\!\!\!\!(p_r^1+q_r^1)\,\rho_{t_i,r}^{1,\epsilon}\,\rho_{t_i,r}^{2,\epsilon}(x',y,u^2(x',r),\eta)\,\phi_\beta(y)\,dx\,dx'\,dy\,d\xi\,d\eta\,dr.
\end{aligned}
\end{align}
Furthermore, after swapping the roles of $\chi^1$ and $\chi^2$, virtually identical arguments prove that
\begin{align}\label{3.18}
\begin{aligned}
\int_{t_i}^{t_{i+1}}\!\!\!\int_{Q\times\R}
&
\!\!\left(\int_{Q\times\R}(p_r^2+q_r^2)\,\partial_{\xi'}\rho_{t_i,r}^{2,\epsilon}\,dx'd\xi'\right)\tilde{\chi}^{1,\epsilon}_{t_i,r}\,\phi_\beta(y)\,dy\,d\eta\,dr
\\
=& IE_i^{\mix2,2}+BE_i^{\mix2,3}
\\
&+
\int_{t_i}^{t_{i+1}}\!\!\!\int_{Q^3\times\R^2}\!\!\!\!\!\!\!\!(p_r^2+q_r^2)\,\rho_{t_i,r}^{2,\epsilon}\,\rho_{t_i,r}^{1,\epsilon}(x,y,0,\eta)\,\phi_\beta(y)\,dx\,dx'dy\,d\xi'd\eta\,dr
\\
&-
\int_{t_i}^{t_{i+1}}\!\!\!\int_{Q^3\times\R^2}\!\!\!\!\!\!\!\!(p_r^2+q_r^2)\,\rho_{t_i,r}^{2,\epsilon}\,\rho_{t_i,r}^{1,\epsilon}(x,y,u^1(x,r),\eta)\,\phi_\beta(y)\,dx\,dx'dy\,d\xi'd\eta\,dr,
\end{aligned}
\end{align}
for internal and boundary errors $IE_i^{\mix2,2}$ and $BE_i^{\mix2,3}$ defined in exact analogy to \eqref{3.14} and \eqref{3.12} respectively.

Going back to \eqref{3.1}, using \eqref{3.9}, \eqref{3.10}, \eqref{3.17} and \eqref{3.18}, and rearranging the terms, we obtain
\small
\begin{align}\label{3.19}
\begin{aligned}
\int_{Q\times\R}&\!\!\!\!\!\!\!\tilde{\chi}^{1,\epsilon}_{t_i,r}(y,\eta)\,\tilde{\chi}^{2,\epsilon}_{t_i,r}(y,\eta)\,\phi_{\beta}(y)\,dy\,d\eta\Big|_{r=t_i}^{r=t_{i+1}}
\!\!\!
\\
=&
\sum_{j=1}^2IE_i^{\mix j,1}\!\!-IE_i^{\mix j,2}
\,+\,
\sum_{j=1}^2BE_i^{\mix j,1}\!\!+BE_i^{\mix j,2}\!\!-BE_i^{\mix j,3}
\\
&
+\!\!\int_{t_i}^{t_{i+1}}\!\!\!\int_{Q^3\times\R^3}\!\!\!\!\!\!\!\!\!\!\!\!\!m|\xi|^{m-1}\chi_r^1\,\chi_r^2\,\rho_{t_i,r}^{1,\epsilon}\,\rho_{t_i,r}^{2,\epsilon}\,\tr\!\left(\!\big(D_xY_{r,r-t_i}^{x,\xi}\big)^{\!T}\! D_y^2\phi_\beta(y)D_xY_{r,r-t_i}^{x,\xi}\!\right)\dsp dx \dsp dx' dy\dsp d\xi\dsp d\xi' d\eta\dsp dr
\\
&
+\!\!\int_{t_i}^{t_{i+1}}\!\!\!\int_{Q^3\times\R^3}\!\!\!\!\!\!\!\!\!\!\!\!\!m|\xi'|^{m-1}\chi_r^1\,\chi_r^2\,\rho_{t_i,r}^{1,\epsilon}\,\rho_{t_i,r}^{2,\epsilon}\,\tr\!\left(\!\big(D_{x'}Y_{r,r-t_i}^{x'\!\!,\xi'}\big)^{\!T}\! D_y^2\phi_\beta(y)D_{x'}Y_{r,r-t_i}^{x'\!\!,\xi'}\!\right)\dsp dx \dsp dx' dy\dsp d\xi\dsp d\xi' d\eta\dsp dr
\\
&  
-\!\!\int_{t_i}^{t_{i+1}}\!\!\!\int_{Q^3\times\R^3}\!\!\!\!\!\!\!\!\!\!\!\!\!m|\xi|^{m-1}\chi_r^1\,\chi_r^2\,\rho_{t_i,r}^{1,\epsilon}\,\nabla_{x'}\rho_{t_i,r}^{2,\epsilon}\,\nabla_y\phi_\beta(y)D_xY_{r,r-t_i}^{x,\xi} dx \dsp dx' dy\dsp d\xi\dsp d\xi' d\eta\dsp dr
\\
&  
-\!\!\int_{t_i}^{t_{i+1}}\!\!\!\int_{Q^3\times\R^3}\!\!\!\!\!\!\!\!\!\!\!\!\!m|\xi'|^{m-1}\chi_r^1\,\chi_r^2\,\nabla_x\rho_{t_i,r}^{1,\epsilon}\,\rho_{t_i,r}^{2,\epsilon}\,\nabla_y\phi_\beta(y)D_{x'}Y_{r,r-t_i}^{x'\!\!,\xi'} dx \dsp dx' dy\dsp d\xi\dsp d\xi' d\eta\dsp dr
\\
&  
-\!\!\int_{t_i}^{t_{i+1}}\!\!\!\int_{Q^3\times\R^3}\!\!\!\!\!\!\!\!\!\!\!\!\!m\big(|\xi|^{m-1}+|\xi'|^{m-1}\big)\chi_r^1\,\chi_r^2\,\nabla_x\rho_{t_i,r}^{1,\epsilon}\,\nabla_{x'}\rho_{t_i,r}^{2,\epsilon}\,\phi_\beta(y)\dsp dx \dsp dx' dy\dsp d\xi\dsp d\xi' d\eta\dsp dr  
\\
&
-\int_{t_i}^{t_{i+1}}\!\!\!\int_{Q^3\times\R^2}\!\!\!\!\!\!\!\!(p_r^1+q_r^1)\,\rho_{t_i,r}^{1,\epsilon}\,\rho_{t_i,r}^{2,\epsilon}(x',y,0,\eta)\,\phi_\beta(y)\,dx\,dx'\,dy\,d\xi\,d\eta\,dr
\\
&
-\int_{t_i}^{t_{i+1}}\!\!\!\int_{Q^3\times\R^2}\!\!\!\!\!\!\!\!(p_r^2+q_r^2)\,\rho_{t_i,r}^{2,\epsilon}\,\rho_{t_i,r}^{1,\epsilon}(x,y,0,\eta)\,\phi_\beta(y)\,dx\,dx'dy\,d\xi'd\eta\,dr
\\
&+
\int_{t_i}^{t_{i+1}}\!\!\!\int_{Q^3\times\R^2}\!\!\!\!\!\!\!\!(p_r^1+q_r^1)\,\rho_{t_i,r}^{1,\epsilon}\,\rho_{t_i,r}^{2,\epsilon}(x',y,u^2(x',r),\eta)\,\phi_\beta(y)\,dx\,dx'\,dy\,d\xi\,d\eta\,dr
\\
&+
\int_{t_i}^{t_{i+1}}\!\!\!\int_{Q^3\times\R^2}\!\!\!\!\!\!\!\!(p_r^2+q_r^2)\,\rho_{t_i,r}^{2,\epsilon}\,\rho_{t_i,r}^{1,\epsilon}(x,y,u^1(x,r),\eta)\,\phi_\beta(y)\,dx\,dx'dy\,d\xi'd\eta\,dr.
\end{aligned}
\end{align}
\normalsize
This completes the analysis of the mixed term.

\medskip
\noindent
\textbf{Step 4: Internal term cancellation.}
We now go back to \eqref{1.4}.
We insert formulas \eqref{2.18}, \eqref{2.19} and \eqref{3.19} into \eqref{1.4} to write
\small
\begin{align}\label{4.1}
\begin{aligned}
    &\hspace{-3mm}\int_{Q\times\R}\!\!\!\!\big(\tilde{\chi}^{1,\epsilon}_{t_i,r}\tilde{\sgn}^{\epsilon}_{t_i,r}+\tilde{\chi}^{2,\epsilon}_{t_i,r}\tilde{\sgn}^{\epsilon}_{t_i,r}-2\tilde{\chi}^{1,\epsilon}_{t_i,r}\tilde{\chi}^{2,\epsilon}_{t_i,r}\big)\phi_{\beta}(y)\,dy\,d\eta\Big|_{r=t_i}^{r=t_{i+1}}
    \\
    \hspace{-3mm}=&
    \sum_{j=1}^2 IE_i^{\sgn j,1}-IE_i^{\sgn j,2}-2IE_i^{\mix j,1}+2IE_i^{\mix j,2}
    \\
    &\hspace{-3mm}
    +\sum_{j=1}^2 BE_i^{\sgn j,1}+BE_i^{\sgn j,2}-BE_i^{\sgn j,3}-2BE_i^{\mix j,1}-2BE_i^{\mix j,2}+2BE_i^{\mix j,3}
    \\
    &\hspace{-3mm}
    +\!\!\msp\int_{t_i}^{t_{i+1}}\!\!\!\int_{Q^3\times\R^3}\!\!\!\!\!\!\!\!\!\!\!\!\!\!\!m|\xi|^{m-1}\!\big(\chi_r^1\sgn(\xi')\!\msp-\msp\!2\chi_r^1\chi_r^2\big)\rho_{t_i,r}^{1,\epsilon}\rho_{t_i,r}^{2,\epsilon}\tr\!\left(\!\!\big(D_x\!Y_{r,r-t_i}^{x,\xi}\msp\big)^{\!T}\msp\!\! D_y^2\phi_\beta(y)D_x\!Y_{r,r-t_i}^{x,\xi}\!\msp\right)\! dx \dsp dx' dy\dsp d\xi\dsp d\xi' d\eta\dsp dr
    \\
    &\hspace{-3mm}
    +\!\!\msp\int_{t_i}^{t_{i+1}}\!\!\!\int_{Q^3\times\R^3}\!\!\!\!\!\!\!\!\!\!\!\!\!\!\!m|\xi'|^{m-1}\!\big(\chi_r^2\sgn(\xi)\!\msp-\msp\!2\chi_r^1\chi_r^2\big)\rho_{t_i,r}^{1,\epsilon}\rho_{t_i,r}^{2,\epsilon}\tr\!\left(\!\!\big(D_{x'}\!Y_{r,r-t_i}^{x'\!\!,\xi'}\msp\big)^{\!T}\msp\!\! D_y^2\phi_\beta(y)D_{x'}\!Y_{r,r-t_i}^{x'\!\!,\xi'}\!\msp\right)\! dx \dsp dx' dy\dsp d\xi\dsp d\xi' d\eta\dsp dr
    \\
    &  \hspace{-3mm}
    +2\!\!\int_{t_i}^{t_{i+1}}\!\!\!\int_{Q^3\times\R^3}\!\!\!\!\!\!\!\!\!\!\!\!\!m|\xi|^{m-1}\chi_r^1\,\chi_r^2\,\rho_{t_i,r}^{1,\epsilon}\,\nabla_{x'}\rho_{t_i,r}^{2,\epsilon}\,\nabla_y\phi_\beta(y)D_xY_{r,r-t_i}^{x,\xi} dx \dsp dx' dy\dsp d\xi\dsp d\xi' d\eta\dsp dr
    \\
    &  \hspace{-3mm}
    +2\!\!\int_{t_i}^{t_{i+1}}\!\!\!\int_{Q^3\times\R^3}\!\!\!\!\!\!\!\!\!\!\!\!\!m|\xi'|^{m-1}\chi_r^1\,\chi_r^2\,\nabla_x\rho_{t_i,r}^{1,\epsilon}\,\rho_{t_i,r}^{2,\epsilon}\,\nabla_y\phi_\beta(y)D_{x'}Y_{r,r-t_i}^{x'\!\!,\xi'} dx \dsp dx' dy\dsp d\xi\dsp d\xi' d\eta\dsp dr
    \\
    &  \hspace{-3mm}
    +2\!\!\int_{t_i}^{t_{i+1}}\!\!\!\int_{Q^3\times\R^3}\!\!\!\!\!\!\!\!\!\!\!\!\!m\big(|\xi|^{m-1}+|\xi'|^{m-1}\big)\chi_r^1\,\chi_r^2\,\nabla_x\rho_{t_i,r}^{1,\epsilon}\,\nabla_{x'}\rho_{t_i,r}^{2,\epsilon}\,\phi_\beta(y)\dsp dx \dsp dx' dy\dsp d\xi\dsp d\xi' d\eta\dsp dr  
    \\
    & \hspace{-3mm}
    -2\int_{t_i}^{t_{i+1}}\!\!\!\int_{Q^3\times\R^2}\!\!\!\!\!\!\!\!(p_r^1+q_r^1)\,\rho_{t_i,r}^{1,\epsilon}\,\rho_{t_i,r}^{2,\epsilon}(x',y,u^2(x',r),\eta)\,\phi_\beta(y)\,dx\,dx'\,dy\,d\xi\,d\eta\,dr
    \\
    & \hspace{-3mm}
    -2\int_{t_i}^{t_{i+1}}\!\!\!\int_{Q^3\times\R^2}\!\!\!\!\!\!\!\!(p_r^2+q_r^2)\,\rho_{t_i,r}^{2,\epsilon}\,\rho_{t_i,r}^{1,\epsilon}(x,y,u^1(x,r),\eta)\,\phi_\beta(y)\,dx\,dx'dy\,d\xi'd\eta\,dr,
\end{aligned}
\end{align}
\normalsize
where we point out that the terms involving the convolution kernels evaluated at $\xi=0$ or $\xi'=0$ in equations \eqref{2.18} and \eqref{2.19} and equation \eqref{3.19} cancel out with each other.
The aim of this step is to observe an additional cancellation among the residual internal terms in \eqref{4.1}.
Namely, among the last three lines of \eqref{4.1} denoted by
\begin{align}\label{4.2}
\begin{aligned}
    IR_i:=&
    2\int_{t_i}^{t_{i+1}}\!\!\!\int_{Q^3\times\R^3}\!\!\!\!\!\!\!\!m\big(|\xi|^{m-1}+|\xi'|^{m-1}\big)\chi_r^1\,\chi_r^2\,\nabla_x\rho_{t_i,r}^{1,\epsilon}\,\nabla_{x'}\rho_{t_i,r}^{2,\epsilon}\,\phi_\beta(y)\dsp dx \dsp dx' dy\dsp d\xi\dsp d\xi' d\eta\dsp dr  
    \\
    & 
    -2\int_{t_i}^{t_{i+1}}\!\!\!\int_{Q^3\times\R^2}\!\!\!\!\!\!\!\!(p_r^1+q_r^1)\,\rho_{t_i,r}^{1,\epsilon}\,\rho_{t_i,r}^{2,\epsilon}(x',y,u^2(x',r),\eta)\,\phi_\beta(y)\,dx\,dx'\,dy\,d\xi\,d\eta\,dr
    \\
    & 
    -2\int_{t_i}^{t_{i+1}}\!\!\!\int_{Q^3\times\R^2}\!\!\!\!\!\!\!\!(p_r^2+q_r^2)\,\rho_{t_i,r}^{2,\epsilon}\,\rho_{t_i,r}^{1,\epsilon}(x,y,u^1(x,r),\eta)\,\phi_\beta(y)\,dx\,dx'dy\,d\xi'd\eta\,dr.
\end{aligned}
\end{align}
This cancellation is an effect of the integration by parts formula \eqref{formula/integration by parts formula}.
The elementary equality
\begin{equation}
    \label{4.3}
    \left(|\xi|^{\frac{m-1}{2}}-|\xi'|^{\frac{m-1}{2}}\right)^2+2|\xi|^{\frac{m-1}{2}}|\xi'|^{\frac{m-1}{2}}=|\xi|^{m-1}+|\xi'|^{m-1}\quad\text{for }\xi,\xi'\in\R
\end{equation}
allows us to rewrite the first term of \eqref{4.2} as
\begin{align}\label{4.4}
\begin{aligned}
    2\int_{t_i}^{t_{i+1}}
    &
    \!\!\!\int_{Q^3\times\R^3}\!\!\!\!\!\!\!\!m\big(|\xi|^{m-1}+|\xi'|^{m-1}\big)\chi_r^1\,\chi_r^2\,\nabla_x\rho_{t_i,r}^{1,\epsilon}\,\nabla_{x'}\rho_{t_i,r}^{2,\epsilon}\,\phi_\beta(y)\dsp dx \dsp dx' dy\dsp d\xi\dsp d\xi' d\eta\dsp dr  
    \\
    =&
    4m\int_{t_i}^{t_{i+1}}\!\!\!\int_{Q^3\times\R^3}\!\!\!|\xi|^{\frac{m-1}{2}}|\xi'|^{\frac{m-1}{2}}\chi_r^1\,\chi_r^2\,\nabla_x\rho_{t_i,r}^{1,\epsilon}\,\nabla_{x'}\rho_{t_i,r}^{2,\epsilon}\,\phi_\beta(y)\dsp dx \dsp dx' dy\dsp d\xi\dsp d\xi' d\eta\dsp dr
    \\
    &+
    2m\int_{t_i}^{t_{i+1}}\!\!\!\int_{Q^3\times\R^3}\!\!\!\left(|\xi|^{\frac{m-1}{2}}-|\xi'|^{\frac{m-1}{2}}\right)^2\chi_r^1\,\chi_r^2\,\nabla_x\rho_{t_i,r}^{1,\epsilon}\,\nabla_{x'}\rho_{t_i,r}^{2,\epsilon}\,\phi_\beta(y)\dsp dx \dsp dx' dy\dsp d\xi\dsp d\xi' d\eta\dsp dr.
\end{aligned}
\end{align}
For the first term on the right-hand side of \eqref{4.4}, after applying the integration by parts formula \eqref{formula/integration by parts formula} both in the $x$ variable and the $x'$ variable, we have
\begin{align}\label{4.5}
\begin{aligned}
    4m\int_{t_i}^{t_{i+1}}&\!\!\!\int_{Q^3\times\R^3}\!\!\!|\xi|^{\frac{m-1}{2}}|\xi'|^{\frac{m-1}{2}}\chi_r^1\,\chi_r^2\,\nabla_x\rho_{t_i,r}^{1,\epsilon}\,\nabla_{x'}\rho_{t_i,r}^{2,\epsilon}\,\phi_\beta(y)\dsp dx \dsp dx' dy\dsp d\xi\dsp d\xi' d\eta\dsp dr
    \\
    =&
    \frac{16m}{(m+1)^2}\!\int_{t_i}^{t_{i+1}}\!\!\!\int_{Q^3\times\R}\!\!\!\!\!\!\!\!\!\!\nabla(u^1)^{\left[\frac{m+1}{2}\right]}\nabla(u^2)^{\left[\frac{m+1}{2}\right]}\bar{\rho}_{t_i,r}^{1,\epsilon}(x,y,\eta)\,\bar{\rho}_{t_i,r}^{2,\epsilon}(x',y,\eta)\,\phi_{\beta}(y)\dsp dx \dsp dx' dy\dsp d\eta\dsp dr,
\end{aligned}
\end{align}
where, for $j=1,2$ and $(x,y,\eta,r)\in Q^2\times\R\times[t_i,\infty)$, we denote
\begin{equation}\label{4.5bis}
    \Bar{\rho}_{t_i,r}^{j,\epsilon}(x,y,\eta):=\rho_{t_i,r}^{\epsilon}(x,u^j(x,r),y,\eta).
\end{equation}
In turn, Cauchy's inequality, the definition of the parabolic defect measure, and the nonnegativity of the entropy defect measure prove that
\begin{align}\label{4.6}
\begin{aligned}
    4m\int_{t_i}^{t_{i+1}}&\!\!\!\int_{Q^3\times\R^3}\!\!\!|\xi|^{\frac{m-1}{2}}|\xi'|^{\frac{m-1}{2}}\chi_r^1\,\chi_r^2\,\nabla_x\rho_{t_i,r}^{1,\epsilon}\,\nabla_{x'}\rho_{t_i,r}^{2,\epsilon}\,\phi_\beta(y)\, dx \dsp dx' dy\dsp d\xi\dsp d\xi' d\eta\dsp dr
    \\
    \leq&
    2\!\int_{t_i}^{t_{i+1}}\!\!\!\int_{Q^3\times\R^2}\!\!\!\!\!\!\!\!\!\!(p_r^1+q_r^1)\,\rho_{t_i,r}^{1,\epsilon}(x,y,\xi,\eta)\,\bar{\rho}_{t_i,r}^{2,\epsilon}(x',y,\eta)\,\phi_{\beta}(y)\, dx \dsp dx' dy\dsp d\xi\dsp d\eta\dsp dr
    \\
    &+
    2\!\int_{t_i}^{t_{i+1}}\!\!\!\int_{Q^3\times\R^2}\!\!\!\!\!\!\!\!\!\!(p_r^2+q_r^2)\,\bar{\rho}_{t_i,r}^{1,\epsilon}(x,y,\eta)\,\rho_{t_i,r}^{2,\epsilon}(x',y,\xi',\eta)\,\phi_{\beta}(y)\, dx \dsp dx' dy\dsp d\xi' d\eta\dsp dr.
\end{aligned}
\end{align}
Therefore, it follows from \eqref{4.4} and \eqref{4.6} that, for the internal residual term $IR_i$ defined in \eqref{4.2},
\begin{equation}
    \label{4.7}
    \limsup_{\epsilon\to0}IR_i\leq\limsup_{\epsilon\to0}IE_i^{\canc},
\end{equation}
for the internal cancellation error term
\begin{equation}
    \label{4.8}
    IE_i^{\canc}:=2m\int_{t_i}^{t_{i+1}}\!\!\!\int_{Q^3\times\R^3}\!\!\!\left(|\xi|^{\frac{m-1}{2}}-|\xi'|^{\frac{m-1}{2}}\right)^2\chi_r^1\,\chi_r^2\,\nabla_x\rho_{t_i,r}^{1,\epsilon}\,\nabla_{x'}\rho_{t_i,r}^{2,\epsilon}\,\phi_\beta(y)\dsp dx \dsp dx' dy\dsp d\xi\dsp d\xi' d\eta\dsp dr.
\end{equation}

\medskip
\noindent
\textbf{Step 5: Boundary term cancellation.}
The aim of this step is to observe a crucial cancellation between the residual boundary terms on the right-hand side of \eqref{4.1}.
Namely, between the third and fourth lines of \eqref{4.1}, denoted by
\begin{align}\label{5.1}
\begin{aligned}
    \hspace{-3mm}
    BR_i^A=&\int_{t_i}^{t_{i+1}}\!\!\!\int_{Q^3\times\R^3}\!\!\!\!\!\!\!\!\!\!\!\!m|\xi|^{m-1}\big(\chi_r^1\sgn(\xi')\!-\!2\chi_r^1\chi_r^2\big)\rho_{t_i,r}^{1,\epsilon}\,\rho_{t_i,r}^{2,\epsilon}\,\tr\!\left(\!\msp\big(D_x\!Y_{r,r-t_i}^{x,\xi}\big)^{\!T}\msp\! D_y^2\phi_\beta(y)D_x\!Y_{r,r-t_i}^{x,\xi}\!\right)\! 
    \\
    &
    +\int_{t_i}^{t_{i+1}}\!\!\!\int_{Q^3\times\R^3}\!\!\!\!\!\!\!\!\!\!\!\!m|\xi'|^{m-1}\big(\chi_r^2\sgn(\xi)\!-\!2\chi_r^1\chi_r^2\big)\rho_{t_i,r}^{1,\epsilon}\,\rho_{t_i,r}^{2,\epsilon}\,\tr\!\left(\msp\!\big(D_{x'}\!Y_{r,r-t_i}^{x'\!\!,\xi'}\big)^{\!T}\msp\! D_y^2\phi_\beta(y)D_{x'}\!Y_{r,r-t_i}^{x'\!\!,\xi'}\!\right),
\end{aligned}
\end{align}
and the fifth and sixth lines of \eqref{4.1}, denoted by
\begin{align}\label{5.2}
\begin{aligned}
    \hspace{-3mm}
    BR_i^B=&2\int_{t_i}^{t_{i+1}}\!\!\!\int_{Q^3\times\R^3}\!\!\!\!\!\!\!\!\!m|\xi|^{m-1}\chi_r^1\,\chi_r^2\,\rho_{t_i,r}^{1,\epsilon}\,\nabla_{x'}\rho_{t_i,r}^{2,\epsilon}\,\nabla_y\phi_\beta(y)D_xY_{r,r-t_i}^{x,\xi} dx \dsp dx' dy\dsp d\xi\dsp d\xi' d\eta\dsp dr
    \\
    & 
    +2\int_{t_i}^{t_{i+1}}\!\!\!\int_{Q^3\times\R^3}\!\!\!\!\!\!\!\!\!m|\xi'|^{m-1}\chi_r^1\,\chi_r^2\,\nabla_x\rho_{t_i,r}^{1,\epsilon}\,\rho_{t_i,r}^{2,\epsilon}\,\nabla_y\phi_\beta(y)D_{x'}Y_{r,r-t_i}^{x'\!\!,\xi'} dx \dsp dx' dy\dsp d\xi\dsp d\xi' d\eta\dsp dr.
\end{aligned}
\end{align}
We begin by analyzing the first term of $BR_i^B$.
First, we add and subtract $|\xi'|^{m-1}$ and then use the integration by parts formula \eqref{formula/integration by parts formula} in the $x'$ variable and formula \eqref{lemma/sobolev regularity of u^m/ formula 1} from Lemma \ref{lemma/sobolev regularity of u^m} below to get
\begin{align}\label{5.3}
\begin{aligned}
    2\int_{t_i}^{t_{i+1}}&\!\!\!\int_{Q^3\times\R^3}\!\!\!\!\!\!\!\!\!m|\xi|^{m-1}\chi_r^1\,\chi_r^2\,\rho_{t_i,r}^{1,\epsilon}\,\nabla_{x'}\rho_{t_i,r}^{2,\epsilon}\,\nabla_y\phi_\beta(y)D_xY_{r,r-t_i}^{x,\xi} dx \dsp dx' dy\dsp d\xi\dsp d\xi' d\eta\dsp dr
    \\
    =&
    BE_i^{\canc B1}
    -2\int_{t_i}^{t_{i+1}}\!\!\!\int_{Q^3\times\R^2}\!\!\!\!\!\!\!\!\nabla\left(u^2\right)^{[m]}\chi_r^1(x,\xi)\rho_{t_i,r}^{1,\epsilon}\Bar{\rho}_{t_i,r}^{2,\epsilon}\,\nabla_y\phi_\beta(y)D_{x}Y_{r,r-t_i}^{x,\xi} dx \dsp dx' dy\dsp d\xi\dsp d\eta\dsp dr,
\end{aligned}
\end{align}
with the notation \eqref{4.5bis} for $\Bar{\rho}_{t_i,r}^{2,\epsilon}$, for the boundary cancellation error term
\begin{align}\label{5.4}
\begin{aligned}
    \hspace{-2.5mm}
    BE_i^{\canc B1}\!\!:=\!2\!\!\int_{t_i}^{t_{i+1}}\!\!\!\int_{Q^3\times\R^3}\!\!\!\!\!\!\!\!\!\!\!\!\!m\left(|\xi|^{m-1}\!\!\!-|\xi'|^{m-1}\right)\chi_r^1\,\chi_r^2\,\rho_{t_i,r}^{1,\epsilon}\,\nabla_{x'}\rho_{t_i,r}^{2,\epsilon}\,\nabla_y\phi_\beta(y)D_xY_{r,r-t_i}^{x,\xi} dx \dsp dx' dy\dsp d\xi\dsp d\xi' d\eta\dsp dr.
\end{aligned}
\end{align}
For the second term in \eqref{5.3}, recalling the definition \eqref{0.16} of the convolution kernels $\rho_{t_i,r}^{j,\epsilon}$ and the inverse property \eqref{formula/inverse relation for smooth characteristics} of characteristics, we observe
\begin{align}\label{5.5}
\begin{aligned}
    \lim_{\epsilon\to0}2\int_{t_i}^{t_{i+1}}\!\!\!\int_{Q^3\times\R^2}\!\!\!\!\!\!\!\!\!\!\!\!\!\nabla\left(u^2\right)^{[m]}\chi_r^1(x,\xi)\rho_{t_i,r}^{1,\epsilon}(x,y,\xi,\eta)\Bar{\rho}_{t_i,r}^{2,\epsilon}(x',y,\eta)\,\nabla_y\phi_\beta(y)D_{x}Y_{r,r-t_i}^{x,\xi} dx \dsp dx' dy\dsp d\xi\dsp d\eta\dsp dr
    \\
    =
    2\int_{t_i}^{t_{i+1}}\!\!\!\int_{Q}\!\!\nabla(u^2)^{[m]}\,\chi_r^1(x',u^2(x',r))\,\nabla_y\phi_\beta\msp\big(Y_{r,r-t_i}^{x'\!\!,u^2}\big)\,D_{x'}Y_{r,r-t_i}^{x'\!\!,u^2} dx' \dsp dr,
\end{aligned}
\end{align}
where we recall the convention \eqref{convention on derivatives}.
The analysis of the second line of \eqref{5.2} is virtually identical tp \eqref{5.3}-\eqref{5.5}, simply swapping the roles of $u^1$ and $u^2$, and it produces a boundary error term $BE_i^{\canc B2}$ defined in exact analogy to \eqref{5.4}.
Then, using \eqref{5.3} and \eqref{5.5}, and the analogous formulas for the second line of \eqref{5.2}, for the boundary residue $BR_i^B$ defined in \eqref{5.2} we conclude that
\begin{align}\label{5.8}
\begin{aligned}
    \limsup_{\epsilon\to0} BR_i^B\leq &\,\,\,\,\limsup_{\epsilon\to0}\,\, BE_i^{\canc B1}\!\!+\!BE_i^{\canc B2}
    \\
    &\,-2\int_{t_i}^{t_{i+1}}\!\!\!\int_{Q}\!\!\nabla(u^2)^{[m]}\,\chi_r^1(x',u^2(x',r))\,\nabla_y\phi_\beta\msp\big(Y_{r,r-t_i}^{x'\!\!,u^2}\big)\,D_{x'}Y_{r,r-t_i}^{x'\!\!,u^2} dx' \dsp dr
    \\
    &\,-2\int_{t_i}^{t_{i+1}}\!\!\!\int_{Q}\!\!\nabla(u^1)^{[m]}\,\chi_r^2(x,u^1(x,r))\,\nabla_y\phi_\beta\msp\big(Y_{r,r-t_i}^{x,u^1}\big)\,D_{x}Y_{r,r-t_i}^{x,u^1} dx \dsp dr.
\end{aligned}
\end{align}

We now consider the residual term $BR_i^A$, defined in \eqref{5.1}.
First of all, we notice that, upon swapping the roles of the variables $(x,\xi)$ and $(x',\xi')$ in the second integral, this term is rewritten as
\begin{align}\label{5.9}
\begin{aligned}
    BR_i^A\!\!=\!\!\int_{t_i}^{t_{i+1}}\!\!\!\int_{Q^3\times\R^3}\!\!\!\!\!\!\!\!\!\!\!\!\!\!\!m|\xi|^{m-1}\big(\chi_r^1(x,\xi)\sgn(\xi')
    &
    +\chi_r^2(x,\xi)\sgn(\xi')\!-\!2\chi_r^1(x,\xi)\chi_r^2(x',\xi')\!-\!2\chi_r^1(x',\xi')\chi_r^2(x,\xi)\big)
    \\
    &
    \cdot
    \rho_{t_i,r}^{1,\epsilon}\,\rho_{t_i,r}^{2,\epsilon}\,\tr\!\left(\!\msp\big(D_x\!Y_{r,r-t_i}^{x,\xi}\big)^{\!T}\msp\! D_y^2\phi_\beta(y)D_x\!Y_{r,r-t_i}^{x,\xi}\!\right) dx\dsp dx'dy\dsp d\xi\dsp d\xi'd\eta\dsp dr. \nonumber
\end{aligned}
\end{align}
Then, recalling the definition \eqref{0.16} of the convolution kernels and the inverse property \eqref{formula/inverse relation for smooth characteristics} of the characteristics, and using properties of the kinetic function, we observe that
\begin{align}\label{5.10}
\begin{aligned}
    \hspace{-3mm}
    \lim_{\epsilon\to0}BR_i^A
    =&
    \int_{t_i}^{t_{i+1}}\!\!\!\int_{Q\times\R}\!\!\!\!\!\!\!\!\!m|\xi|^{m-1}\big(\chi_r^1(x,\xi)\sgn(\xi)+\chi_r^2(x,\xi)\sgn(\xi)\!-\!4\chi_r^1(x,\xi)\chi_r^2(x,\xi)\big)
    \\
    &\qquad\qquad\qquad\qquad\
    \cdot
    \tr\!\left(\!\msp\big(D_x\!Y_{r,r-t_i}^{x,\xi}\big)^{\!T}\msp\! D_y^2\phi_\beta\msp\big(Y_{r,r-t_i}^{x,\xi}\big)D_x\!Y_{r,r-t_i}^{x,\xi}\!\right)dx\,d\xi\,dr
    \\
    =&
    \int_{t_i}^{t_{i+1}}\!\!\!\int_{Q\times\R}\!\!\!\!\!\!\!\!\!m|\xi|^{m-1}\left|\chi_r^1(x,\xi)-\chi_r^2(x,\xi)\right|^2\tr\!\left(\!\msp\big(D_x\!Y_{r,r-t_i}^{x,\xi}\big)^{\!T}\msp\! D_y^2\phi_\beta\msp\big(Y_{r,r-t_i}^{x,\xi}\big)D_x\!Y_{r,r-t_i}^{x,\xi}\!\right)dx\,d\xi\,dr
    \\
    &
    -2\int_{t_i}^{t_{i+1}}\!\!\!\int_{Q\times\R}\!\!\!\!\!\!\!\!\!m|\xi|^{m-1}\chi_r^1(x,\xi)\chi_r^2(x,\xi)\,\tr\!\left(\!\msp\big(D_x\!Y_{r,r-t_i}^{x,\xi}\big)^{\!T}\msp\! D_y^2\phi_\beta\msp\big(Y_{r,r-t_i}^{x,\xi}\big)D_x\!Y_{r,r-t_i}^{x,\xi}\!\right)dx\,d\xi\,dr.
\end{aligned}
\end{align}
For the second term on the right-hand side of \eqref{5.10}, performing an elementary computation, we obtain
\begin{align}\label{5.11}
\begin{aligned}
    -2
    &
    \int_{t_i}^{t_{i+1}}\!\!\!\int_{Q\times\R}\!\!\!\!\!\!\!\!\!m|\xi|^{m-1}\chi_r^1(x,\xi)\chi_r^2(x,\xi)\,\tr\!\left(\!\msp\big(D_x\!Y_{r,r-t_i}^{x,\xi}\big)^{\!T}\msp\! D_y^2\phi_\beta\msp\big(Y_{r,r-t_i}^{x,\xi}\big)D_x\!Y_{r,r-t_i}^{x,\xi}\!\right)dx\,d\xi\,dr
    \\
    &=
    BE_i^{\canc A,1}\!
    -2\int_{t_i}^{t_{i+1}}\!\!\!\int_{Q\times\R}\!\!\!\!\!\!\!\!\!m|\xi|^{m-1}\chi_r^1(x,\xi)\chi_r^2(x,\xi)\,\nabla_x\!\cdot\!\left(\msp\! \nabla_y\phi_\beta\msp\big(Y_{r,r-t_i}^{x,\xi}\big)D_x\!Y_{r,r-t_i}^{x,\xi}\!\right)dx\,d\xi\,dr,
\end{aligned}
\end{align}
for the boundary cancellation error term
\begin{align}\label{5.12}
\begin{aligned}
    BE_i^{\canc A,1}:=2\int_{t_i}^{t_{i+1}}\!\!\!\int_{Q\times\R}\!\!\!\!\!\!\!\!\!m|\xi|^{m-1}\chi_r^1(x,\xi)\chi_r^2(x,\xi)\, \nabla_y\phi_\beta\msp\big(Y_{r,r-t_i}^{x,\xi}\big)\Delta_x\!Y_{r,r-t_i}^{x,\xi}dx\,d\xi\,dr.
\end{aligned}
\end{align}
Furthermore, for the second term in \eqref{5.11}, using the integration by parts formula \eqref{formula/integration by parts formula}, Lemma \ref{lemma/sobolev regularity of u^m} below and the product rule for derivatives, we have
\begin{align}\label{5.13}
\begin{aligned}
    -2\int_{t_i}^{t_{i+1}}\!\!\!\int_{Q\times\R}\!\!\!\!\!\!\!\!\!m
    &
    |\xi|^{m-1}\chi_r^1(x,\xi)\chi_r^2(x,\xi)\,\nabla_x\!\cdot\!\left(\msp\! \nabla_y\phi_\beta\msp\big(Y_{r,r-t_i}^{x,\xi}\big)D_x\!Y_{r,r-t_i}^{x,\xi}\!\right)dx\,d\xi\,dr
    \\
    =&
    2\int_{t_i}^{t_{i+1}}\!\!\!\int_{Q}\nabla\left(u^1\right)^{[m]}\chi_r^2(x,u^1(x,r))\,\nabla_y\phi_\beta\msp\big(Y_{r,r-t_i}^{x,u^1}\big)D_x\!Y_{r,r-t_i}^{x,u^1}dx\,dr
    \\
    &+
    2\int_{t_i}^{t_{i+1}}\!\!\!\int_{Q}\nabla\left(u^2\right)^{[m]}\chi_r^1(x,u^2(x,r))\,\nabla_y\phi_\beta\msp\big(Y_{r,r-t_i}^{x,u^2}\big)D_x\!Y_{r,r-t_i}^{x,u^2}dx\,dr,
\end{aligned}
\end{align}
which is the exact opposite of the last two terms on the right-hand side of \eqref{5.8}.

Finally, we analyze the first term on the right-hand side of \eqref{5.10}.
After adding and subtracting the identity matrix $\id_d$ to $D_xY_{r,r-t_i}^{x,\xi}$ twice, we obtain
\begin{align}\label{5.14}
\begin{aligned}
    \int_{t_i}^{t_{i+1}}\!\!\!\int_{Q\times\R}\!\!\!\!\!\!\!\!\!m|\xi|^{m-1}\left|\chi_r^1(x,\xi)-\chi_r^2(x,\xi)\right|^2\tr\!\left(\!\msp\big(D_x\!Y_{r,r-t_i}^{x,\xi}\big)^{\!T}\msp\! D_y^2\phi_\beta\msp\big(Y_{r,r-t_i}^{x,\xi}\big)D_x\!Y_{r,r-t_i}^{x,\xi}\!\right)dx\,d\xi\,dr
    \\
    =BE_i^{\canc A,2}\!\!
    +
    \int_{t_i}^{t_{i+1}}\!\!\!\int_{Q\times\R}\!\!\!\!\!\!\!\!\!m|\xi|^{m-1}\left|\chi_r^1(x,\xi)-\chi_r^2(x,\xi)\right|^2\Delta_y\phi_\beta\msp\big(Y_{r,r-t_i}^{x,\xi}\big)dx\,d\xi\,dr,
\end{aligned}
\end{align}
for the boundary error term
\begin{align}\label{5.15}
\begin{aligned}
    \hspace{-2mm}
    BE_i^{\canc A,2}\!\!:=\!\!\int_{t_i}^{t_{i+1}}&\!\!\!\int_{Q\times\R}\!\!\!\!\!\!\!\!\!\!m|\xi|^{m-1}\left|\chi_r^1-\chi_r^2\right|^2
    \\
    &\cdot\tr\!\left(\!\msp\big(D_x\!Y_{r,r-t_i}^{x,\xi}\!\!\!\!-\!\id_d\!\big)^{\!T}\msp\! D_y^2\phi_\beta\msp\big(Y_{r,r-t_i}^{x,\xi}\msp\big)D_x\!Y_{r,r-t_i}^{x,\xi}\!\!\!+\!D_y^2\phi_\beta\msp\big(Y_{r,r-t_i}^{x,\xi}\msp\big)\!\big(D_x\!Y_{r,r-t_i}^{x,\xi}\!\!\!\!-\!\id_d\!\big)\!\!\right)\!dx\,d\xi\,dr.
\end{aligned}
\end{align}
Heuristically, the second term in \eqref{5.14} should be mostly negative as it consists of positive terms multiplied by the laplacian of a cutoff function, and this is constantly equal to $1$ within the domain Q and bends downward near the boundary decreasing up to $0$.
Rigorously, recalling the explicit choice \eqref{0.8} of the cutoff, we use formulas \eqref{0.12}-\eqref{0.13} and the nonnegativity of the convolution kernel $\rho_1$ in definition \eqref{0.7} to compute
\begin{align}\label{5.16}
\begin{aligned}
    \int_{t_i}^{t_{i+1}}
    \!\!\!\int_{Q\times\R}&\!\!\!\!\!\!\!\!\!m|\xi|^{m-1}\left|\chi_r^1-\chi_r^2\right|^2\Delta_y\phi_\beta\msp\big(Y_{r,r-t_i}^{x,\xi}\big)dx\,d\xi\,dr
    \\
    =&
    -
    \int_{t_i}^{t_{i+1}}\!\!\!\int_{Q\times\R}\!\!\!\!\!\!\!\!\!m|\xi|^{m-1}\left|\chi_r^1-\chi_r^2\right|^2\Dot{\psi}_\beta\msp\big(d_{\partial Q}\big(Y_{r,r-t_i}^{x,\xi}\big)\big)\nabla_x\!\cdot\!\Tilde{n}\msp\big(Y_{r,r-t_i}^{x,\xi}\big) dx\,d\xi\,dr
    \\
    &+
    \int_{t_i}^{t_{i+1}}\!\!\!\int_{Q\times\R}\!\!\!\!\!\!\!\!\!m|\xi|^{m-1}\left|\chi_r^1-\chi_r^2\right|^2\ddot{\psi}_\beta\msp\big(d_{\partial Q}\big(Y_{r,r-t_i}^{x,\xi}\big)\big) dx\,d\xi\,dr
    \\
    =&
    \,\,BE_i^{\canc A,3}
    \\
    &+
    \int_{t_i}^{t_{i+1}}\!\!\!\int_{Q\times\R}\!\!\!\!\!\!\!\!\!m|\xi|^{m-1}\left|\chi_r^1-\chi_r^2\right|^2\bigg(\frac{1}{\beta}\,\rho_1^{\frac{1}{2}\beta^{\gamma_m}}\!\!\Big(\!d_{\partial Q}\big(Y_{r,r-t_i}^{x,\xi}\big)\!-\!\beta^{\gamma_m}\!\Big)
    \\
    &\qquad\qquad\qquad\qquad\qquad\qquad\qquad\qquad
    -\frac{1}{\beta}\,\rho_1^{\frac{1}{2}\beta^{\gamma_m}}\!\!\Big(\msp d_{\partial Q}\big(Y_{r,r-t_i}^{x,\xi}\big)\!-\!(\beta+\beta^{\gamma_m})\msp\Big)\bigg) dx\,d\xi\,dr
    \\
    \leq&
    \,\,
    BE_i^{\canc A,3}\!\!+BE_i^{\canc A,4},
\end{aligned}
\end{align}
for the boundary error terms
\begin{align}\label{5.17}
\begin{aligned}
    BE_i^{\canc A,3}:=
    -
    \int_{t_i}^{t_{i+1}}\!\!\!\int_{Q\times\R}\!\!\!\!\!\!\!\!\!m|\xi|^{m-1}\left|\chi_r^1-\chi_r^2\right|^2\,\Dot{\psi}_\beta\msp\big(d_{\partial Q}\big(Y_{r,r-t_i}^{x,\xi}\big)\big)\,\nabla_x\!\cdot\!\Tilde{n}\msp\big(Y_{r,r-t_i}^{x,\xi}\big)\, dx\,d\xi\,dr,
\end{aligned}
\end{align}
and
\begin{align}\label{5.18}
\begin{aligned}
    BE_i^{\canc A,4}:=
    \int_{t_i}^{t_{i+1}}\!\!\!\int_{Q\times\R}\!\!\!\!\!\!\!\!\!m|\xi|^{m-1}\left|\chi_r^1-\chi_r^2\right|^2\,\frac{1}{\beta}\,\rho_1^{\frac{1}{2}\beta^{\gamma_m}}\!\!\Big(\!d_{\partial Q}\big(Y_{r,r-t_i}^{x,\xi}\big)\!-\!\beta^{\gamma_m}\!\Big)\, dx\,d\xi\,dr.
\end{aligned}
\end{align}

In conclusion, it follows from \eqref{5.8}, \eqref{5.10}, \eqref{5.11}, \eqref{5.13}, \eqref{5.14} and \eqref{5.16} that, for the residual boundary terms $BR_i^A$ and $BR_i^B$ defined in \eqref{5.1} and \eqref{5.2} respectively,
\begin{align}\label{5.19}
\begin{aligned}
    \limsup_{\epsilon\to0}\Big(BR_i^A\!\!+\!BR_i^B\Big)
    \leq \limsup_{\epsilon\to0}\bigg(\sum_{k=1}^4BE_i^{\canc A,k}\!+\sum_{j=1}^2BE_i^{\canc Bj}\bigg).
\end{aligned}
\end{align}

\medskip
\noindent
\textbf{Step 6: The final inequality.}
We now go back go \eqref{4.1}.
Using \eqref{4.2} and \eqref{4.7}, and \eqref{5.1} \eqref{5.2} and \eqref{5.19}, we obtain that
\begin{align}\label{5.20}
\begin{aligned}
    \limsup_{\epsilon\to0}
    \int_{Q\times\R}
    &
    \!\!\!\!\big(\tilde{\chi}^{1,\epsilon}_{t_i,r}\tilde{\sgn}^{\epsilon}_{t_i,r}+\tilde{\chi}^{2,\epsilon}_{t_i,r}\tilde{\sgn}^{\epsilon}_{t_i,r}-2\tilde{\chi}^{1,\epsilon}_{t_i,r}\tilde{\chi}^{2,\epsilon}_{t_i,r}\big)\phi_{\beta}(y)\,dy\,d\eta\Big|_{r=t_i}^{r=t_{i+1}}
    \\
    \leq
    \limsup_{\epsilon\to0}\Bigg(
    &
    \sum_{j=1}^2 IE_i^{\sgn j,1}-IE_i^{\sgn j,2}-2IE_i^{\mix j,1}+2IE_i^{\mix j,2}
    \\
    &
    \,\,+
    IE_i^{\canc}
    \\
    &
    \,\,+
    \sum_{j=1}^2 BE_i^{\sgn j,1}+BE_i^{\sgn j,2}-BE_i^{\sgn j,3}-2BE_i^{\mix j,1}-2BE_i^{\mix j,2}+2BE_i^{\mix j,3}
    \\
    &
    \,\,+
    \sum_{k=1}^4BE_i^{\canc A,k}
    +
    \sum_{j=1}^2BE_i^{\canc Bj}\Bigg).
\end{aligned}
\end{align}
In turn, from this and \eqref{1.3} it follows that
\begin{align}\label{6.1}
\begin{aligned}
    \hspace{-3mm}\int_Q|u^1(y,r)-u^2(y,r)|\,dy&\Big|_{r=0}^{r=T}
    \\
    \,\leq
    \limsup_{\beta\to0}
    \limsup_{\epsilon\to0}
    \sum_{i=0}^{N-1}\Bigg(
    &
    \sum_{j=1}^2 IE_i^{\sgn j,1}-IE_i^{\sgn j,2}-2IE_i^{\mix j,1}+2IE_i^{\mix j,2}
    \\
    &
    \,+
    IE_i^{\canc}
    \\
    &
    \,+
    \sum_{j=1}^2\!BE_i^{\sgn j,1}\!\!\!+\!BE_i^{\sgn j,2}\!\!\!-\!BE_i^{\sgn j,3}\!\!\!-\!2BE_i^{\mix j,1}\!\!\!-\!2BE_i^{\mix j,2}\!\!\!+\!2BE_i^{\mix j,3}
    \\
    &
    \,+
    \sum_{k=1}^4BE_i^{\canc A,k}
    +
    \sum_{j=1}^2BE_i^{\canc Bj}\Bigg),
\end{aligned}
\end{align}
for the internal, displacement and boundary error terms defined in the previous steps, and where we recall that $\mathcal{P}=\{0=t_0<t_1<\dots<t_N=T\}\subseteq[0,T]\setminus\mathcal{N}$ is an arbitrary partition.
Incidentally, we mention that, according to the relevant definitions, some of these error terms are actually independent of $\epsilon\in(0,1)$.
The aim of the next two steps is to provide estimates for the internal and boundary errors respectively.
These estimates shall depend on $\epsilon,\beta\in(0,1)$ and on the size $|\mathcal{P}|$ of the arbitrary partition, in such a way that, as we let $\epsilon\to0$ first, $\beta\to0$ then, and finally $|\mathcal{P}|\to0$, the right-hand side of \eqref{6.1} vanishes, thus proving the theorem.

\medskip
\noindent
\textbf{Step 7: The internal errors.}
In this step we analyze the internal error terms in \eqref{6.1}.
We begin with the internal errors from Step 2 and Step 3.
These are easily handled using the stability estimates on the characteristics from Section \ref{section/rough path estimates} and the next crucial observation, which follows immediately from the definition \eqref{0.16} of the convolution kernels.
For every $(x,\xi)$, $(x',\xi')$ and $(y,\eta)\in Q\times\R$ and for every $t_i\in\mathcal{P}$ and $r\in[t_i,\infty)$,
\begin{equation}
    \label{7.1}
    \text{if } \rho_{t_i,r}^{1,\epsilon}(x,y,\xi,\eta)\rho_{t_i,r}^{1,\epsilon}\rho_{t_i,r}^{2,\epsilon}(x',y,\xi',\eta)\neq0,\,\,\text{then } \left|Y_{r,r-t_i}^{x,\xi}-Y_{r,r-t_i}^{x'\!,\xi'}\right|+\left|\Pi_{r,r-t_i}^{x,\xi}-\Pi_{r,r-t_i}^{x'\!,\xi'}\right|\leq2\epsilon.
\end{equation}
For the internal error $IE_i^{\sgn 1,1}$ defined in \eqref{2.6}, using the integration by parts formula \eqref{formula/integration by parts formula} and Lemma \ref{lemma/sobolev regularity of u^m} we obtain, with the notation \eqref{4.5bis},
\begin{align}\label{7.2}
\begin{aligned}
IE_i^{\sgn1,1}
\!
=
&
\!
-\int_{t_i}^{t_{i+1}}\!\!\!\int_{Q^3\times\R^2}\!\!\!\!\!\!\!\!\!\nabla\left(u^1\right)^{[m]}\bar{\rho}_{t_i,r}^{1,\epsilon}\sgn(\xi')\phi_\beta(y)
\,\nabla_{y}\rho_{t_i,r}^{2,\epsilon}\big(D_xY_{r,r-t_i}^{x,u^1}\!\!\!-\!D_{x'}Y_{r,r-t_i}^{x'\!\!,\xi'}\big)
\\
&
-
\int_{t_i}^{t_{i+1}}\!\!\!\int_{Q^3\times\R^2}\!\!\!\!\!\!\!\!\!\nabla\left(u^1\right)^{[m]}\bar{\rho}_{t_i,r}^{1,\epsilon}\sgn(\xi')\phi_\beta(y)
\,\partial_{\eta}\rho_{t_i,r}^{2,\epsilon}\big(D_x\Pi_{r,r-t_i}^{x,u^1}\!\!\!-\!D_{x'}\Pi_{r,r-t_i}^{x'\!\!,\xi'}\big).
\end{aligned}
\end{align}
The error term $IE_i^{\mix 1,1}$ defined in \eqref{4.4} is transformed identically, integrating by parts and using Lemma \ref{lemma/sobolev regularity of u^m}, and the specular terms $IE_i^{\sgn 2,1}$ and $IE_i^{\mix 2,1}$ are handled in exact analogy, swapping the roles of $\chi^1$ and $\chi^2$.
Since the functions $\phi_\beta$, $\sgn$ and $\chi^j$ are bounded, and since there exists a constant $C$, only depending on the the standard mollifier $\rho$ chosen in Step 0, such that
\begin{equation}
    \label{7.3}
    \int_{Q\times\R}|\nabla_y\rho_{t_i,r}^{j,\epsilon}(x',y,\xi',\eta)|+|\partial_{\eta}\rho_{t_i,r}^{j,\epsilon}(x',y,\xi',\eta)|\,dx'\,d\xi'\leq C\epsilon^{-1},
\end{equation}
integrating the convolution kernels over the variables $(y,\eta)$ and $(x',\xi')$, using \eqref{7.1} combined with the estimates \eqref{formula/distance DY and DPi in terms of distance between characteristics}, we get
\begin{equation}
    \label{7.4}
    |IE_i^{\sgn j,1}|+|IE_i^{\mix j,1}|\leq C|t_{i+1}-t_i|^{\alpha}\int_{t_i}^{t_{i+1}}\!\!\!\int_{Q}\left|\nabla\left(u^j\right)^{[m]}\right|\,dx\,dr,
\end{equation}
for $j=1,2$, for a constant $C=C(T,A,z)$.

The remaining error terms $IE_i^{\sgn 1,2}$, defined in \eqref{2.16}, and $IE_i^{\mix 1,2}$, defined in \eqref{3.14}, and the specular terms $IE_i^{\sgn 2,2}$ and $IE_i^{\mix 2,2}$, are treated in analogy to \eqref{7.4}.
Indeed, we use the boundedness of $\phi_\beta$, $\sgn$ and $\chi^j$, formula \eqref{7.3}, and formula \eqref{7.1} combined with \eqref{formula/distance DY and DPi in terms of distance between characteristics}, and integrate the convolution kernels to estimate
\begin{equation}
    \label{7.5}
    |IE_i^{\sgn j,2}|+|IE_i^{\mix j,2}|\leq C|t_{i+1}-t_i|^{\alpha}\int_{t_i}^{t_{i+1}}\!\!\!\int_{Q\times\R}\!\!\!\!p_r^j(x,\xi)+q_r^j(x,\xi)\,\,\,dx\,d\xi\,dr,
\end{equation}
for $j=1,2$, for a constant $C=C(T,A,z)$.
Finally, using \eqref{7.4} and \eqref{7.5}, and summing over the partition, we obtain the following estimate, for $C=C(T,A,z)$,
\begin{align}
\begin{aligned}
    \label{7.5bis}
    \limsup_{\beta\to0}\limsup_{\epsilon\to0}\sum_{i=0}^{N-1}\sum_{j=1}^2\sum_{k=1}^2&|IE_i^{\sgn j,k}|+|IE_i^{\mix j,k}|
    \\
    &
    \leq 
    C|\mathcal{P}|^{\alpha}\sum_{j=1}^2\left(\int_{0}^{T}\!\int_{Q}\!\left|\nabla\left(u^j\right)^{[m]}\right|\,dx\,dr
    +\int_{0}^{T}\!\!\!\int_{Q\times\R}\!\!\!\!p_r^j+q_r^j\,dx\,d\xi\,dr\right).
\end{aligned}
\end{align}

We now consider the error term coming from the internal cancellation in Step 4, namely the internal error $IE_i^{\canc}$ defined in \eqref{4.8}.
The analysis is broken down in three cases: $m=1$, $m\in(2,\infty)$ and $m\in(0,1)\cup(1,2]$.

\emph{Case $m=1$.}
This case is trivial.
Indeed, if $m=1$, it follows automatically from definition \eqref{4.8} that $IE_i^{\canc}=0$.

\emph{Case $m\in(2,\infty)$.}
We form a velocity decomposition of the integral.
For each $M>1$, let $K_M:\R\to[0,1]$ be a smooth function satisfying
\begin{align}\label{7.6}
    K_M(\xi)=\empheqlbrace
    \begin{aligned}
    &1\quad\text{if } |\xi|\leq M,
    \\
    &0\quad\text{if } |\xi|\leq M+1.
\end{aligned}
\end{align}
Then, for each $M>1$, we split
\begin{align}
\begin{aligned}
    \label{7.7}
    \hspace{-3mm}
    I\!E_i^{\canc}\!\!\!&=2m\!\!\int_{t_i}^{t_{i+1}}\!\!\!\int_{Q^3\times\R^3}\!\!\!\!\!\!\!\!\!\!\!\!K_M(\xi)\left(|\xi|^{\frac{m-1}{2}}\!\!-\!|\xi'|^{\frac{m-1}{2}}\right)^2\!\!\chi_r^1\chi_r^2\nabla_x\rho_{t_i,r}^{1,\epsilon}\nabla_{x'}\rho_{t_i,r}^{2,\epsilon}\phi_\beta(y)\dsp dx \dsp dx' dy\dsp d\xi\dsp d\xi' d\eta\dsp dr
    \\
    &\,\,\,\,\,+\!\!
    2m\!\!\int_{t_i}^{t_{i+1}}\!\!\!\int_{Q^3\times\R^3}\!\!\!\!\!\!\!\!\!\!\!\!\!\!(1-K_M(\xi))\left(|\xi|^{\frac{m-1}{2}}\!\!-\!|\xi'|^{\frac{m-1}{2}}\right)^2\!\!\chi_r^1\chi_r^2\nabla_x\rho_{t_i,r}^{1,\epsilon}\nabla_{x'}\rho_{t_i,r}^{2,\epsilon}\,\phi_\beta(y)\dsp dx \dsp dx' dy\dsp d\xi\dsp d\xi' d\eta\dsp dr.
\end{aligned}
\end{align}
For the first term on the right-hand side of \eqref{7.7} we write
\begin{align}
\begin{aligned}
    \label{7.8}
    \bigg|2m\!\!\int_{t_i}^{t_{i+1}}&\!\!\!\int_{Q^3\times\R^3}\!\!\!\!\!\!\!\!\!\!\!\!K_M(\xi)\left(|\xi|^{\frac{m-1}{2}}\!\!-\!|\xi'|^{\frac{m-1}{2}}\right)^2\!\!\chi_r^1\chi_r^2\nabla_x\rho_{t_i,r}^{1,\epsilon}\nabla_{x'}\rho_{t_i,r}^{2,\epsilon}\phi_\beta(y)\dsp dx \dsp dx' dy\dsp d\xi\dsp d\xi' d\eta\dsp dr \bigg|
    \\
    &\leq
    C\int_{t_i}^{t_{i+1}}\!\!\!\int_{Q^3\times\R^3\cap\{|\xi-\xi'|\leq \,c\epsilon\}}\!\!\!\!\!\!\!\!\!\!\!\!|\xi-\xi'|^{(m-1)\wedge 2}\left|\nabla_x\rho_{t_i,r}^{1,\epsilon}\right|\left|\nabla_{x'}\rho_{t_i,r}^{2,\epsilon}\right|\dsp dx \dsp dx' dy\dsp d\xi\dsp d\xi' d\eta\dsp dr
    \\
    &\leq
    C\epsilon^{-2}\int_{t_i}^{t_{i+1}}\!\!\!\int_{Q}\int_{-c\epsilon}^{c\epsilon}\!\!\!\!|\theta|^{(m-1)\wedge 2}d\theta\,dy\,dr
    \leq
    C
    |t_{i+1}-t_i|\epsilon^{(3\wedge m)-2},
\end{aligned}
\end{align}
for a constant $C=C(M,m,Q,T,A,z)$.
In the first passage we used the boundedness of $\chi^j$, observation \eqref{7.1} combined with estimate \eqref{formula/distance between points in terms of distance between characteristics}, and the local Lipschitz continuity, if $m\geq3$, or the H\"older continuity, if $m\in(2,3)$, of the map $\R\ni\xi\mapsto|\xi|^{\frac{m-1}{2}}$.
In the second passage we exploited formula \eqref{2.2} for the derivatives of the convolution kernels, the boundedness of the derivatives of the characteristics from Proposition \ref{proposition/stability results for RDEs} and formula \eqref{7.3}, and we changed variables by setting $\theta=\xi-\xi'$.
In the last passage we simply used the boundedness of the domain $Q$.

For the second term on the right-hand side of \eqref{7.7} we use the following elementary inequality
\begin{equation}
    \label{7.9}
    \left(|\xi|^{\frac{m-1}{2}}\!\!-\!|\xi'|^{\frac{m-1}{2}}\right)^2=\left|\int_{\xi'}^{\xi}\frac{m-1}{2}\theta^{\left[\frac{m-3}{2}\right]}d\theta \right|\leq\left|\frac{m-1}{2}\right|^2\left(|\xi|^{m-3}+|\xi'|^{m-3}\right)|\xi-\xi'|^2.
\end{equation}
Then we estimate
\begin{align}
\begin{aligned}
    \label{7.10}
    \bigg|2&m\!\!\int_{t_i}^{t_{i+1}}\!\!\!\int_{Q^3\times\R^3}\!\!\!\!\!\!\!\!\!\!\!\!(1-K_M(\xi))\left(|\xi|^{\frac{m-1}{2}}\!\!-\!|\xi'|^{\frac{m-1}{2}}\right)^2\!\!\chi_r^1\chi_r^2\nabla_x\rho_{t_i,r}^{1,\epsilon}\nabla_{x'}\rho_{t_i,r}^{2,\epsilon}\phi_\beta(y)\dsp dx \dsp dx' dy\dsp d\xi\dsp d\xi' d\eta\dsp dr \bigg|
    \\
    &\leq
    C\int_{t_i}^{t_{i+1}}\!\!\!\int_{Q^3\times\R^3\cap\{|\xi-\xi'|\leq \,c\epsilon\}}\!\!\!\!\!\!\!\!\!\!\!\!\!\!\!\!\!\!\!\!\!\!\!\!\!\!\!\!\!\!\!\!\!\!\!\!\!\!\!\!\!\!\!(1-K_M(\xi))\left(|\xi|^{m-3}+|\xi'|^{m-3}\right)\epsilon^2\left|\chi_r^1\right|\left|\chi_r^2\right|\left|\nabla_x\rho_{t_i,r}^{1,\epsilon}\right|\left|\nabla_{x'}\rho_{t_i,r}^{2,\epsilon}\right|\dsp dx \dsp dx' dy\dsp d\xi\dsp d\xi' d\eta\dsp dr
    \\
    &\leq
    C\left(\int_{t_i}^{t_{i+1}}\!\!\!\int_{Q\cap\{|u^1|\geq M\}}\!\!\!\!\!\!\!\!\!\!\!\!\!\!\!\!\!\!\!\!\!\!\!\!\!\!\left|u^1\right|^{m-2}dx\,dr+\int_{t_i}^{t_{i+1}}\!\!\!\int_{Q\cap\{|u^2|\geq M-c\dsp\epsilon\}}\!\!\!\!\!\!\!\!\!\!\!\!\!\!\!\!\!\!\!\!\!\!\!\!\!\!\left|u^2\right|^{m-2}dx'dr\right),
\end{aligned}
\end{align}
for constants $C=C(m,T,A,z)$ and $c=c(T,A,z)$, independent of $M\geq 1$.
In the first passage we used \eqref{7.9}, and \eqref{7.1} combined with \eqref{formula/distance between points in terms of distance between characteristics}.
In the second passage we exploited properties of the kinetic function, and formula \eqref{2.2} combined with Proposition \ref{proposition/stability results for RDEs} and formula \eqref{7.3}.

Finally, combining \eqref{7.7} with \eqref{7.8} and \eqref{7.10}, and summing over the partition $\mathcal{P}$, we obtain
\begin{align}
\begin{aligned}
    \label{7.11}
    \sum_{i=0}^{N-1}\!\left|I\!E_i^{\canc}\right|
    \leq
    C_1\,\epsilon^{(3\wedge m)-2}
    \!+\!
    C_2\left(\int_{0}^{T}\!\!\!\int_{Q\cap\{|u^1|\geq M\}}\!\!\!\!\!\!\!\!\!\!\!\!\!\!\!\!\!\!\!\!\!\!\!\!\!\!\left|u^1\right|^{m-2}dx\, dr\!+\!\!\int_{0}^{T}\!\!\!\int_{Q\cap\{|u^2|\geq M-c\dsp\epsilon\}}\!\!\!\!\!\!\!\!\!\!\!\!\!\!\!\!\!\!\!\!\!\!\!\!\!\!\!\!\!\!\!\!\!\!\left|u^2\right|^{m-2}dx' dr\right),
\end{aligned}
\end{align}
for constants $C_1=C_1(M,m,Q,T,A,z)$, $C_2=C_2(m,T,A,z)$ and $c=c(T,A,z)$.
Since $m\in(2,\infty)$ and by Definition \ref{definition/pathwise kinetic solution} of kinetic solution we have $u^j\in L^{m+1}([0,T];L^{m+1}(Q))$, and since the constant $C_2$ is independent of $M$, H\"older's inequality and the dominated convergence theorem prove that the last two terms on the right-hand side of \eqref{7.11} vanish in the limit $M\to\infty$, uniformly for $\epsilon\in(0,1)$.
Therefore, passing first to the limit $\epsilon\to0$ and second to the limit $M\to\infty$, formula \eqref{7.11} yields
\begin{equation}
    \label{7.12}
    \limsup_{\epsilon\to0}\sum_{i=0}^{N-1}\left|IE_i^{\canc}\right|=0.
\end{equation}

\emph{Case $m\in(0,1)\cup(1,2]$.}
For this case the idea is to remove the singularity at the origin and to exploit the full regularity of the solution implied by Proposition \ref{proposition/singular moments for defect measures d=1} below.
Using the integration by parts formula \eqref{formula/integration by parts formula}, which is justified exploiting an approximation argument and Proposition \ref{proposition/singular moments for defect measures d in (0,1)} below, both in the $(x,\xi)$ and the $(x',\xi')$ variables, we write
\begin{align}
\begin{aligned}
    \label{7.13}
    \hspace{-3mm}
    I\!E_i^{\canc}\!\!\!=\frac{4m}{(m+1)^2}\!\!\int_{t_i}^{t_{i+1}}\!\!\!\int_{Q^3\times\R}\!\!\!\!\!\!\!\!\!\!\!\!\sigma_m(u^1,u^2)&\left|u^1\right|^{-\nicefrac{1}{2}}\nabla\left(u^1\right)^{\left[\frac{m+1}{2}\right]}\left|u^2\right|^{-\nicefrac{1}{2}}\nabla\left(u^2\right)^{\left[\frac{m+1}{2}\right]}
    \\
    &\cdot\bar{\rho}_{t_i,r}^{1,\epsilon}(x,y,\eta)\bar{\rho}_{t_i,r}^{2,\epsilon}(x'\!,y,\eta)\phi_\beta(y)\dsp dx \dsp dx' dy\dsp d\eta\dsp dr,
\end{aligned}
\end{align}
where we have defined
\begin{equation}\label{7.14}
    \sigma_m(\xi,\xi'):=|\xi|^{\frac{2-m}{2}}|\xi'|^{\frac{2-m}{2}}\left(|\xi|^{\frac{m-1}{2}}-|\xi|^{\frac{m-1}{2}}\right)^2\quad\text{for }\xi,\xi'\in\R.
\end{equation}
We notice that definition \eqref{4.5bis}, observation \eqref{7.1} and estimate \eqref{formula/distance between points in terms of distance between characteristics} imply that,
\begin{equation}
    \label{7.16}
    \text{if }\, \bar{\rho}_{t_i,r}^{1,\epsilon}(x,y,\eta)\bar{\rho}_{t_i,r}^{2,\epsilon}(x'\!,y,\eta)\neq0,\,\,\text{then } \left|x-x'\right|+\left|u^1-u^2\right|\leq c_1 \epsilon,
\end{equation}
for a constant $c_1=c_1(T,A,z)$.
Moreover, we observe that
\begin{equation}
    \label{7.17}
    \text{if } |\xi-\xi'|\leq c_1\epsilon,\,\,\,\text{then } \left|\sigma_m(\xi,\xi')\right|\leq C\,\epsilon,
\end{equation}
for a constant $C=C(c_1,m)$.
Indeed,if $\max\{|\xi|,|\xi'|\}\leq 2c_1\epsilon$, then, recalling $m\in(0,1)\cup(1,2]$, a direct computation yields
\begin{equation}
\label{7.18}
    \sigma_m(\xi,\xi')=|\xi|^{\frac{m}{2}}|\xi'|^{\frac{2-m}{2}}+2|\xi|^{\frac{1}{2}}|\xi'|^{\frac{1}{2}}+|\xi|^{\frac{2-m}{2}}|\xi'|^{\frac{m}{2}}\leq C\,\epsilon.
\end{equation}
Conversely, assume without loss of generality that $|\xi|>2c_1\epsilon$ with $|\xi|\geq|\xi'|$ and $|\xi-\xi'|\leq c_1\epsilon$.
Thus, in particular, $\xi$ and $\xi'$ have the same sign and $|\xi'|\geq\frac{1}{2}|\xi|$.
Then we compute
\begin{equation}
    \label{7.19}
    \sigma_m(\xi,\xi')=|\xi|^{\frac{2-m}{2}}|\xi'|^{\frac{2-m}{2}}\left(\int_{\xi'}^{\xi}\frac{m-1}{2}\theta^{\left[\frac{m-3}{2}\right]}d\theta \right)^2\leq C|\xi|^{-1}\epsilon^2\leq C\epsilon.
\end{equation}

We now form a velocity decomposition of the integral.
For each $\delta\in(0,1)$, let $K^\delta:\R\to[0,1]$ denote a smooth cutoff function satisfying
\begin{align}\label{7.20}
    K^\delta(\xi)=\empheqlbrace
    \begin{aligned}
    &1\quad\text{if } |\xi|\leq \delta \text{ or }\frac{2}{\delta}\leq|\xi|,
    \\
    &0\quad\text{if }2\delta\leq|\xi|\leq\frac{1}{\delta}.
\end{aligned}
\end{align}
Returning to \eqref{7.13}, consider the decomposition
\begin{align}
\begin{aligned}
    \label{7.21}
    \hspace{-3mm}
    I\!E_i^{\canc}\!\!=&\frac{4m}{(m+1)^2}\!\!\int_{t_i}^{t_{i+1}}\!\!\!\int_{Q^3\times\R}\!\!\!\!\sigma_m^\delta(u^1,u^2)\left|u^1\right|^{-\nicefrac{1}{2}}\nabla\left(u^1\right)^{\left[\frac{m+1}{2}\right]}\left|u^2\right|^{-\nicefrac{1}{2}}\nabla\left(u^2\right)^{\left[\frac{m+1}{2}\right]}
    \\
    &\qquad\qquad\qquad\qquad\qquad\qquad\qquad
    \cdot\bar{\rho}_{t_i,r}^{1,\epsilon}(x,y,\eta)\bar{\rho}_{t_i,r}^{2,\epsilon}(x'\!,y,\eta)\phi_\beta(y)\dsp dx \dsp dx' dy\dsp d\eta\dsp dr
    \\
    &+
    \frac{4m}{(m+1)^2}\!\!\int_{t_i}^{t_{i+1}}\!\!\!\int_{Q^3\times\R}\!\!\!\!\!\tilde{\sigma}_m^\delta(u^1,u^2)\left|u^1\right|^{-\nicefrac{1}{2}}\nabla\left(u^1\right)^{\left[\frac{m+1}{2}\right]}\left|u^2\right|^{-\nicefrac{1}{2}}\nabla\left(u^2\right)^{\left[\frac{m+1}{2}\right]}
    \\
    &\qquad\qquad\qquad\qquad\qquad\qquad\qquad
    \cdot\bar{\rho}_{t_i,r}^{1,\epsilon}(x,y,\eta)\bar{\rho}_{t_i,r}^{2,\epsilon}(x'\!,y,\eta)\phi_\beta(y)\dsp dx \dsp dx' dy\dsp d\eta\dsp dr,
\end{aligned}
\end{align}
where, for each $\delta\in(0,1)$, the functions $\sigma_m^{\delta},\tilde{\sigma}_m^{\delta}:\R^2\to\R$ are defined by
\begin{equation}\label{7.22}
    \sigma_m^{\delta}(\xi,\xi')=\left(K^\delta(\xi)+K^\delta(\xi')-K^\delta(\xi)K^\delta(\xi')\right)\sigma_m(\xi,\xi'),
\end{equation}
and
\begin{equation}\label{7.22bis}
    \tilde{\sigma}_m^{\delta}(\xi,\xi')=\left(1-K^\delta(\xi)\right)\left(1-K^\delta(\xi')\right)\sigma_m(\xi,\xi').
\end{equation}
We start with the second term in \eqref{7.21}.
It follows from \eqref{7.14}, \eqref{7.20} and the local Lipschitz continuity of the map $\R\ni\xi\mapsto|\xi|^{\frac{m-1}{2}}$ on the set $\{\delta\leq|\xi|\leq\frac{2}{\delta}\}$ that, for $C=C(m,\delta)$,
\begin{equation}\label{7.23}
    \left|\tilde{\sigma}_m^{\delta}(\xi,\xi')\right|\leq C\,|\xi-\xi'|^2.
\end{equation}
Moreover, recalling observation \eqref{7.16} and the definition \eqref{0.16} and \eqref{4.5bis} of the kernels $\bar{\rho}_{t_i,r}^{j,\epsilon}$, we point out that
\begin{equation}\label{7.24}
    \int_{Q^2\times\R}\!\!\!\!\!\!\!\bar{\rho}_{t_i,r}^{1,\epsilon}(x,y,\eta)\bar{\rho}_{t_i,r}^{2,\epsilon}(x'\!,y,\eta)\,dx'dy\,d\eta=\int_{\{x'\in Q\mid \,|x-x'|\leq\dsp c_1\dsp\epsilon\}\times Q\times\R}\!\!\!\!\!\!\!\!\!\!\!\!\!\!\!\!\!\!\!\!\!\!\!\!\!\!\!\!\!\!\!\!\!\!\!\!\!\!\!\!\!\bar{\rho}_{t_i,r}^{1,\epsilon}(x,y,\eta)\bar{\rho}_{t_i,r}^{2,\epsilon}(x'\!,y,\eta)\,dx'dy\,d\eta \leq C\,\epsilon^{-1},
\end{equation}
for a constant $C=C(d,c_1)$ depending on the constant $c_1=c_1(T,A,z)$ from \eqref{7.16}.
Then, for the second term of \eqref{7.21}, we compute, for a constant $C=C(\delta,m,d,T,A,z)$,
\begin{align}
\begin{aligned}
    \label{7.25}
    \bigg|\frac{4m}{(m+1)^2}&\!\!\int_{t_i}^{t_{i+1}}\!\!\!\int_{Q^3\times\R}\!\!\!\!\!\!\!\!\!\!\!\!\tilde{\sigma}_m^\delta(u^1,u^2)\!\left|u^1\right|^{-\nicefrac{1}{2}}\nabla\!\!\left(u^1\right)^{\left[\frac{m+1}{2}\right]}\!\left|u^2\right|^{-\nicefrac{1}{2}}\nabla\!\!\left(u^2\right)^{\left[\frac{m+1}{2}\right]}\bar{\rho}_{t_i,r}^{1,\epsilon}\bar{\rho}_{t_i,r}^{2,\epsilon}\phi_\beta(y)\dsp dx \dsp dx' dy\dsp d\eta\dsp dr\bigg|
    \\
    &\leq
    C\epsilon^2
    \left(\int_{t_i}^{t_{i+1}}\!\!\!\int_{Q^3\times\R}\!\!\!\!\!\left|u^1\right|^{-1}\left|\nabla\left(u^1\right)^{\left[\frac{m+1}{2}\right]}\right|^2\bar{\rho}_{t_i,r}^{1,\epsilon}\bar{\rho}_{t_i,r}^{2,\epsilon}\phi_\beta(y)\dsp dx \dsp dx' dy\dsp d\eta\dsp dr\right)^{\frac{1}{2}}
    \\
    &\qquad\qquad\qquad\qquad
    \cdot
    \left(\int_{t_i}^{t_{i+1}}\!\!\!\int_{Q^3\times\R}\!\!\!\!\!\left|u^2\right|^{-1}\left|\nabla\left(u^2\right)^{\left[\frac{m+1}{2}\right]}\right|^2\bar{\rho}_{t_i,r}^{1,\epsilon}\bar{\rho}_{t_i,r}^{2,\epsilon}\phi_\beta(y)\dsp dx \dsp dx' dy\dsp d\eta\dsp dr\right)^{\frac{1}{2}}
    \\
    &\leq
    C\epsilon
    \left(\int_{t_i}^{t_{i+1}}\!\!\!\int_{Q\times\R}\!\!\!\!\!|\xi|^{-1}q_r^1(x,\xi)\,dx\,d\xi\,dr\right)^{\frac{1}{2}}
    \left(\int_{t_i}^{t_{i+1}}\!\!\!\int_{Q\times\R}\!\!\!\!\!|\xi'|^{-1}q_r^2(x'\!,\xi')\,dx'd\xi'dr\right)^{\frac{1}{2}}
    \\
    &\leq
    C\,\epsilon\sum_{j=1}^2\left(1+\|u_0^j\|_{L^2(Q)}^2\right).
\end{aligned}
\end{align}
In the first passage we used H\"older's inequality and \eqref{7.16} and \eqref{7.23}.
In the second passage we exploited \eqref{7.24} and the definition of parabolic defect measure.
The last passage follows from Proposition \ref{proposition/singular moments for defect measures d=1} below.

We now consider to the first term in \eqref{7.21}.
It follows immediately from \eqref{7.17}, \eqref{7.20} and \eqref{7.22} that
\begin{equation}
    \label{7.26}
    \text{if } |\xi-\xi'|\leq c_1\epsilon,\,\,\,\text{then } \left|\sigma_m^\delta(\xi,\xi')\right|\leq C\,\epsilon,
\end{equation}
for $C=C(m,c_1)$, and that, for $\epsilon<\delta\in(0,1)$, with $\epsilon$ small enough depending on $c_1$ and $\delta$,
\begin{equation}
    \label{7.27}
    \text{if }\, |\xi-\xi'|\leq c_1\epsilon\,\,\,\text{ and }\sigma_m^\delta(\xi,\xi')\neq0,\,\,\,\text{then }\, 0<|\xi|,|\xi'|<3\delta\,\text{ or }\,\frac{1}{2\delta}<|\xi|,|\xi'|.
\end{equation}
Then, for the first term of \eqref{7.21}, first using H\"older's inequality and observations \eqref{7.16},\eqref{7.26} and \eqref{7.27}, and then using observation \eqref{7.24} and the definition of parabolic defect measure, we compute
\begin{align}
\begin{aligned}
    \label{7.28}
    \bigg|\frac{4m}{(m+1)^2}&\!\!\int_{t_i}^{t_{i+1}}\!\!\!\int_{Q^3\times\R}\!\!\!\!\!\!\!\!\!\!\!\!\sigma_m^\delta(u^1,u^2)\!\left|u^1\right|^{-\nicefrac{1}{2}}\nabla\!\!\left(u^1\right)^{\left[\frac{m+1}{2}\right]}\!\left|u^2\right|^{-\nicefrac{1}{2}}\nabla\!\!\left(u^2\right)^{\left[\frac{m+1}{2}\right]}\bar{\rho}_{t_i,r}^{1,\epsilon}\bar{\rho}_{t_i,r}^{2,\epsilon}\phi_\beta(y)\dsp dx \dsp dx' dy\dsp d\eta\dsp dr\bigg|
    \\
    &\leq
    C\epsilon
    \left(\int_{t_i}^{t_{i+1}}\!\!\!\int_{U_{\delta}^1\times U_{\delta}^2\times Q\times\R}\!\!\!\!\!\!\!\!\!\!\!\!\!\!\!\left|u^1\right|^{-1}\left|\nabla\left(u^1\right)^{\left[\frac{m+1}{2}\right]}\right|^2\bar{\rho}_{t_i,r}^{1,\epsilon}\bar{\rho}_{t_i,r}^{2,\epsilon}\phi_\beta(y)\dsp dx \dsp dx' dy\dsp d\eta\dsp dr\right)^{\frac{1}{2}}
    \\
    &\qquad\qquad\qquad
    \cdot
    \left(\int_{t_i}^{t_{i+1}}\!\!\!\int_{U_{\delta}^1\times U_{\delta}^2\times Q\times\R}\!\!\!\!\!\!\!\!\!\!\!\!\!\!\left|u^2\right|^{-1}\left|\nabla\left(u^2\right)^{\left[\frac{m+1}{2}\right]}\right|^2\bar{\rho}_{t_i,r}^{1,\epsilon}\bar{\rho}_{t_i,r}^{2,\epsilon}\phi_\beta(y)\dsp dx \dsp dx' dy\dsp d\eta\dsp dr\right)^{\frac{1}{2}}
    \\
    &\leq
    C\sum_{j=1}^2\left(\int_{t_i}^{t_{i+1}}\!\!\!\int_{U_{\delta}^j\times\R}\!\!\!\!\!|\xi|^{-1}q_r^j(x,\xi)\,dx\,d\xi\,dr\right),
\end{aligned}
\end{align}
for a constant $C=C(m,d,T,A,z)$, where we have defined, for $r\in[0,\infty)$,
\begin{equation}
    \label{7.29}
    U_{\delta}^j=U_{\delta}^j(r):=\left\{x\in Q \,\,\bigg|\,\, 0<|u^j(x,r)|<3\delta\text{ or }\frac{1}{2\delta}<|u^j(x,r)|\,\right\}.
\end{equation}
In conclusion, summing over the partition $\mathcal{P}$, the splitting \eqref{7.21} and estimates \eqref{7.25} and \eqref{7.28} imply that, for each $\delta\in (0,1)$, for $C=C(m,d,T,A,z)$,
\begin{equation}
    \label{7.30}
    \limsup_{\epsilon\to0}\sum_{i=0}^{N-1}\left|IE_i^{\canc}\right|
    \leq
    C\sum_{j=1}^2\int_{0}^{T}\!\!\!\int_{U_{\delta}^j\times\R}\!\!\!\!\!|\xi|^{-1}q_r^j(x,\xi)\,dx\,d\xi\,dr.
\end{equation}
The dominated convergence theorem, Proposition \ref{proposition/singular moments for defect measures d=1} below and \eqref{7.29} imply that the right-hand side of \eqref{7.30} vanishes in the $\delta\to0$ limit.
Therefore, we deduce
\begin{equation}
    \label{7.31}
    \limsup_{\epsilon\to0}\sum_{i=0}^{N-1}\left|IE_i^{\canc}\right|=0.
\end{equation}
This concludes the analysis of the internal error terms.

\medskip
\noindent
\textbf{Step 8: The boundary errors.}
In this step we analyze the boundary error terms in \eqref{6.1}.
We begin with the errors produced in Step 2 and Step 3.
These are handled using the estimates on the behaviour of the characteristics near the boundary from Section \ref{section/rough path estimates}, properties \eqref{0.10} of the gradient $\nabla \phi_{\beta}$, observation \eqref{7.1}, and the following crucial fact.
Namely, with the notation \eqref{0.5}, it follows from \eqref{0.5tris}, \eqref{0.10} and \eqref{0.16} that, for any $\epsilon<\beta\in(0,1)$, for any $(x,y,\xi,\eta)\in Q^2\times\R^2$ and any $t_i\leq r\in[0,T]$, for $j=1,2$,
\begin{equation}\label{9.1}
    \text{if }\,\nabla_y\phi_\beta(y)\rho_{t_i,r}^{j,\epsilon}(x,y,\xi,\eta)\neq0,\,\,\text{then }\,y,Y_{r,r-t_i}^{x,\xi},x\in Q^\beta.
\end{equation}
For the boundary error term $BE_i^{\sgn 1,1}$ defined in \eqref{2.9}, for a constant $C=C(T,A,z,Q)$, we compute
\begin{align}\label{9.2}
    \begin{aligned}
        \left|BE^{\sgn1,1}_i\right|
        &
        \leq
        \int_{t_i}^{t_{i+1}}\!\!\!\int_{Q^2\times\R^2}\!\!\!m|\xi|^{m-1}\left|\chi_r^1\right|\left|\rho_{t_i,r}^{1,\epsilon}\right|\left|\tilde{\sgn}_{t_i,r}^{\epsilon}\right|\left|\nabla_y\phi_\beta\right|\left|\Delta_xY_{r,r-t_i}^{x,\xi}\right|\,dx\, dy \, d\xi\,d\eta\, dr
        \\
        &\leq
        C
        |t_{i+1}-t_i|^\alpha\beta^{-1}\int_{t_i}^{t_{i+1}}\!\!\!\int_{Q^\beta\times\R}\!\!\!m|\xi|^{m-1}\left|\chi_r^1\right|\,dx \,d\xi\, dr
        \\
        &\leq
        C
        |t_{i+1}-t_i|^\alpha\beta^{-1}\int_{t_i}^{t_{i+1}}\!\!\!\int_{Q^\beta}\!\left|\left(u^1\right)^{[m]}\right|\,dx \, dr
        \\
        &\leq
        C
        |t_{i+1}-t_i|^\alpha\int_{t_i}^{t_{i+1}}\!\!\!\int_{Q^\beta}\!\left|\nabla\left(u^1\right)^{[m]}\right|\,dx \, dr.
    \end{aligned}
\end{align}
In the second inequality we used observation \eqref{9.1}, estimates \eqref{0.10} and \eqref{formula/bound D^2Y and D^2Pi}, and the boundedness of $\Tilde{\sgn}^\epsilon$, and we integrated the mollifier.
In the third passage we simply used properties of the kinetic function.
In the last passage we exploited the Sobolev regularity $(u^1)^{[m]}\in W_0^{1,p_m}(Q)$ from Lemma \ref{lemma/sobolev regularity of u^m} below, which in particular implies that $(u^1)^{[m]}$ vanishes on the boundary, and we used the mean value theorem applied to points $x\in Q^\beta$, which therefore satisfy $\dist(x,\partial Q)\leq C\beta$. 

An identical analysis holds for $BE_i^{\sgn 2,1}$, simply replacing $\chi^1$ and $u^1$ with $\chi^2$ and $u^2$, and for the analogous error terms $BE^{\mix j,1}_i$ defined in \eqref{3.6}, for $j=1,2$, simply replacing $\tilde{\sgn}^{\epsilon}_{t_i,r}$ with $\tilde{\chi}_{t_i,r}^{\cdot,\epsilon}$, which is bounded as well.
Precisely, we have
\begin{align}\label{9.3}
    \begin{aligned}
        \left|BE^{\sgn j,1}_i\right|+\left|BE^{\mix j,1}_i\right|
        \leq
        C
        |t_{i+1}-t_i|^\alpha\int_{t_i}^{t_{i+1}}\!\!\!\int_{Q^\beta}\!\left|\nabla\left(u^j\right)^{[m]}\right|\,dx \, dr,
    \end{aligned}
\end{align}
for $j=1,2$, for a constant $C=C(T,A,z,Q)$.

Next we analyze the boundary error terms $BE_i^{\sgn 1,2}$ and $BE_i^{\mix 1,2}$, defined in \eqref{2.11} and \eqref{3.8} respectively, and the specular terms $BE_i^{\sgn 2,2}$ and $BE_i^{\mix 2,2}$.
Observations \eqref{7.1} and \eqref{9.1} combined with estimate \eqref{formula/bound D^2Y and D^2Pi} imply that, for $C=C(T,A,z)$,
\begin{equation}\label{9.7}
    \text{if }\,\nabla_{\!y}\phi_\beta\,\rho_{t_i,r}^{1,\epsilon}\,\rho_{t_i,r}^{2,\epsilon}\neq0,\,\,\text{then }\,|D_{\!x}Y_{r,r-t_i}^{x,\xi}\!-\!D_{\!x'}Y_{r,r-t_i}^{x',\xi'}|\!+\!|\nabla_{\!x}\Pi_{r,r-t_i}^{x,\xi}\!-\!\nabla_{\!x'}\Pi_{r,r-t_i}^{x',\xi'}|\!\leq\! C|t_{i+1}-t_i|^\alpha\epsilon.
\end{equation}
Then, for $BE_i^{\sgn 1,2}$, for a constant $C=C(T,A,z,Q)$, we compute
\begin{align}\label{9.8}
\begin{aligned}
\left|BE_i^{\sgn1,2}\right|
&
\leq
C\,|t_{i+1}-t_i|^\alpha\epsilon\beta^{-1}
\int_{t_i}^{t_{i+1}}\!\!\!\int_{\left(Q^\beta\right)^3\times\R^3}\!\!\!\!\!\!\!\!\!\!m|\xi|^{m-1}\left|\chi_r^1\right|\,\left|\rho_{t_i,r}^{1,\epsilon}\right|\left(\left|\nabla_y\rho_{t_i,r}^{2,\epsilon}\right|+\left|\partial_{\eta}\rho_{t_i,r}^{2,\epsilon}\right|\right)
\\
&
\leq
C\,|t_{i+1}-t_i|^\alpha\beta^{-1}
\int_{t_i}^{t_{i+1}}\!\!\!\int_{Q^\beta}\left|\left(u^1\right)^{[m]}\right|\,dx\,dr
\\
&
\leq
C\,|t_{i+1}-t_i|^\alpha
\int_{t_i}^{t_{i+1}}\!\!\!\int_{Q^\beta}\left|\nabla\left(u^1\right)^{[m]}\right|\,dx\,dr.
\end{aligned}
\end{align}
In the first passage we used the boundedness of $D_{\!x}Y_{r,r-t_i}^{x,\xi}$ from Proposition \ref{proposition/stability results for RDEs} and of the $\sgn$ function, estimate \eqref{0.10}, and observations \eqref{9.1} and \eqref{9.7}.
In the second passage we integrated the mollifiers over the variables $(y,\eta)$ and $(x'\!,\xi')$, recalling estimate \eqref{7.3}, and then used properties of the kinetic function.
In the last passage we used the Sobolev regularity $(u^1)^{[m]}\in W_0^{1,p_m}(Q)$ from Lemma \ref{lemma/sobolev regularity of u^m} and applied the mean value theorem to points $x\in Q^\beta$. 

Virtually identical analyses and estimates hold for $BE_i^{\sgn2,2}$ and for $BE_i^{\mix j,2}$, for $j=1,2$.
Precisely, we have 
\begin{align}\label{9.9}
\begin{aligned}
\left|BE_i^{\sgn j,2}\right|+\left|BE_i^{\mix j,2}\right|
&
\leq
C\,|t_{i+1}-t_i|^\alpha
\int_{t_i}^{t_{i+1}}\!\!\!\int_{Q^\beta}\left|\nabla\left(u^j\right)^{[m]}\right|\,dx\,dr,
\end{aligned}
\end{align}
for $j=1,2$ and $C=C(Q,T,A,z)$. 

In conclusion, we sum estimates \eqref{9.3} and \eqref{9.9} over the partition $\mathcal{P}$ to conclude that, for $C=C(Q,T,A,z)$,
\begin{align}\label{9.10bis}
    \begin{aligned}
        \sum_{j=1}^2\sum_{k=1}^2\sum_{i=0}^{N-1}\left|BE^{\sgn j,k}_i\right|+\left|BE^{\mix j,k}_i\right|
        &\leq
        C\,|\mathcal{P}|^\alpha\sum_{j=1}^2\int_{0}^{T}\!\!\!\int_{Q^\beta}\!\left|\nabla\left(u^j\right)^{[m]}\right|\,dx \, dr.
    \end{aligned}
\end{align}

Finally, we analyze the boundary errors $BE_i^{\sgn 1,3}$ and $BE_i^{\mix 1,3}$, defined in \eqref{2.14} and \eqref{3.12} respectively, and the specular terms $BE_i^{\sgn 2,3}$ and $BE_i^{\mix 2,3}$, simply obtained by swapping the roles of $u^1$ and $u^2$.
For $BE_i^{\sgn 1,3}$, we use the boundedness of $\tilde{\sgn}_{t_i,r}^{\epsilon}$, we exploit estimate \eqref{0.10}, and estimate \eqref{formula/DxiY near the boundary} combined with observation \eqref{7.1}, and then we integrate the mollifier to obtain, for $C=C(T,A,z)$,
\begin{align}\label{9.5}
    \begin{aligned}
        \left|BE^{\sgn 1,3}_i\right|
        &\leq
        \int_{t_i}^{t_{i+1}}\!\!\!\int_{Q^2\times\R^2}(p_r^1+q_r^1)\,\left|\rho_{t_i,r}^{1,\epsilon}\right|\,\left|\tilde{\sgn}_{t_i,r}^{\epsilon}\right|\left|\nabla_y\phi_\beta(y)\partial_{\xi}Y_{r,r-t_i}^{x,\xi}\right|dx\,dy\,d\xi\,d\eta\,dr
        \\
        &\leq
        C
        |t_{i+1}-t_i|^\alpha\int_{t_i}^{t_{i+1}}\!\!\!\int_{Q^\beta\times\R}\!(p_r^j+q_r^j)\,dx \, d\xi\,dr.
    \end{aligned}
\end{align}
Virtually identical analyses hold for $BE_i^{\sgn 2,3}$ and for $BE_i^{\mix j,3}$, for $j=1,2$. 
Summing over the partition, we conclude that, for $C=C(T,A,z)$, 
\begin{align}\label{9.6}
    \begin{aligned}
        \sum_{j=1}^2\sum_{i=0}^{N-1}\left|BE^{\sgn j,3}_i\right|+\left|BE^{\mix j,3}_i\right|
        &\leq
        C\,|\mathcal{P}|^\alpha\sum_{j=1}^2\int_{0}^{T}\!\!\!\int_{Q^\beta\times\R}\!\!(p_r^j+q_r^j)\,dx \, d\xi\,dr.
    \end{aligned}
\end{align}

In conclusion, combining estimates \eqref{9.10bis} with the integrability of $\nabla\left(u^j\right)^{[m]}$ from Lemma \ref{lemma/sobolev regularity of u^m}, and estimate \eqref{9.6} with the finiteness of the parabolic and entropy defect measures over $Q\times\R\times[0,T]$, and exploiting the dominated convergence theorem, for the boundary error terms produced in Step 2 and Step 3, we conclude that
\begin{align}\label{9.10}
\begin{aligned}
\limsup_{\beta\to0}\limsup_{\epsilon\to0}\sum_{j=1}^2\sum_{k=1}^3\sum_{i=0}^{N-1}\left|BE_i^{\sgn j,k}\right|+\left|BE_i^{\mix j,k}\right|=0.
\end{aligned}
\end{align}

Now we analyze the boundary cancellation error terms produced in Step 5.
First we observe that \eqref{0.5tris} combined with \eqref{0.10}, for $k=1$ or $k=2$, implies that
\begin{equation}\label{9.11}
    \left|D^k_{\!y}\phi_\beta\right|\leq C\beta^{-k},\,\,\text{ and that, if }\,D^k_{\!y}\phi_\beta\msp\big(Y_{r,r-t_i}^{x,\xi}\big)\neq0,\,\,\text{then }x,Y_{r,r-t_i}^{x,\xi}\in Q^\beta. 
\end{equation}
Then, for the error term $BE_i^{\canc A,1}$ defined in \eqref{5.12}, we compute for $C=C(T,A,z,Q)$
\begin{align}\label{9.12}
\begin{aligned}
    \left|BE_i^{\canc A,1}\right|
    &\leq
    C|t_{i+1}-t_i|^{\alpha}\beta^{-1}\int_{t_i}^{t_{i+1}}\int_{Q^\beta}\left|\left(u^1\right)^{[m]}\right|\,dx\,dr
    \\
    &\leq
    C|t_{i+1}-t_i|^{\alpha}\int_{t_i}^{t_{i+1}}\int_{Q^\beta}\left|\nabla\left(u^1\right)^{[m]}\right|\,dx\,dr
\end{aligned}
\end{align}
In the first passage we used use observation \eqref{9.11} and estimate \eqref{formula/bound D^2Y and D^2Pi}, and then we exploited the boundedness of $|\chi^2|$ and integrated $|\chi^1|$ (or viceversa).
In the second passage we used $\left(u^1\right)^{[m]}\in W_0^{1,p_m}(Q)$ and applied the mean value theorem to points $x\in Q^\beta$.

The error term $BE_i^{\canc A,2}$ defined in \eqref{5.15} is handled almost identically.
Indeed, we use the boundedness of $D_xY_{r,r-t_i}^{x,\xi}$, and observation \eqref{9.11} for $k=2$ combined with estimate \eqref{formula/DxY - id near the boundary}, then we integrate the kinetic functions, and finally we exploit again $\left(u^j\right)^{[m]}\in W_0^{1,p_m}(Q)$ and the mean value theorem to compute, for $C=C(T,A,z,Q)$,
\begin{align}\label{9.14}
\begin{aligned}
    \left|BE_i^{\canc A,2}\right|
    &\leq
    C\int_{t_i}^{t_{i+1}}\int_{Q\times\R}m|\xi|^{m-1}\left(\left|\chi_r^1\right|+\left|\chi_r^2\right|\right)\left|D^2_y\phi_\beta\msp\big(Y_{r,r-t_i}^{x,\xi}\big)\right|\left|D_xY_{r,r-t_i}^{x,\xi}-\id_d\right|\,dx\,d\xi\,dr
    \\
    &\leq 
    C\,|t_{i+1}-t_i|^\alpha\beta^{-1}\int_{t_i}^{t_{i+1}}\int_{Q^\beta}\left|\left(u^1\right)^{[m]}\right|+\left|\left(u^2\right)^{[m]}\right|\,dx\,dr
    \\
    &\leq 
    C\,|t_{i+1}-t_i|^\alpha\int_{t_i}^{t_{i+1}}\int_{Q^\beta}\left|\nabla\left(u^1\right)^{[m]}\right|+\left|\nabla\left(u^2\right)^{[m]}\right|\,dx\,dr.
\end{aligned}
\end{align}

In conclusion, we sum estimates \eqref{9.12} and \eqref{9.14} over the partition $\mathcal{P}$ to conclude that, for $C=C(Q,T,A,z)$,
\begin{align}\label{9.15}
\begin{aligned}
    \sum_{i=0}^{N-1}\left|BE_i^{\canc A,1}\right|+\left|BE_i^{\canc A,2}\right|\leq C\,|\mathcal{P}|^\alpha\int_{0}^{T}\int_{Q^\beta}\left|\nabla\left(u^1\right)^{[m]}\right|+\left|\nabla\left(u^2\right)^{[m]}\right|\,dx\,dr.
\end{aligned}
\end{align}

We now consider the error term $BE_i^{\canc A,3}$ defined in \eqref{5.17}.
First, for the function $d_{\partial Q}$ defined in \eqref{0.3} and $\psi_{\beta}$ defined in \eqref{0.7}, recalling \eqref{0.5tris}, we observe that
\begin{equation}\label{9.16}
    \left|\Dot{\psi}_\beta(s)\right|\leq \beta^{-1}\,\,\forall s\in\R\,\,\,\text{ and, if }\,\Dot{\psi}_\beta\big(d_{\partial Q}\big(Y_{r,r-t_i}^{x,\xi}\big)\big)\neq0,\,\,\text{then }x,Y_{r,r-t_i}^{x,\xi}\in Q^\beta. 
\end{equation}
We also mention that we have a uniform bound $|\nabla\cdot\tilde{n}|\leq C$ for some constant $C=C(Q)$, for the extended unit outward normal to $Q$ defined in \eqref{0.4}.
Then we compute
\begin{align}\label{9.17}
\begin{aligned}
    \left|BE_i^{\canc A,3}\right|
    &\leq
    \int_{t_i}^{t_{i+1}}\int_{Q\times\R}m|\xi|^{m-1}\left|\chi_r^1-\chi_r^2\right|^2\left|\Dot{\psi}_\beta\big(d_{\partial Q}\big(Y_{r,r-t_i}^{x,\xi}\big)\big)\right||\nabla\cdot\tilde{n}|\,dx\,d\xi\,dr
    \\
    &\leq
    C\,\beta^{-1}\int_{t_i}^{t_{i+1}}\int_{Q^{\beta}\times\R}m|\xi|^{m-1}\left(\left|\chi_r^1\right|+\left|\chi_r^2\right|\right)dx\,d\xi\,dr
    \\
    &=
    C\,\beta^{-1}\int_{t_i}^{t_{i+1}}\int_{Q^{\beta}}\left|\left(u^1\right)^{[m]}\right|+\left|\left(u^2\right)^{[m]}\right|\,dx\,dr
    \\
    &\leq
    C\int_{t_i}^{t_{i+1}}\int_{Q^{\beta}}\left|\nabla\!\left(u^1\right)^{[m]}\right|+\left|\nabla\!\left(u^2\right)^{[m]}\right|\,dx\,dr,
\end{aligned}
\end{align}
for a constant $C=C(Q,T,A,z)$.
In the second inequality we used observation \eqref{9.16}, the boundedness of $\nabla\cdot\tilde{n}$ and properties of the kinetic function.
In the last inequality we used $\big(u^j\big)^{[m]}\in W_0^{1,p_m}(Q)$ and the mean value theorem.
In conclusion, we sum over the partition to estimate, for $C=C(Q,T,A,z)$,
\begin{align}\label{9.18}
\begin{aligned}
    \sum_{i=0}^{N-1}\left|BE_i^{\canc A,3}\right|
    &\leq
    C\sum_{j=1}^2\int_{0}^{T}\int_{Q^{\beta}}\left|\nabla\!\left(u^1\right)^{[m]}\right|+\left|\nabla\!\left(u^2\right)^{[m]}\right|\,dx\,dr.
\end{aligned}
\end{align}

Finally we consider the error term $BE_i^{\canc A,4}$ defined in \eqref{5.18}.
First, we make the following observation.
Namely, for the standard $1$-dimensional convolution kernel $\rho_1^{\frac{1}{2}\beta^{\gamma_m}}$ introduced in \eqref{0.6}-\eqref{0.7}, rescaled at order $\frac{1}{2}\beta^{\gamma_m}$ with $\gamma_m=(m+2)\vee3$, recalling \eqref{0.5tris}, we notice that
\begin{equation}\label{9.19}
    \left|\rho_1^{\frac{1}{2}\beta^{\gamma_m}}\right|\leq 2\,\beta^{-\gamma_m}\,\,\,\text{ and, if }\,\rho_1^{\frac{1}{2}\beta^{\gamma_m}}\!\!\Big(d_{\partial Q}\big(Y_{r,r-t_i}^{x,\xi}\big)\!-\!\beta^{\gamma_m}\Big)\neq0,\,\,\text{then }x,Y_{r,r-t_i}^{x,\xi}\in Q^{\beta^{\gamma_m}}. 
\end{equation}
Then we compute
\begin{align}\label{9.20}
\begin{aligned}
    \left|BE_i^{\canc A,4}\right|
    &=
    \int_{t_i}^{t_{i+1}}\int_{Q\times\R}m|\xi|^{m-1}\left|\chi_r^1-\chi_r^2\right|^2\beta^{-1}\rho_1^{\frac{1}{2}\beta^{\gamma_m}}\!\!  \Big(d_{\partial Q}\big(Y_{r,r-t_i}^{x,\xi}\big)\!-\!\beta^{\gamma_m}\!\Big)\,dx\,d\xi\,dr
    \\
    &\leq 
    C \beta^{-1}\beta^{-\gamma_m}\int_{t_i}^{t_{i+1}}\int_{Q^{\beta^{\gamma_m}}\times\R}\!\!\!\!\!\!m|\xi|^{m-1}\left(\left|\chi_r^1\right|+\left|\chi_r^2\right|\right)dx\,d\xi\,dr
    \\
    &=
    C\beta^{-(1+\gamma_m)}\int_{t_i}^{t_{i+1}}\int_{Q^{\beta^{\gamma_m}}}\left|\left(u^1\right)^{[m]}\right|+\left|\left(u^2\right)^{[m]}\right|\,dx\,dr
    \\
    &\leq
    C\beta^{-1}\int_{t_i}^{t_{i+1}}\int_{Q^{\beta^{\gamma_m}}}\left|\nabla\!\left(u^1\right)^{[m]}\right|+\left|\nabla\!\left(u^2\right)^{[m]}\right|\,dx\,dr,
\end{aligned}
\end{align}
for a constant $C=C(T,A,z)$.
In the first inequality we used observation \eqref{9.19} and properties of the kinetic function.
In the last inequality we used the Sobolev regularity $\big(u^j\big)^{[m]}\in W_0^{1,p_m}(Q)$ and the mean value theorem applied to points $x\in Q^{\beta^{\gamma_m}}$, which satisfy $\dist(x,\partial Q)\leq C\beta^{\gamma_m}$.

Next, we sum estimate \eqref{9.20} over the partition, then we apply H\"older's inequality combined with Lemma \ref{lemma/sobolev regularity of u^m}, which prescribes the higher integrability $\nabla\left(u^j\right)^{[m]}\in L^{p_m}([0,T];L^{p_m}(Q))$, and observation \eqref{0.5bis}, and we recall that $\gamma_m=(m+2)\vee3$ and $p_m=\frac{m+1}{m}\wedge2$, to estimate
\begin{align}\label{9.21}
\begin{aligned}
    \sum_{i=0}^{N-1}\left|BE_i^{\canc A,4}\right|
    &\leq
    C\beta^{-1}\sum_{j=1}^2\int_{0}^{T}\int_{Q^{\beta^{\gamma_m}}}\left|\nabla\!\left(u^j\right)^{[m]}\right|\,dx\,dr
    \\
    &\leq
    C\beta^{-1}\left(\beta^{\gamma_m}\right)^{\frac{p_m-1}{p_m}}\sum_{j=1}^2\left\|\nabla\!\left(u^j\right)^{[m]}\right\|_{L^{p_m}\left([0,T];L^{p_m}\left(Q^{\beta^{\gamma_m}}\right)\right)}
    \\
    &=
    C\beta^{\frac{1}{2\vee(m+1)}}\sum_{j=1}^2\left\|\nabla\!\left(u^j\right)^{[m]}\right\|_{L^{p_m}\left([0,T];L^{p_m}\left(Q^{\beta^{\gamma_m}}\right)\right)},
\end{aligned}
\end{align}
for a constant $C=C(Q,T,A,z)$.

In conclusion, combining estimates \eqref{9.15}, \eqref{9.18} and \eqref{9.21} with Lemma \ref{lemma/sobolev regularity of u^m}, it follows from the dominated convergence theorem that
\begin{equation}\label{9.22}
    \limsup_{\beta\to0}\limsup_{\epsilon\to0}\sum_{k=1}^4\sum_{i=0}^{N-1}\left|BE_i^{\canc A,k}\right|=\limsup_{\beta\to0}\sum_{k=1}^4\sum_{i=0}^{N-1}\left|BE_i^{\canc A,k}\right|=0.
\end{equation}

We conclude this step by analyzing the last boundary error terms, namely the boundary cancellation errors $BE_i^{\canc B j}$ defined in \eqref{5.4}, for $j=1,2$.
The analysis is broken down in three cases: $m=1$, $m\in(1,\infty)$ and $m\in(0,1)$.

\emph{Case $m=1$.}
This case is trivial.
Indeed, if $m=1$, it follows automatically from definition \eqref{5.4} that $BE_i^{\canc Bj}=0$.

\emph{Case $m\in(1,\infty)$.}
We form a velocity decomposition of the integral.
For each $M>1$, let $K_M:\R\to[0,1]$ be a smooth function satisfying
\begin{align}\label{9.23}
    K_M(\xi)=\empheqlbrace
    \begin{aligned}
    &1\quad\text{if } |\xi|\leq M,
    \\
    &0\quad\text{if } |\xi|\leq M+1.
\end{aligned}
\end{align}
Then, for each $M>1$, we split
\begin{align}
\begin{aligned}
    \label{9.24}
    \hspace{-3mm}
    B\!E_i^{\canc B1}\!\!&=2m\!\int_{t_i}^{t_{i+1}}\!\!\!\int_{Q^3\times\R^3}\!\!\!\!\!\!\!\!\!\!\!\!K_M(\xi)\left(|\xi|^{m-1}\!\!-\!|\xi'|^{m-1}\right)\chi_r^1\chi_r^2\,\rho_{t_i,r}^{1,\epsilon}\nabla_{x'}\rho_{t_i,r}^{2,\epsilon}\nabla_y\phi_\beta(y)D_xY_{r,r-t_i}^{x,\xi}
    \\
    &\,\,\,\,\,+\!
    2m\!\int_{t_i}^{t_{i+1}}\!\!\!\int_{Q^3\times\R^3}\!\!\!\!\!\!\!\!\!\!\!\!\!\!(1-K_M(\xi))\left(|\xi|^{m-1}\!\!-\!|\xi'|^{m-1}\right)\chi_r^1\chi_r^2\,\rho_{t_i,r}^{1,\epsilon}\nabla_{x'}\rho_{t_i,r}^{2,\epsilon}\nabla_y\phi_\beta(y)D_xY_{r,r-t_i}^{x,\xi}.
\end{aligned}
\end{align}
For the first term on the right-hand side of \eqref{9.24} we write
\begin{align}
\begin{aligned}
    \label{9.25}
    \bigg|2m\!\!\int_{t_i}^{t_{i+1}}&\!\!\!\int_{Q^3\times\R^3}\!\!\!\!\!\!\!\!\!\!\!\!K_M(\xi)\left(|\xi|^{m-1}\!\!-\!|\xi'|^{m-1}\right)\chi_r^1\chi_r^2\,\rho_{t_i,r}^{1,\epsilon}\nabla_{x'}\rho_{t_i,r}^{2,\epsilon}\nabla_y\phi_\beta(y)D_xY_{r,r-t_i}^{x,\xi}\dsp dx \dsp dx' dy\dsp d\xi\dsp d\xi' d\eta\dsp dr \bigg|
    \\
    &\leq
    C\beta^{-1}\int_{t_i}^{t_{i+1}}\!\!\!\int_{\left(Q^\beta\right)^3\times\R^3\cap\{|\xi-\xi'|\leq \,c\epsilon\}}\!\!\!\!\!\!\!\!\!\!\!\!|\xi-\xi'|^{(m-1)\wedge 1}\,\,\rho_{t_i,r}^{1,\epsilon}\left|\nabla_{x'}\rho_{t_i,r}^{2,\epsilon}\right|\dsp dx \dsp dx' dy\dsp d\xi\dsp d\xi' d\eta\dsp dr
    \\
    &\leq
    C\beta^{-1}\epsilon^{-1}\int_{t_i}^{t_{i+1}}\!\!\!\int_{Q^{\beta}}\int_{-c\epsilon}^{c\epsilon}\!\!\!\!|\theta|^{(m-1)\wedge 1}d\theta\,dy\,dr
    \leq
    C\,|t_{i+1}-t_i|\,\epsilon^{(m-1)\wedge1},
\end{aligned}
\end{align}
for a constant $C=C(M,m,Q,T,A,z)$.
In the first passage we used the boundedness of $K_M$, $\chi^j$ and $D_xY_{r,r-t_i}^{x,\xi}$, observation \eqref{0.10} and \eqref{9.1}, observation \eqref{7.1} combined with estimate \eqref{formula/distance between points in terms of distance between characteristics}, and the local Lipschitz continuity, if $m\in(2,\infty)$, or the H\"older continuity, if $m\in(1,2)$, of the map $\R\ni\xi\mapsto|\xi|^{m-1}$.
In the second passage we exploited formula \eqref{2.2} for the derivatives of the convolution kernels, the boundedness of the derivatives of the characteristics from Proposition \ref{proposition/stability results for RDEs}, and then we integrated the convolution kernels, recalling observation \eqref{7.3} and changing variables by setting $\theta=\xi-\xi'$.
In the last passage we simply used formula \eqref{0.5bis}.

For the second term on the right-hand side of \eqref{9.24} we use the following elementary inequality
\begin{equation}
    \label{9.26}
    \left||\xi|^{m-1}\!-\!|\xi'|^{m-1}\right|
    =
    \left|\int_{\xi'}^{\xi}(m-1)\theta^{\left[m-2\right]}d\theta \right|
    \leq
    (m-1)\left(|\xi|^{m-2}+|\xi'|^{m-2}\right)|\xi-\xi'|.
\end{equation}
Then we estimate
\begin{align}
\begin{aligned}
    \label{9.27}
    \bigg|2&m\!\!\int_{t_i}^{t_{i+1}}\!\!\!\int_{Q^3\times\R^3}\!\!\!\!\!\!\!\!\!\!\!\!\!(1\!-\!K_M(\xi))\!\left(|\xi|^{m-1}\!\!-\!|\xi'|^{m-1}\right)\!\chi_r^1\chi_r^2\,\rho_{t_i,r}^{1,\epsilon}\nabla_{x'}\rho_{t_i,r}^{2,\epsilon}\nabla_y\phi_\beta(y)D_xY_{r,r-t_i}^{x,\xi} dx \dsp dx' dy\dsp d\xi\dsp d\xi' d\eta\dsp dr \bigg|
    \\
    &\leq
    C\int_{t_i}^{t_{i+1}}\!\!\!\int_{\left(Q^\beta\right)^3\times\R^3\cap\{|\xi-\xi'|\leq \,c\epsilon\}}\!\!\!\!\!\!\!\!\!\!\!\!\!\!\!\!\!\!\!\!\!\!\!\!\!\!\!\!\!\!\!\!\!\!\!\!\!\!\!\!\!\!\!\!\!\!(1-K_M(\xi))\left(|\xi|^{m-2}\!\!+\!|\xi'|^{m-2}\right)\epsilon\left|\chi_r^1\right|\left|\chi_r^2\right|\,\rho_{t_i,r}^{1,\epsilon}\left|\nabla_{x'}\rho_{t_i,r}^{2,\epsilon}\right|\dsp dx \dsp dx' dy\dsp d\xi\dsp d\xi' d\eta\dsp dr
    \\
    &\leq
    C\left(\int_{t_i}^{t_{i+1}}\!\!\!\int_{Q^\beta\cap\{|u^1|\geq M\}}\!\!\!\!\!\!\!\!\!\!\!\!\!\!\!\!\!\!\!\!\!\!\!\!\!\!\left|u^1\right|^{m-1}dx\,dr+\int_{t_i}^{t_{i+1}}\!\!\!\int_{Q^\beta\cap\{|u^2|\geq M-c\dsp\epsilon\}}\!\!\!\!\!\!\!\!\!\!\!\!\!\!\!\!\!\!\!\!\!\!\!\!\!\!\left|u^2\right|^{m-1}dx'dr\right),
\end{aligned}
\end{align}
for constants $C=C(\beta,m,T,A,z)$ and $c=c(T,A,z)$, independent of $M\geq 1$.
In the first passage we used \eqref{9.26}, \eqref{7.1} combined with \eqref{formula/distance between points in terms of distance between characteristics}, and \eqref{0.10}.
In the second passage we used observation \eqref{7.3} and integrated the convolution kernels in the $(x',\xi')$ and $(y,\eta)$ variables, when hitting $|\xi|^{m-1}$, and in the $(x,\xi)$ and $(y,\eta)$ variables, when hitting $|\xi'|^{m-1}$, and then we used properties of the kinetic function.

Finally, combining \eqref{9.24} with \eqref{9.25} and \eqref{9.27}, and summing over the partition $\mathcal{P}$, we obtain
\begin{align}
\begin{aligned}
    \label{9.28}
    \sum_{i=0}^{N-1}\!\left|B\!E_i^{\canc B1}\right|
    \leq
    C_1\,\epsilon^{(m-1)\wedge1}
    \!+\!
    C_2\left(\!\int_{0}^{{T}}\!\!\int_{Q\cap\{|u^1|\geq M\}}\!\!\!\!\!\!\!\!\!\!\!\!\!\!\left|u^1\right|^{m-1}\!\!\!dx\dsp dr\!+\!\!\int_{0}^{T}\!\!\int_{Q\cap\{|u^2|\geq M-c\dsp\epsilon\}}\!\!\!\!\!\!\!\!\!\!\!\!\!\!\left|u^2\right|^{m-1}\!\!\!dx'\msp dr\right),
\end{aligned}
\end{align}
for constants $C_1=C_1(M,m,Q,T,A,z)$, $C_2=C_2(\beta,m,Q,T,A,z)$ and $c=c(T,A,z)$.
An inequality virtually identical to \eqref{9.28} holds for $BE_i^{\canc B2}$, swapping $u^1$ and $u^2$.
Since $m\in(1,\infty)$ and by Definition \ref{definition/pathwise kinetic solution} of kinetic solution we have $u^j\in L^{m+1}([0,T];L^{m+1}(Q))$, and since the constant $C_2$ is independent of $M\geq1$, H\"older's inequality and the dominated convergence theorem prove that the last two terms on the right-hand side of \eqref{9.28} vanish in the limit $M\to\infty$, uniformly for $\epsilon\in(0,1)$.
Therefore, passing first to the limit $\epsilon\to0$ and second to the limit $M\to\infty$, formula \eqref{9.28} yields
\begin{equation}
    \label{9.29}
    \limsup_{\epsilon\to0}\sum_{i=0}^{N-1}\sum_{j=1}^2\left|BE_i^{\canc Bj}\right|=0.
\end{equation}

\emph{Case $m\in(0,1)$.}
For this case the idea is to remove the singularity at the origin and to exploit the full regularity of the solution implied by Proposition \ref{proposition/singular moments for defect measures d=1} below.
Using the integration by parts formula \eqref{formula/integration by parts formula} in the $(x',\xi')$ variable, which is justified exploiting an approximation argument and Proposition \ref{proposition/singular moments for defect measures d in (0,1)} below, we write
\begin{align}
\begin{aligned}
    \label{9.30}
    \hspace{-3mm}
    B\!E_i^{\canc B1}\!\!\!
    &=\!\!2m\!\int_{t_i}^{t_{i+1}}\!\!\!\int_{Q^3\times\R^3}\!\!\!\!\!\!\!\!\!\!\!\!\!\!\!\sigma_m(\xi,\xi')|\xi|^{\frac{m}{2}-1}|\xi'|^{^{\frac{m}{2}-1}}\!\chi_r^1\chi_r^2\rho_{t_i,r}^{1,\epsilon}\nabla_{\!x'}\rho_{t_i,r}^{2,\epsilon}\nabla_{\!y}\phi_\beta(y)D_x\!Y_{r,r-t_i}^{x,\xi} dx\dsp dx'\msp dy\dsp d\xi\dsp d\xi'\msp d\eta\dsp dr
    \\
    &=\!
    -\frac{4m}{(m+1)}\!\msp\int_{t_i}^{t_{i+1}}\!\!\!\int_Q\!\!\msp\left|u^2\right|^{-\frac{1}{2}}\nabla\left(u^2\right)^{\left[\frac{m+1}{2}\right]}
    \\
    &\qquad\qquad\qquad\qquad\,\,\,
    \cdot\!\left(\int_{Q^2\times\R^2}\msp\!\!\!\!\!\!\!\!\!\!\!\!\!\!\!\sigma_m(\xi,u^2)|\xi|^{\frac{m}{2}-1}\chi_r^1\rho_{t_i,r}^{1,\epsilon}\bar{\rho}_{t_i,r}^{2,\epsilon}\nabla_{\!y}\phi_\beta(y)D_x\!Y_{r,r-t_i}^{x,\xi} dx\dsp dy\dsp d\xi\dsp d\eta\!\right)\msp\!dx'\msp dr,
\end{aligned}
\end{align}
where we recall the notation \eqref{4.5bis} for $\bar{\rho}_{t_i,r}^{2,\epsilon}(x',y,\eta)$, and where we have defined
\begin{equation}\label{9.31}
    \sigma_m(\xi,\xi'):=|\xi|^{\frac{m}{2}}|\xi'|^{1-\frac{m}{2}}-|\xi'|^{\frac{m}{2}}|\xi|^{1-\frac{m}{2}}\quad\text{for }\xi,\xi'\in\R.
\end{equation}
Observation \eqref{7.1} and estimate \eqref{formula/distance between points in terms of distance between characteristics} prove that, for $c_1=c_1(T,A,z)$,
\begin{equation}
    \label{9.32}
    \text{if }\, \rho_{t_i,r}^{1,\epsilon}(x,y,\xi,\eta)\bar{\rho}_{t_i,r}^{2,\epsilon}(x'\!,y,\eta)\neq0,\,\,\text{then } \left|x-x'\right|+\left|\xi-u^2\right|\leq c_1 \epsilon.
\end{equation}
Moreover, we observe that
\begin{equation}
    \label{9.33}
    \text{if } |\xi-\xi'|\leq c_1\epsilon,\,\,\,\text{then } \left|\sigma_m(\xi,\xi')\right|\leq C\,\epsilon,
\end{equation}
for a constant $C=C(c_1,m)$.
Indeed,if $\max\{|\xi|,|\xi'|\}\leq 2c_1\epsilon$, recalling $m\in(0,1)$, the result is immediate from \eqref{9.31}.
Conversely, assume without loss of generality that $|\xi|>2c_1\epsilon$ with $|\xi|\geq|\xi'|$ and $|\xi-\xi'|\leq c_1\epsilon$.
Thus, in particular, $\xi$ and $\xi'$ have the same sign and $|\xi'|\geq\frac{1}{2}|\xi|$.
Then, using $\xi\simeq\xi'$, we compute
\begin{equation}
    \label{9.34}
    \sigma_m(\xi,\xi')=|\xi|^{1-\frac{m}{2}}|\xi'|^{1-\frac{m}{2}}\int_{\xi'}^{\xi}(m-1)\theta^{[m-2]}\,d\theta 
    \leq C|\xi-\xi'|\leq C\epsilon.
\end{equation}

We now form a velocity decomposition of the integral.
For each $\delta\in(0,1)$, let $K^\delta:\R\to[0,1]$ denote a smooth cutoff function satisfying
\begin{align}\label{9.35}
    K^\delta(\xi)=\empheqlbrace
    \begin{aligned}
    &1\quad\text{if } |\xi|\leq \delta,
    \\
    &0\quad\text{if }|\xi|\geq2\delta.
\end{aligned}
\end{align}
Returning to \eqref{9.30}, consider the decomposition
\begin{align}
\begin{aligned}
    \label{9.36}
    \hspace{-3mm}
    B\!E_i^{\canc B1}\!\!
    =&\!
    -\!\frac{4m}{(m+1)}\!\msp\int_{t_i}^{t_{i+1}}\!\!\!\int_Q\!\!\msp\left|u^2\right|^{-\frac{1}{2}}\nabla\left(u^2\right)^{\left[\frac{m+1}{2}\right]}
    \\
    &\qquad\qquad\qquad\,\,\,\,\,\,\,
    \cdot\!\left(\int_{Q^2\times\R^2}\msp\!\!\!\!\!\!\!\!\!\!\!\!\!\!\!\sigma_m^{\delta}\left(\xi,u^2\right)|\xi|^{\frac{m}{2}-1}\chi_r^1\rho_{t_i,r}^{1,\epsilon}\bar{\rho}_{t_i,r}^{2,\epsilon}\nabla_{\!y}\phi_\beta(y)D_x\!Y_{r,r-t_i}^{x,\xi} dx\dsp dy\dsp d\xi\dsp d\eta\!\right)\msp\!dx'\msp dr
    \\
    &
    -\!\frac{4m}{(m+1)}\!\msp\int_{t_i}^{t_{i+1}}\!\!\!\int_Q\!\!\msp\left|u^2\right|^{-\frac{1}{2}}\nabla\left(u^2\right)^{\left[\frac{m+1}{2}\right]}
    \\
    &\qquad\qquad\qquad\,\,\,\,\,\,\,
    \cdot\!\left(\int_{Q^2\times\R^2}\msp\!\!\!\!\!\!\!\!\!\!\!\!\!\!\!\tilde{\sigma}_m^{\delta}\left(\xi,u^2\right)|\xi|^{\frac{m}{2}-1}\chi_r^1\rho_{t_i,r}^{1,\epsilon}\bar{\rho}_{t_i,r}^{2,\epsilon}\nabla_{\!y}\phi_\beta(y)D_x\!Y_{r,r-t_i}^{x,\xi} dx\dsp dy\dsp d\xi\dsp d\eta\!\right)\msp\!dx'\msp dr,
\end{aligned}
\end{align}
where, for each $\delta\in(0,1)$, the functions $\sigma_m^{\delta},\tilde{\sigma}_m^{\delta}:\R^2\to\R$ are defined by
\begin{equation}\label{9.37}
    \sigma_m^{\delta}(\xi,\xi')=\left(K^\delta(\xi)+K^\delta(\xi')-K^\delta(\xi)K^\delta(\xi')\right)\sigma_m(\xi,\xi'),
\end{equation}
and
\begin{equation}\label{9.37bis}
    \tilde{\sigma}_m^{\delta}(\xi,\xi')=\left(1-K^\delta(\xi)\right)\left(1-K^\delta(\xi')\right)\sigma_m(\xi,\xi').
\end{equation}
We start with the second term in \eqref{9.36}.
It follows from \eqref{9.32}, \eqref{9.33} and \eqref{9.37bis} that, for a constant $C=C(m,T,A,z)$, for any $(x,x',y,\xi,\eta,r)\in Q^3\times\R^2\times[t_i,\infty)$,
\begin{equation}
    \label{9.38}
    \text{if }\, \rho_{t_i,r}^{1,\epsilon}(x,y,\xi,\eta)\bar{\rho}_{t_i,r}^{2,\epsilon}(x'\!,y,\eta)\tilde{\sigma}_m^{\delta}\left(\xi,u^2\right)\neq0,\,\,\text{then }|\xi|, \left|u^2\right|\geq\delta \text{ and }\left|\tilde{\sigma}_m^{\delta}\left(\xi,u^2\right)\right|\leq C \epsilon.
\end{equation}
Then, for a constant $C=C(\delta,\beta,m,T,A,z)$, we compute
\begin{align}
\begin{aligned}
    \label{9.39}
    \hspace{-4mm}
    \bigg|&\frac{4m}{(m+1)}\!\!\int_{t_i}^{t_{i+1}}\!\!\!\int_Q\!\!\!\left|u^2\right|^{-\frac{1}{2}}\nabla\!\msp\left(u^2\right)^{\left[\frac{m+1}{2}\right]}\!\!\int_{Q^2\times\R^2}\msp\!\!\!\!\!\!\!\!\!\!\!\!\!\!\!\tilde{\sigma}_m^{\delta}\left(\xi,u^2\right)|\xi|^{\frac{m}{2}-1}\chi_r^1\rho_{t_i,r}^{1,\epsilon}\bar{\rho}_{t_i,r}^{2,\epsilon}\nabla_{\!y}\phi_\beta D_x\!Y_{r,r-t_i}^{x,\xi} dx\dsp dy\dsp d\xi\dsp d\eta\,\,dx'\msp dr\bigg|
    \\
    &\qquad\leq
    C
    \int_{t_i}^{t_{i+1}}\!\!\!\int_Q\left|u^2\right|^{-\frac{1}{2}}\left|\nabla\left(u^2\right)^{\left[\frac{m+1}{2}\right]}\right|\left(\int_{Q^2\times\R^2\cap\{|\xi|\geq\delta\}}\!\!\!\!\!\!\!\!\!\!\!\!\!\!\!\epsilon\,|\xi|^{\frac{m}{2}-1}\chi_r^1\,\rho_{t_i,r}^{1,\epsilon}\,\bar{\rho}_{t_i,r}^{2,\epsilon}\,dx\dsp dy\dsp d\xi\dsp d\eta\right)dx'dr
    \\
    &\qquad\leq
    C\epsilon
    \int_{t_i}^{t_{i+1}}\!\!\!\int_{Q}\left|u^2\right|^{-\frac{1}{2}}\left|\nabla\left(u^2\right)^{\left[\frac{m+1}{2}\right]}\right|dx'dr.
\end{aligned}
\end{align}
In the first passage we used \eqref{9.38}, \eqref{0.10} and the boundedness of $\chi^1$ and $D_xY_{r,r-t_i}^{x,\xi}$.
In the second passage we recalled $m\in(0,1)$, so that $\frac{m}{2}-1<0$ and $|\xi|^{\frac{m}{2}-1}\leq\delta^{\frac{m}{2}-1}$, and integrated the convolution kernels. 

We now consider the first term in \eqref{9.36}.
It follows from \eqref{9.32}, \eqref{9.33} and \eqref{9.37} that, for a constant $C=C(m,T,A,z)$, for any $\epsilon<\delta\in(0,1)$, with $\epsilon$ small enough depending on $C$ and $\delta$, and any $(x,x',y,\xi,\eta,r)\in Q^3\times\R^2\times[t_i,\infty)$,
\begin{equation}
    \label{9.40}
    \text{if }\, \rho_{t_i,r}^{1,\epsilon}(x,y,\xi,\eta)\bar{\rho}_{t_i,r}^{2,\epsilon}(x'\!,y,\eta)\sigma_m^{\delta}\left(\xi,u^2\right)\neq0,\,\,\text{then }|\xi|, \left|u^2\right|\leq3\delta \text{ and }\left|\sigma_m^{\delta}\left(\xi,u^2\right)\right|\leq C \epsilon.
\end{equation}
Moreover, the definition \eqref{0.16} of the convolution kernels implies that, for any $(x',r)\in Q\times[t_i,\infty)$,
\begin{equation}
    \label{9.41}
    \int_{Q^2\times\R}\rho_{t_i,r}^{1,\epsilon}(x,y,\xi,\eta)\,\bar{\rho}_{t_i,r}^{2,\epsilon}(x'\!,y,\eta)\,\,dx\,dy\,d\eta\leq\epsilon^{-1}.
\end{equation}
Then, for the first term of \eqref{9.36}, we compute
\begin{align}
\begin{aligned}
    \label{9.42}
    \hspace{-4mm}
    \bigg|&\frac{4m}{(m+1)}\!\!\int_{t_i}^{t_{i+1}}\!\!\!\int_Q\!\!\!\left|u^2\right|^{-\frac{1}{2}}\!\nabla\!\msp\left(u^2\right)^{\left[\frac{m+1}{2}\right]}\!\!\int_{Q^2\times\R^2}\msp\!\!\!\!\!\!\!\!\!\!\!\!\!\!\!\sigma_m^{\delta}\left(\xi,u^2\right)|\xi|^{\frac{m}{2}-1}\!\chi_r^1\rho_{t_i,r}^{1,\epsilon}\bar{\rho}_{t_i,r}^{2,\epsilon}\nabla_{\!y}\phi_\beta D_x\!Y_{r,r-t_i}^{x,\xi} dx\dsp dy\dsp d\xi\dsp d\eta\,\,dx'\msp dr\bigg|
    \\
    &\quad\leq
    C
    \int_{t_i}^{t_{i+1}}\!\!\!\int_Q\left|u^2\right|^{-\frac{1}{2}}\left|\nabla\left(u^2\right)^{\left[\frac{m+1}{2}\right]}\right|\int_{-3\delta}^{3\delta}|\xi|^{\frac{m}{2}-1}\left(\int_{Q^2\times\R}\!\!\!\!\!\!\!\!\epsilon\,\rho_{t_i,r}^{1,\epsilon}\,\bar{\rho}_{t_i,r}^{2,\epsilon}\,dx\dsp dy\dsp d\eta\right)d\xi\,dx'dr
    \\
    &\quad\leq
    C\delta^{\frac{m}{2}}
    \int_{t_i}^{t_{i+1}}\!\!\!\int_{Q}\left|u^2\right|^{-\frac{1}{2}}\left|\nabla\left(u^2\right)^{\left[\frac{m+1}{2}\right]}\right|dx'dr,
\end{aligned}
\end{align}
for a constant $C=C(\beta,m,T,A,z)$, independent of $\delta\in(0,1)$.
In the first passage we used observation \eqref{9.40}, \eqref{0.10} and the boundedness of $\chi^1$ and $D_xY_{r,r-t_i}^{x,\xi}$.
In the second passage we applied observation \eqref{9.41}. 

Finally, combining the splitting \eqref{9.36} with estimates \eqref{9.39} and \eqref{9.42}, summing over the partition and applying H\"older's inequality, we obtain
\begin{equation}
    \label{9.43}
    \sum_{i=0}^{N-1}\left|BE_i^{\canc B1}\right|
    \leq
    \left(C_1\,\epsilon+C_2\,\delta^{\frac{m}{2}}\right)\left(\int_{0}^{T}\!\!\!\int_{Q}\left|u^2\right|^{-1}\left|\nabla\left(u^2\right)^{\left[\frac{m+1}{2}\right]}\right|^2dx'dr\right)^{\frac{1}{2}},
\end{equation}
for constants $C_1=C_1(\delta,\beta,m,Q,T,A,z)$ and $C_2=C_2(\beta,m,Q,T,A,z)$.
An inequality virtually identical to \eqref{9.43} holds for the specular errors $BE_i^{\canc B2}$, simply swapping the roles of $u^1$ and $u^2$.
Proposition \ref{proposition/singular moments for defect measures d=1} below ensures that the right-hand side of \eqref{9.43} is finite.
Since the constant $C_2$ is independent of $\delta$ and $\epsilon$, passing first to the limit $\epsilon\to0$ and then to the limit $\delta\to0$, formula \eqref{9.43} and its analogue version for $j=2$ yield
\begin{equation}
    \label{9.44}
    \limsup_{\epsilon\to0}\sum_{i=0}^{N-1}\sum_{j=1}^2\left|BE_i^{\canc Bj}\right|=0.
\end{equation}
This concludes the analysis of the boundary error terms.

\medskip
\noindent
\textbf{Step 9: The conclusion.}
We are finally ready to conclude the proof.
Returning to inequality \eqref{6.1}, we use estimates \eqref{7.5bis}, \eqref{7.12} and \eqref{7.31}, \eqref{9.10}, \eqref{9.22}, and \eqref{9.29} and \eqref{9.44}, to deduce
\begin{equation}
\label{10.1}
    \int_Q\!\!\left|u^1(x,r)-u^2(x,r)\right|\,dx\Big|_{r=0}^{r=T}
    \!\leq
    C\,\left|\mathcal{P}\right|^{\alpha}\!\sum_{j=1}^2\!\left(\int_{0}^{T}\!\!\int_{Q}\!\left|\nabla\left(u^j\right)^{[m]}\right|\,dx\,dr
    +\int_{0}^{T}\!\!\!\int_{Q\times\R}\!\!\!\!\!\!p_r^j+q_r^j\,dx\,d\xi\,dr\right),
\end{equation}
for a constant $C=C(T,A,z)$.
Definition \ref{definition/pathwise kinetic solution} of pathwise kinetic solution and Lemma \ref{lemma/sobolev regularity of u^m} guarantee that the right-hand side of \eqref{10.1} is finite.
Recalling from Step 0 that the partition $\mathcal{P}\subseteq[0,T]\setminus\mathcal{N}$ is arbitrary, we let $|\mathcal{P}|\to0$ and from \eqref{10.1} we conclude that
\begin{equation}
\label{10.2}
    \int_Q\left|u^1(x,T)-u^2(x,T)\right|\,dx
    \leq
    \int_Q\left|u^1_0(x)-u^2_0(x)\right|\,dx.
\end{equation}
This completes the uniqueness proof.

\end{proof}

\begin{remark}\label{remark/extension to signed data}
We observe that in the proof of Theorem \ref{theorem/uniqueness of pathwise kinetic solutions for nonnegative initial data} the positivity of the initial data was exploited only in Step 7, to handle the error terms $IE_i^{\canc}$ in the case $m\in(0,1)\cup(1,2]$, and in Step 8, to handle the error terms $BE_i^{\canc Bj}$ in the case $m\in(0,1)$.
Namely, we relied on the positivity of the initial data through the application of Proposition \ref{proposition/positive data stay positive} and \ref{proposition/singular moments for defect measures d=1} below.
The remaining arguments of this paper are obtained for general initial data in $L^2(Q)$.
This completes the proof of Theorem \ref{theorem/extension of the results to signed initial data}.
\end{remark}

Repeating the proof of Theorem \ref{theorem/uniqueness of pathwise kinetic solutions for nonnegative initial data} with $u_2=0$, we get the following estimate for the $L^1$-norm of pathwise kinetic solutions with any arbitrary initial data.

\begin{corollary}\label{corollary/L^1 estimate for patwhise kientic solutions}
Let $m\in(0,\infty)$ and let $u_0\in L^2(Q)$ be arbitrary.
Suppose that $u$ is a pathwise kinetic solution of \eqref{formula/stochastic porous media equation} with initial data $u_0$, then
\begin{equation}
    \|u\|_{L^{\infty}([0,T];L^1(Q))}\leq\|u_0\|_{L^1(Q)}. \nonumber
\end{equation}
\end{corollary}
\begin{proof}
Let $u_0\in L^2(Q)$ be arbitrary and let $u$ be a pathwise kinetic solution of \eqref{formula/stochastic porous media equation} with initial data $u_0$.
Repeating the proof of Theorem \ref{theorem/uniqueness of pathwise kinetic solutions for nonnegative initial data} with $\chi^2=0$ implies that
\begin{equation}
    \|u\|_{L^{\infty}([0,T];L^1(Q))}=\|u-0\|_{L^{\infty}([0,T];L^1(Q))}\leq\|u_0-0\|_{L^1(Q)}=\|u_0\|_{L^1(Q)}.
\end{equation}
Indeed, if $\chi^2=0$, in the final inequality \eqref{6.1} the error terms $IE_i^{\canc}$, defined in \eqref{4.8}, and $BE_i^{\canc Bj}$, for $j=1,2$, defined in \eqref{5.4}, are identically zero.
In particular, as observed in Remark \ref{remark/extension to signed data}, these errors are the only terms requiring the positivity of the initial data to tackle the regime $m\in(0,1)\cup(1,2]$.
Thus, when $u^2=0$, the proof works in the full regime $m\in(0,\infty)$ regardless of the sign of $u$.
\end{proof}

We conclude this section with a few auxiliary results.
The following proposition ensures that kinetic solutions with positive initial data stay positive, and it is crucial to obtain Proposition \ref{proposition/singular moments for defect measures d=1} below.

\begin{proposition}\label{proposition/positive data stay positive}
Let $u$ be a pathwise kinetic solution with nonnegative initial data $u_0\in L^2_+(Q)$.
Then we have $u(x,t)\geq0$ almost everywhere in $Q\times[0,\infty)$.
Moreover, for almost every $t\in[0,\infty)$ we have that $\|u(\cdot,t)\|_{L^1(Q)}=\|u_0\|_{L^1(Q)}$.
\end{proposition}
\begin{proof}
Suppose $u_0\in L^2_+(Q)$ and let $u$ be a pathwise kinetic solution to \eqref{formula/stochastic porous media equation} with initial data $u_0$, kinetic function $\chi$ and exceptional set $\mathcal{N}$.
Let $\mathcal{P}=\{0=t_0<t_1<\dots<t_N=T\}$ be an arbitrary partition of $[0,T]$ with $\mathcal{P}\subseteq[0,T]\setminus\mathcal{N}$.
By repeating the same argument leading from \eqref{1.1} to \eqref{1.3}, it follows that
\begin{align}\label{proof/proposition/positive data stay positive/1}
    \begin{aligned}
        \int_{Q\times\R}\!\!\!\!\!\chi_r(y,\eta)\sgn_-(\eta)\,dy\,d\eta\Big|_{r=0}^{r=T}
        \!\!= 
        \lim_{\beta\to0}\lim_{\epsilon\to0}\sum_{i=0}^{N-1}
        \int_{Q\times\R}\!\!\!\!\tilde{\chi}^{\epsilon}_{t_i,r}(y,\eta)\,\widetilde{(\sgn_-)}^{\epsilon}_{t_i,r}\phi_{\beta}(y)\,dy\,d\eta\Big|_{r=t_i}^{r=t_{i+1}},
    \end{aligned}
\end{align}
where $\sgn_-(\xi):=\sgn(\xi)\wedge 0$, and $\widetilde{(\sgn_-)}^{\epsilon}_{t_i,r}$ is defined as in \eqref{0.17}-\eqref{0.18} with $\sgn_-$ replacing $\sgn$.
As regards the second term in the sum, the same reasoning leading from \eqref{2.1} to \eqref{2.18}, with $\sgn_-(\xi')$ replacing $\sgn(\xi')$, show that
\begin{align}\label{p1}
\begin{aligned}
&\int_{Q\times\R}\!\!\!\!\!\!\!\tilde{\chi}^{\epsilon}_{t_i,r}(y,\eta)\,\widetilde{(\sgn_-)}^{\epsilon}_{t_i,r}\,\phi_{\beta}(y)\,dy\,d\eta\Big|_{r=t_i}^{r=t_{i+1}}
\\
&\,=
\widebar{IE}_i^{\sgn1,1}-\widebar{IE}_i^{\sgn1,2}
+
\widebar{BE}_i^{\sgn1,1}+\widebar{BE}_i^{\sgn1,2}-\widebar{BE}_i^{\sgn1,3}
\\
&\quad
+\!\!\int_{t_i}^{t_{i+1}}\!\!\!\int_{Q^3\times\R^3}\!\!\!\!\!\!\!\!\!\!\!\!\!\!\!m|\xi|^{m-1}\chi_r^1\rho_{t_i,r}^{1,\epsilon}\rho_{t_i,r}^{2,\epsilon}\sgn_-(\xi')\,\tr\!\left(\!\big(D_xY_{r,r-t_i}^{x,\xi}\big)^{\!T}\! D_y^2\phi_\beta(y)D_xY_{r,r-t_i}^{x,\xi}\!\right)\dsp dx \dsp dx' dy\dsp d\xi\dsp d\xi' d\eta\dsp dr
\\
&\quad
-2\int_{t_i}^{t_{i+1}}\!\!\!\int_{Q^3\times\R^2}\!\!\!\!\!\!\!\!(p_r^1+q_r^1)\,\rho_{t_i,r}^{1,\epsilon}\,\rho_{t_i,r}^{2,\epsilon}(x',y,0,\eta)\,\phi_\beta(y)\,dx\,dx'\,dy\,d\xi\,d\eta\,dr,
\end{aligned}
\end{align}
for error terms $\widebar{IE}_i^{\sgn 1,1}$, $\widebar{IE}_i^{\sgn 1,2}$, $\widebar{BE}_i^{\sgn 1,1}$, $\widebar{BE}_i^{\sgn 1,2}$ and $\widebar{BE}_i^{\sgn 1,3}$ defined simply by setting $u^1=u$ and replacing $\sgn(\xi')$ with $\sgn_-(\xi')$ in definition \eqref{2.6}, \eqref{2.16}, \eqref{2.9}, \eqref{2.11} and \eqref{2.14} respectively.
For the second to last term in \eqref{p1}, we first notice that
\begin{align}\label{p2}
\begin{aligned}
    \lim_{\epsilon\to0}&\int_{t_i}^{t_{i+1}}\!\!\!\int_{Q^3\times\R^3}\!\!\!\!\!\!\!\!\!\!\!\!\!m|\xi|^{m-1}\chi_r^1\rho_{t_i,r}^{1,\epsilon}\rho_{t_i,r}^{2,\epsilon}\sgn_-(\xi')\,\tr\!\left(\!\big(D_x\!Y_{r,r-t_i}^{x,\xi}\big)^{\!T}\! D_y^2\phi_\beta(y)D_x\!Y_{r,r-t_i}^{x,\xi}\!\right)\dsp dx \dsp dx' dy\dsp d\xi\dsp d\xi' d\eta\dsp dr
    \\
    &=
    \int_{t_i}^{t_{i+1}}\!\!\!\int_{Q\times\R}\!\!\!\!\!\!\!m|\xi|^{m-1}\chi_r^1\sgn_-(\xi)\,\tr\!\left(\!\big(D_x\!Y_{r,r-t_i}^{x,\xi}\big)^{\!T}\! D_y^2\phi_\beta(Y_{r,r-t_i}^{x,\xi})D_x\!Y_{r,r-t_i}^{x,\xi}\!\right)\dsp dx \dsp d\xi\dsp \dsp dr.
\end{aligned}
\end{align}
Then, observing that $|\xi|^{m-1}\chi_r^1\sgn_-(\xi)$ is always nonnegative and mimicking the arguments leading from \eqref{5.14} to \eqref{5.18}, we conclude that
\begin{align}\label{p3}
\begin{aligned}
    \int_{t_i}^{t_{i+1}}\!\!\!\int_{Q\times\R}\!\!\!\!\!\!\!m|\xi|^{m-1}\chi_r^1\sgn_-(\xi)\,\tr\!\left(\!\big(D_x\!Y_{r,r-t_i}^{x,\xi}\big)^{\!T}\! D_y^2\phi_\beta(Y_{r,r-t_i}^{x,\xi})D_x\!Y_{r,r-t_i}^{x,\xi}\!\right)\dsp dx \dsp d\xi\dsp \dsp dr
    \\
    \leq
    \widebar{BE}_i^{\canc A,2}+\widebar{BE}_i^{\canc A,3}+\widebar{BE}_i^{\canc A,4},
\end{aligned}
\end{align}
for error terms $\widebar{BE}_i^{\canc A,2}$, $\widebar{BE}_i^{\canc A,3}$ and $\widebar{BE}_i^{\canc A,4}$ defined simply by replacing $|\xi|^{m-1}|\chi_r^1-\chi_r^2|^2$ with $|\xi|^{m-1}\chi_r\sgn_-(\xi)$ in definition \eqref{5.15}, \eqref{5.17} and \eqref{5.18} respectively.

We now go back to \eqref{proof/proposition/positive data stay positive/1}.
Owing to the nonnegativity of the entropy and parabolic defect measure, we drop the last term in \eqref{p1}, and then we exploit \eqref{p2} combined with \eqref{p3} to obtain
\begin{align}\label{p4}
    \begin{aligned}
        \int_{Q\times\R}\!\!\!\!\!\chi_r(y,\eta)\sgn_-(\eta)\,dy\,d\eta\Big|_{r=0}^{r=T}
        \leq
        \limsup_{\beta\to0}\limsup_{\epsilon\to0}\sum_{i=0}^{N-1}
        \bigg(
        &
        \widebar{IE}_i^{\sgn1,1}-\widebar{IE}_i^{\sgn1,2}
        \\
        &+
        \widebar{BE}_i^{\sgn1,1}+\widebar{BE}_i^{\sgn1,2}-\widebar{BE}_i^{\sgn1,3}
        \\
        &+
        \widebar{BE}_i^{\canc A,2}+\widebar{BE}_i^{\canc A,3}+\widebar{BE}_i^{\canc A,4}
        \bigg).
    \end{aligned}
\end{align}
Arguments virtually identical to those in Step 8 ensure that the sum of the boundary terms $\widebar{BE}_i^{\sgn1,1}$, $\widebar{BE}_i^{\sgn1,2}$, $\widebar{BE}_i^{\sgn1,3}$, $\widebar{BE}_i^{\canc A,2}$, $\widebar{BE}_i^{\canc A,3}$ and $\widebar{BE}_i^{\canc A,4}$ vanishes in the limit as $\epsilon\to0$ first and $\beta\to0$ next.
Finally, recalling that the partition $\mathcal{P}\subseteq[0,T]\setminus \mathcal{N}$ is arbitrary, arguments completely analogous to those in Step 7 and Step 9 prove that the sum of the internal errors $\widebar{IE}_i^{\sgn1,1}$ and $\widebar{IE}_i^{\sgn1,2}$ goes to $0$ as we let $\epsilon\to0$ first, $\beta\to0$ next, and then also $|\mathcal{P}|\to0$.
Putting everything together, we conclude that
\begin{equation}
    0\leq\int_{Q\times\R}\chi(x,\xi,T)\sgn_-(\xi)\,dx\,d\xi \leq \int_{Q\times\R}\bar{\chi}(u_0(x),\xi)\sgn_-(\xi)\,dx\,d\xi\leq 0. \nonumber
\end{equation}
Here, the first inequality follows from the definition of the kinetic function and the last inequality follows from the nonnegativity of $u_0$.
Therefore we deduce that, if $u_0\in L^2_+(Q)$, then $u\geq0$ almost everywhere on $Q\times[0,\infty)$.

To prove the second assertion, we test the kinetic equation against the cutoff $\phi_\beta$, we then let $\beta\to0$ and we exploit the nonnegativity of the solution.
An approximation argument shows that we can take $\rho_0(x,\xi):=\phi_\beta(x)$ in equation \eqref{definition/pathwise kinetic solution/formula 1}. 
For any $t\in [0,\infty)\setminus\,\mathcal{N}$, applying the integration by parts formula \eqref{formula/integration by parts formula}, we obtain
\begin{align}
    \begin{aligned}\label{proof/proposition/positive data stay positive/2}
        \int_{Q\times\R}\chi(x,\xi,r)\phi_\beta(Y_{r,r}^{x,\xi})\,dx\hspace{0.08em}d\xi\Big|_{r=0}^{r=t}
        =
        &-\frac{2m}{m+1}\int_0^t\int_Q|u|^{\frac{m-1}{2}}\nabla u^{\left[\frac{m+1}{2}\right]}\nabla_y\phi_\beta(Y_{r,r}^{x,u})D_xY_{r,r}^{x,u}\,dx\hspace{0.08em}dr
        \\
        &
        -\int_0^t\int_{Q\times\R}(p_r+q_r)\nabla_y\phi_\beta(Y_{r,r}^{x,\xi})\partial_{\xi}Y_{r,r}^{x,\xi}\,dx\hspace{0.08em}d\xi\hspace{0.08em}dr.
    \end{aligned} 
\end{align}
The first term on the right-hand side of \eqref{proof/proposition/positive data stay positive/2} vanishes as we let $\beta\to0$.
Indeed, we exploit observation \eqref{9.11}, the Sobolev regularity $u^{\left[\frac{m+1}{2}\right]}\in H_0^1(Q)$ and the mean value theorem applied to points $x\in Q^\beta$, and the definition of the parabolic defect measure to estimate
\begin{align}
    \begin{aligned}\label{proof/proposition/positive data stay positive/3}
        \left|\int_0^t\!\!\int_Q\!\!\!|u|^{\frac{m-1}{2}}\nabla u^{\left[\frac{m+1}{2}\right]}\nabla_y\phi_\beta(Y_{r,r}^{x,u})\nabla_xY_{r,r}^{x,u}\!dx\hspace{0.05em}dr\right|\!
        &\leq
        C\beta^{-1}\int_0^t\!\int_{{Q^\beta}}\!\!\!|u|^{-1}\!\left|u^{\left[\frac{m+1}{2}\right]}\right| \left|\nabla u^{\left[\frac{m+1}{2}\right]}\right|dx\hspace{0.05em}dr
        \\
        &\leq
        C\int_0^t\int_{{Q^\beta}}|u|^{-1} \left|\nabla u^{\left[\frac{m+1}{2}\right]}\right|^2dx\hspace{0.08em}dr
        \\
        &=
        C\int_0^t\int_{{Q^\beta}\times\R}|\xi|^{-1}q(x,\xi,r)\,dx\hspace{0.08em}d\xi\hspace{0.08em}dr,
    \end{aligned} 
\end{align}
for a constant $C=C(Q,m,T,A,z)$.
Since we know that $u$ is nonnegative, Proposition \ref{proposition/singular moments for defect measures d=1} below and the dominated convergence theorem ensure that, as $\beta\to0$, the last line of \eqref{proof/proposition/positive data stay positive/3} vanishes.
Returning to \eqref{proof/proposition/positive data stay positive/2}, virtually the same estimate as \eqref{9.5}-\eqref{9.6} shows that the second term on the right-hand side of \eqref{proof/proposition/positive data stay positive/2} vanishes in the $\beta\to0$ limit.
Finally, the dominated convergence theorem implies that, as $\beta\to0$, the left-hand side of \eqref{proof/proposition/positive data stay positive/2} converges to $\int_{Q\times\R}\chi(x,\xi,r)\,dx\hspace{0.08em}d\xi\Big|_{r=0}^{r=t}=\|u(\cdot,t)\|_{L^1(Q)}-\|u_0\|_{L^1(Q)}$.
This completes the proof.
\end{proof}

Next, we present a result on the higher integrability of the parabolic and entropy defect measures in a neighbourhood of the origin.
In turn, this result will help us to improve the $H^1$-regularity of $u^{\left[\frac{m+1}{2}\right]}$ prescribed by the definition of pathwise kinetic solution, and obtain Sobolev regularity for $u^{[m]}$ in Lemma \ref{lemma/sobolev regularity of u^m}.
We shall need the following estimate, which follows immediately from Poincar\'e inequality.

\begin{lemma}\label{lemma/poincare inequality for u^(m+1/2)}
Let $v:Q\to\R$ be a measurable function such that $v^{[\nicefrac{m+1}{2}]}\in H_0^1(Q)$.
Then, for $C=C(Q)$, we have 
\begin{equation}
    \|v\|_{L^{m+1}(Q)}^{m+1}=\left\| v^{\left[\frac{m+1}{2}\right]}\right\|_{L^2(Q)}^2\leq C\left\|\nabla v^{\left[\frac{m+1}{2}\right]}\right\|_{L^2(Q)}^2. \nonumber
\end{equation}
\end{lemma}

\begin{proposition}\label{proposition/singular moments for defect measures d in (0,1)}
Let $u_0\in L^2(Q)$ and $\delta \in(0,1]$ be arbitrary.
Suppose that $u$ is a pathwise kinetic solution of \eqref{formula/stochastic porous media equation} with initial data $u_0$.
Then, for each $T>0$, for a constant $C=C(m,Q,T,A,z)$, 
\begin{align}\label{proposition/singular moments for defect measures d in (0,1)/formula 1}
\begin{aligned}
\hspace{-4mm}
    \|u\|_{L^{\infty}([0,T];L^{1+\delta}(Q))}^{1+\delta}\!\!
    +
    \!\delta\!\int_{0}^{T}\!\!\!\int_{Q\times\R}\!\!\!\!\!\!\!\!\!|\xi|^{\delta-1}\!\left(p+q\right)dx\dsp d\xi\dsp dr
    &
    \leq
    C\left(\|u_0\|_{L^{1+\delta}(Q)}^{1+\delta}+\|u_0\|_{L^{2}(Q)}^2\right).
\end{aligned}
\end{align}
\end{proposition}
\begin{proof}
Let $\delta\in(0,1]$ be arbitrary.
Let $u$ be a pathwise kinetic solution with initial data $u_0\in L^2(Q)$.
We shall write $\chi$ for its kinetic function, $p$ and $q$ for its entropy and parabolic defect measure respectively, and $\mathcal{N}\subseteq (0,\infty)$ for its null set.
The proof is based on an iterative argument along a suitably fine partition of $[0,T]$ and on the idea of formally testing the kinetic equation \eqref{definition/pathwise kinetic solution/formula 1} against the function $\R\ni\xi\mapsto\xi^{[\delta]}$.
Consider a partition $\mathcal{P}=\{t_0=0<t_1\dots<t_N=T\}$ with the constraints that $\mathcal{P}\subseteq [0,T]\setminus\mathcal{N}$ and that the diameter $|\mathcal{P}|=\max_i|t_{i+1}-t_i|$ is suitably small, as specified by conditions \eqref{d.16} and \eqref{d.20} below.
For each $\beta\in(0,1)$, consider the cutoff $\phi_\beta\in C_c^{\infty}(Q)$ introduced in \eqref{0.8}.
We stress that we consider this particular cutoff only to ease the referencing to follow, but its peculiar shape is not needed in this proof and any cutoff of scale $\beta$ would work.
According to Lemma \ref{lemma/velocity of characteristics comparable to initial data}, there exists a positive constant, which we denote by $C_T>1$, such that
\begin{equation}\label{d.1}
    C_T^{-1}|\xi|\leq\left|\Pi_{t,s}^{x,\xi}\right|\leq C_T|\xi|, \quad\text{for each $(x,\xi)\in\R^d\times\R$ and $s\leq t\in[0,T]$.}
\end{equation}
For each $\theta\in(0,1)$, we shall consider an approximation to the function $\R\ni\xi\to\xi^{[\delta]}$ with bounded derivatives.
Namely, we introduce the piecewise $C^1$ function $g_{\theta}:\R\to\R$ defined by
\begin{align}\label{d.2}
    g_{\theta}(\xi):=\int_0^{\xi}\Dot{g}_{\theta}(\eta)\,d\eta, \quad\text{ and } \quad
    \Dot{g}_{\theta}(\xi)=\empheqlbrace
    \begin{aligned}
    &\delta \left(C_T\theta\right)^{\delta-1}\quad\text{if } |\xi|\leq C_T\theta,
    \\
    &\delta|\xi|^{\delta-1}\quad\qquad\text{if }C_T\theta\leq|\xi|\leq C_T\theta^{-1},
    \\
    &0\quad\qquad\qquad\text{ else}.
    \end{aligned}
\end{align}
The following properties follow immediately
\begin{align}\label{d.3}
    0\leq\dot{g}_{\theta}(\xi)\leq\delta\left(C_T\theta\right)^{\delta-1}\,\,\text{for each }\theta\in(0,1),\text{ and }
    \Dot{g}_{\theta}(\xi)\uparrow\delta|\xi|^{\delta-1} \text{ as }\theta\to0,\,\, \forall\xi\in\R.
\end{align}
Similarly, we have
\begin{align}\label{d.4}
    |g_{\theta}(\xi)|\leq C_T^\delta\theta^{-\delta}\,\,\,\text{for each }\theta\in(0,1),\text{ and }\lim_{\theta\to0}g_{\theta}(\xi)=\xi^{[\delta]},\,\,\, \forall\xi\in\R.
\end{align}
Now, for each $i=0,\dots,N-1$ and each $\theta,\beta\in(0,1)$, we test the kinetic equation \eqref{definition/pathwise kinetic solution/formula 1} of $u$ against the transport along characteristics, started from time $t_i\in\mathcal{P}$, of the function $\phi_\beta(x)g_{\theta}(\xi)\in C(Q\times\R)$.
An approximation argument shows that $\phi_\beta(x)g_{\theta}(\xi)$ is indeed an admissible test function. 
For any $t\in[t_i,t_{i+1}]\setminus\mathcal{N}$, after using the integration by parts formula \eqref{formula/integration by parts formula}, the kinetic equation becomes
\begin{align}\label{d.5}
\begin{aligned}
            0=
            &
            \int_{Q\times\R}\!\!\!\!\!\!\chi_r(x,\xi)\phi_\beta(Y_{r,r-t_i}^{x,\xi})g_{\theta}(\Pi_{r,r-t_i}^{x,\xi})\,dx\,d\xi\bigg|_{r=t_i}^{r=t}
            \\
            &
            +
            \int_{t_i}^{t}\!\int_{Q\times\R}\!\!\!\!\!\!\phi_\beta(Y_{r,r-t_i}^{x,\xi})\dot{g}_{\theta}(\Pi_{r,r-t_i}^{x,\xi})\partial_{\xi}\Pi_{r,r-t_i}^{x,\xi}\left(p_r+q_r\right)\,dx\,d\xi\,dr
            +
            BE_i^{\text{vel}}
            \\
            &      
            +\frac{2m}{m+1}\int_{t_i}^{t}\!\int_{Q}\!\!\!|u|^{\frac{m-1}{2}}\nabla u^{\left[\frac{m+1}{2}\right]}\phi_\beta(Y_{r,r-t_i}^{x,u})\dot{g}_{\theta}(\Pi_{r,r-t_i}^{x,u})\nabla_x\Pi_{r,r-t_i}^{x,u}\,dx\,dr
            +
            BE_i^{\text{space}},
\end{aligned}
\end{align}
for the boundary velocity error relative to the interval $[t_i,t_{i+1}]$
\begin{align}\label{d.6}
\begin{aligned}
        BE_i^{\text{vel}}:=
        \int_{t_i}^{t}\!\int_{Q\times\R}\!\!\!\!\!\!\nabla_{\!y}\phi_\beta(Y_{r,r-t_i}^{x,\xi})\partial_{\xi}Y_{r,r-t_i}^{x,\xi}g_{\theta}(\Pi_{r,r-t_i}^{x,\xi})\left(p_r+q_r\right)\,dx\,d\xi\,dr,
\end{aligned}
\end{align}
and the boundary space error relative to the interval $[t_i,t_{i+1}]$
\begin{align}\label{d.7}
\begin{aligned}
        BE_i^{\text{space}}:=
        \frac{2m}{m+1}\int_{t_i}^{t}\!\int_{Q}\!\!\!|u|^{\frac{m-1}{2}}\nabla u^{\left[\frac{m+1}{2}\right]}\nabla_{\!y}\phi_\beta(Y_{r,r-t_i}^{x,u})D_x\!Y_{r,r-t_i}^{x,u}g_{\theta}(\Pi_{r,r-t_i}^{x,u})\,dx\,dr.
\end{aligned}
\end{align}
Now the idea is to let $\beta\to0$ first and $\theta\to0$ next, and to show that \eqref{d.5} yields the desired inequality \eqref{proposition/singular moments for defect measures d in (0,1)/formula 1} in the interval $[t_i,t_{i+1}]$.
Then, an iteration argument along the partition $\mathcal{P}$ completes the proof.

Let us first show that the boundary errors vanish in the $\beta\to0$ limit.
The boundary velocity error term is handled with the same arguments \eqref{9.5}-\eqref{9.6} as for the error terms $BE^{\sgn 1, 3}_i$ in the uniqueness proof.
Namely, using \eqref{9.11} with $k=1$, estimate \eqref{formula/DxiY near the boundary}, and the bound \eqref{d.4} for $g_{\theta}$ with $\theta$ fixed, we conclude that
\begin{align}\label{d.8bis}
\begin{aligned}
        \lim_{\beta\to0}\left|BE_i^{\text{vel}}\right|=0.
\end{aligned}
\end{align}

For the boundary space error, we start by splitting the integral according to the parameter $\theta\in(0,1)$:
\begin{align}\label{d.9}
\begin{aligned}
        BE_i^{\text{space}}
        =&
        \frac{2m}{m+1}\int_{t_i}^{t}\!\int_{Q\cap\{|u|\geq\theta\}}\!\!\!\!\!\!|u|^{\frac{m-1}{2}}\nabla u^{\left[\frac{m+1}{2}\right]}\nabla_{\!y}\phi_\beta(Y_{r,r-t_i}^{x,u})D_x\!Y_{r,r-t_i}^{x,u}g_{\theta}(\Pi_{r,r-t_i}^{x,u})\,dx\,dr
        \\
        &+
        \frac{2m}{m+1}\int_{t_i}^{t}\!\int_{Q\cap\{|u|<\theta\}}\!\!\!\!\!\!|u|^{\frac{m-1}{2}}\nabla u^{\left[\frac{m+1}{2}\right]}\nabla_{\!y}\phi_\beta(Y_{r,r-t_i}^{x,u})D_x\!Y_{r,r-t_i}^{x,u}g_{\theta}(\Pi_{r,r-t_i}^{x,u})\,dx\,dr.
\end{aligned}
\end{align}
For the first term in \eqref{d.9} we compute
\begin{align}\label{d.10}
\begin{aligned}
        \bigg| \frac{2m}{m+1}&\int_{t_i}^{t}\!\int_{Q\cap\{|u|\geq\theta\}}\!\!\!\!\!\!|u|^{\frac{m-1}{2}}\nabla u^{\left[\frac{m+1}{2}\right]}\nabla_{\!y}\phi_\beta(Y_{r,r-t_i}^{x,u})D_x\!Y_{r,r-t_i}^{x,u}g_{\theta}(\Pi_{r,r-t_i}^{x,u})\,dx\,dr\bigg|
        \\
        &\leq
        C \int_{t_i}^{t_{i+1}}\!\!\!\int_{Q^\beta\cap\{|u|\geq\theta\}}\!\!\!\!\!\!\theta^{-1}\left|u^{\left[\frac{m+1}{2}\right]}\right|\left|\nabla u^{\left[\frac{m+1}{2}\right]}\right|\beta^{-1}\,dx\,dr
        \\
        &\leq
        C \int_{t_i}^{t_{i+1}}\!\!\!\int_{Q^\beta\cap\{|u|\geq\theta\}}\!\!\left|\nabla u^{\left[\frac{m+1}{2}\right]}\right|^2\,dx\,dr,
\end{aligned}
\end{align}
for a constant $C=C(\theta,\delta,m,Q,T,A,z)$ independent of $\beta\in(0,1)$.
In the first passage we used \eqref{9.11}, \eqref{d.4} and $|u|^{\frac{m-1}{2}}\leq\theta^{-1}|u|^{\frac{m+1}{2}}$.
In the second passage we used the Sobolev regularity $u^{\left[\frac{m+1}{2}\right]}\in H_0^1(Q)$ and the mean value theorem applied to points $x\in Q^\beta$.
For the second term in \eqref{d.9}, we preliminarily observe that, for any $(x,\xi)\in Q\times[t_i,T]$,
\begin{equation}\label{d.11}
    \text{if } |u|\leq\theta,\text{ then }\left|g_{\theta}(\Pi_{r,r-t_i}^{x,u})\right|\leq\delta C_T^{\delta}\theta^{\delta-1}|u|.
\end{equation}
Indeed, when $|u|\leq\theta$, Lemma \ref{lemma/velocity of characteristics comparable to initial data} guarantees that $\left|\Pi_{r,r-t_i}^{x,u}\right|\leq C_T|u|\leq C_T\theta$, and in turn definition \eqref{d.2} implies that $\left|g_{\theta}(\Pi_{r,r-t_i}^{x,u})\right|\leq\delta C_T^{\delta-1}\theta^{\delta-1}\left|\Pi_{r,r-t_i}^{x,u}\right|$.
Then we calculate
\begin{align}\label{d.12}
\begin{aligned}
        \bigg| \frac{2m}{m+1}&\int_{t_i}^{t}\!\int_{Q\cap\{|u|<\theta\}}\!\!\!\!\!\!|u|^{\frac{m-1}{2}}\nabla u^{\left[\frac{m+1}{2}\right]}\nabla_{\!y}\phi_\beta(Y_{r,r-t_i}^{x,u})D_x\!Y_{r,r-t_i}^{x,u}g_{\theta}(\Pi_{r,r-t_i}^{x,u})\,dx\,dr\bigg|
        \\
        &\leq
        C \int_{t_i}^{t_{i+1}}\!\!\!\int_{Q^\beta\cap\{|u|<\theta\}}\!\!\left|u^{\left[\frac{m+1}{2}\right]}\right|\left|\nabla u^{\left[\frac{m+1}{2}\right]}\right|\beta^{-1}\,dx\,dr
        \\
        &\leq
        C \int_{t_i}^{t_{i+1}}\!\!\!\int_{Q^\beta\cap\{|u|<\theta\}}\!\!\left|\nabla u^{\left[\frac{m+1}{2}\right]}\right|^2\,dx\,dr,
\end{aligned}
\end{align}
for a constant $C=C(\theta,\delta,m,Q,T,A,z)$ independent of $\beta\in(0,1)$.
In the first passage we used observation \eqref{9.11} and observation \eqref{d.11}.
In the second passage we used $u^{\left[\frac{m+1}{2}\right]}\in H_0^1(Q)$ and the mean value theorem applied to points $x\in Q^\beta$.
In conclusion, combining the splitting \eqref{d.9} with \eqref{d.10} and \eqref{d.12}, the Sobolev regularity $u^{\left[\frac{m+1}{2}\right]}\in L^2([0,T];H_0^1(Q))$ and the dominated convergence theorem yield
\begin{equation}\label{d.13}
    \lim_{\beta\to0}\left|BE_i^{\space}\right|=0.
\end{equation}
This concludes the analysis of the error terms.

We now consider the remaining terms in \eqref{d.5}.
For the first term on the right-hand side of \eqref{d.5}, owing to the integrability $u\in L^{\infty}([0,T];L^2(Q))$ from Definition \ref{definition/pathwise kinetic solution}, properties of the kinetic function, definition \eqref{d.2} and the proportionality \eqref{d.1}, we apply the dominated convergence theorem and then use again formula \eqref{d.1} to get
\begin{align}\label{d.14}
    \begin{aligned}
        \lim_{\theta\to0}\lim_{\beta\to0}\int_{Q\times\R}\!\!\!\!\!\!\!\!\chi_r\,\phi_\beta(Y_{r,r-t_i}^{x,\xi})g_{\theta}(\Pi_{r,r-t_i}^{x,\xi})\,dx\,d\xi\bigg|_{r=t_i}^{r=t}\!\!\!
        &=
        \int_{Q\times\R}\!\!\!\!\!\!\chi_r\,(\Pi_{r,r-t_i}^{x,\xi})^{[\delta]}\,dx\,d\xi\bigg|_{r=t_i}^{r=t}
        \\
        &\geq
        C_T^{-1}\!\int_Q\!\!|u(x,t)|^{1+\delta}dx
        -
        C_T\int_Q\!\!|u(x,t_i)|^{1+\delta}dx.
    \end{aligned}
\end{align}

For the second term in \eqref{d.5}, we first make the following observation.
Namely, Proposition \ref{proposition/stability results for RDEs} and the equality $\partial_{\xi}\Pi_{r,0}^{x,\xi}\equiv1\,\,\forall(x,\xi,r)\in\R^d\times\R\times[0,\infty)$ imply there exists a suitably small value, which we denote by $t^*>0$, such that, for any $(x,\xi)\in\R^d\times\R$ and any $s\leq r\in[0,T]$,
\begin{equation}
    \label{d.15}
    \text{if }\, s\leq t^*,\, \text{ then }\,\,\partial_{\xi}\Pi_{r,s}^{x,\xi}\geq\frac{3}{4}.
\end{equation}
Therefore, if the partition $\mathcal{P}$ satisfies the condition
\begin{equation}
    \label{d.16}
    |\mathcal{P}|=\max_i|t_{i+1}-t_i|\leq t^*,
\end{equation}
then observation \eqref{d.15} and the nonnegativity of the terms integrated yield
\begin{align}\label{d.17}
\begin{aligned}
     \int_{t_i}^{t}\!\int_{Q\times\R}&\!\!\!\!\!\!\!\phi_\beta(Y_{r,r-t_i}^{x,\xi})\,\dot{g}_{\theta}(\Pi_{r,r-t_i}^{x,\xi})\,\partial_{\xi}\Pi_{r,r-t_i}^{x,\xi}\left(p_r\!+\!q_r\right)\,dx\dsp d\xi\dsp dr
     \\
     &
     \geq
     \frac{3}{4} \int_{t_i}^{t}\!\int_{Q\times\R}\!\!\!\!\!\!\!\!\phi_\beta(Y_{r,r-t_i}^{x,\xi})\,\dot{g}_{\theta}(\Pi_{r,r-t_i}^{x,\xi})\left(p_r\!+\!q_r\right)\,dx\dsp d\xi\dsp dr.
\end{aligned}
\end{align}

For the fourth term on the right-hand side of \eqref{d.5}, using \eqref{d.1} and \eqref{d.3}, we preliminarily notice that
\begin{equation}
    \label{d.18}
    0\leq\dot{g}_{\theta}(\Pi_{r,r-t_i}^{x,u})\leq\delta\, C_T^{\delta-1}\,|u|^{\delta-1}\,\,\,\,\forall(x,r)\in Q\times[t_i,T].
\end{equation}
Then we compute, for a constant $C^*=C^*(m,Q,T,A,z)$,
\begin{align}\label{d.19}
\begin{aligned}
\hspace{-4mm}
    \bigg|&\frac{2m}{m+1}\int_{t_i}^{t}\!\int_{Q}\!\!\!|u|^{\frac{m-1}{2}}\nabla u^{\left[\frac{m+1}{2}\right]}\phi_\beta(Y_{r,r-t_i}^{x,u})\dot{g}_{\theta}(\Pi_{r,r-t_i}^{x,u})\nabla_x\Pi_{r,r-t_i}^{x,u}\,dx\,dr\bigg|
    \\
    &\leq
    C^*\!\int_{t_i}^{t}\!\int_{Q}\!\!\!\!|u|^{\frac{m-1}{2}}\!\left|\nabla_x\Pi_{r,r-t_i}^{x,u}\!\right|\dot{g}_{\theta}(\Pi_{r,r-t_i}^{x,u})^{\frac{1}{2}}\phi_\beta(Y_{r,r-t_i}^{x,u})^{\frac{1}{2}}\left|\nabla u^{\left[\frac{m+1}{2}\right]}\right|\dot{g}_{\theta}(\Pi_{r,r-t_i}^{x,u})^{\frac{1}{2}}\phi_\beta(Y_{r,r-t_i}^{x,u})^{\frac{1}{2}}\,dx\,dr
    \\
    &\leq
    C^*|t_{i+1}-t_i|^\alpha\int_{t_i}^{t}\!\int_{Q}\!\!\!\!|u|^{\frac{m+\delta}{2}}\delta^{\frac{1}{2}}\phi_\beta(Y_{r,r-t_i}^{x,u})^{\frac{1}{2}}\left|\nabla u^{\left[\frac{m+1}{2}\right]}\right|\dot{g}_{\theta}(\Pi_{r,r-t_i}^{x,u})^{\frac{1}{2}}\phi_\beta(Y_{r,r-t_i}^{x,u})^{\frac{1}{2}}\,dx\,dr
    \\
    &\leq
    C^*|\mathcal{P}|^\alpha\left(
    \delta\int_{t_i}^{t}\!\int_{Q}\!\!\!\!|u|^{m+\delta}dx\,dr
    +
    \int_{t_i}^{t}\!\int_{Q\times\R}\!\!\!\!\!\!\dot{g}_{\theta}(\Pi_{r,r-t_i}^{x,\xi})\phi_\beta(Y_{r,r-t_i}^{x,\xi})q_r(x,\xi)\,dx\,dr
    \right)
    \\
    &\leq
    C^*|\mathcal{P}|^\alpha\left(
    \delta\left(\int_{t_i}^{t}\!\int_{Q\times\R}\!\!\!\!\!\!q_r\,dx\,d\xi\,dr\right)^{\frac{m+\delta}{m+1}}
    +
    \int_{t_i}^{t}\!\int_{Q\times\R}\!\!\!\!\!\!\dot{g}_{\theta}(\Pi_{r,r-t_i}^{x,\xi})\phi_\beta(Y_{r,r-t_i}^{x,\xi})q_r(x,\xi)\,dx\,dr
    \right).
\end{aligned}
\end{align}
In the second passage we used \eqref{d.18} and \eqref{lemma/velocity of characteristics comparable to initial data/formula 2}.
The third passage follows from H\"older's inequality and the definition of parabolic defect measure.
In the last passage we controlled the first term with H\"older's inequality and Lemma \ref{lemma/poincare inequality for u^(m+1/2)}.

Therefore, if the partition $\mathcal{P}$ satisfies the further condition, for the constant $C^*$ from \eqref{d.19},
\begin{equation}
    \label{d.20}
    C^*|\mathcal{P}|<\frac{1}{4},
\end{equation}
then, for the second and fourth term in \eqref{d.5}, combining \eqref{d.17} and \eqref{d.19}, the monotone convergence theorem, property \eqref{d.3} and the proportionality \eqref{d.1} yield
\begin{align}\label{d.21}
\begin{aligned}
    &\liminf_{\theta\to0}\liminf_{\beta\to0}\int_{t_i}^{t}\!\int_{Q\times\R}\!\!\!\!\!\!\phi_\beta(Y_{r,r-t_i}^{x,\xi})\dot{g}_{\theta}(\Pi_{r,r-t_i}^{x,\xi})\partial_{\xi}\Pi_{r,r-t_i}^{x,\xi}\left(p_r+q_r\right)\,dx\,d\xi\,dr
    \\
    &\qquad\qquad\qquad\quad
    +\frac{2m}{m+1}\int_{t_i}^{t}\!\int_{Q}\!\!\!|u|^{\frac{m-1}{2}}\nabla u^{\left[\frac{m+1}{2}\right]}\phi_\beta(Y_{r,r-t_i}^{x,u})\dot{g}_{\theta}(\Pi_{r,r-t_i}^{x,u})\nabla_x\Pi_{r,r-t_i}^{x,u}\,dx\,dr
    \\
    &
    \quad
    \geq
    \lim_{\theta\to0}\lim_{\beta\to0}\frac{2}{4}\int_{t_i}^{t}\!\int_{Q\times\R}\!\!\!\!\!\!\!\phi_\beta(Y_{r,r-t_i}^{x,\xi})\dot{g}_{\theta}(\Pi_{r,r-t_i}^{x,\xi})\left(p_r+q_r\right)\,dx\,d\xi\,dr
    -\frac{\delta}{4}\left(\int_{t_i}^{t}\!\int_{Q\times\R}\!\!\!\!\!\!q_r\,dx\,d\xi\,dr\right)^{\frac{m+\delta}{m+1}}
    \\
    &
    \quad
    \geq
    \frac{1}{2}\,\delta\, C_T^{\delta-1}\int_{t_i}^{t}\!\int_{Q\times\R}\!\!\!\!\!\!|\xi|^{\delta-1}\left(p_r+q_r\right)\,dx\,d\xi\,dr
    -\frac{\delta}{4}\left(\int_{t_i}^{t}\!\int_{Q\times\R}\!\!\!\!\!\!q_r\,dx\,d\xi\,dr\right)^{\frac{m+\delta}{m+1}}.
\end{aligned}   
\end{align}
We now conclude the proof with an iterative argument along the partition $\mathcal{P}$.
Let us first consider the case $\delta=1$.
Passing to the limit $\beta\to0$ first and $\theta\to0$ next in equation \eqref{d.5}, using \eqref{d.8bis}, \eqref{d.13}, \eqref{d.14}, and \eqref{d.21} with $\delta=1$, and recalling that $t\in[t_{i},t_{i+1}]\setminus\mathcal{N}$ is arbitrary, we obtain
\begin{equation}
    \label{d.22}
    C_T^{-1}\|u\|_{L^{\infty}([t_i,t_{i+1}];L^2(Q))}^2+\frac{1}{4}\int_{t_i}^{t_{i+1}}\!\!\!\int_{Q\times\R}\!\!\!\!\!\!\left(p_r+q_r\right)\,dx\,d\xi\,dr
    \leq
    C_T\|u(\cdot,t_i)\|_{L^2(Q)}^2.
\end{equation}
Then, arguing iteratively first in the interval $[0,t_1]$, then in $[t_1,t_2]$, and so forth up to $[t_{N-1},T]$, we conclude that
\begin{equation}
    \label{d.23}
    \|u\|_{L^{\infty}([0,T];L^2(Q))}^2+\int_{0}^{T}\!\!\!\int_{Q\times\R}\!\!\!\!\!\!\left(p_r+q_r\right)\,dx\,d\xi\,dr
    \leq
    C\|u_0\|_{L^2(Q)}^2,
\end{equation}
for a constant $C=C(m,Q,T,A,z)$ depending on the constant $C_T$ from \eqref{d.1} and on the cardinality of the partition $\mathcal{P}$, that is on the conditions \eqref{d.16} and \eqref{d.20}.

Finally, let us consider any arbitrary $\delta\in(0,1]$.
Taking again the limit $\beta\to0$ first and $\theta\to0$ next in equation \eqref{d.5}, using \eqref{d.8bis}, \eqref{d.13}, \eqref{d.14} and \eqref{d.21}, and recalling that $t\in[t_{i},t_{i+1}]\setminus\mathcal{N}$ is arbitrary, we obtain
\begin{align}
\begin{aligned}
    \label{d.24}
    C_T^{-1}\|u\|_{L^{\infty}([t_i,t_{i+1}];L^{1+\delta}(Q))}^{1+\delta}+\frac{\delta}{2}C_T^{\delta-1}\int_{t_i}^{t_{i+1}}\!\!\!\int_{Q\times\R}\!\!\!\!\!\!|\xi|^{\delta-1}\left(p_r+q_r\right)\,dx\,d\xi\,dr
    \\
    \leq
    C_T\|u(\cdot,t_i)\|_{L^{1+\delta}(Q)}^{1+\delta}+\frac{\delta}{4}\left(\int_{t_i}^{t_{i+1}}\!\!\!\int_{Q\times\R}\!\!\!\!\!\!q_r\,dx\,d\xi\,dr\right)^{\frac{m+\delta}{m+1}}.
\end{aligned}
\end{align}
In conclusion, arguing iteratively over the intervals $[t_i,t_{i+1}]$, and then using Young's inequality and the newly found estimate \eqref{d.23}, we deduce 
\begin{align}\label{d.25}
\begin{aligned}
\hspace{-4mm}
    \|u\|_{L^{\infty}([0,T];L^{1+\delta}(Q))}^{1+\delta}\!\!
    +
    \!\delta\!\int_{0}^{T}\!\!\!\int_{Q\times\R}\!\!\!\!\!\!\!\!\!|\xi|^{\delta-1}\!\!\left(p_r+q_r\right)dx\dsp d\xi\dsp dr
    &
    \leq
    C\!\left(\!\!\|u_0\|_{L^{1+\delta}(Q)}^{1+\delta}
    \!\!+\!
    \delta\!
    \left(\int_{0}^{T}\!\!\!\int_{Q\times\R}\!\!\!\!\!\!\!\!\!\!q_r\dsp dx\dsp d\xi\dsp dr\!\right)^{\frac{m+\delta}{m+1}}
    \right)
    \\
    &
    \leq
    C\left(\|u_0\|_{L^{1+\delta}(Q)}^{1+\delta}+\|u_0\|_{L^{2}(Q)}^2\right),
\end{aligned}
\end{align}
for a constant $C=C(m,Q,T,A,z)$ depending on $C_T$ from \eqref{d.1} and on the cardinality of $\mathcal{P}$.
\end{proof}

As anticipated, we now exploit the above proposition to prove that, if $u$ is a pathwise kinetic solution, the power $u^{[m]}$ lies in a suitable Sobolev space $W^{1,p_m}_0(Q)$, for an exponent $p_m$ depending on the diffusion regime $m$.
In particular, we can pass the vanishing boundary conditions at the level of $u^{[m]}$.
We used this property in the proof of Theorem \ref{theorem/uniqueness of pathwise kinetic solutions for nonnegative initial data} to handle the error terms coming from the cutoff procedure.
Namely, Proposition \ref{proposition/singular moments for defect measures d in (0,1)} and a straightforward modification of \cite[Lemma A.1]{fehrman-gess-Path-by-path-well-posedness-of-nonlinear-diffusion-equations-with-multiplicative-noise} prove the following.

\begin{lemma}\label{lemma/sobolev regularity of u^m}
Let $u_0\in L^2(Q)$ and suppose that $u$ is a pathwise kinetic solution of \eqref{formula/stochastic porous media equation} with initial data $u_0$.
Then, for $p_m:=(\frac{m+1}{m}\wedge 2)$, for any $T>0$ we have
\begin{align}\label{lemma/sobolev regularity of u^m/ formula 1}
    u^{[m]}\in L^{p_m}\left([0,T];W_0^{1,p_m}(Q)\right)\quad\text{with}\quad\nabla u^{[m]}=\frac{2m}{m+1}|u|^{\frac{m-1}{2}}\nabla u^{\left[\frac{m+1}{2}\right]}.
\end{align}
In particular $u^{[m]}$ has vanishing trace, and we have the estimate, for $C=C(m,Q,T)$,
\begin{align}\label{lemma/sobolev regularity of u^m/ formula 2}
    \|u^{[m]}\|_{L^{p_m}([0,T];W_0^{1,p_m}(Q))}\leq
    C\left(1+\|u_0\|_{L^2(Q)}^2\right).
\end{align}
\end{lemma}

Finally, we extend Proposition \ref{proposition/singular moments for defect measures d in (0,1)} to the case $\delta=0$ and establish a bound on the first singular moment of the defect measures.
This result is used in the proof of uniqueness for diffusion exponents $m\in(0,1)\cup(1,2]$.
Informally, it implies the local $L^2$-integrability of $\nabla u^{\frac{m}{2}}$.

\begin{remark}
We require the nonnegativity of the initial data, and indeed Proposition \ref{proposition/singular moments for defect measures d=1} is false for signed data.
Consider for simplicity the case $d=m=1$ and $A(x,\xi)=0$. Suppose that $u_0(x)=x$ in a neighbourhood of the origin.
Then, since the heat flow preserves the linear behaviour of the initial data locally in time, the failure of Proposition \ref{proposition/singular moments for defect measures d=1} manifests as the non-integrability of the map $\R\ni x\mapsto\nicefrac{1}{|x|}$ in a neighbourhood of the origin.
\end{remark}

\begin{proposition}\label{proposition/singular moments for defect measures d=1}
Suppose that $u_0\in L^2_+(Q)$ and let $u$ be a pathwise kinetic solution of \eqref{formula/stochastic porous media equation} with initial data $u_0$.
Then, for each $T>0$, there exists $C=C(m,Q,T,A,z)$ such that
\begin{align}\label{formula/singular moments for defect measures d=1}
\begin{aligned}
    \int_0^T\int_{\R\times Q}\!\!\!\!|\xi|^{-1} (p+q)\,dx\,d\xi\,dr
    \leq
    C\left(1+\|u_0\|_{L^2(Q)}^2\right).
\end{aligned}
\end{align}
\end{proposition}
\begin{proof}
The proof is essentially the same as for Proposition \ref{proposition/singular moments for defect measures d in (0,1)}, except that now $\delta=0$ and the idea is to formally test the kinetic equation against the function $\xi\mapsto\log(\xi)$.
We stress that in order for this to work, that is in order to find an approximation of $\xi\mapsto\log(\xi)$ whose derivative is bounded below and increases to the function $\xi\mapsto|\xi|^{-1}$, so as to justify the analogous of the monotone convergence passage \eqref{d.21}, it is crucial to consider positive values of $\xi$ only.
In other words, it is crucial that the kinetic solution $u$ is nonnegative, as ensured by Proposition \ref{proposition/positive data stay positive}.

We now sketch the relevant modifications.
With the notation from the proof of Proposition \ref{proposition/singular moments for defect measures d in (0,1)}, fix a partition $\mathcal{P}$ satisfying conditions \eqref{d.16} and \eqref{d.20}.
The approximation $g_{\theta}:\R\to\R$ is now defined by
\begin{align}\label{d0.1}
    g_{\theta}(\xi):=\int_1^{\xi}\Dot{g}_{\theta}(\eta)\,d\eta, \quad\text{ and }\quad
    \Dot{g}_{\theta}(\xi)=\empheqlbrace
    \begin{aligned}
    &C_T^{-1}\theta^{-1}\quad\text{if } 0\leq\xi\leq C_T\theta,
    \\
    &|\xi|^{-1}\quad\quad\,\,\text{if }C_T\theta\leq\xi\leq C_T\theta^{-1},
    \\
    &0\qquad\quad\,\,\,\,\text{ if } \xi\leq0 \text{ or } C_T\theta^{-1}<\xi.
    \end{aligned}
\end{align}
The properties \eqref{d.3} and \eqref{d.4} are replaced respectively by
\begin{align}\label{d0.2}
    0\leq\dot{g}_{\theta}(\xi)\leq C_T^{-1}\theta^{-1}\,\,\text{for each }\theta\in(0,1),\text{ and }
    \Dot{g}_{\theta}(\xi)\uparrow|\xi|^{-1} \text{ as }\theta\to0,\,\, \forall\xi\geq0,
\end{align}
and by
\begin{align}\label{d0.3}
    |g_{\theta}|\leq C\log(C_T\theta^{-1})\,\,\text{for each }\theta\in(0,1),\text{ and }\lim_{\theta\to0}g_{\theta}(\xi)=\log(\xi),\,\,\forall\xi\geq0.
\end{align}
Testing the kinetic equation against $\phi_\beta(x)g_{\theta}(\xi)\in C(Q\times\R)$ we still get equation \eqref{d.5}, with the only difference that $g_{\theta}$ is now defined by \eqref{d0.1}. 
The error term $BE_i^{\text{vel}}$, still defined by \eqref{d.6}, is handled exactly as before, simply using \eqref{d0.3} in place of \eqref{d.4}.
For the error term $BE_i^{\text{space}}$, still defined by \eqref{d.7}, we consider again the splitting \eqref{d.9}.
The first term in \eqref{d.9} is handled exactly as in \eqref{d.10}, simply using \eqref{d0.3} in place of \eqref{d.4}.
The only difference is that, for the second term in \eqref{d.9}, we now compute
\begin{align}\label{d0.4}
\begin{aligned}
        \bigg|&\frac{2m}{m+1}\int_{t_i}^{t}\!\int_{Q\cap\{|u|<\theta\}}\!\!\!\!\!\!|u|^{\frac{m-1}{2}}\nabla u^{\left[\frac{m+1}{2}\right]}\nabla_{\!y}\phi_\beta(Y_{r,r-t_i}^{x,u})D_x\!Y_{r,r-t_i}^{x,u}g_{\theta}(\Pi_{r,r-t_i}^{x,u})\,dx\,dr\bigg|
        \\
        &\leq
        C \int_{t_i}^{t_{i+1}}\!\!\!\int_{Q^\beta\cap\{|u|<\theta\}}\!\!\!\!\!\!|u|^{\frac{m-1}{m+1}}\left|u^{\left[\frac{m+1}{2}\right]}\right|\left|\nabla u^{\left[\frac{m+1}{2}\right]}\right|\beta^{-1}\,dx\,dr
        \\
        &\leq
        C \int_{t_i}^{t_{i+1}}\!\!\!\int_{Q^\beta\cap\{|u|<\theta\}}\!\!\!\!\!\!|u|^{\frac{m-1}{m+1}}\left|\nabla u^{\left[\frac{m+1}{2}\right]}\right|^2\,dx\,dr
        \\
        &\leq
        C \int_{t_i}^{t_{i+1}}\!\!\!\int_{Q^\beta\times\R\cap\{|\xi|<\theta\}}\!\!\!\!\!\!|\xi|^{\frac{m-1}{m+1}\wedge0}\,\theta^{\frac{m-1}{m+1}\vee0}\,q_r(x,\xi)\,dx\,d\xi\,dr,
\end{aligned}
\end{align}
for a constant $C=C(\theta,m,Q,T,A,z)$ independent of $\beta\in(0,1)$.
In the first passage we used observation \eqref{9.11} and \eqref{d0.3}.
In the second passage we used the Sobolev regularity $u^{\left[\frac{m+1}{2}\right]}\in H_0^1(Q)$ and the mean value theorem. 
The last passage follows from the definition of parabolic defect measure.
Now, formulas \eqref{d.9}, \eqref{d.10}, and \eqref{d0.4} coupled with Proposition \ref{proposition/singular moments for defect measures d in (0,1)}, if $m\in(0,1)$, or simply with the finiteness of the measure $q$, if $m\in[1,\infty)$, and the dominated convergence theorem yield  
\begin{equation}\label{d0.5}
    \lim_{\beta\to0}\left|BE_i^{\space}\right|=0.
\end{equation}

For the first term on the right-hand side of \eqref{d.5}, the integrability $u\in L^{\infty}([0,T];L^2(Q))$, the nonnegativity of $u$, which follows from Proposition \ref{proposition/positive data stay positive}, the integrability of $\xi\mapsto\log(\xi)$ near $0$ and its growth at infinity, properties of the kinetic function, the proportionality \eqref{d.1} and the dominated convergence theorem imply that
\begin{align}\label{d0.6}
    \begin{aligned}
        \lim_{\theta\to0}\lim_{\beta\to0}\int_{Q\times\R}\!\!\!\!\!\!\chi_r(x,\xi)\,\phi_\beta(Y_{r,r-t_i}^{x,\xi})g_{\theta}(\Pi_{r,r-t_i}^{x,\xi})\,dx\,d\xi\bigg|_{r=t_i}^{r=t_{i+1}}
        &=
        \int_{Q\times\R}\!\!\!\!\!\!\chi_r(x,\xi)\,\log(\Pi_{r,r-t_i}^{x,\xi})\,dx\,d\xi\bigg|_{r=t_i}^{r=t_{i+1}}
        \\
        &\leq
        C\left(1+\|u\|_{L^{\infty}([t_i,t_{i+1}];L^2(Q))}^2\right),
    \end{aligned}
\end{align}
for a constant $C=C(T,A,z)$.
The second and fourth term in \eqref{d.5} are handled with the same arguments as in \eqref{d.17}-\eqref{d.21}, except that $g_{\theta}$ is now given by \eqref{d0.1} and estimate \eqref{d.18} is replaced by
\begin{equation}
    \label{d0.7}
    0\leq\dot{g}_{\theta}(\Pi_{r,r-t_i}^{x,u})\leq C_T^{-1}|u|^{-1}\,\,\,\forall(x,r)\in Q\times[t_i,T].
\end{equation}
In turn, the final inequality becomes
\begin{align}\label{d0.8}
\begin{aligned}
    &\liminf_{\theta\to0}\liminf_{\beta\to0}\int_{t_i}^{t_{i+1}}\!\!\!\int_{Q\times\R}\!\!\!\!\!\!\phi_\beta(Y_{r,r-t_i}^{x,\xi})\,\dot{g}_{\theta}(\Pi_{r,r-t_i}^{x,\xi})\,\partial_{\xi}\Pi_{r,r-t_i}^{x,\xi}\left(p_r+q_r\right)\,dx\,d\xi\,dr
    \\
    &\qquad\qquad\qquad
    +\frac{2m}{m+1}\int_{t_i}^{t_{i+1}}\!\!\!\int_{Q}\!\!\!|u|^{\frac{m-1}{2}}\nabla u^{\left[\frac{m+1}{2}\right]}\phi_\beta(Y_{r,r-t_i}^{x,u})\dot{g}_{\theta}(\Pi_{r,r-t_i}^{x,u})\nabla_x\Pi_{r,r-t_i}^{x,u}\,dx\,dr
    \\
    &
    \quad
    \geq
    \frac{1}{2} C_T^{-1}\int_{t_i}^{t_{i+1}}\!\!\!\int_{Q\times\R}\!\!\!\!\!\!|\xi|^{-1}\left(p_r+q_r\right)\,dx\,d\xi\,dr
    -\frac{1}{4}\left(\int_{t_i}^{t_{i+1}}\!\!\!\int_{Q\times\R}\!\!\!\!\!\!q_r\,dx\,d\xi\,dr\right)^{\frac{m}{m+1}}.
\end{aligned}   
\end{align}
In conclusion, passing to the limit as $\beta\to0$ first and $\theta\to0$ then in \eqref{d.5}, using \eqref{d.8bis}, \eqref{d0.5}, \eqref{d0.6} and \eqref{d0.8}, we obtain that
\begin{align}\label{d0.9}
\begin{aligned}
    \int_{t_i}^{t_{i+1}}\!\!\!\int_{Q\times\R}\!\!\!\!\!\!|\xi|^{-1}\left(p_r+q_r\right)\,dx\,d\xi\,dr
    \leq
    C\left(
    1+\|u\|_{L^{\infty}([t_i,t_{i+1}]L^2(Q))}^2
    +
    \left(\int_{t_i}^{t_{i+1}}\!\!\!\int_{Q\times\R}\!\!\!\!\!\!q_r\,dx\,d\xi\,dr\right)^{\frac{m}{m+1}}
    \right),
\end{aligned}   
\end{align}
for a constant $C=C(T,A,z)$.
Iterating \eqref{d0.9} over the partition $\mathcal{P}$, using Young's inequality and estimate \eqref{d.23}, we obtain formula \eqref{formula/singular moments for defect measures d=1}.
\end{proof}



\section{Existence of pathwise kinetic solutions}\label{section/existence of pathwise kinetic solutions}

In this section we establish the existence of pathwise kinetic solutions to equation \eqref{formula/stochastic porous media equation}.
For this, we consider the setting outlined in Section \ref{section/definition and motivation of pathwise kinetic solutions} and derive stable estimates for the regularized equation \eqref{formula/regularized porous media equation 1}, defined for each $\eta\in(0,1)$ and $\epsilon\in(0,1)$.

Let $u_0\in L^2(Q)$ and let $\eta,\epsilon\in(0,1)$ be fixed but arbitrary.
Recall that $z^\epsilon:[0,\infty)\to\R^n$ are smooth paths converging to $z$ with respect to the $\alpha$-H\"older metric on the space of geometric rough paths $C^{0,\alpha}\big([0,T];G^{\left\lfloor\nicefrac{1}{\alpha}\right\rfloor}(\R^n)\big)$, for each $T>0$.
Let $\ueta$ be the solution of \eqref{formula/regularized porous media equation 1} from Proposition \ref{proposition/existence of classical solutions for smoothed equation} and let $\chieta(x,\xi,t):=\bar{\chi}(\ueta(x,t),\xi)$ be its kinetic function and $\peta$ and $\qeta$ the associated entropy and parabolic defect measures.
In Proposition \ref{proposition/weak kinetic formulation of the smoothed equation} we established the kinetic function is a distributional solution of equation \eqref{formula/strong kinetic formulation of the smoothed equation}.
In Corollary \ref{corollary/transported weak kinetic formulation of the smoothed equation} we got rid of the noise in the equation by the testing it against test functions transported along the smooth backward characteristics \eqref{formula/backward smooth characteristics} and showed the kinetic function solves equation \eqref{formula/transported weak kinetic formulation of the smoothed equation 1}.

The idea is now to establish stable estimates for the solutions $\ueta$ and the associated kinetic functions $\chieta$, that allow us to pass to the limit $\eta,\epsilon\to0$ and find a coherent limit for the functions $\ueta$ and the equations \eqref{formula/transported weak kinetic formulation of the smoothed equation 1}.

As in Section 4, consider for motivation to the kinetic formulation \eqref{formula/deterministic porous media equation kinetic formulation} of the deterministic porous media equation.
Following \cite{chen-perthame-Well-posedness-for-non-isotropic-degenerate-parabolic-hyperbolic-equations}, estimates for the solution are obtained by testing the equation against the maps $\R\ni \xi\mapsto\sgn(\xi)$ and $\R\ni\xi\mapsto\xi$.
In the first case, owing to the positivity of the parabolic and entropy defect measures, we informally get
\begin{equation}
    \|u\|_{L^{\infty}\left([0,\infty);L^1(Q)\right)}=\|\chi\|_{L^{\infty}\left([0,\infty);L^1(Q\times\R)\right)}\leq\|\chi(u_0,\xi)\|_{L^1(Q\times\R)}=\|u_0\|_{L^1(Q)}. \nonumber
\end{equation}
In the second case, we informally observe the estimate
\begin{equation}
   \frac{1}{2}\|u\|^2_{L^{\infty}([0,\infty);L^2(Q))}+\int_0^{\infty}\int_{\R\times Q} (p(x,\xi,r)+q(x,\xi,r))\,dx\,d\xi\,dr
   \leq
   \frac{1}{2}\|u_0\|_{L^2(Q)}^2. \nonumber
\end{equation}

\begin{remark}
In the discussion to follow, we will first establish estimates and the existence of pathwise kinetic solutions for initial data $u_0\in C^{\infty}_c(Q)$.
The general results will follow by density, repeating the arguments presented.
\end{remark}

In Proposition \ref{proposition/stable estimate for L^1 norm of smoothed solutions} we obtain the analogue of the $L^1$-estimate, and in Proposition \ref{proposition/stable estimate for L^2 norm and defect measures of smoothed solutions} we obtain the analogue of the $L^2$-estimate and the estimate for the defect measures.
The argument for Proposition \ref{proposition/stable estimate for L^1 norm of smoothed solutions} is just a small modification of the relevant details of Corollary \ref{corollary/L^1 estimate for patwhise kientic solutions}.
The proof of Proposition \ref{proposition/stable estimate for L^2 norm and defect measures of smoothed solutions} is essentially identical to that of Proposition \ref{proposition/singular moments for defect measures d in (0,1)}.
We therefore omit the details.

\begin{proposition}\label{proposition/stable estimate for L^1 norm of smoothed solutions}
For each $u_0\in L^2(Q)$, $\eta\in(0,1)$ and $\epsilon\in(0,1)$, the solution $\ueta$ of \eqref{formula/regularized  porous media equation 1} from Proposition \ref{proposition/existence of classical solutions for smoothed equation} satisfies
\begin{equation}
    \|\ueta\|_{L^{\infty}([0,\infty);L^1(Q))}\leq\|u_0\|_{L^1(Q)}. \nonumber
\end{equation}
\end{proposition}

\begin{proposition}\label{proposition/stable estimate for L^2 norm and defect measures of smoothed solutions}
For each $u_0\in L^2(Q)$, $\eta\in(0,1)$ and $\epsilon\in(0,1)$, let $\ueta$ be the solution of \eqref{formula/regularized  porous media equation 1} from Proposition \ref{proposition/existence of classical solutions for smoothed equation}.
Let $\delta\in (0,1)$.
For each $T>0$, there exists $C=C(m,Q,T,A,z)$ such that
\begin{align}
    \begin{aligned}
        \|\ueta\|_{L^{\infty}([0,T];L^2(Q))}^2\!\!+\!\int_0^T\!\!\int_{\R\times Q}\!\!\!\!\!\big(1+\delta|\xi|^{\delta-1}\big)& (\peta+\qeta)\,dx d\xi dr
       \leq
        C\left(1+\|u_0\|_{L^2(Q)}^2\right). \nonumber
    \end{aligned}
\end{align}
\end{proposition}

In general, we do not expect to obtain a stable estimate in time for the solutions $\ueta$.
However, we can obtain some regularity for the time derivative of the transported kinetic functions $\Tilde{\chi}^{\eta,\epsilon}$, defined by
\begin{equation}\label{formula/transported smoothed kinetic function}
    \Tilde{\chi}^{\eta,\epsilon}(x,\xi,t):=\chieta(X_{0,t}^{x,\xi,\epsilon},\Xi_{0,t}^{x,\xi,\epsilon},t)\quad\text{for}\quad (x,\xi,t)\in Q\times\R^d\times[0,\infty).
\end{equation}
In practice, the transport cancels the oscillations introduced by the noise.
The following proposition proves the functions $\partial_t\Tilde{\chi}^{\eta,\epsilon}$ are uniformly bounded in the negative Sobolev space $H^{-s}$, for $s$ big enough.

\begin{proposition}\label{proposition/stable estimate for time derivative of transported kinetic functions}
For each $u_0\in L^2(Q)$, $\eta\in(0,1)$ and $\epsilon\in(0,1)$, consider the transported kinetic function \eqref{formula/transported smoothed kinetic function}.
For each $T>0$, for any Sobolev exponent $s>\frac{d}{2}+1$, there exists $C=C(m,z,Q,T,s)$ such that
\begin{equation}
    \|\partial_t\Tilde{\chi}^{\eta,\epsilon}\|_{L^1([0,T];H^{-s}(Q\times\R))}\leq\left(1+\|u_0\|_{L^2(Q)}^2\right).\nonumber
\end{equation}
\end{proposition}
\begin{proof}
For any $\rho_0\in C_c^{\infty}(Q\times\R)$ denote $\rho_{0,r}^{\epsilon}(x,\xi):=\rho_0\left(Y_{r,r}^{x,\xi,\epsilon},\Pi_{r,r}^{x,\xi,\epsilon}\right)$.
Take any $\varphi\in C_c^{\infty}([0,\infty))$.
Testing equation \eqref{formula/weak kinetic formulation of the smoothed equation} against $\psi(x,\xi,t):=\rho_{0,r}^{\epsilon}(x,\xi)\varphi(t)$ and exploiting the cancellations coming from the transport of $\rho_0$ along the characteristics, for any $a\leq b\in[0,\infty)$, we get
\begin{align}
        \hspace{-2mm}\int_{a}^{b}\!\!\int_{Q\times\R}\!\!\!\!\!\chieta(x,\xi,t)\rho_{0,t}^{\epsilon}(x,\xi)\,dx\,d\xi\,\dot{\varphi}(t)\,dt\!=&\!\int_{Q\times\R}\!\!\!\!\!\!\!\!\chieta(x,\xi,t)\rho_{0,t}^{\epsilon}(x,\xi)\,dx\,d\xi\,\varphi(t)\Big|_{t=a}^{t=b}
         \nonumber
        \\
        &-\!\!
        \int_{a}^{b}\!\!\int_{Q\times\R}\!\!\!\left(m|\xi|^{m-1}+\eta\right)\chieta(x,\xi,t)\Delta_x\rho_{0,t}^{\epsilon}(x,\xi)\,dx\,d\xi\,\varphi(t)\,dt \nonumber
        \\
        &+
        \int_{a}^{b}\!\!\int_{Q\times\R}\!\!\!\!\left(\peta(x,\xi,t)+\qeta(x,\xi,t)\right)\partial_{\xi}\rho_{0,t}^{\epsilon}(x,\xi)\,dx\,d\xi\,\varphi(t)\,dt. \nonumber
\end{align}
Using the conservative property \eqref{formula/lebesgue measure is preserved by smooth caharacteristics} of the characteristics and integrating by parts the second term on the right-hand side of the equation, we obtain
\begin{align}
        \int_{a}^{b}\!\!\int_{Q\times\R}\!\!\!\!\!\Tilde{\chi}^{\eta,\epsilon}(x,\xi,t)\rho_0(x,\xi)\,dx\,d\xi\,\dot{\varphi}(t)\,dt\!
        =&\!\int_{Q\times\R}\!\!\!\!\!\!\!\!\chieta(x,\xi,t)\rho_{0,t}^{\epsilon}(x,\xi)\,dx\,d\xi\,\varphi(t)\Big|_{t=a}^{t=b} \nonumber
        \\
        &\!\!\!+\!\!
        \int_{a}^{b}\!\!\msp\int_{Q}\!\!\!\left(\!\text{\scalebox{0.85}{$\displaystyle\frac{2m}{m+1}\left|\ueta\right|^{\frac{m-1}{2}}\!\nabla\left(\ueta\right)^{\left[\frac{m+1}{2}\right]}\!+\eta\nabla\ueta$}}\!\msp\right)\!\!\nabla_{\!x}\rho_{0,t}^{\epsilon}\msp(x,\ueta)dx\dsp\varphi(t)\dsp dt \nonumber
        \\
        &\!\!\!+\!\!
        \int_{a}^{b}\!\!\int_{Q\times\R}\!\!\!\!\left(\peta(x,\xi,t)+\qeta(x,\xi,t)\right)\partial_{\xi}\rho_{0,t}^{\epsilon}(x,\xi)\,dx\,d\xi\,\varphi(t)\,dt. \nonumber
\end{align}
Since $\varphi\in C_c^{\infty}([0,\infty))$ is arbitrary, this shows that, for any $\rho_0\in C_c^{\infty}(Q\times\R)$, the mapping
\begin{equation}
    [t_0,\infty)\ni t\mapsto \int_{Q\times\R}\!\!\!\!\!\Tilde{\chi}^{\eta,\epsilon}(x,\xi,t)\rho_0(x,\xi)\,dx\,d\xi  \nonumber
\end{equation}
has a weak derivative given by 
\begin{align}
    \begin{aligned}
    [t_0,\infty)\ni t\mapsto 
        &-\int_{Q}\!\!\!\left(\text{\scalebox{0.9}{$\displaystyle\frac{2m}{m+1}\left|\ueta\right|^{\frac{m-1}{2}}\nabla\left(\ueta\right)^{\left[\frac{m+1}{2}\right]}+\eta\nabla\ueta$}}\right)\nabla_{\!x}\rho_{0,t}^{\epsilon}(x,\ueta)\,dx
        \\
        &\,\,-\int_{Q\times\R}\!\!\!\!\left(\peta(x,\xi,t)+\qeta(x,\xi,t)\right)\partial_{\xi}\rho_{0,t}^{\epsilon}(x,\xi)\,dx\,d\xi. \nonumber
    \end{aligned}
\end{align}

The embedding theorem for Sobolev spaces ensures that for any $s>\frac{d}{2}+1$, for $C=C(Q,s)$, $\|\rho\|_{W^{1,\infty}(Q\times\R)}\leq C\|\rho_0\|_{H^{s}(Q\times\R)}$.
Then we use this bound, Cauchy's inequality and the definition of the parabolic defect measure to estimate
\begin{align}
    \begin{aligned}
        \left|\langle\partial_t\Tilde{\chi}^{\eta,\epsilon},\rho_0\rangle\right|
        \leq
        C
        \left\{
        \eta+\int_{Q}\left|\ueta\right|^{(m-1)\vee0}\,dx
        +\int_{Q\times\R}\!\!\!\!\left(\peta+\left(1+|\xi|^{m-1\wedge0}\right)\qeta\right)dx\,d\xi
        \right\}
        \cdot
        \|\rho_0\|_{H^{s}(Q\times\R)}, \nonumber
    \end{aligned}
\end{align}
for $C=C(m,z,Q,s)$.
Finally, we use the density of $C_c(Q\times\R)$ in $H_0^s(Q\times\R)$ and then we integrate in time to get
\begin{align}
    \begin{aligned}
        \left\|\partial_t\Tilde{\chi}^{\eta,\epsilon}\right\|_{L^1([0,T];H^{-s}(Q\times\R))}
        &\leq
        C
        \int_0^T\left\{
        \eta+\int_{Q}\left|\ueta\right|^{(m-1)\vee0}\,dx
        +\int_{Q\times\R}\!\!\!\!\left(\peta+\left(1+|\xi|^{m-1\wedge0}\right)\qeta\right)dx\,d\xi
        \right\}dt
        \\
        &\leq
        C\left(1+\|u_0\|_{L^2(Q)}\right), \nonumber
    \end{aligned}
\end{align}
for $C=C(m,z,Q,T,s)$.
In the last inequality we used Lemma \ref{lemma/poincare inequality for u^(m+1/2)}, Proposition \ref{proposition/stable estimate for L^2 norm and defect measures of smoothed solutions}, and Young's and H\"older's inequalities.
\end{proof}

It remains to establish the regularity of the kinetic functions with respect to the spatial and velocity variables.
Actually, we shall consider the transported kinetic functions, since these are the ones we proved regularity in time for.
Straightforward modifications to \cite[Corollary 5.5]{fehrman-gess-Wellposedness-of-Nonlinear-Diffusion-Equations-with-Nonlinear-Conservative-Noise} prove the following estimate in the fractional Sobolev space $W^{\ell,1}(Q\times\R)$, for $\ell$ suitably small.

\begin{proposition}\label{proposition/stable estimate for transported kinetic functions}
For each $u_0\in L^2(Q)$, $\eta\in(0,1)$ and $\epsilon\in(0,1)$, consider the transported kinetic function \eqref{formula/transported smoothed kinetic function}.
For each $T>0$ and for each $\ell\in(0,\frac{2}{m+1}\wedge 1)$, there exists $C=C(m,Q,T,\ell)$ such that
\begin{equation}
    \|\Tilde{\chi}^{\eta,\epsilon}\|_{L^1([0,T];W^{\ell,1}(Q\times\R))}\leq C \left(1+\|u_0\|_{L^2(Q)}^2\right). \nonumber
\end{equation}
\end{proposition}

We can now establish the existence of pathwise kinetic solutions with initial data $u_0\in L^2(Q)$.
The proof is a consequence of Propositions \ref{proposition/stable estimate for L^2 norm and defect measures of smoothed solutions}, \ref{proposition/stable estimate for time derivative of transported kinetic functions} and \ref{proposition/stable estimate for transported kinetic functions}, and the Aubin-Lions-Simon Lemma \cite{simon-Compact-sets-in-the-spaceLp}.
We remark that the necessity of an entropy defect measure in Definition \ref{definition/pathwise kinetic solution} arises from the fact that, after passing to a subsequence, the gradients $\nabla(\ueta)^{[\nicefrac{m+1}{2}]}$ will converge only weakly in the $\epsilon,\eta\to0$ limit.
Due to the weak lower semicontinuity of the norm, the limit of the parabolic defect measures $\qeta$ may therefore overestimate the energy of the signed power of the limiting solution.
The total mass of the entropy defect measure quantifies this loss.

\begin{proof}[\textbf{Proof of Theorem \ref{theorem/existence of pathwise kinetic solutions for signed data}}]
\hfill

\noindent
Let $u_0\in L^2(Q)$ be arbitrary.
For any $\eta,\epsilon\in (0,1)$, let $\ueta$ be the solution of the regularized equation \eqref{formula/regularized porous media equation 1} with initial data $u_0$, with transported kinetic function $\Tilde{\chi}^{\eta,\epsilon}$, entropy defect measure $\peta$ and parabolic defect measure $\qeta$.
We recall that, for each $\ell\in(0,\frac{2}{m+1}\wedge 1)$ and each $R>0$, the embedding of $W^{\ell,1}(Q\times[-R,R])$ into $L^1(Q\times[-R,R])$ is compact, and that $L^1(Q\times[-R,R])$ embeds continuously into $H^{-s}(Q\times\R)$ for $s>\frac{d}{2}+1$.
Then, Propositions \ref{proposition/stable estimate for time derivative of transported kinetic functions} and \ref{proposition/stable estimate for transported kinetic functions}, the Aubin--Lions--Simon Lemma \cite{simon-Compact-sets-in-the-spaceLp} and a diagonal argument needed to control the tail of $Q\times\R$ show that, for each $T>0$, the family
\begin{equation}
    \{\Tilde{\chi}^{\eta,\epsilon}\}_{\eta,\epsilon\in(0,1)} \text{ is precompact in } L^1([0,T];L^1(Q\times\R)). \nonumber
\end{equation}
The conservative property of the characteristics \eqref{formula/lebesgue measure is preserved by smooth caharacteristics} then implies that, for each $T>0$, the family
\begin{equation}
    \{\chi^{\eta,\epsilon}\}_{\eta,\epsilon\in(0,1)} \text{ is precompact in } L^1([0,T];L^1(Q\times\R)). \nonumber
\end{equation}
In turn, the definition of kinetic function immediately shows that, for each $T>0$, the family
\begin{equation}
    \{\ueta\}_{\eta,\epsilon\in(0,1)} \text{ is precompact in } L^1([0,T];L^1(Q)). \nonumber
\end{equation}
Furthermore, Proposition \ref{proposition/stable estimate for L^2 norm and defect measures of smoothed solutions} and the Riesz--Markov Theorem imply that the sequence of measures
\begin{equation}
    \{(\peta,\qeta)\}_{\eta,\epsilon\in(0,1)} \text{ is weakly precompact in } C_c(Q\times\R\times[0,T])^*, \nonumber
\end{equation}
and that the family
\begin{equation}
    \{(\ueta)^{\left[\frac{m+1}{2}\right]}\}_{\eta,\epsilon\in(0,1)} \text{ is weakly precompact in } L^2([0,T];H^1_0(Q)). \nonumber
\end{equation}

After passing to a subsequence $\{(\eta_k,\epsilon_k)\}_{k\in\N}$ with $\lim_{k\to\infty}(\eta_k,\epsilon_k)=(0,0)$, it follows that there exists a function $u$ such that $u\in L^1([0,T];L^1(Q))$ for any $T>0$, and such that as $k\to\infty$
\begin{equation}\label{proof/theorem/existence of pathwise kinetic solutions for signed data/1}
    u^{\eta_k,\epsilon_k}\!\!\to u \quad \text{ strongly in } L^1([0,T];L^1(Q))
\end{equation}
and
\begin{equation}\label{proof/theorem/existence of pathwise kinetic solutions for signed data/2}
    (u^{\eta_k,\epsilon_k})^{\left[\nicefrac{m+1}{2}\right]}\!\!\rightharpoonup u^{\left[\nicefrac{m+1}{2}\right]} \,\,\,\text{weakly in } L^2([0,T];H_0^1(Q)).
\end{equation}
Furthermore, there exist positive measures $(p',q')$ such that as $k\to\infty$, for any $T>0$,
\begin{equation} \label{proof/theorem/existence of pathwise kinetic solutions for signed data/3}
    (\peta,\qeta)\rightharpoonup(p',q')\quad \text{weakly in } C_c(Q\times\R\times[0,T])^*.
\end{equation}
Recalling the definition \eqref{formula/parabolic defect measure smooth} of the parabolic defect measure, it follows from the strong convergence \eqref{proof/theorem/existence of pathwise kinetic solutions for signed data/1} and the weak lower semicontinuity of the Sobolev norm that, in the sense of measures,
\begin{equation} \label{proof/theorem/existence of pathwise kinetic solutions for signed data/4}
    \delta_0(\xi-u(x,t))\frac{4m}{(m+1)^2}\left|\nabla u^{\left[\frac{m+1}{2}\right]}(x,t)\right|^2\leq q'(x,\xi,t) \quad\text{for } (x,\xi,t)\in Q\times\R\times[0,\infty).
\end{equation}
To see this, let $\psi\in C_c(Q\times\R\times[0,T])$ be an arbitrary nonnegative function.
The strong convergence \eqref{proof/theorem/existence of pathwise kinetic solutions for signed data/1} implies that as $k\to\infty$
\begin{equation}
    \sqrt{\psi(x,u^{\eta_k,\epsilon_k}(x,t),t)}\,\to\sqrt{\psi(x,u(x,t),t)} \quad\text{strongly in } L^2(Q\times[0,T]). \nonumber
\end{equation}
In turn, this and the weak convergence \eqref{proof/theorem/existence of pathwise kinetic solutions for signed data/2} yield
\begin{equation}
    \sqrt{\psi(x,u^{\eta_k,\epsilon_k},t)}\,\nabla\left(u^{\eta_k,\epsilon_k}\right)^{\left[\frac{m+1}{2}\right]}\to\sqrt{\psi(x,u,t)}\,\nabla u^{\left[\frac{m+1}{2}\right]} \quad\text{weakly in } L^2(Q\times[0,T]). \nonumber
\end{equation}
Therefore, the weak convergence \eqref{proof/theorem/existence of pathwise kinetic solutions for signed data/3}, the definition of the measures $q^{\eta_k,\epsilon_k}$ and the weak lower semicontinuity of the $L^2$-norm prove that
\begin{align}
    \begin{aligned}
    \frac{4m}{(m+1)^2}\int_0^T\int_Q\psi(x,u,t)\left|\nabla u^{\left[\frac{m+1}{2}\right]}\right|^2
    & \leq
    \liminf_{k\to\infty}\frac{4m}{(m+1)^2}\int_0^T\int_Q\psi(x,u^{\eta_k,\epsilon_k},t)\left|\nabla \left(u^{\eta_k,\epsilon_k}\right)^{\left[\frac{m+1}{2}\right]}\right|^2
    \\
    & =
    \liminf_{k\to\infty}\int_0^T\int_Q\int_{\R}\psi(x,\xi,t) \, q^{\eta_k,\epsilon_k}
    \\
    &=
    \int_0^T\int_Q\int_{\R}\psi(x,\xi,t)\, q'.
    \end{aligned}
\end{align}
Since $\psi$ was arbitrary, this establishes the inequality \eqref{proof/theorem/existence of pathwise kinetic solutions for signed data/4}.

Now we define the parabolic defect measure by the usual formula \eqref{formula/parabolic defect measure}, and, since \eqref{proof/theorem/existence of pathwise kinetic solutions for signed data/4} implies that $q'-q$ is nonnegative, we define a positive entropy defect measure
\begin{equation}
    p:=p'+q'-q\geq 0\quad\text{on } Q\times\R\times[0,\infty).\nonumber
\end{equation}
Finally, the regularity assumptions \eqref{formula/assumption on A smooth}, the convergence \eqref{formula/smooth paths converging to rough path} of the paths $z^\epsilon$ and Proposition \ref{proposition/stability results for RDEs} implies that, for each $T>0$ and each $k=0,1,2$,
\begin{equation} \label{proof/theorem/existence of pathwise kinetic solutions for signed data/5}
    \lim_{\epsilon\to0}\left|D^k_{\!\!\text{\scalebox{0.7}{$(x,\xi)$}}}\!Y_{t_0,t}^{x,\xi,\epsilon}-D^k_{\!\!\text{\scalebox{0.7}{$(x,\xi)$}}}\!Y_{t_0,t}^{x,\xi}\right|+\left|D^k_{\!\!\text{\scalebox{0.7}{$(x,\xi)$}}}\Pi_{t_0,t}^{x,\xi,\epsilon}-D^k_{\!\!\text{\scalebox{0.7}{$(x,\xi)$}}}\Pi_{t_0,t}^{x,\xi}\right|=0,
\end{equation}
uniformly for $(x,\xi)\in \R^d\!\times\!\R$ and $t_0\!\leq\! t \in[0,T]$. 

For the kinetic function $\chi$ of $u$, the convergence \eqref{proof/theorem/existence of pathwise kinetic solutions for signed data/1} proves that, for a subset $\mathcal{N}\subset (0,\infty)$ of measure zero, for each $t\in[0,\infty)\setminus\,\mathcal{N}$,
\begin{equation}
\lim_{k\to\infty}\|u^{\eta_k,\epsilon_k}(\cdot,t)-u(\cdot,t)\|_{L^1(Q)}=0. \nonumber
\end{equation}
This and the additional convergences \eqref{proof/theorem/existence of pathwise kinetic solutions for signed data/2}, \eqref{proof/theorem/existence of pathwise kinetic solutions for signed data/3} and \eqref{proof/theorem/existence of pathwise kinetic solutions for signed data/5} imply that, for every $t_0\leq t_1\in[0,\infty)\setminus\,\mathcal{N}$ and every $\rho_0\in C_c^\infty(Q\times\R)$, we can pass to the limit in equation \eqref{formula/transported weak kinetic formulation of the smoothed equation 1} and get
\begin{align}
        \vspace{-2mm}
        \begin{aligned}
            \int_{Q\times\R}\!\!\!\!\!\!\chi(x,\xi,r)\rho_{t_0,r}(x,\xi)\,dx\,d\xi\bigg|_{r=t_0}^{r=t_1}
            =&
            \int_{t_0}^{t_1}\!\!\!\int_{Q\times\R}\!\!\!\!\!m|\xi|^{m-1}\chieta(x,\xi,r)\,\Delta_x\rho_{t_0,r}(x,\xi)\,dx\,d\xi\,dr
            \\
            &-
            \int_{t_0}^{t_1}\!\!\!\int_{Q\times\R}\!\!\!\!\!\!\left(p(x,\xi,r)+q(x,\xi,r)\right)\,\partial_{\xi}\rho_{t_0,r}(x,\xi)\,dx\,d\xi\,dr, \nonumber
    \end{aligned}
\end{align}
where $\rho_{t_0,t}$, given by \eqref{formula/transported stochastic test function}, is the solution of \eqref{formula/stochastic underlying transport equation} with initial data $\rho_0$.
Moreover, when $t_0=0$,
\begin{equation}
    \int_{Q\times\R}\!\!\!\!\!\!\chi(x,\xi,0)\,\rho_{0,0}(x,\xi)\,dx\,d\xi
    =\lim_{k\to\infty}\int_{Q\times\R}\!\!\!\!\!\!\chi^{\eta_k,\epsilon_k}(x,\xi,0)\,\rho_{0,0}(x,\xi)\,dx\,d\xi
    =\int_{Q\times\R}\!\!\!\!\!\!\bar{\chi}(u_0(x),\xi)\,\rho_0(x,\xi)\,dx\,d\xi. \nonumber
\end{equation}
This completes the proof that $u$ is a pathwise kinetic solution.
\end{proof}

Finally, we show that pathwise kinetic solutions depend continuously on the driving noise.
The proof will follow from a compactness argument relying on the estimates used in the proof of Theorem \ref{theorem/existence of pathwise kinetic solutions for signed data}, the rough path estimates of Proposition \ref{proposition/stability results for RDEs}, and the uniqueness of solutions from Theorem \ref{theorem/uniqueness of pathwise kinetic solutions for nonnegative initial data}.
Unfortunately, we remark that these methods do not yield an explicit estimate quantifying the convergence of the solutions in terms of the convergence of the noise.

\begin{proof}[\textbf{Proof of Theorem \ref{theorem/continuous dependence on the noise for nonnegative initial data}}]
\hfill

\noindent
Let $u_0\in L_+^2(Q)$ and $T>0$.
Let $\{z^n\}_{n\in\N}$ and $z$ be $\alpha$-H\"older continuous geometric rough paths on $[0,T]$ satisfying 
\begin{equation}
    \lim_{n\to\infty}d_\alpha(z^n,z)=0. \nonumber
\end{equation}
This ensures that we can find $R_0\geq 0$ such that condition \eqref{formula/paths z^epsilon are uniformly close to zero path} from Section A holds.

For each $n\in\N$, let $u^n$ denote the solution of \eqref{formula/stochastic porous media equation} constructed in Theorem \ref{theorem/existence of pathwise kinetic solutions for signed data} with initial data $u_0$ and driving signal $z^n$ respectively.
It follows from \eqref{formula/paths z^epsilon are uniformly close to zero path} and Proposition \ref{proposition/stability results for RDEs} that the solutions $u^n$ satisfy the estimates of Propositions \ref{proposition/stable estimate for L^1 norm of smoothed solutions}, \ref{proposition/stable estimate for L^2 norm and defect measures of smoothed solutions}, \ref{proposition/stable estimate for time derivative of transported kinetic functions} and \ref{proposition/stable estimate for transported kinetic functions} on the interval $[0,T]$ for a constant that is independent of $n\in\N$.

A repetition of the proof of Theorem \ref{theorem/existence of pathwise kinetic solutions for signed data} proves that, after passing to a subsequence $\{u^{n_k}\}_{k\in\N}$, there exists a pathwise kinetic solution $u$ of \eqref{formula/stochastic porous media equation} with initial condition $u_0$ and driving noise $z$ such that
\begin{equation}
    \lim_{k\to\infty}\|u^{n_k}-u\|_{L^1([0,T];L^1(Q))}=0.\nonumber
\end{equation}
However, since it follows from Theorem \ref{theorem/uniqueness of pathwise kinetic solutions for nonnegative initial data} that $u$ is the unique solution of \eqref{formula/stochastic porous media equation} with initial condition $u_0$ and driving noise $z$, we conclude that, along the full sequence $\displaystyle\lim_{n\to\infty}\|u^{n}-u\|_{L^1([0,T];L^1(Q))}=0$.
\end{proof}


\appendix

\section{Rough path estimates}\label{section/rough path estimates}

In this section we present some stability results for rough differential equations.
Then we apply these results to the systems of characteristics \eqref{formula/forward smooth characteristics} and \eqref{formula/forward stochastic characteristics} to obtain several properties and estimates needed throughout the paper.
We refer the reader to Friz and Hairer \cite{fritz-hairer-a-course-on-rough-apths} and Friz and Victoir \cite{friz-victoir-multidimensiona-stochastic-processes-as-rough-paths} for detailed expositions of the theory of rough paths, originally introduced by Lyons \cite{Lyons-differential-equations-driven-by-rough-signals}.

For $d\geq 1$, $T\geq0$ and $\beta\in(0,1)$, we denote by $C^{0,\beta}\big([0,T];G^{\left\lfloor\nicefrac{1}{\beta}\right\rfloor}(\R^n)\big)$ the space of $\beta$-H\"older continuous geometric rough paths, and by $d_{\beta}$ the associated $\beta$-H\"older metric defined on this space (cf. \cite[Definition 7.41]{friz-victoir-multidimensiona-stochastic-processes-as-rough-paths}).
For each $x\in \R^d$ and $z\in C^{0,\beta}\big([0,T];G^{\left\lfloor\nicefrac{1}{\beta}\right\rfloor}(\R^n)\big)$, let $X^{x,z}$ be the solution of the rough differential equation
\begin{align}[left ={\empheqlbrace}]\label{formula/typical RDE}
    \begin{aligned}
        &dX^{x,z}_t= V(X_t^{x,z})\circ dz_t &\text{on }(0,\infty),
        \\
        &X_0^{x,z}=x.&
    \end{aligned}
\end{align}
The collection \eqref{formula/typical RDE} defines a flow map $\psi^z:\R^d\times[0,T]\to\R^d$ by the rule
\begin{equation}
    \psi_t^z(x)=X^{x,z}_t\quad\text{for } (x,t)\in\R^d\times[0,T].\nonumber
\end{equation}
The next proposition encodes the regularity of the flow map with respect to the initial condition and the driving signal. 
The regularity is inherited from the nonlinearity $V$, which must be sufficiently regular to overcome the roughness of the noise.
A proof of the proposition is given in \cite[Lemma 13]{crisan-diehl-friz-oberhauser}.
In the statement below, we shall write $e=1\oplus0\oplus\dots\oplus0$ to denote the signature of the zero path.
\begin{proposition}\label{proposition/stability results for RDEs}
Fix $T\geq0$, $\beta\in(0,1)$, $\gamma>\frac{1}{\beta}\geq1$, and $k\in \N$.
Assume $V\in C^{k+\gamma}_b(\R^d;\R^{d\times n})$.
For any $R\geq 0$ there exist constants $C=C(R,\|V\|_{C_b^{k+\gamma}})>0$ and $M=M(R,\|V\|_{C_b^{k+\gamma}})>0$ such that, for any $z^1,z^2\in C^{0,\beta}\big([0,T];G^{\left\lfloor\nicefrac{1}{\beta}\right\rfloor}(\R^n)\big)$ with
\begin{equation}
    d_{\beta}(z^i,e)\leq R\quad\text{for $i=1,2$}, \nonumber
\end{equation}
the following properties hold.
Here $\|\cdot\|_{\beta}$ is the standard H\"older norm in $C^{\beta}([0,T];\R^N)$ for some $N\in\N$.
\begin{itemize}
    \item [i)] For every $0\leq j\leq k$,
        \begin{equation}
            \sup_{x\in\R^d}\,\left\|D_x^j\Big(\psi_t^{z^{1}}\!-\!\psi_t^{z^{2}}\Big)(x)\right\|_{\beta}+\left\|D_x^j\Big(\big(\psi_t^{z^{1}}\big)^{-1}\!\!-\!\big(\psi_t^{z^2}\big)^{-1}\Big)(x)\right\|_{\beta}\leq C d_{\beta}(z^1,z^2).
        \end{equation}
    \item [ii)] For every $0\leq j\leq k$,
        \begin{equation}
            \sup_{x\in\R^d}\,\left\|D_x^j\Big(\psi_t^{z^{1}}\Big)(x)\right\|_{\beta}+\left\|D_x^j\Big(\big(\psi_t^{z^{1}}\big)^{-1}\Big)(x)\right\|_{\beta}\leq M.
        \end{equation}
\end{itemize}
\end{proposition}

Now we consider the setting outlined in Sections \ref{section/preliminaries and main results} and \ref{section/definition and motivation of pathwise kinetic solutions}, and apply this regularity result to the rough differential equations \eqref{formula/forward smooth characteristics} and \eqref{formula/forward stochastic characteristics} defining the characteristics.

\begin{remark}
In the setting of Section \ref{section/preliminaries and main results} and \ref{section/definition and motivation of pathwise kinetic solutions}, the assumption that, for each $T>0$, the smooth paths $z^{\epsilon}$ converge to $z$ in the metric $d_{\alpha}$ on $C^{0,\alpha}\big([0,T];G^{\left\lfloor\nicefrac{1}{\alpha}\right\rfloor}(\R^n)\big)$, ensures that we can find $R_0\geq 0$ such that
\begin{equation}\label{formula/paths z^epsilon are uniformly close to zero path}
    d_{\alpha}(z,e)+\sup_{\epsilon\in(0,1)}d_{\alpha}(z^{\epsilon},e)\leq R_0. 
\end{equation}
Therefore, all the consequences of Proposition \ref{proposition/stability results for RDEs} holds uniformly for $z^{\epsilon}$, for any $\epsilon\in(0,1)$, and for $z$, with the same constants.
That is, they hold uniformly for the smooth characteristics \eqref{formula/forward smooth characteristics}, for any $\epsilon\in(0,1)$, and for the rough characteristics \eqref{formula/forward stochastic characteristics}.
In the remainder of the section we shall consider the path $z$ and the system \eqref{formula/forward stochastic characteristics}. 
We remark that the exact same arguments work for $z^{\epsilon}$ and the system \eqref{formula/forward smooth characteristics}, just inserting $\epsilon$ where needed.
\end{remark}

First we present a lemma which asserts that the velocity characteristics are locally in time comparable to their initial condition.
The proof is given in \cite[Lemma B.2]{fehrman-gess-Wellposedness-of-Nonlinear-Diffusion-Equations-with-Nonlinear-Conservative-Noise}.

\begin{lemma}\label{lemma/velocity of characteristics comparable to initial data}
For each $T>0$ there exists $C=C(T,A,R_0)\geq 1$ such that, for each $(x,\xi)\in\R^d\times\R$ and $t\leq t_0\in[0,T]$,
\begin{equation}\label{lemma/velocity of characteristics comparable to initial data/formula 1}
    C^{-1}|\xi|\leq\left|\Pi_{t_0,t}^{x,\xi}\right|\leq C|\xi|. \nonumber
\end{equation}
Furthermore, there exists $C=C(T,A,R_0)>0$ such that, for each $(x,\xi)\in\R^d\times\R$ and $t\leq t_0\in[0,T]$, 
\begin{equation}\label{lemma/velocity of characteristics comparable to initial data/formula 2}
   \left|\nabla_{\!x}\Pi_{t_0,t}^{x,\xi}\right|\leq C \,t^{\alpha}(|\xi|\wedge 1). 
\end{equation}
\end{lemma}

\medskip
We now exploit Proposition \ref{proposition/stability results for RDEs} to study the continuity of the characteristics with respect to the initial data.
First, recalling \eqref{formula/inverse relation for smooth characteristics}, we observe that for each $(x,\xi),(x',\xi')\in\R^d\times\R$, each $T\geq0$ and each $s\leq t_0\in [0,T]$,
\begin{align}
    \begin{aligned}
    |x-x'|=&\left|X_{t_0-s,t_0}^{Y^{x,\xi}_{t_0,s},\Pi^{x,\xi}_{t_0,s}}-X_{t_0-s,t_0}^{Y^{x',\xi'}_{t_0,s},\Pi^{x',\xi'}_{t_0,s}}\right|
    \\
    \leq&
    \sup_{(y,\eta)\in\R^d\times\R}\left|D_{\!y}X_{t_0-s,t_0}^{y,\eta}\right|\left|Y^{x,\xi}_{t_0,s}-Y^{x',\xi'}_{t_0,s}\right|
    +\sup_{(y,\eta)\in\R^d\times\R}\left|\partial_{\eta}X_{t_0-s,t_0}^{y,\eta}\right|\left|\Pi^{x,\xi}_{t_0,s}-\Pi^{x',\xi'}_{t_0,s}\right|.
    \end{aligned}\nonumber
\end{align}
An identical estimate holds for $|\xi-\xi'|$.
Therefore, assumption \eqref{formula/assumption on A smooth} and Proposition \ref{proposition/stability results for RDEs} with $k=1$ imply that, for a constant $C=C(T,A,R_0)$,
\begin{equation}\label{formula/distance between points in terms of distance between characteristics}
    |x-x'|+|\xi-\xi'|\leq C\left(\left|Y^{x,\xi}_{t_0,s}-Y^{x',\xi'}_{t_0,s}\right|+\left|\Pi^{x,\xi}_{t_0,s}-\Pi^{x',\xi'}_{t_0,s}\right|\right).
\end{equation}

In turn, we can use this bound to estimate the difference between derivatives of the characteristics starting from distinct points in terms of the characteristics themselves.
First, notice that the equalities $D^2_{\!(y,\eta)}Y_{t_0,0}^{y,\eta}=0$ and $D^2_{\!(y,\eta)}\Pi_{t_0,0}^{y,\eta}=0$, which follow immediately from the initial conditions, and Proposition \ref{proposition/stability results for RDEs} with $k=2$ imply that, for $C=C(T,A,R_0)$,
\begin{equation}\label{formula/bound D^2Y and D^2Pi}
    \sup_{x\in\R^d,\xi\in\R}\left|D^2_{(x,\xi)}Y^{x,\xi}_{t_0,s}\right|+\left|D^2_{(x,\xi)}\Pi^{x,\xi}_{t_0,s}\right|
    \leq
    C|s|^{\alpha}\quad \forall s\leq t_0\in[0,T].
\end{equation}
Then, for the same constant $C$, we compute
\begin{align}
    \begin{aligned}
    |D_{\!x}Y_{t_0,s}^{x,\xi}-D_{\!x'}Y_{t_0,s}^{x',\xi'}|
    &\leq
    \sup_{y\in\R^d,\eta\in\R}\left(\left|D^2_{\!y}Y_{t_0,s}^{y,\eta}\right|+\left|\partial_{\eta}D_{\!y}Y_{t_0,s}^{y,\eta}\right|\right)(|x-x'|+|\xi-\xi'|)
    \\
    &
    \leq
    C|s|^{\alpha}(|x-x'|+|\xi-\xi'|)
    \\
    &
    \leq
    C|s|^{\alpha}\left(\left|Y^{x,\xi}_{t_0,s}\!\!-\!Y^{x',\xi'}_{t_0,s}\right|\!+\!\left|\Pi^{x,\xi}_{t_0,s}\!\!-\!\Pi^{x',\xi'}_{t_0,s}\right|\right).
    \end{aligned}
\end{align}
Identical computations holds for the other derivatives of $Y^{x,\xi}_{t_0,s}$ and $\Pi_{t_0,s}^{x,\xi}$ and we conclude that, for $C=C(T,A,R_0)$,
\begin{align}\label{formula/distance DY and DPi in terms of distance between characteristics}
    |D_{\!x}Y_{t_0,s}^{x,\xi}-D_{\!x'}Y_{t_0,s}^{x',\xi'}|
    +
    |\partial_{\xi}Y_{t_0,s}^{x,\xi}-\partial_{\xi'}Y_{t_0,s}^{x',\xi'}|
    +
    |\nabla_{\!x}\Pi_{t_0,s}^{x,\xi}-\nabla_{\!x'}\Pi_{t_0,s}^{x',\xi'}|
    +
    |\partial_{\xi}\Pi_{t_0,s}^{x,\xi}-\partial_{\xi'}\Pi_{t_0,s}^{x',\xi'}|
    \\
    \leq
    C|s|^{\alpha}\left(\left|Y^{x,\xi}_{t_0,s}\!\!-\!Y^{x',\xi'}_{t_0,s}\right|\!+\!\left|\Pi^{x,\xi}_{t_0,s}\!\!-\!\Pi^{x',\xi'}_{t_0,s}\right|\right).
\end{align}

We conclude this section by examining the consequences of the conditions $\partial_{\xi}A_{|\partial Q\times\R}\equiv0$ and $D_x\partial_{\xi}A_{|\partial Q\times\R}\equiv0$.
Then we exploit the continuity properties from Proposition \ref{proposition/stability results for RDEs} to obtain information on the behaviour of the space characteristics near the boundary.
These estimates are crucial in the proof of Theorem \ref{theorem/uniqueness of pathwise kinetic solutions for nonnegative initial data} to tackle the boundary error terms coming from the introduction of the cutoff and the transport along characteristics.

The assumption $\partial_{\xi}A(x,\xi)_{|\partial Q\times\R}\equiv 0$, combined with the first line of \eqref{formula/forward stochastic characteristics} and \eqref{formula/backward stochastic characteristics}, ensures that space characteristics starting from the boundary do not move. That is, for each $t_0\geq0$,
\begin{equation}\label{formula/space characteristics stand still on the boundary}
   \text{if \, $(x,\xi)\in\partial Q\times\R$,\, then}\,\,X^{x,\xi}_{t_0,t}=Y^{x,\xi}_{t_0,s}=x\,\,\, \text{for all $t\geq 0$ and all $s\in[0,t_0]$}. 
\end{equation}
The uniqueness of solutions then guarantees that space characteristics cannot cross the boundary.
Thus, when starting from $x\in Q$ they never leave the domain.
That is, for every $t_0\geq0$, 
\begin{equation} 
    \text{if \,\,$(x,\xi)\in Q\times\R$, \, then }\,X^{x,\xi}_{t_0,t}\in Q\,\,\,\text{for all $t\geq 0$, \, and }\,\, Y^{x,\xi}_{t_0,s}\in Q\,\,\, \text{for all $s\in[0,t_0]$}. 
\end{equation}
In fact, owing to the smoothness hypothesis \eqref{formula/assumption on A smooth} and Proposition \ref{proposition/stability results for RDEs} with $k=1$, more is true: the closer to $\partial Q$ is the space initial data $x\in \R^d$, the slower the associated space characteristic moves.
Rigorously, given any $x\in\R^d$, let $x^*\in\partial Q$ be such that $|x-x^*|=\dist(x,\partial Q)$.
Then, for any $\xi\in\R$, any $T\geq0$ and any $s\leq t_0\in[0,T]$, for $C=C(T,A,R_0)$, we compute
\begin{align}\label{formula/characteristics move slower the closer they get to the boundary}
\begin{aligned}
    \left|Y_{t_0,s}^{x,\xi}-x\right|\leq\left|Y_{t_0,s}^{x,\xi}-x^*\right|+\left|x^*-x\right|
    &=
    \left|Y_{t_0,s}^{x,\xi}-Y_{t_0,s}^{x^*,\xi}\right|+\left|x^*-x\right|
    \\
    &\leq
    \sup_{(y,\eta)\in\R^d\times\R}\!\!\left|D_{\!x}Y^{y,\eta}_{t_0,s}\right||x-x^*|\,+|x-x^*|
    \\
    &\leq
    C\,|x-x^*|=C\,\dist(x,\partial Q).
\end{aligned}
\end{align}
Here we used $Y_{t_0,s}^{x^*\xi}\equiv x^*$ from \eqref{formula/space characteristics stand still on the boundary}.
An identical computation holds for $X^{x,\xi}_{t_0,t}$.

The constancy of the space characteristics along the boundary has consequences on their velocity and space derivatives.
Indeed, the smoothness of characteristics combined with \eqref{formula/space characteristics stand still on the boundary} implies that the $\xi$-derivative of space characteristics vanishes on the boundary.
Namely, for each $t_0\in[0,\infty)$,
\begin{equation}\label{formula/velocity derivative of space characteristics vanishes on the boundary}
    \text{if $(x,\xi)\in\partial Q\times\R$,\, then}\,\,\partial_{\xi}X^{x,\xi}_{t_0,t}=\partial_{\xi}Y^{x,\xi}_{t_0,s}=0\quad\text{for all $t\geq 0$ and all $s\in[0,t_0]$.}
\end{equation}

Moreover, Proposition \ref{proposition/stability results for RDEs} allows us to derive \eqref{formula/backward stochastic characteristics} with respect to $x$, and to show that, for any $(x_0,\xi_0)\in\R^d\times\R$, the derivative $D_{\!x}Y_{t_0,s}^{x_0,\xi_0}$ solves the rough differential equation
\begin{align}[left ={\empheqlbrace}]\label{formula/RDE for x-derivative of stochastic backward space characteristics}
    \begin{aligned}
        &dD_{\!x}Y_{t_0,s}^{x_0,\xi_0}\!\!=\!\!-\!\left(\!D_{\!x}\partial_{\xi}A(Y_{t_0,s}^{x_0,\xi_0},\Pi_{t_0,s}^{x_0,\xi_0}) D_{\!x}Y_{t_0,s}^{x_0,\xi_0}\!\!+\!\partial^2_{\xi}A(Y_{t_0,s}^{x_0,\xi_0},\Pi_{t_0,s}^{x_0,\xi_0})\nabla_{\!x}\Pi_{t_0,s}^{x_0,\xi_0}\!\right)\!\circ\! dz_{t_0,s}&\!\!\!\text{in } (0,t_0),
        \\
        &D_{\!x}Y_{t_0,0}^{x_0,\xi_0}=\id_d.
    \end{aligned}
\end{align}
When $x_0\in\partial Q$, we have $Y_{t_0,s}^{x_0,\xi_0}= x_0$ for all $s\in[0,t_0]$.
Then, using $\partial_{\xi}A_{|\partial Q\times\R}\equiv 0$, which trivially implies also $\partial^2_{\xi}A_{|\partial Q\times\R}\equiv 0$, and the further assumption $D_x\partial_{\xi}A_{|\partial Q\times\R}\equiv 0$ from \eqref{formula/assumption on A xi derivative zero on the boundary}, the previous equation simply reduces to
\begin{align}
    dD_{\!x}Y_{t_0,s}^{x_0,\xi_0}\!\!=0\,\,\text{in } (0,t_0),\quad D_{\!x}Y_{t_0,0}^{x_0,\xi_0}=\id_d. \nonumber
\end{align}
We conclude that, for each $t_0\geq0$, 
\begin{equation}\label{formula/x-derivative of space characteristics on the boundary}
    \text{if $(x_0,\xi_0)\in\partial Q\times\R$,\, then}\,\,D_xY_{t_0,s}^{x_0,\xi_0}\equiv \id_d\quad\text{for every } s\in[0,t_0].
\end{equation}

Finally, we use these results and the stability properties from Proposition \ref{proposition/stability results for RDEs} with $k=2$ to obtain estimates on the derivatives of the space characteristics near the boundary.
Given any $x\in\R^d$, let $x^*\in\partial Q$ be such that $|x-x^*|=\dist(x,\partial Q)$.
Then, using \eqref{formula/bound D^2Y and D^2Pi} and \eqref{formula/velocity derivative of space characteristics vanishes on the boundary}, for any $\xi\in\R$, any $T\geq0$ and any $s\leq t_0\in[0,T]$, we compute for the the $\xi$-derivative
\begin{align}\label{formula/DxiY near the boundary}
\begin{aligned}
    \left|\partial_{\xi}Y_{t_0,s}^{x,\xi}\right|=\left|\partial_{\xi}Y_{t_0,s}^{x,\xi}-\partial_{\xi}Y_{t_0,s}^{x^*,\xi}\right|
    &\leq
    \sup_{y\in\R^d}\left|D_{\!x}\partial_{\xi}Y^{y,\xi}_{t_0,s}\right||x-x^*|
    \\
     &\leq
    C\,|s|^{\alpha}\dist(x,\partial Q),
\end{aligned}
\end{align}
for a constant $C=C(T,A,R_0)$.
As regards the $x$-derivative, the same argument using \eqref{formula/bound D^2Y and D^2Pi} and \eqref{formula/x-derivative of space characteristics on the boundary} implies that, for $C=C(T,A,R_0)$,
\begin{align}\label{formula/DxY - id near the boundary}
\begin{aligned}
    \left|D_{\!x}Y_{t_0,s}^{x,\xi}-\id_d\right|
    =
    \left|D_xY_{t_0,s}^{x,\xi}-D_xY_{t_0,s}^{x^*,\xi}\right|
    \leq
    C\,|s|^{\alpha}\dist(x,\partial Q).
\end{aligned}
\end{align}


\section*{Acknowledgments}
The author has no competing interests.
This research has been supported by the EPSRC Centre for Doctoral Training in Mathematics of Random Systems: Analysis, Modelling and Simulation (EP/S023925/1).
The author would like to thank his PhD advisors Benjamin Fehrman (with the support of EPSRC Early Career Fellowship EP/V027824/1) and Jos\'e Carrillo for the unbelievable support, both mathematical and personal.
This work would have not been possible without them.


\bibliographystyle{alpha}

\end{document}